\documentclass[10.5pt, one sided]{amsart}
\RequirePackage[colorlinks,citecolor=blue,urlcolor=blue]{hyperref}

\usepackage{latexsym,amssymb,amsmath,amsfonts,amsthm}
\usepackage{amsthm,amsmath}
\usepackage{float}
\usepackage{enumerate}
\usepackage{scalerel}[2016/12/29]

\usepackage[margin=2.5cm]{geometry}
\usepackage{bbm}
\usepackage{subfigure}
\usepackage{graphicx}
\usepackage[toc,page]{appendix}
\usepackage{multirow}
\usepackage{amsfonts}
\usepackage[all]{xy}
\usepackage{caption}
 \usepackage{geometry}                
\usepackage{graphicx}
\usepackage{amssymb}
\usepackage{epstopdf}
\usepackage{color}
\usepackage{bbm, dsfont}
\usepackage{hyperref}
\usepackage[utf8]{inputenc}
\usepackage{wasysym}
\usepackage{comment}

\usepackage{wrapfig}

\usepackage{amssymb,amsthm,amsmath,amssymb,dsfont}
\usepackage[dvipsnames]{xcolor}
\definecolor{myred}{RGB}{251,154,133}
\definecolor{myblue}{RGB}{153,206,227}
\definecolor{mylightblue}{RGB}{0, 150, 255}
\definecolor{mygreen}{RGB}{32, 210, 64}
\definecolor{mygray}{RGB}{220, 220, 220}

\usepackage{tikz}
\usetikzlibrary{decorations.pathmorphing}
\tikzset{snake it/.style={decorate, decoration=snake}}
\usetikzlibrary{shapes.geometric,positioning,decorations.pathreplacing} 
\usepackage{pgfplots}

\newtheorem{theorem}{Theorem}
\newtheorem{definition}{Definition}[section]
\newtheorem{lemma}{Lemma}[section]
\newtheorem{remark}{Remark}[section]

\newtheorem{corollary}[lemma]{Corollary}

\def\beq{ \begin{equation} }
\def\eeq{ \end{equation} }

\def\square{\vcenter{\vbox{\hrule height .4pt
  \hbox{\vrule width .4pt height 5pt \kern 5pt
        \vrule width .4pt} \hrule height .4pt}}}

\newcommand{\smallnearrow}{\scaleobj{0.6}{\nearrow}}
\newcommand{\smallsearrow}{\scaleobj{0.6}{\searrow}}
\newcommand{\smallrightarrow}{\scaleobj{0.6}{\rightarrow}}

\newcommand{\bae}{\begin{equation}\begin{aligned}}
\newcommand{\eae}{\end{aligned}\end{equation}}

\DeclareFontFamily{OML}{rsfs}{\skewchar\font'177}
\DeclareFontShape{OML}{rsfs}{m}{n}{ <5> <6> rsfs5 <7> <8> <9>
rsfs7 <10> <10.95> <12> <14.4> <17.28> <20.74> <24.88> rsfs10 }{}
\DeclareMathAlphabet{\mathfs}{OML}{rsfs}{m}{n}

\newcommand{\note}[1]{{\color{red}{ \bf{ [Note: #1]}}}}

\begin{document}

\title{The Double Bubble Problem in the Hexagonal Norm}

\author{Parker Duncan}
\address[Parker Duncan]{Faculty of Data and Decision Sciences, Technion - Israel Institute of Technology}
\email{parkerduncan@campus.technion.ac.il}

\author{Rory O'Dwyer}
\address[Rory O'Dwyer]{Department of Physics, Stanford University}
\email{rodwyer@stanford.edu}

\author{Eviatar B. Procaccia}
\address[Eviatar B. Procaccia]{Faculty of Data and Decision Sciences, Technion - Israel Institute of Technology}
\urladdr{https://procaccia.net.technion.ac.il}
\email{eviatarp@technion.ac.il}

\maketitle
\begin{abstract}

We study the double bubble problem where the perimeter is taken with respect to the hexagonal norm, i.e. the norm whose unit circle in $\mathbb{R}^2$ is the regular hexagon. We provide an elementary proof for the existence of minimizing sets for volume ratio parameter $\alpha\in (0,1]$ by arguing that any minimizer must belong to a small family of parameterized sets. This family is further simplified by showing that $60^{\circ}$ angles are not optimal as well as other geometric exclusions. We then provide a  minimizer for all $\alpha\in(0,1]$ except at a single point, for which we find two minimizing configurations.

\end{abstract}



\section{Introduction}

The Double Bubble problem asks what configuration or configurations of two sets of given volumes minimize their surface area; this is a natural extension of the classical isoperimetric problem. In the case of the Euclidean norm on $\mathbb{R}^2$ and $\mathbb{R}^3$, it has been shown in \cite{foisy1993standard} and \cite{hutchings2002proof} that the minimizing geometry consists of three spherical caps that intersect at an angle of $\frac{2\pi}{3}$. This was extended in \cite{reichardt2007proof} to dimensions $n\geq4$. More recently, it was shown by Milman and Neeman in \cite{milman2022gaussian} that the same solution holds when measuring volume and perimeter with respect to the Gaussian measure. Other natural variants on the classical double bubble problem in Euclidean space have also been solved. These include the solution to the problem in the $\ell^1$ norm on the plane in \cite{morgan1998wulff}; more recently an elementary solution was given to this problem in \cite{duncan2023elementary}. A generalization to this problem, in which the interface between the two sets is given a potentially different weight than the non-joint boundary, was recently solved in \cite{friedrich2023characterization}. 

In this paper we study the double bubble problem with respect to the hexagonal norm, which is to say the norm whose unit ball is the regular hexagon. The hexagonal norm is a natural norm to study in problems related to the classical double bubble problem. In \cite{hales2001honeycomb}, it was shown that partitioning the plane into regions of equal-area using the regular hexagonal honeycomb tiling minimizes, in some sense, the perimeter of the tiling. For a large, but finite number of equal area regions, hexagons also appear. It has been shown in \cite{Cox_2013} that for a large number of sets, the structure that minimizes perimeter consists of many interior hexagons, while for the sets on the periphery it is best to round some of the edges. 

The hexagonal norm comes from a family of norms which presents many interesting features. In this family, there is a norm that gives as its unit ball the square, the regular hexagon, the regular octagon, and in general any regular $2n$-gon ($n>1$). Thus, as the number of sides increases, the balls in these norms converge, in the Hausdorff sense, to the Euclidean ball. We would, therefore, suppose that the solution to the double bubble problem also converges to the solution in the Euclidean norm. Likely, the most difficult part of this problem would be to show that there is a unique solution to the double bubble problem for each norm for $n$ large enough. Note that, in this paper, we show that uniqueness fails for at least one volume ratio.  

As mentioned, in the Euclidean norm all three caps in the solution meet at $120^{\circ}$ angles. It arises naturally in the hexagonal norm that the solution also consists of three paths that meet at $120^{\circ}$ angles. Indeed, a great part of the difficulty presented by the problem in this paper is to eliminate the possibility of $60^{\circ}$ angles in the solution's shape and to show optimality of $120^{\circ}$ angles.

To account for the possibility that the double bubble minimizing shapes might contain interior $60^{\circ}$ angles, we form a nice family of configurations that depends on a finite number of parameters that represent side lengths, and that we hope contains a shape that is a solution. The plan was to follow the argument in \cite{duncan2023elementary} and use the KKT method to optimize the length of each of the sides in this family. This is a powerful optimization method and is presented in \cite{Karush2014}. However, the number of parameters is too great to be able to apply the KKT method. Therefore, we considered the optimization problem of each set individually given a fixed joint boundary\footnote{Inspiration for this came from the story of Princess Dido and the founding of Carthage, made famous by Virgil in \cite{virgil2009aeneid}}. That is to say, we fixed the joint boundary of the two sets and then optimized the side lengths of the non-joint boundary. This so significantly reduced the number of parameters required that the KKT method was no longer necessary, and we were able to use the methods of elementary calculus instead. 

Other double bubble work has been done in the discrete setting, which presents its own challenges. In relation to this paper, natural questions to ask might be about the structure of the double bubble minimizing configurations on the hexagonal lattice, and how much larger is the double bubble perimeter of these configurations than their continuous case counterparts. Such question have been asked in  \cite{duncan2023discrete} and \cite{friedrich2021double}. In \cite{duncan2023discrete}, it was shown that the discrete solution to the double bubble problem in the $\ell^1$ norm is at most two more than the ceiling function of the continuous case solution. While in \cite{friedrich2021double}, some of the features of double bubble minimizers on the lattice are described, and the case when both sets have equal volume is thoroughly analysed. Furthermore, in this latter result, as in \cite{friedrich2023characterization}, the weight given to the interface is allowed to vary.

\section{Notation and Results}

For any Lebesgue-measurable set $A\subset\mathbb{R}^2$, let $\mu(A)$ be its Lebesgue measure.  Further define the hexagonal norm on $\mathbb{R}^2$ to be: 

$$\mathcal{D}\left(x,y\right)=\max\left\{\left|x\right|+\frac{1}{\sqrt{3}}\left|y\right|,\frac{2}{\sqrt{3}}\left|y\right|\right\},\text{ }(x,y)\in\mathbb{R}^2.$$

We will see later that the unit circle to this norm is the regular hexagon as is the solution to the isoperimetric problem. Thus, the norm is self dual by \cite{taylor1975unique}. 

For a simple curve $\lambda:[a,b]\rightarrow\mathbb{R}^2$, not necessarily closed, where $\lambda(t)=(x(t),y(t))$, define its length $\rho$ with respect to $\mathcal{D}$ by 
$$\rho(\lambda)=\sup_{N\geq1}\sup_{a\leq t_1\leq...\leq t_N\leq b}\sum_{i=1}^N\mathcal{D}\Big(x(t_{i+1})-x(t_{i}),y(t_{i+1})-y(t_{i})\Big).$$

If we wish to measure only a portion of the curve $\lambda$, it will be denoted $\rho(\lambda
([t,t']))$, where $[t,t']\subset[a,b]$.  For simplicity we assume that $[a,b]=[0,1]$ unless otherwise stated.

We say that two curves $\lambda,\lambda':[0,1]\rightarrow\mathbb{R}^2$ intersect nontrivially if there are intervals $[s,s'],[t,t']\subset[0,1]$ such that $\lambda([s,s'])=\lambda'([t,t'])$; their nontrivial intersection can be written as the union of curves $\lambda_i$ such that $\lambda_i([0,1])=\lambda([s_i,s_{i+1}])\rightarrow\mathbb{R}^2$ for some intervals $[s_i,s_{i+1}]$, and we define the length of the nontrivial intersection to be $\rho(\lambda\cap\lambda'):=\sum_i\rho(\lambda_i)$.  \\

Here we are interested in the double bubble perimeter of two simply connected open sets $A,B\subset\mathbb{R}^2$ where the boundary of $A$, denoted as $\partial A$, is a closed, simple, rectifiable curve, and similarly for $B$, and where the intersection of the boundaries of $A$ and $B$ is a union of disjoint, rectifiable curves.  The double bubble perimeter is defined as $$\rho_{\text{DB}}(A,B)=\rho(\lambda)+\rho(\lambda')-\rho(\lambda\cap\lambda'),$$ where $\partial A=\lambda([0,1])$, and $\partial B=\lambda'([0,1])$.  We will also use the notation $\rho(\lambda)=\rho(\partial A)$.

For $\alpha\in(0,1]$, define:
$\gamma_{\alpha}=\{(A,B):A,B\subset\mathbb{R}^2$, where $ A,B$ are disjoint, simply connected open sets, and $\partial A$,$\partial B$,$\partial A\cap \partial B$ are unions of closed, continuous, simple, rectifiable curves, with $\mu(A)=1,\mu(B)=\alpha\}$.  

Let $$\rho_{\text{DB}}(\alpha):=\inf\{\rho_{\text{DB}}(A,B):(A,B)\in\gamma_{\alpha}\},$$ be the infimum of the double bubble perimeter (bounded below by zero).

The main result in this paper is summarized by the following theorem:

\begin{theorem}\label{thm:maintheorem}

For $0<\alpha\le 1$, 
\begin{enumerate}[I.]
\item \label{existencepart}The set $\Gamma_{\alpha}:=\{(A,B)\in\gamma_{\alpha}: \rho_{\text{DB}}(A,B)=\rho_{\text{DB}}(\alpha)\}$ is not empty.  

\item \label{part1} Furthermore, there is a phase transition at an $\alpha_0\in(0,1)$ such that for $\alpha\leq\alpha_0$, the infimum is given by the lefthand side of Figure \ref{fig:TheSolutions}, and for $\alpha\geq\alpha_0$, the minimizing configuration is given by the righthand side of Figure \ref{fig:TheSolutions}. Finally, $\alpha_0$ is calculated, as are the side lengths as a function of $\alpha$. 

\begin{figure}[H]

\begin{tikzpicture}[scale=0.3]

\draw[red, thin] (0,0) to (-1,1.732050808);
\draw[red, thin] (-1,1.732050808) to (-3.5,1.732050808);
\draw[red, thin] (-3.5,1.732050808) to (-4.5,0);
\draw[red, thin] (-4.5,0) to (-3.5,-1.732050808);
\draw[red, thin] (-3.5,-1.732050808) to (-1,-1.732050808);
\draw[red, thin] (-1,-1.732050808) to (0,0);

\draw (0.2,1) node{$\scaleobj{0.8}{y_1}$};
\draw (0.2,-1) node{$\scaleobj{0.8}{y_1}$};
\draw (-2,2.2) node{$\scaleobj{0.8}{y_2}$};
\draw (-2,-2.2) node{$\scaleobj{0.8}{y_2}$};
\draw (-4.5,1) node {\scaleobj{0.8}{y_1}};
\draw (-4.5,-1) node {\scaleobj{0.8}{y_1}};
\draw (-3.4,2.6) node {\scaleobj{0.8}{y_3}};
\draw (-3.4,-2.6) node {\scaleobj{0.8}{y_3}};
\draw (-5.1,3.9) node {\scaleobj{0.8}{y_4}};
\draw (-5.1,-3.9) node {\scaleobj{0.8}{y_4}};
\draw (-7.3,2.6) node {\scaleobj{0.8}{y_5}};
\draw (-7.3,-2.6) node {\scaleobj{0.8}{y_5}};

\draw[blue, thin] (-3.5,1.732050808) to (-4.5,3.464101615);
\draw[blue, thin] (-4.5,3.464101615) to (-6,3.464101615);
\draw[blue, thin] (-6,3.464101615) to (-8,0);
\draw[blue, thin] (-8,0) to (-6,-3.464101615);
\draw[blue, thin] (-4.5,-3.464101615) to (-6,-3.464101615);
\draw[blue, thin] (-4.5,-3.464101615) to (-3.5,-1.732050808);

\draw[blue, thin] (6,3.46410161) to (10,3.46410161);
\draw[blue, thin] (10,3.46410161) to (11,1.7320508);
\draw[blue, thin] (11,1.7320508) to (10,0);
\draw[blue, thin] (10,0) to (6,0);
\draw[blue, thin] (6,0) to (5,1.7320508);
\draw[blue, thin] (5,1.7320508) to (6,3.46410161);

\draw (8,4) node{$\scaleobj{0.8}{x_1}$};
\draw (11,3) node{$\scaleobj{0.8}{x_1}$};
\draw (5,3) node{$\scaleobj{0.8}{x_1}$};
\draw (5,0.8) node{$\scaleobj{0.8}{x_2}$};
\draw (11,0.8) node{$\scaleobj{0.8}{x_2}$}; 
\draw (8,0.35) node{$\scaleobj{0.8}{x_3}$};
\draw (5,-0.8) node{$\scaleobj{0.8}{x_4}$};
\draw (11,-0.8) node{$\scaleobj{0.8}{x_4}$};
\draw (5,-2.9) node{$\scaleobj{0.8}{x_5}$};
\draw (11,-2.9) node{$\scaleobj{0.8}{x_5}$};
\draw (8,-3.9) node{$\scaleobj{0.8}{x_5}$};

\draw[red, thin] (10,0) to (11,-1.7320508);
\draw[red, thin] (11,-1.7320508) to (10,-3.46410161);
\draw[red, thin] (10,-3.46410161) to (6,-3.46410161);
\draw[red, thin] (6,-3.46410161) to (5,-1.7320508);
\draw[red, thin] (5,-1.7320508) to (6,0);

\draw (8,1.6) node{$A$};
\draw (8,-1.6) node{$B$};

\end{tikzpicture}

\caption{\label{fig:TheSolutions}}

\end{figure}

\end{enumerate}

\end{theorem}

\begin{remark}
The expressions giving $\alpha_0$ as well as the side lengths are calculated implicitly but are too long to include here. The method used to find the side lengths are presented in Appendix \ref{sec:TheBigPolynomial}. A numeric approximation for the point of the phase transition is $\alpha_0\approx0.152$. 
\end{remark}

We are now ready to outline the strategy for proving Theorem \ref{thm:maintheorem}, which follows 5 main steps:
\begin{enumerate}
\item Begin with any two sets $(A,B)\in \gamma_{\alpha}$ and find sets $(\tilde{A},\tilde{B})$ that belong to a nicer family of sets, call it $\mathfs{F}_{\alpha}$ (see Figure \ref{fig:GeneralCase}), with $\rho_{\text{DB}}(\tilde{A},\tilde{B})\le \rho_{\text{DB}}(A,B)$. This part is done in Section \ref{sec:optimalsets}. The idea is to replace $B$ with the smallest hexagon containing it, call it $B^{\varhexagon}$, and $A$ with the smallest hexagon containing it, call it $A^{\varhexagon}$. Then we want to replace $(A,B)$ with something as close to $\left(A^{\varhexagon}\setminus B^{\varhexagon},B^{\varhexagon}\right)$ or $\left(A^{\varhexagon},B^{\varhexagon}\setminus A^{\varhexagon}\right)$ as we can. However, it may be that doing this creates an incorrect volume ratio. We may therefore have to adjust the volume of one or both of the new sets. Included in these arguments is the fact that if there is a minimizer, its joint boundary cannot consist of four or more line segments. 
\item The sets in $\mathfs{F}_{\alpha}$ are too complicated to analyse because there are too many parameters. So, we reduce the possible configurations we need to look at by eliminating many instances where there is a sixty degree angle formed by the line segments in $\partial A$ and $\partial B$ ($(A,B)\in\mathfs{F}_{\alpha}$). This is done in section \ref{sec:Elimination}. In these arguments, we also begin to show that the joint boundary of any minimizer can consist of one or two line segments, but not three. 

\item In section \ref{sec:NoSixties}, we finish showing that the joint boundary cannot consist of three line segments. We also categorize the remaining configurations that have no $60^{\circ}$ angles. At the end of this section, we are left only with configurations in a nicer family than $\mathfs{F}_{\alpha}$, call it $\mathfs{G}_{\alpha}$.

\item The family $\mathfs{G}_{\alpha}$ is simple enough to analyse using methods from calculus. Therefore, we can show the existence of $$\arg\inf\{\rho_{\text{DB}}({A},{B}):(A,B)\in\mathfs{G}_{\alpha}\}.$$ This part is done in Section \ref{sec:kkt}.

\item Finally, by the previous points, these sets achieve the infimum over all of $\gamma_\alpha$ proving the existence of an element in $\Gamma_{\alpha}$. Moreover we get a phase transition in $\alpha$. This is done in Section \ref{sec:ProofOfMainTheorem}.
\end{enumerate}

\section{Properties of $\mathcal{D}$}

In this section we include some of the pertinent properties of the norm $\mathcal{D}$ that will be helpful in the forthcoming sections. We begin with a description of the unit ball under $\mathcal{D}$.

\begin{lemma}\label{lem:UnitBall}
The unit ball in the metric $\mathcal{D}$ is a regular hexagon, call it $\mathfs{H}$, that is the convex hull of the following points: $\left\lbrace(1,0),\left(\frac{1}{2},\frac{\sqrt{3}}{2}\right),\left(\frac{-1}{2},\frac{\sqrt{3}}{2}\right),\left(-1,0\right),\left(\frac{-1}{2},\frac{-\sqrt{3}}{2}\right),\left(\frac{1}{2},\frac{-\sqrt{3}}{2}\right)\right\rbrace$. 

\end{lemma}

\begin{proof}
It is easily verified that $|x|+\frac{1}{\sqrt{3}}|y|$ and $\frac{2}{\sqrt{3}}|y|$ agree on the lines $y=\pm\sqrt{3}x$. These lines divide the plane into four sections.  A simple calculation shows that above both $y=\sqrt{3}x$ and $y=-\sqrt{3}x$, $\frac{2}{\sqrt{3}}|y|\geq|x|+\frac{1}{\sqrt{3}}|y|$, and similarly for below both of these lines. On the other hand, between the lines $y=\sqrt{3}x$ and $y=-\sqrt{3}x$, $|x|+\frac{1}{\sqrt{3}}|y|\geq\frac{2}{\sqrt{3}}|y|$. 

The lines $\frac{2}{\sqrt{3}}\cdot|y|=1$ result in the top and bottom edges of our unit ball.  Similarly, $|x|+\frac{1}{\sqrt{3}}|y|=1$ produces four lines. This makes six lines in total, and some simple arithmetic shows that they meet at the indicated points. 
\end{proof}

We now describe a family of geodesics under the metric $\mathcal{D}$. To do this, we divide the plane into sextants. Let $S_1$ be the set of points of $\mathbb{R}^2$ that are between the positive $x$-axis and the ray emanating from the origin and with an angle of $\frac{\pi}{3}$ (measured from the positive $x$-axis). We let $S_1$ contain both the $x$-axis, and the ray with angle $\frac{\pi}{3}$. Similarly, define $S_2$ to be the set of points in $\mathbb{R}^2$ that fall between the two rays emanating from the origin and with angles $\frac{\pi}{3}$ and $\frac{2\pi}{3}$. Again, we let $S_2$ contain both of these two rays, which means that $S_1$ and $S_2$ have non-empty intersection. Continuing in this way, we also get $S_3$, $S_4$, $S_5$, and $S_6$.

\begin{lemma}\label{lem:geodesics}

Let $(x,y)\in S_1$, and $(x',0)$ be the point of intersection of the $x$-axis and the line $L(t)=(x,y)+t\cdot\left(\frac{1}{2},\frac{\sqrt{3}}{2}\right)$. The path from the origin to $(x',0)$ and then from $(x',0)$ to $(x,y)$ along the line $L(t)$ is a geodesic between $(0,0)$ and $(x,y)$. Similar statements are true for $(x,y)\in S_2,...,S_6$.
\end{lemma}

\begin{proof}

The proofs for each of $S_1,...,S_6$ are all similar, and essentially elementary. We will therefore include only the proof for $(x,y)\in S_1$. So, let $(x,y)\in S_1$. It is easily seen that $|x|\geq\frac{1}{\sqrt{3}}|y|$, whence $|x|+\frac{1}{\sqrt{3}}|y|\geq\frac{2}{\sqrt{3}}|y|$. Let $(x',0)$ be the projection of $(x,y)$ onto the $x$-axis in the direction of the vector $\left(\frac{1}{2},\frac{\sqrt{3}}{2}\right)$. Then, the length of $\left(x',0\right)$ plus the length of $\left(x-x',y\right)$, under the metric $\mathcal{D}$, is $\mathcal{D}\left((x',0),(0,0)\right)+\mathcal{D}\left((x,y),(x',0)\right)=|x'|+|x-x'|+\frac{1}{\sqrt{3}}|y|=|x|+\frac{1}{\sqrt{3}}|y|=\mathcal{D}\left((x,y),(0,0)\right)$. Thus, the path from $(0,0)$ to $(x',0)$ and then from $(x',0)$ to $(x,y)$ is a geodesic. It follows that any path from $(0,0)$ to $(x,y)$ that consists of line segments that are parallel to either the $x$-axis or the vector $\left(\frac{1}{2},\frac{\sqrt{3}}{2}\right)$, and such that the lengths of the segments in the direction of $\left(1,0\right)$ adds to the length of $\left(x',0\right)$, and the lengths of the segments in the direction of $\left(\frac{1}{2},\frac{\sqrt{3}}{2}\right)$ adds to the length of $\left(x-x',y\right)$ is a geodesic between $(0,0)$ and $(x,y)$.

\end{proof}

A necessary step towards solving the double bubble problem is to first solve the isoperimetric problem in this metric. The following definitions will help us achieve this. 

\begin{definition}\label{def:ImportantLines}
For $b\in\mathbb{R}$, let $\Lambda_{\scaleobj{0.6}{\nearrow}}(b)$ be the line whose slope is $\sqrt{3}$ and intersects the vertical axis at $b$. Similarly, define $\Lambda_{\scaleobj{0.6}{\searrow}}(b)$ to be the line whose slope is $-\sqrt{3}$ and intersects the vertical axis at $b$, and $\Lambda_{\scaleobj{0.6}{\rightarrow}}(b)$ to be the line whose slope is zero and intersects the vertical axis at $b$. 
\end{definition}

Each of these lines divides the plane into two half-planes, an upper half-plane and a lower half-plane. 
\begin{definition}\label{def:Halfplanes}
Let $U_{\scaleobj{0.6}{\nearrow}}(b)$ be the upper half-plane formed by removing $\Lambda_{\smallnearrow}(b)$ from $\mathbb{R}^2$, and $L_{\smallnearrow}(b)$ be the lower half-plane formed by removing $\Lambda_{\smallnearrow}(b)$ from $\mathbb{R}^2$.  Similarly define $U_{\smallsearrow}(b)$,$L_{\smallsearrow}(b)$, $U_{\smallrightarrow}(b)$, and $L_{\smallrightarrow}(b)$. 
\end{definition}

\begin{definition}\label{def:ImportantIntercepts}
For a bounded set $A\subset\mathbb{R}^2$, let $s_{\smallnearrow}^A=\sup\{b:A\subset U_{\smallnearrow}(b)\}$, $i_{\smallnearrow}^A=\inf\{b:A\subset L_{\smallnearrow}(b)\}$, $s_{\smallsearrow}^A=\sup\{b:A\subset U_{\smallsearrow}(b)\}$, $i_{\smallsearrow}^A=\inf\{b:A\subset L_{\smallsearrow}(b)\}$, $s_{\smallrightarrow}^A=\sup\{b:A\subset U_{\smallrightarrow}(b)\}$, and $i_{\smallrightarrow}^A=\inf\{b:A\subset L_{\smallrightarrow}(b)\}$. 
\end{definition}

\begin{definition}
The closure of the connected component of  
$$\mathbb{R}^2\setminus\left[\Lambda_{\smallnearrow}(s_{\smallnearrow}^A)\cup\Lambda_{\smallnearrow}\left(i_{\smallnearrow}^A\right)\cup\Lambda_{\smallsearrow}\left(s_{\smallsearrow}^A\right)\cup\Lambda_{\smallsearrow}\left(i_{\smallsearrow}^A\right)\cup\Lambda\left(s_{\smallrightarrow}^A\right)\cup\Lambda_H\left(i_{\smallrightarrow}^A\right)\right],$$
which contains $A$, is what we call $A^{\varhexagon}$, pronounced $A$-hexagon. 
\end{definition}

The following figure illustrates the previous four definitions:\\

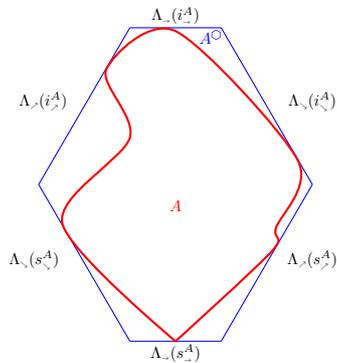
\begin{figure}[H]
\begin{tikzpicture}[scale=0.6]

\draw[blue, thin] (-1,0) to (1,0);
\draw[blue, thin] (1,0) to (3,3.4641);
\draw[blue, thin] (3,3.4641) to (1,6.9282);
\draw[blue, thin] (1,6.9282) to (-1,6.9282);
\draw[blue, thin] (-1,6.9282) to (-3,3.4641);
\draw[blue, thin] (-3,3.4641) to (-1,0);

\draw [red,thick] plot [smooth, tension=0.5] coordinates {(0,0) (-2.48,2.6)(-1,4.5) (-1.52,6)(0,6.855) (2.68,4)(2.2,2.5) (2.12,2)(0,0)};

\draw[red] (0,3) node{$\scaleobj{0.6}{A}$};
\draw[blue] (0.775,6.725) node{$\scaleobj{0.6}{A^{\varhexagon}}$};

\draw (3,1.8) node{$\scaleobj{0.6}{\Lambda_{\smallnearrow}(s_{\smallnearrow}^A)}$};
\draw (3,5.3) node{$\scaleobj{0.6}{\Lambda_{\smallsearrow}(i_{\smallsearrow}^A)}$};
\draw (0,7.2) node{$\scaleobj{0.6}{\Lambda_{\smallrightarrow}(i_{\smallrightarrow}^A)}$};
\draw (-2.9,5.3) node{$\scaleobj{0.6}{\Lambda_{\smallnearrow}(i_{\smallnearrow}^A)}$};
\draw (-3.1,1.8) node{$\scaleobj{0.6}{\Lambda_{\smallsearrow}(s_{\smallsearrow}^A)}$};
\draw (0,-0.26) node{$\scaleobj{0.6}{\Lambda_{\smallrightarrow}(s_{\smallrightarrow}^A)}$};

\end{tikzpicture}
\end{figure}

Notice that $A^{\varhexagon}$ is not necessarily a shape with six sides of positive length.  For example, in the following figure $A=A^{\varhexagon}$, and there are only four sides of positive length.

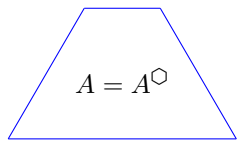
\begin{figure}[H]
\begin{tikzpicture}[scale=0.5]

\draw[blue, thin] (-3,3.4641) to (3,3.4641);
\draw[blue, thin] (3,3.4641) to (1,6.9282);
\draw[blue, thin] (1,6.9282) to (-1,6.9282);
\draw[blue, thin] (-1,6.9282) to (-3,3.4641);


\draw (0,5) node{$A=A^{\varhexagon}$};

\end{tikzpicture}
\end{figure}

In this case $\Lambda_{\smallsearrow}(s_{\smallsearrow})\cap\partial A^{\varhexagon}$ has length zero, as does $\Lambda_{\smallnearrow}\left(s_{\smallnearrow}^A\right)\cap\partial A^{\varhexagon}$. However, we still consider $A^{\varhexagon}$ to have six sides, but two of them have length zero. 

It is also possible for $A^{\varhexagon}$ to have three or five sides of positive length. With the definition of $A^{\varhexagon}$ we can now solve the isoperimetric problem in this metric.

\begin{lemma}\label{lem:IsoperimetricProblem}
In the metric $\mathcal{D}$, the solution to the isoperimetric problem is the renormalized unit ball $\frac{\mathfs{H}}{\mu\left(\mathfs{H}\right)}$, where $\mathfs{H}$ is the unit ball as described in Lemma \ref{lem:UnitBall}. 
\end{lemma}

\begin{proof}
Let $A$ be any set in $\mathbb{R}^2$ such that $\mu(A)=1$, and consider $A^{\varhexagon}$. Since $A\subset A^{\varhexagon}$, $\mu(A)\leq\mu\left(A^{\varhexagon}\right)$. Let $p_{\smallnearrow}(i_{\smallnearrow}^A)$ be any point where $\Lambda_{\smallnearrow}\left(i_{\smallnearrow}^A\right)$ intersects $\partial A$.  Similarly define $p_{\smallrightarrow}(i_{\smallrightarrow}^A)$, $p_{\smallsearrow}\left(i_{\smallsearrow}^A\right)$, $p_{\smallnearrow}\left(s_{\smallnearrow}^A\right)$, $p_{\smallrightarrow}(s_{\smallrightarrow}^A)$, and $p_{\smallnearrow}\left(s_{\smallnearrow}^A\right)$. By lemma \ref{lem:geodesics}, the portion of $\partial A^{\varhexagon}$ along $\Lambda_{\smallnearrow}\left(i_{\smallnearrow}^A\right)$ and $\Lambda_{\smallrightarrow}\left({i_{\smallrightarrow}^A}\right)$ that connects $p_{\smallnearrow}\left(i_{\smallnearrow}^A\right)$ and $p_{\smallrightarrow}\left(i_{\smallrightarrow}^A\right)$ is a geodesic.  Similarly, we can show that the portion of $\partial A^{\varhexagon}$ along $\Lambda_{\smallrightarrow}\left(i_{\smallrightarrow}^A\right)$ and $\Lambda_{\smallnearrow}\left(i_{\smallnearrow}^A\right)$ that connects $p_{\smallrightarrow}\left(i_{\smallrightarrow}^A\right)$ and $p_{\smallnearrow}\left(i_{\smallnearrow}^A\right)$ is a geodesic. We can repeat this process to connect $p_{\smallsearrow}\left(s_{\smallsearrow}^A\right)$ to $p_{\smallrightarrow}\left(s_{\smallrightarrow}^A\right)$, $p_{\smallrightarrow}\left(s_{\smallrightarrow}^A\right)$ to $p_{\smallnearrow}\left(s_{\smallnearrow}^A\right)$, and $p_{\smallnearrow}\left(s_{\smallnearrow}^A\right)$ to $p_{\smallsearrow}\left(s_{\smallsearrow}^A\right)$. Each of these geodesics is no longer than the portion of $\partial A$ that connects the same points. Therefore, $\rho\left(\partial A^{\varhexagon}\right)\leq\rho\left(\partial A\right)$. Since $\mu\left(A^{\varhexagon}\right)\geq1$, we can rescale $A^{\varhexagon}$ so that it has volume $1$, which can only reduce the length of the perimeter. We now have a figure that looks like the following (again, it is possible that some of the parameters equal zero): \\

\begin{figure}[H]

\begin{tikzpicture}[scale=0.4]

\draw[blue, thin] (-1,0) to (1,0);
\draw[blue, thin] (1,0) to (3,3.4641);
\draw[blue, thin] (3,3.4641) to (1,6.9282);
\draw[blue, thin] (1,6.9282) to (-1,6.9282);
\draw[blue, thin] (-1,6.9282) to (-3,3.4641);
\draw[blue, thin] (-3,3.4641) to (-1,0);

\draw (2.7,5) node{$x_1$};
\draw (0,7.35) node{$x_2$};
\draw (-2.75,4.95) node{$x_3$};
\draw (-2.7,2) node{$x_4$};
\draw (0,-0.4) node{$x_5$};
\draw (2.75,2) node{$x_6$};

\end{tikzpicture}
\caption{\label{fig:params}}
\end{figure}

We can now use methods from calculus to minimize the perimeter of this figure with the constraints that the volume must equal $1$ and $x_1,...,x_6\geq0$ to obtain a hexagon with six sides of equal length as a solution to the isoperimetric problem. This is done in section \ref{sec:kkt} along with the rest of the calculations. We have yet, however, to show that this solution is unique. Notice, though, that if $A\neq A^{\varhexagon}$, then $\mu(A)<\mu\left(A^{\varhexagon}\right)$. Therefore, when we rescale $A^{\varhexagon}$ so that its volume equals $1$, the perimeter will strictly decrease. This means that the original figure could not have been a solution to the isoperimetric problem. Therefore, the solution is unique and is a hexagon with six sides of equal length. 

\end{proof}

Note that by \cite{taylor1975unique}, the solution of the isoperimetric problem is given by the renormalized unit ball in the dual norm. Thus, by the previous lemma, we obtain that $\mathcal{D}$ is self dual.

\section{Further Definitions}

Now, we define in Figure \ref{fig:GeneralCase} the family of sets $\mathfs{F}_{\alpha}\subset\gamma_{\alpha}$, and in Figure \ref{fig:TheGoodFamily}, the subfamily of $\mathfs{F}_{\alpha}$, which we call $\mathfs{G}_{\alpha}$, that were just mentioned in the outline of the proof strategy. The definition of $\pi_1$ and $\pi_2$ in Figure \ref{fig:GeneralCase} will be given presently.

\begin{figure}[H]

\begin{tikzpicture}[scale=0.6]

\draw[blue, thin] (-1,0) to (1,0);
\draw[blue, thin] (1,0) to (2.2,2.0784609);
\draw[blue, thin] (2.5,4.330127) to (1,6.9282);
\draw[blue, thin] (1,6.9282) to (-1,6.9282);
\draw[blue, thin] (-1,6.9282) to (-3,3.4641);
\draw[blue, thin] (-3,3.4641) to (-1,0);

\draw[red, thin] (3.7,-0.5) to (5.7,-0.5);
\draw[red, thin] (5.7,-0.5) to (7.7,2.9641);
\draw[red, thin] (7.7,2.9641) to (5.7,6.4282);
\draw[red, thin] (5.7,6.4282) to (3.7,6.4282);
\draw[red, thin] (1.7,2.96410135) to (3.7,-0.5);
\draw[red, thin] (3.7,6.4282) to (1.7,2.96410135);

\draw (4.7,6.7) node{$x_1$};
\draw (6.9,5) node{$x_2$};
\draw (6.9,0.9) node{$x_3$};
\draw (4.8,-0.8) node{$x_4$};
\draw (2.5,0.9) node{$x_5$};
\draw (1.8,2.25) node{$x_6$};
\draw (1.4,4.6) node{$x_7$};
\draw[black, thin] (1.5,4.4) to (2.38,4.16);
\draw (3,6) node{$x_{8}$};

\draw (2,6) node{$x_{9}$};
\draw (0.15,7.2) node{$x_{10}$};
\draw (-2.45,5.5) node{$x_{11}$};
\draw (-2.45,1.8) node{$x_{12}$};
\draw (0,-0.2) node{$x_{13}$};
\draw (1.15,1) node{$x_{14}$};

\draw (-0.2,3.4) node{$A$};
\draw (5,3) node{$B$};

\draw[blue, thin] (14.5,4.330127) to (13,6.9282);
\draw[blue, thin] (13,6.9282) to (11,6.9282);
\draw[blue, thin] (11,6.9282) to (9,3.4641);
\draw[blue, thin] (9,3.4641) to (12,-1.7320508);
\draw[blue, thin] (12,-1.7320508) to (20,-1.7320508);
\draw[blue, thin] (20,-1.7320508) to (21,0);
\draw[blue, thin] (21,0) to (19.5,2.5980762);

\draw (19.5,2.5980762) node{$\cdot$};
\draw[black, thin] (19.5,2.5980762) to (19,2.6);
\draw (18.5,2.6) node{$\pi_2$};

\draw (14.5,4.330127) node{$\cdot$};
\draw[black, thin] (14.5,4.330127) to (15,4.5);
\draw (15.4,4.5) node{$\pi_1$};

\draw[red, thin] (15.7,-0.5) to (17.7,-0.5);
\draw[red, thin] (17.7,-0.5) to (19.7,2.9641);
\draw[red, thin] (19.7,2.9641) to (17.7,6.4282);
\draw[red, thin] (17.7,6.4282) to (15.7,6.4282);
\draw[red, thin] (13.7,2.96410135) to (15.7,-0.5);
\draw[red, thin] (15.7,6.4282) to (13.7,2.96410135);

\draw (16.7,6.7) node{$x_1$};
\draw (18.9,5) node{$x_2$};
\draw (20.7,2.4) node{$x_3$};
\draw[black, thin] (20.3,2.4) to (19.6,2.75);
\draw (19,1) node{$x_4$};
\draw (16.8,-0.8) node{$x_5$};
\draw (14.5,0.9) node{$x_6$};
\draw (13.4,4.6) node{$x_7$};
\draw[black, thin] (13.5,4.4) to (14.38,4.16);
\draw (15,6) node{$x_{8}$};

\draw (14,6) node{$x_{9}$};
\draw (12.15,7.2) node{$x_{10}$};
\draw (9.55,5.5) node{$x_{11}$};
\draw (9.55,1.8) node{$x_{12}$};
\draw (15.8,-2.2) node{$x_{13}$};
\draw (21,-1) node{$x_{14}$};
\draw (21,1) node{$x_{15}$};

\draw (11.8,3.4) node{$A$};
\draw (17,3) node{$B$};










\end{tikzpicture}

\caption{\label{fig:GeneralCase}}

\end{figure}

The family $\mathfs{F}_{\alpha}$ consists of all pairs of sets $\left(A,B\right)$ that can be formed from Figure \ref{fig:GeneralCase}, where any collection of the parameters are allowed to equal zero so long as one of $A$ or $B$ has volume $\alpha$, and the other one has volume $1$. Admittedly, the configuration on the righthand side does not look very nice. It is more general than the configuration on the left in the sense that any configuration that can be formed from the configuration on the left can also be formed from the configuration on the right. However, the cases that arise from the configuration on the right that can't be created from the configuration on the left do not result in double bubble minimizers. So, we include the lefthand side of Figure \ref{fig:GeneralCase} to emphasize that this is the important configuration from which we can find our minimizers. 

The configurations in Figure \ref{fig:GeneralCase} have too many parameters to be easily analysed. So, we define in Figure \ref{fig:TheGoodFamily} a subfamily $\mathfs{G}_{\alpha}\subset\mathfs{F}_{\alpha}$. 

\begin{figure}[H]

\begin{tikzpicture}[scale=0.6]

\draw[red, thin] (0,0) to (-1,1.732050808);
\draw[red, thin] (-1,1.732050808) to (-3.5,1.732050808);
\draw[red, thin] (-3.5,1.732050808) to (-4.5,0);
\draw[red, thin] (-4.5,0) to (-3.5,-1.732050808);
\draw[red, thin] (-3.5,-1.732050808) to (-1,-1.732050808);
\draw[red, thin] (-1,-1.732050808) to (0,0);

\draw[blue, thin] (-3.5,1.732050808) to (-4.5,3.464101615);
\draw[blue, thin] (-4.5,3.464101615) to (-6,3.464101615);
\draw[blue, thin] (-6,3.464101615) to (-8,0);
\draw[blue, thin] (-8,0) to (-6,-3.464101615);
\draw[blue, thin] (-4.5,-3.464101615) to (-6,-3.464101615);
\draw[blue, thin] (-4.5,-3.464101615) to (-3.5,-1.732050808);

\draw (-6,0) node{$A$};
\draw (-2.5,0) node{$B$};

\draw (-2.3,2) node{$x_1$};
\draw (-0.1,1) node{$x_2$};
\draw (-0.1,-1) node{$x_3$};
\draw (-2.3,-2) node{$x_4$};
\draw (-4.5,-1) node{$x_5$};
\draw (-4.5,1) node{$x_6$};
\draw (-3.5,2.5) node{$x_7$};
\draw (-5.1,3.8) node{$x_8$};
\draw (-7,2.5) node{$x_9$};
\draw (-7.1,-2.5) node{$x_{10}$};
\draw (-5.1,-3.8) node{$x_{11}$};
\draw (-3.4,-2.5) node{$x_{12}$};

\draw[blue, thin] (6,3.46410161) to (10,3.46410161);
\draw[blue, thin] (10,3.46410161) to (11,1.7320508);
\draw[blue, thin] (11,1.7320508) to (10,0);
\draw[blue, thin] (10,0) to (6,0);
\draw[blue, thin] (6,0) to (5,1.7320508);
\draw[blue, thin] (5,1.7320508) to (6,3.46410161);

\draw (11,0.8) node{$x_1$};
\draw (11,2.4) node{$x_2$};
\draw (8,3.65) node{$x_3$};
\draw (5,2.4) node{$x_4$};
\draw (5,0.8) node{$x_5$};
\draw (8,0.25) node{$x_6$};
\draw[black, thin] (6,0.2) to (7.6,0.2);
\draw[black, thin] (8.4,0.2) to (10,0.2);

\draw (8,-0.3) node{$x_7$};
\draw[black, thin] (7,-0.2) to (7.6,-0.2);
\draw[black, thin] (8.4,-0.2) to (9,-0.2);
\draw (6.15,-0.8) node{$x_8$};
\draw (6.15,-2.5) node{$x_9$};
\draw (8,-3.8) node{$x_{10}$};
\draw (10,-2.5) node{$x_{11}$};
\draw (10,-0.8) node{$x_{12}$};

\draw[red, thin] (9,0) to (10,-1.7320508);
\draw[red, thin] (10,-1.7320508) to (9,-3.46410161);
\draw[red, thin] (9,-3.46410161) to (7,-3.46410161);
\draw[red, thin] (7,-3.46410161) to (6,-1.7320508);
\draw[red, thin] (6,-1.7320508) to (7,0);

\draw (8,1.6) node{$A$};
\draw (8,-1.6) node{$B$};

\end{tikzpicture}

\caption{\label{fig:TheGoodFamily}}

\end{figure}

Throughout the paper, there are two points and two lines which we repeatedly use. We give their definitions here. 

\begin{definition}\label{def:TwoImportantPoints}
Let $\left(A,B\right)\in\mathfs{F}_{\alpha}$, the nice family of configurations in Figure \ref{fig:GeneralCase}. The joint boundary ($\partial A\cap\partial B$) consists of a path with two endpoints. Let $\pi_1\left(A,B\right),\pi_2\left(A,B\right)$ be these points, chosen in the following way: if the second coordinates of the two points in question are different, we let $\pi_1$ be the point with larger second coordinate, and $\pi_2$ the other point. If the second coordinate of each point is the same, then we let $\pi_1$ be the point with larger first coordinate, and $\pi_2$ be the other point. It does not occur that $\pi_1$ and $\pi_2$ are the same point. 
\end{definition}

\begin{definition}\label{def:ImportantLineSegments}
Let $\left(A,B\right)\in\mathfs{F}_{\alpha}$. Define $l_1^A$ to be the line segment in $\partial A$ that is not joint boundary, and has one endpoint at $\pi_1$. Define $a_1$ to be the line segment in $\partial A$ that meets the endpoint of $l_1^A$ that is not $\pi_1$. Similarly define $l_1^B$ and $b_1$. We also define the line segments $l_2^A$, $a_2$, $l_2^B$, and $b_2$ to be the analogous line segments associated with $\pi_2$. 

\end{definition}

To illustrate the previous definition, in the righthand side of Figure \ref{fig:GeneralCase}, $x_9=l_1^A$, $x_{10}=a_1$, $x_8=l_1^B$, and $x_1=b_2$.

We end this section with a final definition. 

\begin{definition}
Let $A\subset \mathbb{R}^2$ be open and bounded. We call the line segments $\Lambda_{\smallnearrow}\left(i_{\smallnearrow}^A\right)\cap A^{\varhexagon}$ and $\Lambda_{\smallsearrow}\left(s_{\smallsearrow}^A\right)\cap\partial A^{\varhexagon}$ the "left sides" of $A^{\varhexagon}$, even if one of them has length zero. Similarly, we call $\Lambda_{\smallrightarrow}\left(i_{\smallrightarrow}^A\right)\cap\partial A^{\varhexagon}$ the "top" of $A^{\varhexagon}$, and $\Lambda_{\smallrightarrow}\left(s_{\smallrightarrow}^A\right)\cap\partial A^{\varhexagon}$ the "bottom" of $A^{\varhexagon}$, even if one or both of these line segments has length zero. Finally, we call $\Lambda_{\smallnearrow}\left(s_{\smallnearrow}^A\right)\cap\partial A^{\varhexagon}$, and $\Lambda_{\smallnearrow}\left(s_{\smallnearrow}^A\right)\cap\partial A^{\varhexagon}$ the "right sides" of $\partial A^{\varhexagon}$, even if one of their lengths is zero. 
\end{definition}

We will frequently have to adjust sides of the sets we analyse. This results in sets that are not necessarily the smallest hexagon containing any set. However, we will still refer to the line segments forming these new sets as the "left" or "right" sides in a natural way. Furthermore, we will often have to move line segments either to the left, right, up, or down. If we say that we are moving a line segment contained in $\Lambda_{\smallnearrow}\left(b\right)$ to the right, it means that we are decreasing the value of $b$. We give similar definitions to "moving to the left", "moving up", and so forth. Finally, whenever we say that we are moving some line segment, it must be understood that, unless otherwise specifically stated, this may require another line segment to become shorter or longer so that we always have two bounded sets whose boundaries are closed, simple, rectifiable curves.

\section{Comparing general configurations to $\mathfs{F}_{\alpha}$}\label{sec:optimalsets}

Our goal in this section is to show that given any two sets $(A,B)\in\gamma_{\alpha}$, we can find two new sets $(\tilde{A},\tilde{B})\in\mathfs{F}_{\alpha}$ such that $\rho_{{DB}}(\tilde{A},\tilde{B})\leq\rho_{{DB}}(A,B)$. Note that, with two exceptional configurations, $\left(A^{\varhexagon}\setminus B^{\varhexagon},B^{\varhexagon}\right)$ is of the correct form to belong to $\mathfs{F}_{\alpha'}$ for some $\alpha'\in(0,1]$, and $\left(A^{\varhexagon},B^{\varhexagon}\setminus A^{\varhexagon}\right)$ has the correct form to belong to $\mathfs{F}_{\alpha''}$ for some $\alpha''\in(0,1]$. The problem that may occur is that the sets, besides having the incorrect volume ratio, in these two configurations may not have the correct volumes. The two exceptional cases are when $A^{\varhexagon}\cap B^{\varhexagon}=\emptyset$, and when $A^{\varhexagon}$ is contained in $B^{\varhexagon}$ (or vice versa). Apart from these two exceptions, in this section, we aim to begin with $\left(A^{\varhexagon}\setminus B^{\varhexagon},B^{\varhexagon}\right)$ or $\left(A^{\varhexagon},B^{\varhexagon}\setminus A^{\varhexagon}\right)$ and alter the lengths of the already existing line segments in a way that results in sets that are in $\mathfs{F}_{\alpha}$.

We treat three cases.  First, we consider when $\mu\left(A^{\varhexagon}\cap B^{\varhexagon}\right)=0$. Second, we consider when \newline $\mu\left(B^{\varhexagon}\setminus A^{\varhexagon}\right),\mu\left(A^{\varhexagon}\cap B^{\varhexagon}\right),\mu\left(A^{\varhexagon}\setminus B^{\varhexagon}\right)>0$. Finally, we consider when $B^{\varhexagon}\subset A^{\varhexagon}$ (or vice versa). \\

\subsection{The first case}

The first case, in which $\mu\left(A^{\varhexagon}\cap B^{\varhexagon}\right)=0$, is simple.  We know that $\mu\left(A\right)\leq\mu\left(A^{\varhexagon}\right)$, and $\rho\left(\partial A^{\varhexagon}\right)\leq\rho\left(A\right)$ (similarly for $B$). So, we begin by replacing $A$ with $A^{\varhexagon}$, $B$ with $B^{\varhexagon}$. This may increase joint boundary, but since we know that $\rho\left(A^{\varhexagon}\right)\leq\rho\left(A\right)$, and $\rho\left(B^{\varhexagon}\right)\leq\rho\left(B\right)$, it follows that $\rho_{DB}\left(A^{\varhexagon},B^{\varhexagon}\right)=\rho\left(\partial A^{\varhexagon}\right)+\rho\left(\partial B^{\varhexagon}\right)-\rho\left(\partial A^{\varhexagon}\cap\partial B^{\varhexagon}\right)\leq\rho\left(\partial A\right)+\rho\left(\partial B\right)-\rho\left(\partial A\cap\partial B\right)=\rho_{DB}\left(A,B\right)$. Since it may be the case that $\mu\left(A\right)<\mu\left(A^{\varhexagon}\right)$, and $\mu\left(B\right)<\mu\left(B^{\varhexagon}\right)$, we may need to rescale both of these sets until their volumes are $1$ and $\alpha$, respectively. Call these rescaled sets $A_r^{\varhexagon}$ (to replace $A^{\varhexagon}$), and $B_r^{\varhexagon}$ (to replace $B^{\varhexagon}$). Rescaling $A^{\varhexagon}$ can only reduce the length of its perimeter, and similarly for $B^{\varhexagon}$. This rescaling may reduce the joint boundary, i.e. $\rho\left(\partial A^{\varhexagon}_r\cap\partial B^{\varhexagon}_r\right)\leq\rho\left(\partial A^{\varhexagon}\cap\partial B^{\varhexagon}\right)$, or even result in two sets such that $\partial A^{\varhexagon}_r\cap\partial B^{\varhexagon}_r=\emptyset$. In this latter case, we can translate one or both of $A^{\varhexagon}_r,B^{\varhexagon}_r$ so that their longest sides share as much boundary as possible. This process can only reduce the double bubble perimeter because $\rho\left(A^{\varhexagon}_r\right)\leq\rho\left(A^{\varhexagon}\right)$, and $\rho\left(\partial B^{\varhexagon}\setminus\partial A^{\varhexagon}_r\right)\leq\rho\left(\partial B^{\varhexagon}\setminus \partial A^{\varhexagon}\right)$.  Thus, the result is two sets of the correct volume, and whose double bubble perimeter is no more than that of $(A,B)$. Furthermore, the resulting configuration is a member of the family $\mathfs{F}_{\alpha}$. Notice that if $A$ and $B$ did not have any joint boundary, then placing $A^{\varhexagon}_r$ and $B^{\varhexagon}_r$ next to each other so that they share a side strictly decreases the double bubble perimeter.  It follows that $(A,B)$ cannot be in $\Gamma_{\alpha}$, and therefore we only need to interest ourselves in configurations whose joint boundary has strictly positive length. \\

\subsection{The second case}

We now consider when $\mu\left(A^{\varhexagon}\setminus B^{\varhexagon}\right),\mu\left(A^{\varhexagon}\cap B^{\varhexagon}\right),\mu\left(B^{\varhexagon}\setminus A^{\varhexagon}\right)>0$. This means that two sides of $\partial B^{\varhexagon}$ cross one or two sides of $\partial A^{\varhexagon}$ (or vice versa). We must decide whether to analyse $\left(A^{\varhexagon}\setminus B^{\varhexagon},B^{\varhexagon}\right)$ or $\left(A^{\varhexagon},B^{\varhexagon}\setminus A^{\varhexagon}\right)$, and it is beneficial for us to choose the configuration with larger joint boundary; this will be helpful when we have to adjust volumes to ensure that we have a configuration consisting of two sets of the correct volumes. Let us assume that this is $\left(A^{\varhexagon}\setminus B^{\varhexagon},B^{\varhexagon}\right)$. The arguments when $\left(A^{\varhexagon},B^{\varhexagon}\setminus A^{\varhexagon}\right)$ has larger joint boundary are similar. The important point to notice is that if a side of $B^{\varhexagon}$ crosses a side of $A^{\varhexagon}$, then the angle (measured inside $A^{\varhexagon}$) is either $60^{\circ}$ or $120^{\circ}$. 
We begin with a bit of notation and a figure to illustrate Definitions \ref{def:TwoImportantPoints} and \ref{def:ImportantLineSegments} (see Figure \ref{fig:OneCornerTwoHexagons}). The line segment $l_1^A$ is part of a potentially longer line segment in $\partial A^{\varhexagon}$, part of which is contained in the interior of $B^{\varhexagon}$. We denote the part of this longer line segment that is contained in $B^{\varhexagon}$ by $l_1^{A_{ext}}$. The notation means that, in some sense, $l_1^{A_{ext}}$ is the extension of $l_1^A$. We make similar definitions for $l_2^A$, $l_2^{A_{ext}}$, $l_1^B$, $l_2^B$, $l_1^{B_{ext}}$, and $l_2^{B_{ext}}$. Figure \ref{fig:OneCornerTwoHexagons} illustrates this definition for $l_2^{B_{ext}}$. Now, define $\theta_1$ to be the angle between $l_1^A$ and $l_1^{B_{ext}}$ (measured inside $A^{\varhexagon}$), and similarly define $\theta_2$ to be the angle between $l_2^A$ and $l_2^{B_{ext}}$. Note that if $\left(A^{\varhexagon},B^{\varhexagon}\setminus A^{\varhexagon}\right)$ has greater joint boundary than $\left(A^{\varhexagon}\setminus B^{\varhexagon},B^{\varhexagon}\right)$, then $\theta_1$ and $\theta_2$ would be measured in the interior of $B^{\varhexagon}$.

\begin{figure}

\begin{tikzpicture}[scale=0.45]

\draw[blue, thin] (-1,0) to (1,0);
\draw[blue, thin] (1,0) to (3,3.4641);
\draw[blue, thin] (3,3.4641) to (1,6.9282);
\draw[blue, thin] (1,6.9282) to (-1,6.9282);
\draw[blue, thin] (-1,6.9282) to (-3,3.4641);
\draw[blue, thin] (-3,3.4641) to (-1,0);

\draw[red, thin] (3.7,-0.5) to (5.7,-0.5);
\draw[red, thin] (5.7,-0.5) to (7.7,2.9641);
\draw[red, thin] (7.7,2.9641) to (5.7,6.4282);
\draw[red, thin] (5.7,6.4282) to (3.7,6.4282);
\draw[red, thin] (3.7,6.4282) to (1.7,2.9641);
\draw[red, thin] (1.7,2.9641) to (3.7,-0.5);

\draw (0,4.5) node{$A^{\varhexagon}$};
\draw (4.8,3) node{$B^{\varhexagon}$};

\draw[black, thin] (4.5,6.4282) to (5,7.5);
\draw (5.3,7.8) node{$b_1$};

\draw[black, thin] (4.5,-0.5) to (5.5,-1.5);
\draw (5.7,-1.8) node{$b_2$};

\draw [black,thin] plot [smooth, tension=0.2] coordinates {(2.55,-0.7) (2.5,0.2)(2.7,1.2)};
\draw (2.6,-1) node{$l_2^B$};

\draw [black,thin] plot [smooth, tension=0.2] coordinates {(3,7.2) (3.3,6.5)(3.34,5.8)};
\draw (3,7.6) node{$l_1^B$};

\draw [black,thin] plot [smooth, tension=0.2] coordinates {(0.8,-1) (1.4,-0.4)(1.3,0.5)};
\draw (0.6,-1) node{$l_2^A$};

\draw [black,thin] plot [smooth, tension=0.2] coordinates {(1.35,7.45) (1.6,6.8)(1.4,6.2)};
\draw (1.2,7.6) node{$l_1^A$};

\draw (0.18,-0.26) node{$a_2$};
\draw (0.18,7.15) node{$a_1$};

\draw (0,2.55) node{$l_2^{B_{ext}}$};
\draw[black, thin] (1.95,2.5) to (0.8,2.6);
\draw (1.78,2.15) node{$\theta_2$};
\draw (2.1,4.35) node{$\theta_1$};

\draw (2.208,2.067) node{$\cdot$};
\draw (2.1,0.5) node{$\pi_2$};
\draw[black, thin] (2.208,2.067) to (2.15,0.8);

\draw (2.49,4.33) node{$\cdot$};
\draw (2.3,5.7) node{$\pi_1$};
\draw[black, thin] (2.49,4.33) to (2.4,5.5);

\end{tikzpicture}

\caption{\label{fig:OneCornerTwoHexagons}}

\end{figure}

\begin{lemma}

Let $(A,B)\in\gamma_{\alpha}$ be such that $\mu\left(A^{\varhexagon}\setminus B^{\varhexagon}\right),\mu\left(A^{\varhexagon}\cap B^{\varhexagon}\right),\mu\left(B^{\varhexagon}\setminus A^{\varhexagon}\right)>0$.  If, using the notation presented before this lemma, both of $\theta_1$ and $\theta_2$ are $120^{\circ}$ angles, then \newline$\max\left\{\rho_{DB}\left(A^{\varhexagon}\setminus B^{\varhexagon},B^{\varhexagon}\right),\rho_{DB}\left(A^{\varhexagon},B^{\varhexagon}\setminus A^{\varhexagon}\right)\right\}\leq\rho_{DB}\left(A,B\right)$. 


\label{lem:Two120DegreeAngles}

\end{lemma}

\begin{proof}

We may assume, after a series of reflections and/or $60^{\circ}$ rotations, that the line segments in $\partial A^{\varhexagon}$ and $\partial B^{\varhexagon}$ that form $\theta_2$ are contained in $\Lambda_{\smallnearrow}\left(s_{\smallnearrow}^A\right)$ and $\Lambda_{\smallsearrow}\left(s_{\smallsearrow}^B\right)$, respectively. This is as in Figure \ref{fig:OneCornerTwoHexagons}. This means that there are two possibilities for the formation of $\theta_1$. Either it is formed by the intersection of the lines $\Lambda_{\smallsearrow}\left(i_{\smallsearrow}^A\right)$ and $\Lambda_{\smallnearrow}\left(i_{\smallnearrow}^B\right)$ (this is again as in Figure \ref{fig:OneCornerTwoHexagons}), or the intersection of the lines $\Lambda_{\smallnearrow}\left(s_{\smallnearrow}^A\right)$ and $\Lambda_{\smallrightarrow}\left(i_{\smallrightarrow}^B\right)$. 

As mentioned above, we assume that the joint boundary of $\left(A^{\varhexagon}\setminus B^{\varhexagon},B^{\varhexagon}\right)$ is greater than the joint boundary of $\left(A^{\varhexagon},B^{\varhexagon}\setminus A^{\varhexagon}\right)$. In this case, we replace $B$ with $B^{\varhexagon}$ (we know that the perimeter of $B^{\varhexagon}$ is no more than that of $B$), and we want to argue that there is enough uncounted perimeter to make the boundary of $A^{\varhexagon}\setminus B^{\varhexagon}$ that does not form joint boundary with $B^{\varhexagon}$. 

We use the notation introduced before this lemma. We only need to do the proof for $\theta_2$ as the process is the same for both angles; the process we perform for $\theta_2$ does not affect the process for $\theta_1$, and vice versa. Therefore, we include only the proof for $\theta_2$ with the understanding that, after an appropriate change of notation, the proof is the same for $\theta_1$. So, suppose that $\theta_2=120^{\circ}$. The situation resembles something like Figure \ref{fig:OneCornerTwoHexagons}. Note, however, that the angle formed by $a_2$ and $l_2^A$ could be either $60^{\circ}$ or $120^{\circ}$. Similarly for $b_2$ and $l_2^B$. 

Now, since $B$ has already been replaced by $B^{\varhexagon}$, consider the sides of $A^{\varhexagon}$ that do not intersect $B^{\varhexagon}$. It follows from Lemma \ref{lem:geodesics} that since $\partial A$ must touch each of these line segments in at least one point, the path in $\partial A^{\varhexagon}$ that connects them (that doesn't intersect $B^{\varhexagon}$) can be no longer than the path in $\partial A$ that connects them (and doesn't intersect $B^{\varhexagon}$). See the lefthand side of the Figure \ref{fig:OneCornerConstruction1}.

\begin{figure}[H]

\begin{tikzpicture}[scale=0.4]

\draw[blue, thin] (-1,0) to (-0.4,0);
\draw[blue, thin, dashed] (1,0) to (3,3.4641);
\draw[blue, thin, dashed] (3,3.4641) to (1,6.9282);
\draw[blue, thin, dashed] (-0.4,0) to (1,0);
\draw[blue, thin] (0.1,6.9282) to (-1,6.9282);
\draw[blue, thin, dashed] (0.1,6.9282) to (1,6.9282);
\draw[blue, thin] (-1,6.9282) to (-3,3.4641);
\draw[blue, thin] (-3,3.4641) to (-1,0);

\draw[red, thin] (3.7,-0.5) to (5.7,-0.5);
\draw[red, thin] (5.7,-0.5) to (7.7,2.9641);
\draw[red, thin] (7.7,2.9641) to (5.7,6.4282);
\draw[red, thin] (5.7,6.4282) to (3.7,6.4282);
\draw[red, thin] (3.7,6.4282) to (1.7,2.9641);
\draw[red, thin] (1.7,2.9641) to (3.7,-0.5);

\draw (-0.1,3.5) node{$A^{\varhexagon}$};
\draw (4.3,3) node{$B^{\varhexagon}$};

\draw (-0.4,0) node{$\bullet$};
\draw (-0.4,-0.8) node{$p\left(s_{\smallrightarrow}^A\right)$};

\draw (0.1,6.9282) node{$\bullet$};
\draw (0.1,7.7) node{$p\left(i_{\smallrightarrow}^A\right)$};

\draw (-2,5.15) node{$\bullet$};
\draw (-4,5.15) node{$p\left(i_{\smallnearrow}^A\right)$};

\draw (-2.5,2.58) node{$\bullet$};
\draw (-4.7,2.58) node{$p\left(s_{\smallsearrow}^A\right)$};

\draw[->] (8.6,2.9) to (9.4,2.9);


\draw[blue, thin] (13,0) to (13.6,0);

\draw[black, thin, dashed] (14.5,0) to (16,2.59808);

\draw[blue, thin, dashed] (16.2,2.07846) to (17,3.4641);

\draw[blue, thin, dashed] (17,3.4641) to (15,6.9282);
\draw[blue, thin, dashed] (13.6,0) to (15,0);
\draw[blue, thin] (14.1,6.9282) to (13,6.9282);
\draw[blue, thin, dashed] (14.1,6.9282) to (15,6.9282);
\draw[blue, thin] (13,6.9282) to (11,3.4641);
\draw[blue, thin] (11,3.4641) to (13,0);

\draw[red, thin] (17.7,-0.5) to (19.7,-0.5);
\draw[red, thin] (19.7,-0.5) to (21.7,2.9641);
\draw[red, thin] (21.7,2.9641) to (19.7,6.4282);
\draw[red, thin] (19.7,6.4282) to (17.7,6.4282);
\draw[red, thin] (17.7,6.4282) to (15.7,2.9641);
\draw[red, thin] (15.7,2.9641) to (17.7,-0.5);



\draw [black,thin] plot [smooth, tension=0.2] coordinates {(15.1,1.03923) (14,1.2)(10,1)};
\draw (9.6,1) node{$p$};

\draw (15.1,1.03923) node{$\bullet$};

\draw (13.6,0) node{$\bullet$};
\draw (12.9,-0.8) node{$p\left(s_{\smallrightarrow}^A\right)$};




\draw[->] (22.6,2.9) to (23.4,2.9);


\draw[black, thick, dotted] (24,8.66025) to (32,-5.1961524);

\draw[blue, thin] (27,0) to (27.6,0);

\draw[blue, thin] (28.5,0) to (29.96,2.54);

\draw[blue, thin, dashed] (29,0) to (31,3.4641);


\draw[blue, thin, dashed] (31,3.4641) to (29,6.9282);
\draw[blue, thin] (27.6,0) to (28.5,0);
\draw[blue, thick, dotted] (28.5,0) to (29,0);
\draw[blue, thin] (28.1,6.9282) to (27,6.9282);
\draw[blue, thin, dashed] (28.1,6.9282) to (29,6.9282);
\draw[blue, thin] (27,6.9282) to (25,3.4641);
\draw[blue, thin] (25,3.4641) to (27,0);

\draw[red, thin] (31.7,-0.5) to (33.7,-0.5);
\draw[red, thin] (33.7,-0.5) to (35.7,2.9641);
\draw[red, thin] (35.7,2.9641) to (33.7,6.4282);
\draw[red, thin] (33.7,6.4282) to (31.7,6.4282);
\draw[red, thin] (31.7,6.4282) to (29.7,2.9641);
\draw[red, thin] (29.7,2.9641) to (31.7,-0.5);

\draw (32.15,-4.5) node{$\Lambda$};

\end{tikzpicture}

\caption{\label{fig:OneCornerConstruction1}}

\end{figure}

Here we are using $p\left(i_{\smallrightarrow}^A\right)$ to represent a point where $\partial A$ intersects $\Lambda_{\smallrightarrow}\left(i_{\smallrightarrow}^A\right)$, $p\left(i_{\smallnearrow}^A\right)$ to represent a point where $\partial A$ intersects $\Lambda_{\smallnearrow}\left(i_{\smallnearrow}^A\right)$, etc. The dashed lines in the lefthand configuration in Figure \ref{fig:OneCornerConstruction1} represent what was the remaining part of $\partial A^{\varhexagon}$ but for which we have not yet found enough perimeter in $\left(A,B\right)$ to form.

Now, recalling the notation shown in Figure \ref{fig:OneCornerTwoHexagons}, we translate $l_2^A$ up and to the left so that one of its endpoints is always contained in the line segment in $\partial B^{\varhexagon}$ consisting of $l_2^B\cup l_2^{B_{ext}}$. We do this until the line passing through $l_2^A$ intersects $\partial A$ at some point $p$ (which may be before we move $l_2^A$ at all). Notice first that this may increase the length of $l_2^A$. And notice second that we will not have to slide $l_2^A$ passed the end of $l_2^B\cup l_2^{B_{ext}}$ because we know that eventually $\partial A$ must intersect $\Lambda_{\smallnearrow}\left(i_{\smallnearrow}^B\right)$ (the discussion at the beginning of this section showed that we only need to analyse configurations that have joint boundary of positive length). Let's denote this translated version of $l_2^A$ by $\tilde{l}_2^A$. Then, we have created a geodesic from $p\left(s_{\smallrightarrow}^A\right)$ to $p$ that is partially in $a_2$ and partially in $\tilde{l}_2^A$. We also know that there is some path in $\partial A$ that connects $p$ to $\Lambda_{\smallsearrow}\left(s_{\smallsearrow}^B\right)$, and this path must be at least as long as the path in $\tilde{l}_2^A$ that connects this same point and line, which is a geodesic. The middle configuration in Figure \ref{fig:OneCornerConstruction1} illustrates the translation of $l_2^A$.

We now have a new path from $p\left(s_{\smallrightarrow}^A\right)$ to $l_2^B\cup l_2^{B_{ext}}$, and we want to argue that it actually has the same length as the original path in $\partial A^{\varhexagon}$ that connected this same point and line. To this end, let $\Lambda$ be the line with slope $-\sqrt{3}$ that passes through the intersection of $\Lambda_{\smallrightarrow}\left(s_{\smallrightarrow}^A\right)$ and $\Lambda_{\smallnearrow}\left(s_{\smallnearrow}^A\right)$. See the righthand side of Figure \ref{fig:OneCornerConstruction1}.

Notice that, up to now, we have moved $l_2^A$ up and broken $a_2$ at an appropriate point and reoriented it so that it has slope $\sqrt{3}$. We know that this reoriented piece of $a_2$ has not changed length because, along with its original orientation as well as $\Lambda$, it forms a triangle whose interior angles are all $60^{\circ}$, whence the side lengths of this triangle are all equal. So, we know that we can reorient this broken piece of $a_2$ and move $l_2^A$ back to its original location without increasing the double bubble perimeter. 

Note that this argument still works if the angle between $a_2$ and $l_2^A$ is $60^{\circ}$. In this case the boundary of $A$ must intersect the point of intersection of $a_2$ and $l_2^A$, which means that we can make this point $p\left(s_{\smallrightarrow}^A\right)$. Therefore, the line $l_2^A$ is a geodesic between $p\left(s_{\smallrightarrow}^A\right)$ and $l_2^B$. That is, the path in the boundary of $\partial A$ that connects $p\left(s_{\smallrightarrow}^A\right)$ to $\Lambda_{\smallsearrow}\left(s_{\smallsearrow}^B\right)$ must be at least as long as $l_2^A$. 

Since we are assuming that $\theta_1=120^{\circ}$ we can make the same argument, with appropriate changes in notation, to show that there is still enough uncounted perimeter in $A\setminus B^{\varhexagon}$ to create the configuration $\left(A^{\varhexagon}\setminus B^{\varhexagon},B^{\varhexagon}\right)$ without increasing the double bubble perimeter.

\end{proof}

The following lemma shows that we can take the configuration resulting from Lemma \ref{lem:Two120DegreeAngles} and adjust the volumes, if necessary, to create a new configuration that has the correct volume ratio. 

\begin{lemma}\label{lem:Two120AnglesVolAdj}
Let $\left(A,B\right)\in\gamma_{\alpha}$ be as in Lemma \ref{lem:Two120DegreeAngles}. Suppose, without loss of generality, that $\left(A^{\varhexagon}\setminus B^{\varhexagon},B^{\varhexagon}\right)$ has greater joint boundary than $\left(A^{\varhexagon},B^{\varhexagon}\setminus A^{\varhexagon}\right)$. Then we can adjust the volumes of the two sets in the former configuration to obtain a new configuration $\left(\tilde{A},\tilde{B}\right)\in\mathfs{F}_{\alpha}$, with $\rho_{DB}\left(\tilde{A},\tilde{B}\right)\leq\rho_{DB}\left(A,B\right)$. 

\end{lemma}


 \begin{proof}

After a sequence of reflections and/or $60^{\circ}$ rotations, we can assume that $\theta_2$ is formed by the two lines $\Lambda_{\smallnearrow}\left(s_{\smallnearrow}^A\right)$ and $\Lambda_{\smallsearrow}\left(s_{\smallsearrow}^B\right)$. The angle $\theta_1$ must also be formed by the intersection of two lines, and there are now only two possibilities. Either $\theta_1$ is formed by the two lines $\Lambda_{\smallsearrow}\left(i_{\smallsearrow}^A\right)$ and $\Lambda_{\smallnearrow}\left(i_{\smallnearrow}^B\right)$, or it is formed by the lines $\Lambda_{\smallnearrow}\left(s_{\smallnearrow}^A\right)$ and $\Lambda_{\smallrightarrow}\left(i_{\smallrightarrow}^B\right)$. Figure \ref{fig:Two120sVolAdj} demonstrates these two options:

\begin{figure}[H]

\begin{tikzpicture}[scale=0.2]

\draw[blue, thin] (0,0) to (2,3.464101615);
\draw[blue, thin] (2,3.464101615) to (4,3.464101615);
\draw[blue, thin] (4,3.464101615) to (6,0);
\draw[blue, thin] (6,0) to (4,-3.464101615);
\draw[blue, thin] (4,-3.464101615) to (2,-3.464101615);
\draw[blue, thin] (2,-3.464101615) to (0,0);

\draw[red, thin] (4,0) to (7,5.19615242);
\draw[red, thin] (7,5.19615242) to (9,5.19615242);
\draw[red, thin] (9,5.19615242) to (12,0);
\draw[red, thin] (12,0) to (9,-5.19615242);
\draw[red, thin] (9,-5.19615242) to (7,-5.19615242);
\draw[red, thin] (7,-5.19615242) to (4,0);

\draw[blue, thin] (15,0) to (18,5.19615242);
\draw[blue, thin] (18,5.19615242) to (22,5.19615242);
\draw[blue, thin] (22,5.19615242) to (23,3.464101615);
\draw[blue, thin] (23,3.464101615) to (18,-5.19615242);
\draw[blue, thin] (18,-5.19615242) to (17,-5.19615242);
\draw[blue, thin] (17,-5.19615242) to (15,0);

\draw[red, thin] (19,-1.7320508) to (21,1.7320508);
\draw[red, thin] (21,1.7320508) to (23,1.7320508);
\draw[red, thin] (23,1.7320508) to (24,0);
\draw[red, thin] (24,0) to (21,-5.19615242);
\draw[red, thin ](21,-5.19615242) to (20,-5.19615242);
\draw[red, thin] (20,-5.19615242) to (19,-1.7320508);

\end{tikzpicture}

\caption{\label{fig:Two120sVolAdj}}

\end{figure}

In both configurations we can analyse $\left(A^{\varhexagon}\setminus B^{\varhexagon},B^{\varhexagon}\right)$. Let us first analyse the configuration on the left. Note that $\mu\left(B^{\varhexagon}\right)\geq\mu\left(B\right)$, and that $\mu\left(A^{\varhexagon}\cup B^{\varhexagon}\right)\geq\mu\left(A\cup B\right)$. Thus, we argue that we can reduce the volume of $B^{\varhexagon}$ to obtain a set of volume $\alpha$. Also, all volume that is removed from $A^{\varhexagon}\cap B^{\varhexagon}$ is added to $A^{\varhexagon}\setminus B^{\varhexagon}$. This will guarantee that we have a set of volume $\alpha$, and another set of volume at least $1$. This latter set we then adjust until we have a set of volume $1$. So, we always begin by reducing the volume of $B^{\varhexagon}$ first. 

There are essentially two situations to examine. Either $s_{\smallrightarrow}^A\geq s_{\smallrightarrow}^B$ and $i_{\smallrightarrow}^A\leq i_{\smallrightarrow}^B$ (the highest point of $A^{\varhexagon}$ is higher than the highest point of $B^{\varhexagon}$ and the lowest point of $A^{\varhexagon}$ is lower than the lowest point of $B^{\varhexagon}$), or $s_{\smallrightarrow}^A>s_{\smallrightarrow}^B$ and $i_{\smallrightarrow}^A\geq i_{\smallrightarrow}^B$ (the highest point in $A^{\varhexagon}$ is higher than the highest point in $B^{\varhexagon}$ and the lowest point in $A^{\varhexagon}$ is also higher than the lowest point in $B^{\varhexagon}$). 

Let us take the first situation. With reference to Figure \ref{fig:OneCornerVolumeAdjustment1}, we merely have to reorient $l_1^B$ so that it is horizontal (if it isn't already), has one endpoint in $\partial A^{\varhexagon}\setminus B^{\varhexagon}$, and is contained in $\Lambda_{\smallrightarrow}\left(i_{\smallrightarrow}^B\right)$. Similarly, we reorient $l_2^B$. The following figure demonstrates this process:

\begin{figure}[H]

\begin{tikzpicture}[scale=0.2]

\draw[red, thin] (0,0) to (2,3.4641016);
\draw[red, thin] (2,3.4641016) to (4,3.4641016);
\draw[red, thin] (4,3.4641016) to (6,0);
\draw[red, thin] (6,0) to (4,-3.4641016);
\draw[red, thin] (4,-3.4641016) to (2,-3.4641016);
\draw[red, thin] (2,-3.4641016) to (0,0);

\draw[blue, thin] (1,1.7320508) to (-1,5.1961524);
\draw[blue, thin] (-1,5.1961524) to (-3,5.1961524);
\draw[blue, thin] (-3,5.1961524) to (-6,0);
\draw[blue, thin] (-6,0) to (-3,-5.1961524);
\draw[blue, thin] (-3,-5.1961524) to (-1,-5.1961524);
\draw[blue, thin] (-1,-5.1961524) to (1,-1.7320508);

\draw[->] (6.5,0) to (7.5,0);

\draw[red, thin] (14,0) to (15,1.7320508);
\draw[red, thin] (14,3.4641016) to (16,3.4641016);
\draw[red, thin] (16,3.4641016) to (18,3.4641016);
\draw[red, thin] (18,3.4641016) to (20,0);
\draw[red, thin] (20,0) to (18,-3.4641016);
\draw[red, thin] (18,-3.4641016) to (16,-3.4641016);
\draw[red, thin] (16,-3.4641016) to (14,-3.4641016);
\draw[red, thin] (15,-1.7320508) to (14,0);

\draw[blue, thin] (15,1.7320508) to (13,5.1961524);
\draw[blue, thin] (13,5.1961524) to (11,5.1961524);
\draw[blue, thin] (11,5.1961524) to (8,0);
\draw[blue, thin] (8,0) to (11,-5.1961524);
\draw[blue, thin] (11,-5.1961524) to (13,-5.1961524);
\draw[blue, thin] (13,-5.1961524) to (15,-1.7320508);

\end{tikzpicture}

\caption{\label{fig:OneCornerVolumeAdjustment1}}

\end{figure}


We no longer have $B^{\varhexagon}$, so we rename this new set $B_1$. Now, we rearrange the joint boundary by replacing $B_1$ with $B_1^{\varhexagon}$. We began with a set $B^{\varhexagon}$ whose volume was too great, and $B_1^{\varhexagon}$ has even greater volume. So, we move $\Lambda_{\smallsearrow}\left(s_{\smallsearrow}^{B_1^{\varhexagon}}\right)$ and $\Lambda_{\smallnearrow}\left(i_{\smallnearrow}^{B_1^{\varhexagon}}\right)$ to the right until we have a set whose volume is the same as the volume of $B$. See Figure \ref{fig:OneCornerVolumeAdjustment2}. We no longer have $B_1$ or $B_1^{\varhexagon}$, so we rename this new set $\tilde{B}$. Note that by this construction $\tilde{B}=\tilde{B}^{\varhexagon}$. We also no longer have $A^{\varhexagon}\setminus B^{\varhexagon}$. So, we call this new set $A_1$. We then replace $A_1$ with $A_1^{\varhexagon}\setminus\tilde{B}$. The following figure shows this process. 

\begin{figure}[H]

\begin{tikzpicture}[scale=0.2]

\draw[red, thin] (0,0) to (1,1.7320508);
\draw[red, thin] (4,3.4641016) to (0,3.4641016);
\draw[red, thin] (4,3.4641016) to (6,0);
\draw[red, thin] (6,0) to (4,-3.4641016);
\draw[red, thin] (4,-3.4641016) to (2,-3.4641016);
\draw[red, thin] (2,-3.4641016) to (0,-3.4641016);
\draw[red, thin] (1,-1.7320508) to (0,0);

\draw[blue, thin] (1,1.7320508) to (-1,5.1961524);
\draw[blue, thin] (-1,5.1961524) to (-3,5.1961524);
\draw[blue, thin] (-3,5.1961524) to (-6,0);
\draw[blue, thin] (-6,0) to (-3,-5.1961524);
\draw[blue, thin] (-3,-5.1961524) to (-1,-5.1961524);
\draw[blue, thin] (-1,-5.1961524) to (1,-1.7320508);

\draw (-2.2,0) node{$\scaleobj{0.5}{A^{\varhexagon}\setminus B^{\varhexagon}}$};
\draw (2.6,0) node{$\scaleobj{0.5}{B_1}$};

\draw[->] (6.5,0) to (7.5,0);


\draw[red, thin] (12,0) to (14,3.4641016);
\draw[red, thin] (14,3.4641016) to (16,3.4641016);
\draw[red, thin] (16,3.4641016) to (18,3.4641016);
\draw[red, thin] (18,3.4641016) to (20,0);
\draw[red, thin] (20,0) to (18,-3.4641016);
\draw[red, thin] (18,-3.4641016) to (16,-3.4641016);
\draw[red, thin] (16,-3.4641016) to (14,-3.4641016);
\draw[red, thin] (14,-3.4641016) to (12,0);

\draw[blue, thin] (14,3.4641016) to (13,5.1961524);
\draw[blue, thin] (13,5.1961524) to (11,5.1961524);
\draw[blue, thin] (11,5.1961524) to (8,0);
\draw[blue, thin] (8,0) to (11,-5.1961524);
\draw[blue, thin] (11,-5.1961524) to (13,-5.1961524);
\draw[blue, thin] (13,-5.1961524) to (14,-3.4641016);

\draw (10.1,0) node{$\scaleobj{0.5}{A^{\varhexagon}\setminus B_1^{\varhexagon}}$};
\draw (15.6,0) node{$\scaleobj{0.5}{B_1^{\varhexagon}}$};

\draw[->] (20.5,0) to (21.5,0);


\draw[red, thin] (27,0) to (29,3.4641016);
\draw[red, thin] (29,3.4641016) to (30,3.4641016);
\draw[red, thin] (30,3.4641016) to (32,3.4641016);
\draw[red, thin] (32,3.4641016) to (34,0);
\draw[red, thin] (34,0) to (32,-3.4641016);
\draw[red, thin] (32,-3.4641016) to (30,-3.4641016);
\draw[red, thin] (30,-3.4641016) to (28,-3.4641016);
\draw[red, thin] (29,-3.4641016) to (27,0);

\draw[blue, thin] (28,3.4641016) to (29,3.4641016);
\draw[blue, thin] (28,3.4641016) to (27,5.1961524);
\draw[blue, thin] (27,5.1961524) to (25,5.1961524);
\draw[blue, thin] (25,5.1961524) to (22,0);
\draw[blue, thin] (22,0) to (25,-5.1961524);
\draw[blue, thin] (25,-5.1961524) to (27,-5.1961524);
\draw[blue, thin] (27,-5.1961524) to (28,-3.4641016);
\draw[blue, thin] (28,-3.4641016) to (29,-3.4641016);

\draw (24.8,0) node{$\scaleobj{0.5}{A_1}$};
\draw (30.2,0) node{$\scaleobj{0.5}{\tilde{B}}$};

\draw[->] (34.5,0) to (35.5,0);


\draw[red, thin] (41,0) to (43,3.4641016);
\draw[red, thin] (43,3.4641016) to (44,3.4641016);
\draw[red, thin] (44,3.4641016) to (46,3.4641016);
\draw[red, thin] (46,3.4641016) to (48,0);
\draw[red, thin] (48,0) to (46,-3.4641016);
\draw[red, thin] (46,-3.4641016) to (44,-3.4641016);
\draw[red, thin] (44,-3.4641016) to (43,-3.4641016);
\draw[red, thin] (43,-3.4641016) to (41,0);

\draw[blue, thick, dotted] (42,3.4641016) to (43,3.4641016);
\draw[blue, thin] (41,5.1961524) to (42,5.1961524);
\draw[blue, thick, dotted] (42,3.4641016) to (41,5.1961524);
\draw[blue, thin] (43,3.4641016) to (42,5.1961524);

\draw[blue, thin] (41,5.1961524) to (39,5.1961524);
\draw[blue, thin] (39,5.1961524) to (36,0);
\draw[blue, thin] (36,0) to (39,-5.1961524);
\draw[blue, thin] (39,-5.1961524) to (41,-5.1961524);

\draw[blue, thick, dotted] (41,-5.1961524) to (42,-3.4641016);
\draw[blue, thin] (42,-5.1961524) to (43,-3.4641016);

\draw[blue, thin] (41,-5.1961524) to (42,-5.1961524);
\draw[blue, thick, dotted] (42,-3.4641016) to (43,-3.4641016);

\draw (38.8,0) node{$\scaleobj{0.5}{A_1^{\varhexagon}\setminus\tilde{B}}$};
\draw (44.2,0) node{$\scaleobj{0.5}{\tilde{B}}$};

\end{tikzpicture}

\caption{\label{fig:OneCornerVolumeAdjustment2}}

\end{figure}

Notice that $A_1^{\varhexagon}\setminus\tilde{B}$ may now have too much volume. However, we can decrease its volume by moving its bottom side up until, if necessary, it is collinear with $\Lambda_{\smallrightarrow}\left(s_{\smallrightarrow}^{\tilde{B}}\right)\cap\partial\tilde{B}$. Then, if necessary, we can move its top side down until it is collinear with $\Lambda_{\smallrightarrow}\left(i_{\smallrightarrow}^{\tilde{B}}\right)$. Finally, if necessary, we can move its left sides to the right until we have a set $\tilde{A}$ of volume $1$. Note that $\left(A_1^{\varhexagon}\setminus\tilde{B},\tilde{B}\right)$ is of the correct form to belong to some $\mathfs{F}_{\alpha'}$, except that possibly $\mu\left(A_1^{\varhexagon}\setminus\tilde{B}\right)>1$. All of the operations used on $A_1^{\varhexagon}\setminus\tilde{B}$ to obtain $\tilde{A}$ result in a $\tilde{A}$ such that $\left(\tilde{A},\tilde{B}\right)\in\mathfs{F}_{\alpha}$. 

On the other hand, suppose that $i_{\smallrightarrow}^A>i_{\smallrightarrow}^B$, and $s_{\smallrightarrow}^A\geq s_{\smallrightarrow}^B$. See Figure \ref{fig:OneCornerVolumeAdjustment3} for illustration. In this case, we begin by reorienting $l_1^B$ so that its endpoint that is not in the joint boundary remains the same, and so that it is contained in $\Lambda_{\smallrightarrow}\left(i_{\smallrightarrow}^B\right)$. Similarly, we reorient $l_2^A$. This results in two new sets that we call $A_1$ and $B_1$. See Figure \ref{fig:OneCornerVolumeAdjustment3}:
\begin{figure}[H]

\begin{tikzpicture}[scale=0.2]

\draw[red, thin] (-2,3.4641016) to (5,3.4641016);
\draw[red, thick, dotted] (-1,1.7320508) to (0,3.4641016);
\draw[red, thin] (5,3.4641016) to (7,0);
\draw[red, thin] (7,0) to (4,-5.1961524);
\draw[red, thin] (4,-5.1961524) to (1,-5.1961524);
\draw[red, thin] (1,-5.1961524) to (-1,-1.7320508);
\draw[red, thin] (-1,1.7320508) to (-2,0);
\draw[red, thin] (-2,0) to (-1,-1.7320508);

\draw[blue, thin] (-1,1.7320508) to (-3,5.1961524);
\draw[blue, thin] (-3,5.1961524) to (-6,5.1961524);
\draw[blue, thin] (-6,5.1961524) to (-9,0);
\draw[blue, thin] (-9,0) to (-7,-3.4641016);
\draw[blue, thin] (-7,-3.4641016) to (-2,-3.4641016);

\draw[blue, thin] (-2,-3.4641016) to (0,-3.4641016);
\draw[blue, thick, dotted] (-2,-3.4641016) to (-1,-1.7320508);


\draw (-5,0) node{$\scaleobj{0.5}{A_1}$};
\draw (2,0) node{$\scaleobj{0.5}{B_1}$};

\end{tikzpicture}

\caption{\label{fig:OneCornerVolumeAdjustment3}}

\end{figure}
We now replace $B_1$ and $A_1$ by, respectively, $B_1^{\varhexagon}$ and $A_1\setminus B_1^{\varhexagon}$. Then, we move the part of $\partial B_1^{\varhexagon}$ that is contained in $\Lambda_{\smallnearrow}\left(i_{\smallnearrow}^{B_1^{\varhexagon}}\right)$ to the right so that we simultaneously decrease the volume of $B_1^{\varhexagon}$ and increase the volume of $A_1$. This process decreases the length of the line segment $\partial B_1^{\varhexagon}\cap \Lambda_{\smallrightarrow}\left(i_{\smallrightarrow}^{B_1^{\varhexagon}}\right)$ and adds it to $\partial A_1$. We can do this until we get a set $A_2$ to replace $A_1\setminus B_1^{\varhexagon}$, and $B_2$ to replace $B_1^{\varhexagon}$ such that $\mu\left(B_2\right)\geq\alpha$, and $\mu\left(A_1\right)\geq1$. We know that we can do this because we can move $\partial B_1\cap\Lambda_{\smallnearrow}\left(i_{\smallnearrow}^{B_1^{\varhexagon}}\right)$ to the right either until our new set $B_2$ has volume $\alpha$, or until all of the volume of $A^{\varhexagon}\cap B^{\varhexagon}$ is contained in $A_2$, whichever comes first.

Similarly, $A_1\setminus B_1^{\varhexagon}$ has been replaced, and we call this new set $A_2$. Then, we replace $A_2$ with $A_2^{\varhexagon}\setminus B_2$. See the following figure for illustration:

\begin{figure}[H]

\begin{tikzpicture}[scale=0.2]

\draw[red, thin] (-2,3.4641016) to (5,3.4641016);
\draw[red, thick, dotted] (-1,1.7320508) to (-2,0);
\draw[red, thin] (5,3.4641016) to (7,0);
\draw[red, thin] (7,0) to (4,-5.1961524);
\draw[red, thin] (4,-5.1961524) to (1,-5.1961524);
\draw[red, thin] (1,-5.1961524) to (-1,-1.7320508);
\draw[red, thin] (-2,0) to (-1,-1.7320508);
\draw[red, thin] (-2,3.4641016) to (-3,1.7320508);
\draw[red, thin] (-2,0) to (-3,1.7320508);

\draw[red, thick, dotted] (-1,1.7320508) to (-2,0);

\draw[blue, thin] (-2,3.4641016) to (-3,5.1961524);
\draw[blue, thick, dotted] (-1,1.7320508) to (-2,3.4641016);

\draw[blue, thin] (-3,5.1961524) to (-6,5.1961524);
\draw[blue, thin] (-6,5.1961524) to (-9,0);
\draw[blue, thin] (-9,0) to (-7,-3.4641016);
\draw[blue, thin] (-7,-3.4641016) to (-2,-3.4641016);

\draw[blue, thin] (-2,-3.4641016) to (0,-3.4641016);


\draw (-4.6,0) node{$\scaleobj{0.5}{A_1\setminus B_1^{\varhexagon}}$};
\draw (1.4,0) node{$\scaleobj{0.5}{B_1^{\varhexagon}}$};

\draw[->] (7.5,0) to (8.5,0);


\draw[red, thin] (18,3.4641016) to (23,3.4641016);

\draw[red, thin] (18,3.4641016) to (16,0);

\draw[red, thin] (23,3.4641016) to (25,0);
\draw[red, thin] (25,0) to (22,-5.1961524);
\draw[red, thin] (22,-5.1961524) to (19,-5.1961524);
\draw[red, thin] (19,-5.1961524) to (17,-1.7320508);
\draw[red, thin] (16,0) to (17,-1.7320508);
\draw[red, thick, dotted] (16,3.4641016) to (15,1.7320508);
\draw[red, thick, dotted] (16,0) to (15,1.7320508);

\draw[blue, thin] (16,3.4641016) to (15,5.1961524);

\draw[blue, thin] (18,3.4641016) to (16,3.4641016);

\draw[blue, thin] (15,5.1961524) to (12,5.1961524);
\draw[blue, thin] (12,5.1961524) to (9,0);
\draw[blue, thin] (9,0) to (11,-3.4641016);
\draw[blue, thin] (11,-3.4641016) to (16,-3.4641016);

\draw[blue, thin] (16,-3.4641016) to (18,-3.4641016);


\draw[->] (25.5,0) to (26.5,0);


\draw[red, thin] (40,3.46410161) to (41,3.46410161);

\draw[red, thin] (40,3.4641016) to (36,-3.46410161);

\draw[red, thin] (41,3.4641016) to (43,0);

\draw[red, thin] (43,0) to (40,-5.1961524);
\draw[red, thin] (40,-5.1961524) to (37,-5.1961524);
\draw[red, thin] (37,-5.1961524) to (36,-3.46410161);
\draw[red, thick, dotted] (36,-3.46410161) to (34,0);
\draw[red, thick, dotted] (34,0) to (36,3.46410161);



\draw[blue, thin] (34,3.4641016) to (33,5.1961524);

\draw[blue, thin] (33,5.1961524) to (30,5.1961524);
\draw[blue, thin] (30,5.1961524) to (27,0);
\draw[blue, thin] (27,0) to (29,-3.4641016);
\draw[blue, thin] (29,-3.4641016) to (34,-3.4641016);

\draw[blue, thin] (34,-3.4641016) to (36,-3.4641016);


\draw[blue, thin] (34,3.4641016) to (40,3.4641016);

\draw (31.5,0) node{$\scaleobj{0.5}{A_2}$};
\draw (39,0) node{$\scaleobj{0.5}{B_2}$};

\draw[->] (43.5,0) to (44.5,0);


\draw[red, thin] (58,3.4641016) to (59,3.4641016);
\draw[red, thin] (58,3.4641016) to (54,-3.46410161);

\draw[red, thin] (59,3.4641016) to (61,0);
\draw[red, thin] (61,0) to (58,-5.1961524);
\draw[red, thin] (58,-5.1961524) to (55,-5.1961524);
\draw[red, thin] (55,-5.1961524) to (54,-3.46410161);

\draw[blue, thick, dotted] (52,3.4641016) to (51,5.1961524);
\draw[blue, thick, dotted] (52,3.46410161) to (58,3.46410161);
\draw[blue, thin] (57,5.1961524) to (58,3.46410161);

\draw[blue, thin] (57,5.1961524) to (51,5.1961524);
\draw[blue, thin] (51,5.1961524) to (48,5.1961524);
\draw[blue, thin] (48,5.1961524) to (45,0);
\draw[blue, thin] (45,0) to (47,-3.4641016);
\draw[blue, thin] (47,-3.4641016) to (54,-3.4641016);

\draw[blue, thin] (34,-3.4641016) to (36,-3.4641016);


\draw (50,0) node{$\scaleobj{0.5}{A_2^{\varhexagon}\setminus B_2}$};
\draw (57,0) node{$\scaleobj{0.5}{B_2}$};


\end{tikzpicture}

\caption{\label{fig:OneCornerVolumeAdjustment4}}

\end{figure}

If the volume of $B_2$ is still too much, we can move its bottom side up to form a new set $B_3$ that either has volume $\alpha$, or such that $s_{\smallrightarrow}^{B_3}=s_{\smallrightarrow}^{A_2}$, whichever comes first. Finally, if necessary, we can move the right sides of $B_3$ to the left until we have a set $\tilde{B}$ of volume $\alpha$. 

To adjust the volume of $A_2^{\varhexagon}\setminus B_2$, we can lower its top side until we get a new set $A_3$ that either has volume $1$, or such that $i_{\smallrightarrow}^{A_3}=i_{\smallrightarrow}^{B_3}$, whichever comes first. We can then move the left sides of $A_3$ to the right to obtain a set $\tilde{A}$ with volume $1$. 

We still need to argue that we can adjust the volumes of the sets in the right side of Figure \ref{fig:Two120sVolAdj}. This, however, is simple. We begin with the configuration $\left(A^{\varhexagon}\setminus B^{\varhexagon},B^{\varhexagon}\right)$. Then, we can move the line segment $\Lambda_{\smallnearrow}\left(i_{\smallnearrow}^B\right)$ to the right, if necessary, until it is contained in $\Lambda_{\smallnearrow}\left(s_{\smallnearrow}^A\right)$. This creates a new set, call it $B_1$, in place of $B^{\varhexagon}$. If the volume of $B_1$ is still too great, then we can move its right sides to the left until we have a set $\tilde{B}$ of volume $\alpha$. We may now have to adjust the volume of $A^{\varhexagon}\setminus\tilde{B}$. To do this, we can lower the top side of $A^{\varhexagon}\setminus\tilde{B}$ until it is level with the top of $\tilde{B}$. We can then raise the bottom side of this set until it is level with the bottom side of $\tilde{B}$, and finally, if necessary, we can move its left sides to the right until we have a set $\tilde{A}$ of volume $1$.

\end{proof}

We now discuss the situation when exactly one of $\theta_1$ or $\theta_2$ is $60^{\circ}$. 

\begin{lemma}\label{lemma:ExactlyOneSixty}
Let $\left(A,B\right)\in\gamma_{\alpha}$, and $\mu\left(A^{\varhexagon}\setminus B^{\varhexagon}\right),\mu\left(A^{\varhexagon}\cap B^{\varhexagon}\right),\mu\left(B^{\varhexagon}\setminus A^{\varhexagon}\right)>0$. Using the notation introduced before Lemma \ref{lem:Two120DegreeAngles}, suppose that exactly one of $\theta_1$ and $\theta_2$ is $60^{\circ}$, meaning that the other angle must be $120^{\circ}$. Then there is either a translation of $B^{\varhexagon}$, call it $B_t^{\varhexagon}$, or of $A^{\varhexagon}$, call it $A_t^{\varhexagon}$, such that $\rho_{DB}\left(A^{\varhexagon}\setminus B_t^{\varhexagon},B_t^{\varhexagon}\right)\leq\rho_{DB}\left(A,B\right)$, or $\rho_{DB}\left(A_t^{\varhexagon},B^{\varhexagon}\setminus A_t^{\varhexagon}\right)\leq\rho_{DB}\left(A,B\right)$.

\end{lemma}

\begin{proof}
As before, we begin with $\left(A,B\right)\in\gamma_{\alpha}$ and first must decide whether to analyse $\left(A^{\varhexagon}\setminus B^{\varhexagon},B^{\varhexagon}\right)$ or $\left(A^{\varhexagon},B^{\varhexagon}\setminus A^{\varhexagon}\right)$. We choose the configuration with the larger joint boundary. Let us assume that it is $\left(A^{\varhexagon}\setminus B^{\varhexagon},B^{\varhexagon}\right)$. After a sequence of reflections and/or $60^{\circ}$ rotations, we may assume that the $60^{\circ}$ degree angle is created by the intersection of the lines $\Lambda_{\smallrightarrow}\left(s_{\smallrightarrow}^A\right)$ and $\Lambda_{\smallsearrow}\left(s_{\smallsearrow}^B\right)$ (which means that it is $\theta_2$ that is the $60^{
\circ}$ angle). This leaves three possibilities for the two lines that together form $\theta_1$. The following figure demonstrates the options:

\begin{figure}[H]

\begin{tikzpicture}[scale=0.2]

\draw[blue, thin] (3,-5.19615242) to (8,-5.19615242);
\draw[blue, thin] (8,-5.19615242) to (9,-3.46410161);
\draw[blue, thin] (9,-3.46410161) to (4,5.19615242);
\draw[blue, thin] (4,5.19615242) to (3,5.19615242);
\draw[blue, thin] (3,5.19615242) to (1,5.19615242);
\draw[blue, thin] (1,5.19615242) to (-1,1.7320508);
\draw[blue, thin] (-1,1.7320508) to (1,-1.7320508);
\draw[blue, thin] (1,-1.7320508) to (3,-5.19615242);

\draw[red, thin] (4,0) to (7,5.19615242);
\draw[red, thin] (7,5.19615242) to (10,5.19615242);
\draw[red, thin] (10,5.19615242) to (13,0);
\draw[red, thin] (13,0) to (10,-5.19615242);
\draw[red, thin] (10,-5.19615242) to (7,-5.19615242);
\draw[red, thin] (7,-5.19615242) to (4,0);

\draw[blue, thin] (19,-1.7320508) to (23,-1.7320508);
\draw[blue, thin] (23,-1.7320508) to (28,6.92820323);
\draw[blue, thin] (28,6.92820323) to (27,8.66025403);
\draw[blue, thin] (27,8.66025403) to (19,8.66025403);
\draw[blue, thin] (19,8.66025403) to (16,3.46410161);
\draw[blue, thin] (16,3.46410161) to (19,-1.7320508);

\draw[red, thin] (21,0) to (24,5.19615242);
\draw[red, thin] (24,5.19615242) to (28,5.19615242);
\draw[red, thin] (28,5.19615242) to (31,0);
\draw[red, thin] (31,0) to (28,-5.19615242);
\draw[red, thin] (28,-5.19615242) to (24,-5.19615242);
\draw[red, thin] (24,-5.19615242) to (21,0);

\draw (20.5,4) node{$\scaleobj{0.5}{A^{\varhexagon}}$};
\draw (27,0) node{$\scaleobj{0.5}{B^{\varhexagon}}$};


\draw[red, thin] (38,1.7320508) to (40,5.1961524);
\draw[red, thin] (40,5.1961524) to (42,5.1961524);
\draw[red, thin] (42,5.1961524) to (46,-1.7320508);
\draw[red, thin] (46,-1.7320508) to (45,-3.46410161);
\draw[red, thin] (45,-3.46410161) to (41,-3.46410161);
\draw[red, thin] (41,-3.46410161) to (38,1.7320508);

\draw[blue, thin] (35,0) to (46,0);
\draw[blue, thin] (46,0) to (47,1.7320508);
\draw[blue, thin] (47,1.7320508) to (44,6.92820323);
\draw[blue, thin] (44,6.92820323) to (35,6.92820323);
\draw[blue, thin] (35,6.92820323) to (33,3.4641016);
\draw[blue, thin] (33,3.4641016) to (35,0);

\draw (35.5,4) node{$\scaleobj{0.5}{A^{\varhexagon}}$};
\draw (42,-1) node{$\scaleobj{0.5}{B^{\varhexagon}}$};


\end{tikzpicture}

\caption{\label{fig:One60One120}}

\end{figure}

As a side note, there is a slight variation on the configuration on the left of Figure \ref{fig:One60One120}, which occurs when the bottom line segment of $A^{\varhexagon}$ is above the bottom line segment of $B^{\varhexagon}$. In this case, the bottom line segment of $A^{\varhexagon}$ cannot extend into the interior of $B^{\varhexagon}$ because then the joint boundary of $\left(A^{\varhexagon},B^{\varhexagon}\setminus A^{\varhexagon}\right)$ would be greater than that of $\left(A^{\varhexagon}\setminus B^{\varhexagon},B^{\varhexagon}\right)$, and we would therefore analyse the former of these two configurations. This would result in a configuration (after some rotations of $60^{\circ}$) that is of the same form as the configuration in the middle of Figure \ref{fig:One60One120}. Figure \ref{fig:AnomalousOne60One120} gives an illustration of this variation of the lefthand side of Figure \ref{fig:One60One120}. The process that treats this case is similar, but simpler, to that for the lefthand side of Figure \ref{fig:One60One120}. 

\begin{figure}[H]

\begin{tikzpicture}[scale=0.2]

\draw[blue, thin] (1,-1.7320508) to (5,-1.7320508);
\draw[blue, thin] (5,-1.7320508) to (6,0);
\draw[blue, thin] (6,0) to (3,5.19615242);
\draw[blue, thin] (3,5.19615242) to (1,5.19615242);
\draw[blue, thin] (1,5.19615242) to (-1,1.7320508);
\draw[blue, thin] (-1,1.7320508) to (1,-1.7320508);

\draw[red, thin] (4,0) to (7,5.19615242);
\draw[red, thin] (7,5.19615242) to (9,5.19615242);
\draw[red, thin] (9,5.19615242) to (12,0);
\draw[red, thin] (12,0) to (9,-5.19615242);
\draw[red, thin] (9,-5.19615242) to (7,-5.19615242);
\draw[red, thin] (7,-5.19615242) to (4,0);

\draw (2,2) node{$\scaleobj{0.5}{A^{\varhexagon}}$};
\draw (8,0) node{$\scaleobj{0.5}{B^{\varhexagon}}$};

\end{tikzpicture}

\caption{\label{fig:AnomalousOne60One120}}

\end{figure}

In all situations, the beginning of the construction is the same as in Lemma \ref{lem:Two120DegreeAngles}. We begin by replacing $B$ with $B^{\varhexagon}$. Then, we consider each of the line segments in $\partial A^{\varhexagon}$ that do not intersect $B^{\varhexagon}$. There must be a point in each of these line segments that $\partial A$ intersects. The path connecting them that does not pass through $B^{\varhexagon}$ must be at least as long as the path in $\partial A^{\varhexagon}$ that connects them and doesn't pass through $B^{\varhexagon}$. Figure \ref{fig:One60One120DifficultCases1} illustrates the process up to this point.

\begin{figure}[H]

\begin{tikzpicture}[scale=0.3]

\draw[blue, thick, dotted] (3,-5.19615242) to (8,-5.19615242);
\draw[blue, thick, dotted] (8,-5.19615242) to (9,-3.46410161);
\draw[blue, thick, dotted] (9,-3.46410161) to (4,5.19615242);
\draw[blue, thick, dotted] (4,5.19615242) to (3,5.19615242);
\draw[blue, thin] (3,5.19615242) to (1,5.19615242);
\draw[blue, thin] (1,5.19615242) to (-1,1.7320508);
\draw[blue, thin] (-1,1.7320508) to (1,-1.7320508);
\draw[blue, thick, dotted] (1,-1.7320508) to (3,-5.19615242);

\draw (3,5.19615242) node{$\bullet$};
\draw[black, thin] (3,5.19615242) to (2,6);
\draw (1,6.7) node{$p\left(i_{\smallrightarrow}^A\right)$};

\draw (0,3.46410161) node{$\bullet$};
\draw[black, thin] (0,3.46410161) to (-1,4);
\draw (-2.5,4.2) node{$p\left(i_{\smallnearrow}^A\right)$};

\draw (1,-1.7320508) node{$\bullet$};
\draw[black, thin] (1,-1.7320508) to (-0.2,-1.7320508);
\draw (-2,-1.7320508) node{$p\left(s_{\smallsearrow}^A\right)$};

\draw (7.5,-5.19615242) node{$\bullet$};
\draw[black, thin] (7.5,-5.19615242) to (7.5,-6.5);
\draw (7.5,-7) node{$p\left(s_{\smallrightarrow}^A\right)$};

\draw[black, thin] (4,-5.19615242) to (3,-6.5);
\draw (2.8,-7.4) node{$l_2^A$};

\draw[red, thin] (4,0) to (7,5.19615242);
\draw[red, thin] (7,5.19615242) to (10,5.19615242);
\draw[red, thin] (10,5.19615242) to (13,0);
\draw[red, thin] (13,0) to (10,-5.19615242);
\draw[red, thin] (10,-5.19615242) to (7,-5.19615242);
\draw[red, thin] (7,-5.19615242) to (4,0);

\draw[blue, thick, dotted] (19,-1.7320508) to (24,-1.7320508);
\draw[blue, thick, dotted] (24,-1.7320508) to (29,6.92820323);
\draw[blue, thick, dotted] (29,6.92820323) to (28.5,7.7942286);
\draw[blue, thin] (28.5,7.7942286) to (28,8.66025403);
\draw[blue, thin] (28,8.66025403) to (19,8.66025403);
\draw[blue, thin] (19,8.66025403) to (16,3.46410161);
\draw[blue, thin] (16,3.46410161) to (17,1.7320508);
\draw[blue, thick, dotted] (17,1.7320508) to (19,-1.7320508);

\draw (17,1.7320508) node{$\bullet$};
\draw[black, thin] (17,1.7320508) to (16.5,0);
\draw (16.5,-0.8) node{$p\left(s_{\smallsearrow}^A\right)$};

\draw (17,5.19615242) node{$\bullet$};
\draw[black, thin] (17,5.19615242) to (16,6.4);
\draw (15,7.3) node{$p\left(i_{\smallnearrow}^A\right)$};

\draw (23,8.66025403) node{$\bullet$};
\draw[black, thin] (23,8.66025403) to (22,9.4);
\draw (21,10.2) node{$p\left(i_{\smallrightarrow}^A\right)$};

\draw (28.5,7.7942286) node{$\bullet$};
\draw[black, thin] (28.5,7.7942286) to (30,8.5);
\draw (32,8.7) node{$p\left(i_{\smallsearrow}^A\right)$};

\draw (23.5,-1.7320508) node{$\bullet$};
\draw[black, thin] (23.5,-1.7320508) to (23,-5);
\draw (23,-6.2) node{$p\left(s_{\smallrightarrow}^A\right)$};

\draw[black, thin] (20,-1.7320508) to (19,-3);
\draw (18.7,-4) node{$l_2^A$};

\draw[red, thin] (21,0) to (24,5.19615242);
\draw[red, thin] (24,5.19615242) to (29,5.19615242);
\draw[red, thin] (29,5.19615242) to (32,0);
\draw[red, thin] (32,0) to (29,-5.19615242);
\draw[red, thin] (29,-5.19615242) to (24,-5.19615242);
\draw[red, thin] (24,-5.19615242) to (21,0);

\draw[red, thin] (38,1.7320508) to (40,5.1961524);
\draw[red, thin] (40,5.1961524) to (42,5.1961524);
\draw[red, thin] (42,5.1961524) to (46,-1.7320508);
\draw[red, thin] (46,-1.7320508) to (45,-3.46410161);
\draw[red, thin] (45,-3.46410161) to (41,-3.46410161);
\draw[red, thin] (41,-3.46410161) to (38,1.7320508);

\draw[blue, thick, dotted] (35,0) to (46,0);
\draw[blue, thick, dotted] (46,0) to (46.5,0.866025403);
\draw[blue, thin] (46.5,0.866025403) to (47,1.7320508);
\draw[blue, thin] (47,1.7320508) to (44,6.92820323);
\draw[blue, thin] (44,6.92820323) to (35,6.92820323);
\draw[blue, thin] (35,6.92820323) to (33,3.4641016);
\draw[blue, thin] (33,3.4641016) to (34,1.7320508);
\draw[blue, thick, dotted] (34,1.7320508) to (35,0);

\draw (40,0) node{$\bullet$};
\draw[black, thin] (40,0) to (39,-1.5);
\draw (38.2,-2.4) node{$p\left(s_{\smallrightarrow}^A\right)$};

\draw (46.5,0.866025403) node{$\bullet$};
\draw[black, thin] (46.5,0.866025403) to (47.5,0.866025403);
\draw (49.5,0.866025403) node{$p\left(s_{\smallnearrow}^A\right)$};

\draw (45,5.19615242) node{$\bullet$};
\draw[black, thin] (45,5.19615242) to (46,5.5);
\draw (47.9,5.7) node{$p\left(i_{\smallsearrow}^A\right)$};

\draw (40,6.92820323) node{$\bullet$};
\draw[black, thin] (40,6.92820323) to (40,8.1);
\draw (40,9) node{$p\left(i_{\smallrightarrow}^A\right)$};

\draw (34,5.19615242) node{$\bullet$};
\draw[black, thin] (34,5.19615242) to (35,5);
\draw (36.5,5) node{$p\left(i_{\smallnearrow}^A\right)$};

\draw (34,1.7320508) node{$\bullet$};
\draw[black, thin] (34,1.7320508) to (34,-2);
\draw (34,-3.2) node{$p\left(s_{\smallsearrow}^A\right)$};

\draw[black, thin] (36,0) to (36,-5);
\draw (36,-6) node{$l_2^A$};

\end{tikzpicture}

\caption{\label{fig:One60One120DifficultCases1}}

\end{figure}

We assumed that $\theta_1=120^{\circ}$. Therefore, we can use the same argument as in Lemma \ref{lem:Two120DegreeAngles} to show that the path in $\partial A$ that connects $p\left(i_{\smallrightarrow}^A\right)$ (in the leftmost configuration), $p\left(i_{\smallsearrow}^A\right)$ (in the central configuration), and $p\left(s_{\smallnearrow}^A\right)$ (in the rightmost configuration) to the line passing through $l_1^B$ is at least as long as the geodesic in $\partial A^{\varhexagon}$ that connects this point and line. The result looks as follows:

\begin{figure}[H]

\begin{tikzpicture}[scale=0.3]

\draw[blue, thick, dotted] (3,-5.19615242) to (8,-5.19615242);
\draw[blue, thick, dotted] (8,-5.19615242) to (9,-3.46410161);
\draw[blue, thick, dotted] (9,-3.46410161) to (5.5,2.5980762);
\draw[blue, thin] (5.5,2.5980762) to (4,5.19615242);
\draw[blue, thin] (4,5.19615242) to (3,5.19615242);
\draw[blue, thin] (3,5.19615242) to (1,5.19615242);
\draw[blue, thin] (1,5.19615242) to (-1,1.7320508);
\draw[blue, thin] (-1,1.7320508) to (1,-1.7320508);
\draw[blue, thick, dotted] (1,-1.7320508) to (3,-5.19615242);

\draw (3,5.19615242) node{$\bullet$};
\draw[black, thin] (3,5.19615242) to (2,6);
\draw (1,6.7) node{$p\left(i_{\smallrightarrow}^A\right)$};

\draw (0,3.46410161) node{$\bullet$};
\draw[black, thin] (0,3.46410161) to (-1,4);
\draw (-2.5,4.2) node{$p\left(i_{\smallnearrow}^A\right)$};

\draw (1,-1.7320508) node{$\bullet$};
\draw[black, thin] (1,-1.7320508) to (-0.2,-1.7320508);
\draw (-2,-1.7320508) node{$p\left(s_{\smallsearrow}^A\right)$};

\draw (7.5,-5.19615242) node{$\bullet$};
\draw[black, thin] (7.5,-5.19615242) to (7.5,-6.5);
\draw (7.5,-7) node{$p\left(s_{\smallrightarrow}^A\right)$};

\draw[black, thin] (4,-5.19615242) to (3,-6.5);
\draw (2.8,-7.4) node{$l_2^A$};

\draw[red, thin] (4,0) to (7,5.19615242);
\draw[red, thin] (7,5.19615242) to (10,5.19615242);
\draw[red, thin] (10,5.19615242) to (13,0);
\draw[red, thin] (13,0) to (10,-5.19615242);
\draw[red, thin] (10,-5.19615242) to (7,-5.19615242);
\draw[red, thin] (7,-5.19615242) to (4,0);

\draw[blue, thick, dotted] (19,-1.7320508) to (24,-1.7320508);
\draw[blue, thick, dotted] (24,-1.7320508) to (28,5.19615242);
\draw[blue, thin] (28,5.19615242) to (29,6.92820323); 
\draw[blue, thin] (29,6.92820323) to (28.5,7.7942286);
\draw[blue, thin] (28.5,7.7942286) to (28,8.66025403);
\draw[blue, thin] (28,8.66025403) to (19,8.66025403);
\draw[blue, thin] (19,8.66025403) to (16,3.46410161);
\draw[blue, thin] (16,3.46410161) to (17,1.7320508);
\draw[blue, thick, dotted] (17,1.7320508) to (19,-1.7320508);

\draw (17,1.7320508) node{$\bullet$};
\draw[black, thin] (17,1.7320508) to (16.5,0);
\draw (16.5,-0.8) node{$p\left(s_{\smallsearrow}^A\right)$};

\draw (17,5.19615242) node{$\bullet$};
\draw[black, thin] (17,5.19615242) to (16,6.4);
\draw (15,7.3) node{$p\left(i_{\smallnearrow}^A\right)$};

\draw (23,8.66025403) node{$\bullet$};
\draw[black, thin] (23,8.66025403) to (22,9.4);
\draw (21,10.2) node{$p\left(i_{\smallrightarrow}^A\right)$};

\draw (28.5,7.7942286) node{$\bullet$};
\draw[black, thin] (28.5,7.7942286) to (30,8.5);
\draw (32,8.7) node{$p\left(i_{\smallsearrow}^A\right)$};

\draw (23.5,-1.7320508) node{$\bullet$};
\draw[black, thin] (23.5,-1.7320508) to (23,-5);
\draw (23,-6.2) node{$p\left(s_{\smallrightarrow}^A\right)$};

\draw[black, thin] (20,-1.7320508) to (19,-3);
\draw (18.7,-4) node{$l_2^A$};

\draw[red, thin] (21,0) to (24,5.19615242);
\draw[red, thin] (24,5.19615242) to (29,5.19615242);
\draw[red, thin] (29,5.19615242) to (32,0);
\draw[red, thin] (32,0) to (29,-5.19615242);
\draw[red, thin] (29,-5.19615242) to (24,-5.19615242);
\draw[red, thin] (24,-5.19615242) to (21,0);

\draw[red, thin] (38,1.7320508) to (40,5.1961524);
\draw[red, thin] (40,5.1961524) to (42,5.1961524);
\draw[red, thin] (42,5.1961524) to (46,-1.7320508);
\draw[red, thin] (46,-1.7320508) to (45,-3.46410161);
\draw[red, thin] (45,-3.46410161) to (41,-3.46410161);
\draw[red, thin] (41,-3.46410161) to (38,1.7320508);


\draw[blue, thick, dotted] (35,0) to (45,0);
\draw[blue, thin] (45,0) to (46,0);
\draw[blue, thin] (46,0) to (47,1.7320508);
\draw[blue, thin] (47,1.7320508) to (44,6.92820323);
\draw[blue, thin] (44,6.92820323) to (35,6.92820323);
\draw[blue, thin] (35,6.92820323) to (33,3.4641016);
\draw[blue, thin] (33,3.4641016) to (34,1.7320508);
\draw[blue, thick, dotted] (34,1.7320508) to (35,0);

\draw (40,0) node{$\bullet$};
\draw[black, thin] (40,0) to (39,-1.5);
\draw (38.2,-2.4) node{$p\left(s_{\smallrightarrow}^A\right)$};

\draw (46.5,0.866025403) node{$\bullet$};
\draw[black, thin] (46.5,0.866025403) to (47.5,0.866025403);
\draw (49.5,0.866025403) node{$p\left(s_{\smallnearrow}^A\right)$};

\draw (45,5.19615242) node{$\bullet$};
\draw[black, thin] (45,5.19615242) to (46,5.5);
\draw (47.9,5.7) node{$p\left(i_{\smallsearrow}^A\right)$};

\draw (40,6.92820323) node{$\bullet$};
\draw[black, thin] (40,6.92820323) to (40,8.1);
\draw (40,9) node{$p\left(i_{\smallrightarrow}^A\right)$};

\draw (34,5.19615242) node{$\bullet$};
\draw[black, thin] (34,5.19615242) to (35,5);
\draw (36.5,5) node{$p\left(i_{\smallnearrow}^A\right)$};

\draw (34,1.7320508) node{$\bullet$};
\draw[black, thin] (34,1.7320508) to (34,-2);
\draw (34,-3.2) node{$p\left(s_{\smallsearrow}^A\right)$};

\draw[black, thin] (36,0) to (36,-5);
\draw (36,-6) node{$l_2^A$};

\end{tikzpicture}

\caption{\label{fig:One60One120DifficultCases2}}

\end{figure}

Now, we move $l_2^A$ up and left so that one of its endpoints intersects $\partial B^{\varhexagon}$ until this shifted $l_2^A$ intersects $\partial A$ at some point $p$. We will still need to refer to the original $l_2^A$, so let's call this shifted line $\tilde{l}_2^A$. The following figure shows the procedure:

\begin{figure}[H]

\begin{tikzpicture}[scale=0.3]

\draw[blue, thick, dotted] (2,-3.46410161) to (6,-3.46410161);
\draw[blue, thick, dotted] (8,-5.19615242) to (9,-3.46410161);
\draw[blue, thick, dotted] (9,-3.46410161) to (5.5,2.5980762);
\draw[blue, thin] (5.5,2.5980762) to (4,5.19615242);
\draw[blue, thin] (4,5.19615242) to (3,5.19615242);
\draw[blue, thin] (3,5.19615242) to (1,5.19615242);
\draw[blue, thin] (1,5.19615242) to (-1,1.7320508);
\draw[blue, thin] (-1,1.7320508) to (1,-1.7320508);
\draw[blue, thick, dotted] (1,-1.7320508) to (3,-5.19615242);

\draw (3,5.19615242) node{$\bullet$};
\draw[black, thin] (3,5.19615242) to (2,6);
\draw (1,6.7) node{$p\left(i_{\smallrightarrow}^A\right)$};

\draw (0,3.46410161) node{$\bullet$};
\draw[black, thin] (0,3.46410161) to (-1,4);
\draw (-2.5,4.2) node{$p\left(i_{\smallnearrow}^A\right)$};

\draw (1,-1.7320508) node{$\bullet$};
\draw[black, thin] (1,-1.7320508) to (-0.2,-1.7320508);
\draw (-2,-1.7320508) node{$p\left(s_{\smallsearrow}^A\right)$};

\draw (7.5,-5.19615242) node{$\bullet$};
\draw[black, thin] (7.5,-5.19615242) to (7.5,-6.5);
\draw (7.5,-7) node{$p\left(s_{\smallrightarrow}^A\right)$};

\draw[black, thin] (3,-3.46410161) to (3,-6.5);
\draw (2.8,-7.4) node{$\tilde{l}_2^A$};

\draw (4,-3.46410161) node{$\bullet$};
\draw[black, thin] (4,-3.46410161) to (3,-2.5);
\draw (3,-2) node{$p$};

\draw[red, thin] (4,0) to (7,5.19615242);
\draw[red, thin] (7,5.19615242) to (10,5.19615242);
\draw[red, thin] (10,5.19615242) to (13,0);
\draw[red, thin] (13,0) to (10,-5.19615242);
\draw[red, thin] (10,-5.19615242) to (7,-5.19615242);
\draw[red, thin] (7,-5.19615242) to (4,0);

\draw[blue, thick, dotted] (24,-1.7320508) to (22,-1.7320508);
\draw[blue, thick, dotted] (18,0) to (21,0);
\draw[blue, thick, dotted] (24,-1.7320508) to (28,5.19615242);
\draw[blue, thin] (28,5.19615242) to (29,6.92820323); 
\draw[blue, thin] (29,6.92820323) to (28.5,7.7942286);
\draw[blue, thin] (28.5,7.7942286) to (28,8.66025403);
\draw[blue, thin] (28,8.66025403) to (19,8.66025403);
\draw[blue, thin] (19,8.66025403) to (16,3.46410161);
\draw[blue, thin] (16,3.46410161) to (17,1.7320508);
\draw[blue, thick, dotted] (17,1.7320508) to (19,-1.7320508);

\draw (17,1.7320508) node{$\bullet$};
\draw[black, thin] (17,1.7320508) to (16.5,0);
\draw (16.5,-0.8) node{$p\left(s_{\smallsearrow}^A\right)$};

\draw (17,5.19615242) node{$\bullet$};
\draw[black, thin] (17,5.19615242) to (16,6.4);
\draw (15,7.3) node{$p\left(i_{\smallnearrow}^A\right)$};

\draw (23,8.66025403) node{$\bullet$};
\draw[black, thin] (23,8.66025403) to (22,9.4);
\draw (21,10.2) node{$p\left(i_{\smallrightarrow}^A\right)$};

\draw (28.5,7.7942286) node{$\bullet$};
\draw[black, thin] (28.5,7.7942286) to (30,8.5);
\draw (32,8.7) node{$p\left(i_{\smallsearrow}^A\right)$};

\draw (23.5,-1.7320508) node{$\bullet$};
\draw[black, thin] (23.5,-1.7320508) to (23,-5);
\draw (23,-6.2) node{$p\left(s_{\smallrightarrow}^A\right)$};

\draw[black, thin] (19,0) to (19,-3);
\draw (18.7,-4) node{$\tilde{l}_2^A$};

\draw (20,0) node{$\bullet$};
\draw[black, thin] (20,0) to (20,1.4);
\draw (20,2) node{$p$};

\draw[red, thin] (21,0) to (24,5.19615242);
\draw[red, thin] (24,5.19615242) to (29,5.19615242);
\draw[red, thin] (29,5.19615242) to (32,0);
\draw[red, thin] (32,0) to (29,-5.19615242);
\draw[red, thin] (29,-5.19615242) to (24,-5.19615242);
\draw[red, thin] (24,-5.19615242) to (21,0);

\draw[red, thin] (38,1.7320508) to (40,5.1961524);
\draw[red, thin] (40,5.1961524) to (42,5.1961524);
\draw[red, thin] (42,5.1961524) to (46,-1.7320508);
\draw[red, thin] (46,-1.7320508) to (45,-3.46410161);
\draw[red, thin] (45,-3.46410161) to (41,-3.46410161);
\draw[red, thin] (41,-3.46410161) to (38,1.7320508);

\draw[blue, thick, dotted] (34.5,0.86602540) to (38.5,0.86602540);
\draw[blue, thick, dotted] (39,0) to (45,0);
\draw[blue, thin] (45,0) to (46,0);
\draw[blue, thin] (46,0) to (47,1.7320508);
\draw[blue, thin] (47,1.7320508) to (44,6.92820323);
\draw[blue, thin] (44,6.92820323) to (35,6.92820323);
\draw[blue, thin] (35,6.92820323) to (33,3.4641016);
\draw[blue, thin] (33,3.4641016) to (34,1.7320508);
\draw[blue, thick, dotted] (34,1.7320508) to (35,0);

\draw (40,0) node{$\bullet$};
\draw[black, thin] (40,0) to (39,-1.5);
\draw (38.2,-2.4) node{$p\left(s_{\smallrightarrow}^A\right)$};

\draw (46.5,0.866025403) node{$\bullet$};
\draw[black, thin] (46.5,0.866025403) to (47.5,0.866025403);
\draw (49.5,0.866025403) node{$p\left(s_{\smallnearrow}^A\right)$};

\draw (45,5.19615242) node{$\bullet$};
\draw[black, thin] (45,5.19615242) to (46,5.5);
\draw (47.9,5.7) node{$p\left(i_{\smallsearrow}^A\right)$};

\draw (40,6.92820323) node{$\bullet$};
\draw[black, thin] (40,6.92820323) to (40,8.1);
\draw (40,9) node{$p\left(i_{\smallrightarrow}^A\right)$};

\draw (34,5.19615242) node{$\bullet$};
\draw[black, thin] (34,5.19615242) to (35,5);
\draw (36.5,5) node{$p\left(i_{\smallnearrow}^A\right)$};

\draw (34,1.7320508) node{$\bullet$};
\draw[black, thin] (34,1.7320508) to (34,-2);
\draw (34,-3.2) node{$p\left(s_{\smallsearrow}^A\right)$};

\draw[black, thin] (36,0.86602540) to (36,-4);
\draw (36.4,-4.9) node{$\tilde{l}_2^A$};

\draw (37,0.86602540) node{$\bullet$};
\draw[black, thin] (37,0.86602540) to (36.8,1.6);
\draw (36.7,2) node{$p$};

\end{tikzpicture}

\caption{\label{fig:One60One120DifficultCasesFinal}}

\end{figure}

Doing this creates a new geodesic, part of which is in $a_2$, and part of which is in $\tilde{l}_2^A$, between the points $p\left(s_{\smallsearrow}^A\right)$ and $p$. The path in $\partial A$ that connects these two points must be at least as long as this geodesic. Furthermore, there is a path in $\partial A$ from $p$ to the line $\Lambda_{\smallsearrow}\left(s_{\smallsearrow}^B\right)$, and this path must be at least as long as the geodesic in $\tilde{l}_2^A$ that connects this same point and line. 

Let $\Lambda$ be the line passing through the line segment $\tilde{l}_2^A$. In the leftmost configuration and the middle configuration, there must be two crossings in $\partial A$ from $p\left(s_{\smallrightarrow}^A\right)$ to this line $\Lambda$, which is to say that there are two crossings between the lines $\Lambda_{\smallrightarrow}\left(s_{\smallrightarrow}^A\right)$ and $\Lambda$. Similarly, there must be two crossings in $\partial B$ between these same lines. This is a total of four such crossings. Taking into account the possibility of joint boundary, this means there must be at least three separate crossings between $\Lambda_{\smallrightarrow}\left(s_{\smallrightarrow}^A\right)$ and $\Lambda$. Only two such crossings are required to make $B^{\varhexagon}$, which means that there is a third crossing that we have not counted in our construction up to this point. Therefore, for these two configurations, we know that there is enough boundary to move $\tilde{l}_2^A$ back to its original position and regain $l_2^A$. This means we can construct $\left(A^{\varhexagon}\setminus B^{\varhexagon},B^{\varhexagon}\right)$ in the cases of the left and middle configurations without increasing the double bubble perimeter. 

In the configuration on the right, however, there are already four crossings between $\Lambda_{\smallrightarrow}\left(s_{\smallrightarrow}^A\right)$ and $\Lambda$ in $\partial\left(A^{\varhexagon}\setminus B^{\varhexagon}\right)\cup\partial B^{\varhexagon}$. What we can do, however, is translate $B^{\varhexagon}$ down between the two lines $\Lambda_{\smallsearrow}\left(s_{\smallsearrow}^B\right)$ and $\Lambda_{\smallsearrow}\left(i_{\smallsearrow}^B\right)$ until $\Lambda_{\smallsearrow}\left(s_{\smallsearrow}^B\right)\cap\Lambda_{\smallnearrow}\left(i_{\smallnearrow}^B\right)$ is an endpoint of $\tilde{l}_2^A$. Then, it can be translated to the left until it reaches $\Lambda_{\smallsearrow}\left(s_{\smallsearrow}^A\right)$. While doing this, we can add all of the length of $\tilde{l}_2^A$ to $l_1^A$. See the following figure:

\begin{figure}[H]

\begin{tikzpicture}[scale=0.25]

\draw[red, thin] (4.5,0.86602540) to (6.5,4.33012701);
\draw[red, thin] (6.5,4.33012701) to (8.5,4.33012701);
\draw[red, thin] (8.5,4.33012701) to (12.5,-2.59807621);
\draw[red, thin] (12.5,-2.59807621) to (11.5,-4.330127018);
\draw[red, thin] (11.5,-4.330127018) to (7.5,-4.330127018);
\draw[red, thin] (7.5,-4.330127018) to (4.5,0.86602540);

\draw[blue, thin] (0.5,0.86602540) to (4.5,0.86602540);
\draw[blue, thin] (11,0) to (12,0);
\draw[blue, thin] (12,0) to (13,1.7320508);
\draw[blue, thin] (13,1.7320508) to (10,6.92820323);
\draw[blue, thin] (10,6.92820323) to (1,6.92820323);
\draw[blue, thin] (1,6.92820323) to (-1,3.4641016);
\draw[blue, thin] (-1,3.4641016) to (0,1.7320508);
\draw[blue, thin] (0,1.7320508) to (0.5,0.86602540);

\draw[->] (14,1) to (15,1);


\draw[red, thin] (19.5,0.86602540) to (21.5,4.33012701);
\draw[red, thin] (21.5,4.33012701) to (23.5,4.33012701);
\draw[red, thin] (23.5,4.33012701) to (27.5,-2.59807621);
\draw[red, thin] (27.5,-2.59807621) to (26.5,-4.330127018);
\draw[red, thin] (26.5,-4.330127018) to (22.5,-4.330127018);
\draw[red, thin] (22.5,-4.330127018) to (19.5,0.86602540);

\draw[blue, thin] (17.5,0.86602540) to (19.5,0.86602540);
\draw[blue, thick, dotted] (19.5,0.86602540) to (21.5,0.86602540);
\draw[blue, thin] (26,0) to (29,0);
\draw[blue, thin] (29,0) to (30,1.7320508);
\draw[blue, thin] (30,1.7320508) to (27,6.92820323);
\draw[blue, thin] (27,6.92820323) to (18,6.92820323);
\draw[blue, thin] (18,6.92820323) to (16,3.4641016);
\draw[blue, thin] (16,3.4641016) to (17,1.7320508);
\draw[blue, thin] (17,1.7320508) to (17.5,0.86602540);

\draw[->] (31,1) to (32,1);


\draw[red, thin] (34.5,0.86602540) to (36.5,4.33012701);
\draw[red, thin] (36.5,4.33012701) to (38.5,4.33012701);
\draw[red, thin] (38.5,4.33012701) to (42.5,-2.59807621);
\draw[red, thin] (42.5,-2.59807621) to (41.5,-4.330127018);
\draw[red, thin] (41.5,-4.330127018) to (37.5,-4.330127018);
\draw[red, thin] (37.5,-4.330127018) to (34.5,0.86602540);

\draw[blue, thick, dotted] (34.5,0.86602540) to (36.5,0.86602540);
\draw[blue, thick, dotted] (36.5,0.86602540) to (38.5,0.86602540);
\draw[blue, thin] (41,0) to (46,0);
\draw[blue, thin] (46,0) to (47,1.7320508);
\draw[blue, thin] (47,1.7320508) to (44,6.92820323);
\draw[blue, thin] (44,6.92820323) to (35,6.92820323);
\draw[blue, thin] (35,6.92820323) to (33,3.4641016);
\draw[blue, thin] (33,3.4641016) to (34,1.7320508);
\draw[blue, thin] (34,1.7320508) to (34.5,0.86602540);

\end{tikzpicture}

\caption{\label{fig:One60One120DifficultCasesFinal}}

\end{figure}

This leaves us with a translated version of $B^{\varhexagon}$, which we call $B_t^{\varhexagon}$, and $A^{\varhexagon}\setminus B_t^{\varhexagon}$. Thus, in all cases we can replace $(A,B)$ with $\left(A^{\varhexagon}\setminus B^{\varhexagon},B^{\varhexagon}\right)$ or $\left(A^{\varhexagon}\setminus B_t^{\varhexagon},B_t^{\varhexagon}\right)$, as we wanted to show.

\end{proof}

\begin{lemma}\label{lem:OneSixtyVolAdj}
Let $\left(A,B\right)\in\gamma_{\alpha}$ and $\tilde{A}^{\varhexagon},\tilde{B}^{\varhexagon}$ be as in Lemma \ref{lemma:ExactlyOneSixty}. If $\left(A^{\varhexagon}\setminus\tilde{B}^{\varhexagon},\tilde{B}^{\varhexagon}\right)$ and $\left(\tilde{A}^{\varhexagon},B^{\varhexagon}\setminus\tilde{A}^{\varhexagon}\right)$ do not have the correct volume ratio, then we can adjust the volumes of the two sets in the configuration with larger joint boundary to obtain a new configuration $\left(A',B'\right)\in\mathfs{F}_{\alpha}$, with $\rho_{DB}\left(A',B'\right)\leq\rho_{DB}\left(A,B\right)$.

\end{lemma}

\begin{proof}

It follows from Lemma \ref{lemma:ExactlyOneSixty} that we merely have to adjust the volumes of the following three configurations, along with a slight variation on the leftmost configuration:

\begin{figure}[H]

\begin{tikzpicture}[scale=0.2]

\draw[blue, thin] (1,-1.7320508) to (5,-1.7320508);
\draw[blue, thin] (5,1.7320508) to (3,5.19615242);
\draw[blue, thin] (3,5.19615242) to (1,5.19615242);
\draw[blue, thin] (1,5.19615242) to (-1,1.7320508);
\draw[blue, thin] (-1,1.7320508) to (1,-1.7320508);

\draw[red, thin] (4,0) to (7,5.19615242);
\draw[red, thin] (7,5.19615242) to (9,5.19615242);
\draw[red, thin] (9,5.19615242) to (12,0);
\draw[red, thin] (12,0) to (9,-5.19615242);
\draw[red, thin] (9,-5.19615242) to (7,-5.19615242);
\draw[red, thin] (7,-5.19615242) to (4,0);

\draw[blue, thin] (19,-1.7320508) to (22,-1.7320508);
\draw[blue, thin] (27,5.19615242) to (28,6.92820323);
\draw[blue, thin] (28,6.92820323) to (27,8.66025403);
\draw[blue, thin] (27,8.66025403) to (19,8.66025403);
\draw[blue, thin] (19,8.66025403) to (16,3.46410161);
\draw[blue, thin] (16,3.46410161) to (19,-1.7320508);

\draw[red, thin] (21,0) to (24,5.19615242);
\draw[red, thin] (24,5.19615242) to (28,5.19615242);
\draw[red, thin] (28,5.19615242) to (31,0);
\draw[red, thin] (31,0) to (28,-5.19615242);
\draw[red, thin] (28,-5.19615242) to (24,-5.19615242);
\draw[red, thin] (24,-5.19615242) to (21,0);


\draw[red, thin] (35,0) to (37,3.46410161);
\draw[red, thin] (37,3.46410161) to (39,3.46410161);
\draw[red, thin] (39,3.46410161) to (42,-1.7320508);
\draw[red, thin] (42,-1.7320508) to (40,-5.19615242);
\draw[red, thin] (40,-5.19615242) to (38,-5.19615242);
\draw[red, thin] (38,-5.19615242) to (35,0);

\draw[blue, thin] (41,0) to (46,0);
\draw[blue, thin] (46,0) to (47,1.7320508);
\draw[blue, thin] (47,1.7320508) to (44,6.92820323);
\draw[blue, thin] (44,6.92820323) to (35,6.92820323);
\draw[blue, thin] (35,6.92820323) to (33,3.4641016);
\draw[blue, thin] (33,3.4641016) to (35,0);

\end{tikzpicture}

\caption{\label{fig:One60ToBeAdjusted}}

\end{figure}

The rightmost configuration is the simplest to adjust. First, since $\mu\left(\tilde{B}^{\varhexagon}\right)=\mu\left(B^{\varhexagon}\right)$, and $B^{\varhexagon}\supset B$, we may have to reduce the volume of $\tilde{B}^{\varhexagon}$, and we must do this in a way so that all of the volume removed from $\tilde{B}^{\varhexagon}$ is given to $A^{\varhexagon}\setminus\tilde{B}^{\varhexagon}$. All that is required is that we move the top of $\tilde{B}^{\varhexagon}$ down, adjusting the other sides accordingly. See the following figure:

\begin{figure}[H]

\begin{tikzpicture}[scale=0.2]

\draw[red, thin] (2,0) to (4,3.46410161);
\draw[red, thin] (4,3.46410161) to (6,3.46410161);
\draw[red, thin] (6,3.46410161) to (9,-1.7320508);
\draw[red, thin] (9,-1.7320508) to (7,-5.19615242);
\draw[red, thin] (7,-5.19615242) to (5,-5.19615242);
\draw[red, thin] (5,-5.19615242) to (2,0);

\draw[blue, thin] (8,0) to (13,0);
\draw[blue, thin] (13,0) to (14,1.7320508);
\draw[blue, thin] (14,1.7320508) to (11,6.92820323);
\draw[blue, thin] (11,6.92820323) to (2,6.92820323);
\draw[blue, thin] (2,6.92820323) to (0,3.4641016);
\draw[blue, thin] (0,3.4641016) to (2,0);

\draw[->] (15,1) to (16,1);


\draw[red, thin] (19,0) to (20,1.7320508);
\draw[red, thick, dotted] (20,1.7320508) to (21,3.46410161);
\draw[red, thin] (20,1.7320508) to (24,1.7320508);
\draw[red, thick, dotted] (21,3.46410161) to (23,3.46410161);
\draw[red, thick, dotted] (23,3.46410161) to (24,1.7320508);
\draw[red, thin] (24,1.7320508) to (26,-1.7320508);
\draw[red, thin] (26,-1.7320508) to (24,-5.19615242);
\draw[red, thin] (24,-5.19615242) to (22,-5.19615242);
\draw[red, thin] (22,-5.19615242) to (19,0);

\draw[blue, thin] (25,0) to (30,0);
\draw[blue, thin] (30,0) to (31,1.7320508);
\draw[blue, thin] (31,1.7320508) to (28,6.92820323);
\draw[blue, thin] (28,6.92820323) to (19,6.92820323);
\draw[blue, thin] (19,6.92820323) to (17,3.4641016);
\draw[blue, thin] (17,3.4641016) to (19,0);

\draw[->] (32,1) to (33,1);

\draw[red, thick, dotted] (36,0) to (37,1.7320508);
\draw[red, thick, dotted] (37,1.7320508) to (38,3.46410161);
\draw[red, thick, dotted] (38,3.46410161) to (40,3.46410161);
\draw[red, thick, dotted] (40,3.46410161) to (41,1.7320508);
\draw[red, thick, dotted] (41,1.7320508) to (42,0);
\draw[red, thin] (36,0) to (42,0);
\draw[red, thin] (42,0) to (43,-1.7320508);
\draw[red, thin] (43,-1.7320508) to (41,-5.19615242);
\draw[red, thin] (41,-5.19615242) to (39,-5.19615242);
\draw[red, thin] (39,-5.19615242) to (36,0);

\draw[blue, thin] (42,0) to (47,0);
\draw[blue, thin] (47,0) to (48,1.7320508);
\draw[blue, thin] (48,1.7320508) to (45,6.92820323);
\draw[blue, thin] (45,6.92820323) to (36,6.92820323);
\draw[blue, thin] (36,6.92820323) to (34,3.4641016);
\draw[blue, thin] (34,3.4641016) to (36,0);

\end{tikzpicture}

\caption{\label{fig:First60Adjusted}}

\end{figure}

As shown in Figure \ref{fig:First60Adjusted}, we can do this until, if necessary, we have lowered the top of $B^{\varhexagon}$ until it is collinear with $\Lambda_{\smallrightarrow}\left(s_{\smallrightarrow}^A\right)$. Let's call this resulting set $B_1$. If $\mu\left(B_1\right)>\alpha$, we can raise its bottom side until we obtain a set $B'$ with volume $\alpha$. To reduce $A^{\varhexagon}\setminus B'$ we can move its right sides to the left, and then its top side down. Taken together, these two processes guarantee that we can obtain a set $A'$ with volume $1$.

The arguments to adjust the leftmost and central configurations in Figure \ref{fig:One60ToBeAdjusted} are similar to each other. We begin with the left one. 

First, recalling the definition of $l_1^A$, we replace $A^{\varhexagon}$ with a new set $A_1$ by reorienting this line segment so that either all of it is horizontal (if $s_{\smallrightarrow}^A\leq s_{\smallrightarrow}^B$), or so that part of it is horizontal and intersects the top line segment in $B^{\varhexagon}$, and part of it is in the original orientation (this we do if $s_{\smallrightarrow}^A>s_{\smallrightarrow}^B$). In the latter case, we can then replace $A_1$ with $A_1^{\varhexagon}\setminus B^{\varhexagon}$. The following figure demonstrates the procedure: 

\begin{figure}[H]

\begin{tikzpicture}[scale=0.2]

\draw[blue, thin] (0,-1.7320508) to (5,-1.7320508);
\draw[blue, thin] (5,1.7320508) to (2,6.92820323);
\draw[blue, thin] (2,6.92820323) to (1,6.92820323);
\draw[blue, thin] (1,6.92820323) to (-2,1.7320508);
\draw[blue, thin] (-2,1.7320508) to (0,-1.7320508);

\draw[red, thin] (4,0) to (7,5.19615242);
\draw[red, thin] (7,5.19615242) to (9,5.19615242);
\draw[red, thin] (9,5.19615242) to (12,0);
\draw[red, thin] (12,0) to (9,-5.19615242);
\draw[red, thin] (9,-5.19615242) to (7,-5.19615242);
\draw[red, thin] (7,-5.19615242) to (4,0);

\draw (1,2) node{$\scaleobj{0.5}{A^{\varhexagon}\setminus B^{\varhexagon}}$};
\draw (7.6,0.8) node{$\scaleobj{0.5}{B^{\varhexagon}}$};

\draw[->] (13,1) to (14,1);


\draw[blue, thin] (17,-1.7320508) to (22,-1.7320508);
\draw[blue, thin] (24,5.19615242) to (20,5.19615242);
\draw[blue, thin] (20,5.19615242) to (19,6.92820323);
\draw[blue, thin] (19,6.92820323) to (18,6.92820323);
\draw[blue, thin] (18,6.92820323) to (15,1.7320508);
\draw[blue, thin] (15,1.7320508) to (17,-1.7320508);

\draw[red, thin] (21,0) to (24,5.19615242);
\draw[red, thin] (24,5.19615242) to (26,5.19615242);
\draw[red, thin] (26,5.19615242) to (29,0);
\draw[red, thin] (29,0) to (26,-5.19615242);
\draw[red, thin] (26,-5.19615242) to (24,-5.19615242);
\draw[red, thin] (24,-5.19615242) to (21,0);

\draw (18.6,2) node{$\scaleobj{0.5}{A_1}$};
\draw (24.4,0.8) node{$\scaleobj{0.5}{B^{\varhexagon}}$};

\draw[->] (30,1) to (31,1);


\draw[blue, thin] (34,-1.7320508) to (39,-1.7320508);
\draw[blue, thin] (40,6.92820323) to (36,6.92820323);
\draw[blue, thin] (41,5.19615242) to (40,6.92820323);

\draw[blue, thin] (36,6.92820323) to (35,6.92820323);
\draw[blue, thin] (35,6.92820323) to (32,1.7320508);
\draw[blue, thin] (32,1.7320508) to (34,-1.7320508);

\draw[red, thin] (38,0) to (41,5.19615242);
\draw[red, thin] (41,5.19615242) to (43,5.19615242);
\draw[red, thin] (43,5.19615242) to (46,0);
\draw[red, thin] (46,0) to (43,-5.19615242);
\draw[red, thin] (43,-5.19615242) to (41,-5.19615242);
\draw[red, thin] (41,-5.19615242) to (38,0);

\draw (35.8,2) node{$\scaleobj{0.5}{A_1^{\varhexagon}\setminus B^{\varhexagon}}$};
\draw (41.8,0.8) node{$\scaleobj{0.5}{B^{\varhexagon}}$};

\end{tikzpicture}

\caption{\label{fig:Second60Adjusted}}

\end{figure}

We must now adjust the volumes to create a configuration in $\mathfs{F}_{\alpha}$. We begin by moving the line segment $\Lambda_{\smallnearrow}\left(i_{\smallnearrow}^B\right)\cap B^{\varhexagon}$ to the right, and extending the top of $A_1$ by the same amount as demonstrated in Figure \ref{fig:Second60Adjusted}:

\begin{figure}[H]

\begin{tikzpicture}[scale=0.2]

\draw[blue, thin] (0,-1.7320508) to (5,-1.7320508);
\draw[blue, thin] (6,6.92820323) to (2,6.92820323);
\draw[blue, thin] (7,5.19615242) to (6,6.92820323);
\draw[blue, thin] (2,6.92820323) to (1,6.92820323);
\draw[blue, thin] (1,6.92820323) to (-2,1.7320508);
\draw[blue, thin] (-2,1.7320508) to (0,-1.7320508);

\draw[red, thin] (4,0) to (7,5.19615242);
\draw[red, thin] (7,5.19615242) to (9,5.19615242);
\draw[red, thin] (9,5.19615242) to (12,0);
\draw[red, thin] (12,0) to (9,-5.19615242);
\draw[red, thin] (9,-5.19615242) to (7,-5.19615242);
\draw[red, thin] (7,-5.19615242) to (4,0);

\draw[->] (13,0) to (14,0);

\draw[blue, thin] (17,-1.7320508) to (22,-1.7320508);
\draw[blue, thin] (23,6.92820323) to (19,6.92820323);
\draw[blue, thin] (24,5.19615242) to (23,6.92820323);
\draw[blue, thin] (19,6.92820323) to (18,6.92820323);
\draw[blue, thin] (18,6.92820323) to (15,1.7320508);
\draw[blue, thin] (15,1.7320508) to (17,-1.7320508);

\draw[blue, thin] (24,5.19615242) to (25,5.19615242);

\draw[red, thick, dotted] (21,0) to (24,5.19615242);
\draw[red, thin] (21.5,-0.86602540) to (25,5.19615242);

\draw[red, thin] (25,5.19615242) to (26,5.19615242);
\draw[red, thin] (26,5.19615242) to (29,0);
\draw[red, thin] (29,0) to (26,-5.19615242);
\draw[red, thin] (26,-5.19615242) to (24,-5.19615242);
\draw[red, thin] (24,-5.19615242) to (21.5,-0.86602540);
\draw[red, thick, dotted] (21.5,-0.86602540) to (21,0);

\draw[->] (30,0) to (31,0);

\draw[blue, thin] (34,-1.7320508) to (39,-1.7320508);
\draw[blue, thin] (40,6.92820323) to (36,6.92820323);
\draw[blue, thin] (42,5.19615242) to (41,6.92820323);
\draw[blue, thin] (36,6.92820323) to (35,6.92820323);
\draw[blue, thin] (35,6.92820323) to (32,1.7320508);
\draw[blue, thin] (32,1.7320508) to (34,-1.7320508);
\draw[blue, thin] (40,6.92820323) to (41,6.92820323);

\draw[red, thick, dotted] (38,0) to (41,5.19615242);
\draw[red, thin] (38.5,-0.86602540) to (42,5.19615242);

\draw[red, thick, dotted] (41,5.19615242) to (42,5.19615242);

\draw[red, thin] (42,5.19615242) to (43,5.19615242);
\draw[red, thin] (43,5.19615242) to (46,0);
\draw[red, thin] (46,0) to (43,-5.19615242);
\draw[red, thin] (43,-5.19615242) to (41,-5.19615242);
\draw[red, thin] (41,-5.19615242) to (38.5,-0.86602540);
\draw[red, thick, dotted] (38.5,-0.86602540) to (38,0);

\end{tikzpicture}

\caption{\label{fig:Second60Adjusted}}

\end{figure}

This process creates a new set to replace $B^{\varhexagon}$, which we call $B_1$, and a set to replace $A_1^{\varhexagon}\setminus B^{\varhexagon}$, which we call $A_2$. If necessary, we can move $\Lambda_{\smallnearrow}\left(i_{\smallnearrow}^B\right)\cap B^{\varhexagon}$ far enough that $A^{\varhexagon}\subset A_2$, which guarantees that $\mu\left(A_2\right)\geq1$. If $\mu\left(B_1\right)>\alpha$, we can raise its bottom side until it is level with the bottom side of $A_2$, and then move its right sides to the left. We stop this process as soon as we have a set $B'$ that has volume $\alpha$. Once we have $B'$, we must ensure that we can reduce the volume of $A_2$ so that it has volume $1$. We can begin by moving the top of $A_2$ down until it is collinear with the top side of $B'$. Then we can move its left sides to the right until we obtain a set $A'$ with volume $1$.

Note that if $s_{\smallrightarrow}^A=s_{\smallrightarrow}^B$, then adjusting volumes is even easier, as we can move both left sides, i.e. $\Lambda_{\smallsearrow}\left(s_{\smallsearrow}^B\right)\cap\partial B^{\varhexagon}$ and $\Lambda_{\smallnearrow}\left(i_{\smallnearrow}^B\right)\cap\partial B^{\varhexagon}$, of $\partial B^{\varhexagon}$ to the right. 

We still need to demonstrate how to adjust the volume of the lefthand configuration in Figure \ref{fig:One60ToBeAdjusted} when $i_{\smallrightarrow}^A<i_{\smallrightarrow}^B$. First, we adjust $l_1^A$ so that it is horizontal, similarly to the way described above, to create a new set in place of $A^{\varhexagon}\setminus B^{\varhexagon}$, which we call $A_1$. Then, the situation is analogous to the volume adjustment in Lemma \ref{lem:Two120AnglesVolAdj}, see Figure \ref{fig:OneCornerVolumeAdjustment2}, but we switch the roles of $A$ and $B$. The following figure gives a visual aid:

\begin{figure}[H]

\begin{tikzpicture}[scale=0.2]

\draw[blue, thin] (1,-3.46410161) to (6,-3.46410161);
\draw[blue, thin] (-1,0) to (1,-3.46410161);
\draw[blue, thin] (1,3.46410161) to (-1,0);
\draw[blue, thin] (1,3.46410161) to (4,3.46410161);
\draw[blue, thin] (4,3.46410161) to (5,1.7320508);

\draw[red, thin] (4,0) to (7,5.19615242);
\draw[red, thin] (7,5.19615242) to (9,5.19615242);
\draw[red, thin] (9,5.19615242) to (12,0);
\draw[red, thin] (12,0) to (9,-5.19615242);
\draw[red, thin] (9,-5.19615242) to (7,-5.19615242);
\draw[red, thin] (7,-5.19615242) to (4,0);

\draw[->] (13,0) to (14,0);

\draw[blue, thin] (18,-3.46410161) to (23,-3.46410161);
\draw[blue, thin] (16,0) to (18,-3.46410161);
\draw[blue, thin] (18,3.46410161) to (16,0);
\draw[blue, thin] (18,3.46410161) to (21,3.46410161);
\draw[blue, thick, dotted] (21,3.46410161) to (22,1.7320508);
\draw[blue, thin] (21,3.46410161) to (23,3.46410161);

\draw[red, thin] (21,0) to (24,5.19615242);
\draw[red, thin] (24,5.19615242) to (26,5.19615242);
\draw[red, thin] (26,5.19615242) to (29,0);
\draw[red, thin] (29,0) to (26,-5.19615242);
\draw[red, thin] (26,-5.19615242) to (24,-5.19615242);
\draw[red, thin] (24,-5.19615242) to (21,0);

\draw[->] (30,0) to (31,0);

\draw[blue, thin] (35,-3.46410161) to (40,-3.46410161);
\draw[blue, thin] (33,0) to (35,-3.46410161);
\draw[blue, thin] (35,3.46410161) to (33,0);
\draw[blue, thin] (35,3.46410161) to (38,3.46410161);
\draw[blue, thin] (38,3.46410161) to (40,3.46410161);
\draw[blue, thin] (40,3.46410161) to (42,0);
\draw[blue, thin] (42,0) to (40,-3.46410161);

\draw[red, thick, dotted] (38,0) to (40,3.46410161);
\draw[red, thin] (40,3.46410161) to (41,5.19615242);
\draw[red, thin] (41,5.19615242) to (43,5.19615242);
\draw[red, thin] (43,5.19615242) to (46,0);
\draw[red, thin] (46,0) to (43,-5.19615242);
\draw[red, thin] (43,-5.19615242) to (41,-5.19615242);
\draw[red, thin] (41,-5.19615242) to (40,-3.46410161);
\draw[red, thick, dotted] (40,-3.46410161) to (38,0);

\draw[->] (47,0) to (48,0);

\draw[blue, thin] (52,-3.46410161) to (57,-3.46410161);
\draw[blue, thin] (50,0) to (52,-3.46410161);
\draw[blue, thin] (52,3.46410161) to (50,0);
\draw[blue, thin] (52,3.46410161) to (55,3.46410161);
\draw[blue, thin] (55,3.46410161) to (56,3.46410161);
\draw[blue, thin] (56,3.46410161) to (58,0);
\draw[blue, thin] (58,0) to (56,-3.46410161);

\draw[red, thin] (57,3.46410161) to (58,5.19615242);
\draw[red, thin] (58,5.19615242) to (60,5.19615242);
\draw[red, thin] (60,5.19615242) to (63,0);
\draw[red, thin] (63,0) to (60,-5.19615242);
\draw[red, thin] (60,-5.19615242) to (58,-5.19615242);
\draw[red, thin] (58,-5.19615242) to (57,-3.46410161);

\draw[red, thin] (56,3.46410161) to (57,3.46410161);
\draw[red, thin] (56,-3.46410161) to (57,-3.46410161);

\draw[->] (64,0) to (65,0);

\draw[blue, thin] (69,-3.46410161) to (73,-3.46410161);
\draw[blue, thin] (67,0) to (69,-3.46410161);
\draw[blue, thin] (69,3.46410161) to (67,0);
\draw[blue, thin] (69,3.46410161) to (72,3.46410161);
\draw[blue, thin] (72,3.46410161) to (73,3.46410161);
\draw[blue, thin] (73,3.46410161) to (75,0);
\draw[blue, thin] (75,0) to (73,-3.46410161);

\draw[red, thin] (73,3.46410161) to (74,5.19615242);
\draw[red, thin] (75,5.19615242) to (77,5.19615242);
\draw[red, thin] (77,5.19615242) to (80,0);
\draw[red, thin] (80,0) to (77,-5.19615242);
\draw[red, thin] (77,-5.19615242) to (75,-5.19615242);
\draw[red, thin] (74,-5.19615242) to (73,-3.46410161);

\draw[red, thin] (74,5.19615242) to (75,5.19615242);
\draw[red, thin] (74,-5.19615242) to (75,-5.19615242);

\draw[red, thick, dotted] (74,3.46410161) to (73,3.46410161);
\draw[red, thick, dotted] (74,-3.46410161) to (73,-3.46410161);

\draw[red, thick, dotted] (74,-3.46410161) to (75,-5.19615242);
\draw[red, thick, dotted] (74,3.46410161) to (75,5.19615242);

\end{tikzpicture}

\caption{\label{fig:Second60Adjusted}}

\end{figure}

Finally, we adjust the volume of the central configuration in Figure \ref{fig:One60ToBeAdjusted}. First, we move $\Lambda_{\smallnearrow}\left(i_{\smallnearrow}^B\right)$ to the right to create a new set $B_1$. If necessary, we can do this until one of the endpoints of this line is contained in $\Lambda_{\smallrightarrow}\left(s_{\smallrightarrow}^A\right)$. Suppose that at this point the volume of $A^{\varhexagon}\setminus B_1$ is still less than $1$. Then, we can continue to move this same line to the right, decreasing the length of the top of $B_1$ and increasing the length of the bottom of $A^{\varhexagon}\setminus B_1$. This is shown in the top right configuration in Figure \ref{fig:One60HardestAdjustment}. As the length of the bottom side of $A^{\varhexagon}\setminus B_1$ increases, we can reorient the part of it that is joint boundary so that it has slope $-\sqrt{3}$. This continuously removes volume from $B_1$ and adds it to $A^{\varhexagon}\setminus B_1$ to create two new sets, which we call $B_2$ and $A_1$, respectively. This process is shown in the bottom two configurations of Figure \ref{fig:One60HardestAdjustment}. We can ensure that $\mu\left(A_1\right)\geq1$, and $\mu\left(B_2\right)\geq\alpha$. This is because the joint volumes of these two sets is at least $1+\alpha$, and we can continue this process until $A^{\varhexagon}\subset A_1$, if necessary.

\begin{figure}[H]

\begin{tikzpicture}[scale=0.2]

\draw[blue, thin] (3,-1.7320508) to (6,-1.7320508);
\draw[blue, thin] (11,5.19615242) to (12,6.92820323);
\draw[blue, thin] (12,6.92820323) to (11,8.66025403);
\draw[blue, thin] (11,8.66025403) to (3,8.66025403);
\draw[blue, thin] (3,8.66025403) to (0,3.46410161);
\draw[blue, thin] (0,3.46410161) to (3,-1.7320508);

\draw[red, thin] (5,0) to (8,5.19615242);
\draw[red, thin] (8,5.19615242) to (12,5.19615242);
\draw[red, thin] (12,5.19615242) to (15,0);
\draw[red, thin] (15,0) to (12,-5.19615242);
\draw[red, thin] (12,-5.19615242) to (8,-5.19615242);
\draw[red, thin] (8,-5.19615242) to (5,0);

\draw (3.5,3) node{$\scaleobj{0.7}{A^{\varhexagon}\setminus B^{\varhexagon}}$};
\draw (10.5,1) node{$\scaleobj{0.8}{B^{\varhexagon}}$};

\draw[->] (16,0) to (17,0);


\draw[blue, thin] (21,-1.7320508) to (24,-1.7320508);
\draw[blue, thin] (29,5.19615242) to (30,6.92820323);
\draw[blue, thin] (30,6.92820323) to (29,8.66025403);
\draw[blue, thin] (29,8.66025403) to (21,8.66025403);
\draw[blue, thin] (21,8.66025403) to (18,3.46410161);
\draw[blue, thin] (18,3.46410161) to (21,-1.7320508);

\draw[red, thick, dotted] (23,0) to (26,5.19615242);
\draw[red, thin] (23.5,-0.86602540) to (27,5.19615242);
\draw[red, thick, dotted] (26,5.19615242) to (27,5.19615242);
\draw[red, thin] (27,5.19615242) to (30,5.19615242);
\draw[red, thin] (30,5.19615242) to (33,0);
\draw[red, thin] (33,0) to (30,-5.19615242);
\draw[red, thin] (30,-5.19615242) to (26,-5.19615242);
\draw[red, thin] (26,-5.19615242) to (23.5,-0.86602540);
\draw[red, thick, dotted] (23.5,-0.86602540) to (23,0);

\draw[->] (34,0) to (35,0);


\draw[blue, thin] (39,-1.7320508) to (42,-1.7320508);
\draw[blue, thin] (47,5.19615242) to (48,6.92820323);
\draw[blue, thin] (48,6.92820323) to (47,8.66025403);
\draw[blue, thin] (47,8.66025403) to (39,8.66025403);
\draw[blue, thin] (39,8.66025403) to (36,3.46410161);
\draw[blue, thin] (36,3.46410161) to (39,-1.7320508);

\draw[red, thick, dotted] (41,0) to (44,5.19615242);
\draw[red, thin] (42,-1.7320508) to (46,5.19615242);
\draw[red, thick, dotted] (44,5.19615242) to (46,5.19615242);
\draw[red, thin] (46,5.19615242) to (48,5.19615242);
\draw[red, thin] (48,5.19615242) to (51,0);
\draw[red, thin] (51,0) to (48,-5.19615242);
\draw[red, thin] (48,-5.19615242) to (44,-5.19615242);
\draw[red, thin] (44,-5.19615242) to (42,-1.7320508);
\draw[red, thick, dotted] (42,-1.7320508) to (41,0);

\draw (40,3.8) node{$\scaleobj{0.7}{A^{\varhexagon}\setminus B_1}$};
\draw (47,1) node{$\scaleobj{0.7}{B_1}$};

\draw[->] (52,0) to (53,0);


\draw[blue, thin] (57,-1.7320508) to (60.5,-1.7320508);
\draw[blue, thin] (65,5.19615242) to (66,6.92820323);

\draw[blue, thin] (66,6.92820323) to (65,8.66025403);
\draw[blue, thin] (65,8.66025403) to (57,8.66025403);
\draw[blue, thin] (57,8.66025403) to (54,3.46410161);
\draw[blue, thin] (54,3.46410161) to (57,-1.7320508);

\draw[red, thick, dotted] (59,0) to (62,5.19615242);

\draw[red, thin] (60.5,-1.7320508) to (64.5,5.19615242);
\draw[red, thin] (64.5,5.19615242) to (65,5.19615242);

\draw[red, thick, dotted] (62,5.19615242) to (65,5.19615242);
\draw[red, thin] (65,5.19615242) to (66,5.19615242);
\draw[red, thin] (66,5.19615242) to (69,0);
\draw[red, thin] (69,0) to (66,-5.19615242);
\draw[red, thin] (66,-5.19615242) to (62,-5.19615242);
\draw[red, thin] (62,-5.19615242) to (60,-1.7320508);
\draw[red, thick, dotted] (60,-1.7320508) to (59,0);

\draw (71.5,-6) node{$\hookleftarrow$};


\draw[blue, thin] (4,-17.3205080) to (7,-17.3205080);

\draw[blue, thin] (7,-17.3205080) to (7.25,-17.7535207);
\draw[blue, thin] (7.25,-17.7535207) to (7.5,-17.3205080);

\draw[blue, thin] (7.5,-17.3205080) to (11.5,-10.3923048);
\draw[blue, thin] (11.5,-10.3923048) to (12,-10.3923048);

\draw[blue, thin] (12,-10.3923048) to (13,-8.66025403);

\draw[blue, thin] (13,-8.66025403) to (12,-6.92820323);
\draw[blue, thin] (12,-6.92820323) to (4,-6.92820323);
\draw[blue, thin] (4,-6.92820323) to (1,-12.124355);
\draw[blue, thin] (1,-12.124355) to (4,-17.3205080);

\draw[red, thick, dotted] (6,-15.58845726) to (9,-10.3923048);

\draw[red, thick, dotted] (9,-10.3923048) to (12,-10.3923048);
\draw[red, thin] (12,-10.3923048) to (13,-10.3923048);
\draw[red, thin] (13,-10.3923048) to (16,-15.58845726);
\draw[red, thin] (16,-15.58845726) to (13,-20.7846096);
\draw[red, thin] (13,-20.7846096) to (9,-20.7846096);

\draw[red, thin] (9,-20.7846096) to (7.25,-17.7535207);

\draw[red, thick, dotted] (7,-17.3205080) to (6,-15.58845726);

\draw[->] (17,-16) to (18,-16);


\draw[blue, thick, dotted] (22,-17.3205080) to (25,-17.3205080);%
\draw[blue, thin] (22.25,-17.7535207) to (25.25,-17.7535207);%
\draw[blue, thick, dotted] (25,-17.3205080) to (25.25,-17.75352077);

\draw[blue, thin] (22,-17.3205080) to (22.25,-17.75352077);

\draw[blue, thin] (25.25,-17.75352077) to (25.5,-17.3205080);
\draw[blue, thin] (29.5,-10.3923048) to (30,-10.3923048);
\draw[blue, thin] (25.5,-17.3205080) to (29.5,-10.3923048);
\draw[blue, thin] (30,-10.3923048) to (31,-8.66025403);
\draw[blue, thin] (31,-8.66025403) to (30,-6.92820323);
\draw[blue, thin] (30,-6.92820323) to (22,-6.92820323);
\draw[blue, thin] (22,-6.92820323) to (19,-12.124355);
\draw[blue, thin] (19,-12.124355) to (22,-17.3205080);

\draw[red, thick, dotted] (24,-15.58845726) to (27,-10.3923048);

\draw[red, thick, dotted] (27,-10.3923048) to (30,-10.3923048);
\draw[red, thin] (30,-10.3923048) to (31,-10.3923048);
\draw[red, thin] (31,-10.3923048) to (34,-15.58845726);
\draw[red, thin] (34,-15.58845726) to (31,-20.7846096);
\draw[red, thin] (31,-20.7846096) to (27,-20.7846096);
\draw[red, thin] (27,-20.7846096) to (25.25,-17.75352077);
\draw[red, thick, dotted] (25,-17.3205080) to (24,-15.58845726);

\draw (24,-12) node{$\scaleobj{0.7}{A_1}$};
\draw (28.6,-15.2) node{$\scaleobj{0.7}{B_2}$};

\end{tikzpicture}

\caption{\label{fig:One60HardestAdjustment}}

\end{figure}

If we need to reduce the volume of $B_2$, we can raise its bottom side until we obtain a set $B_3$ that either has volume $\alpha$, or such that $s_{\smallrightarrow}^{B_3}=s_{\smallrightarrow}^{A_1}$. Then, if necessary, we can move the right sides of $B_3$ to the left until we have a set $\tilde{B}$ with volume $\alpha$. If we need to reduce the volume of $A_1$, we can move its top side down to obtain a new set $A_2$ that either has volume $1$ or such that $i_{\smallrightarrow}^{A_2}=i_{\smallrightarrow}^{\tilde{B}}$, whichever comes first. Then, if necessary, we move the left sides of $A_2$ to the right until we have a set $\tilde{A}$ with volume $1$. Then, the configuration $\left(\tilde{A},\tilde{B}\right)\in\mathfs{F}_{\alpha}$ has double bubble perimeter no more than that of $(A,B)$.

\end{proof}

The last case to analyse when $\mu\left(A^{\varhexagon}\setminus B^{\varhexagon}\right),\mu\left(A^{\varhexagon}\cap B^{\varhexagon}\right),\mu\left(B^{\varhexagon}\setminus A^{\varhexagon}\right)>0$, is when both $\theta_1$, and $\theta_2$ are sixty degree angles. The analysis is somewhat delicate, and depends on the number of sides of either $\partial A^{\varhexagon}$ that intersect the interior of $B^{\varhexagon}$, or vice versa. We first consider which of $\left(A^{\varhexagon}\setminus B^{\varhexagon},B^{\varhexagon}\right)$ or $\left(A^{\varhexagon},B^{\varhexagon}\setminus A^{\varhexagon}\right)$ has greater joint boundary, and select that one.  As usual, we suppose, without loss of generality, that $\left(A^{\varhexagon}\setminus B^{\varhexagon},B^{\varhexagon}\right)$ has greater joint boundary.

The following technical lemma will be helpful with the proof of this case.

\begin{lemma}\label{lem:fourcrossings}
Let $A\subset\mathbb{R}^2$ be simply connected and such that $\partial A$ is closed, continuous, simple, and rectifiable. Furthermore, if there are two parallel lines $\Lambda_1$ and $\Lambda_2$ that have slope $\pm\sqrt{3}$ or slope $0$, with $\Lambda_1\neq\Lambda_2$, such that $\partial A$ crosses between these two lines at least four times, then $\rho\left(\partial A\right)$ is greater than $\rho\left(\partial A^{\varhexagon}\right)$ by at least the distance (in $\mathcal{D}$) between $\Lambda_1$ and $\Lambda_2$. 

\end{lemma}

\begin{proof}

Without loss of generality, suppose that the slope of both lines in question is $-\sqrt{3}$, and that the $y$-intercept of $\Lambda_1$ is less than the $y$-intercept of $\Lambda_2$. Let $\delta$ be the distance between these two lines. The following figure can be used as an illustration:

\begin{figure}[H]

\begin{tikzpicture}[scale=0.5]

\draw[blue, thin] (0,0) to (2,3.46410161);
\draw[blue, thin] (2,3.46410161) to (6,3.46410161);
\draw[blue, thin] (6,3.46410161) to (8,0);
\draw[blue, thin] (8,0) to (6,-3.46410161);
\draw[blue, thin] (6,-3.46410161) to (2,-3.46410161);
\draw[blue, thin] (2,-3.46410161) to (0,0);

\draw[black, thick, dotted] (1,3.46410161) to (6,-5.19615242);
\draw (1,3.66) node{$\scaleobj{0.6}{\Lambda_1}$};

\draw[black, thick, dotted] (2,5.19615242) to (7,-3.46410161);
\draw (2,5.4) node{$\scaleobj{0.6}{\Lambda_2}$};

\draw (3.4,2.74) node{$\bullet$};
\draw[black, thin] (3.4,2.74) to (4.7,4);
\draw (5.2,4.2) node{$\scaleobj{0.65}{\gamma\left(t_1\right)}$};

\draw (1.8,2.1) node{$\bullet$};
\draw[black, thin] (1.8,2.1) to (0.2,2.1);
\draw (-0.4,2.1) node{$\scaleobj{0.65}{\gamma\left(t_2\right)}$};

\draw (2.6,0.75) node{$\bullet$};
\draw[black, thin] (2.6,0.75) to (1.4,0.75);
\draw (1,0.75) node{$\scaleobj{0.65}{\gamma\left(t_3\right)}$};

\draw (4.71,0.545) node{$\scaleobj{0.65}{\bullet}$};
\draw[black, thin] (4.71,0.545) to (5.8,1);
\draw (6,1.3) node{$\scaleobj{0.65}{\gamma\left(t_4\right)}$};

\draw (4.8,0.35) node{$\scaleobj{0.65}{\bullet}$};
\draw[black, thin] (4.8,0.35) to (6,0.1);
\draw (6.5,0) node{$\scaleobj{0.65}{\gamma\left(t_5\right)}$};

\draw (3,0) node{$\bullet$};
\draw[black, thin] (3,0) to (2,0);
\draw (1.2,0) node{$\scaleobj{0.65}{\gamma\left(t_6\right)}$};

\draw (4.75,-3.04) node{$\bullet$};

\draw [black, thick] plot [smooth, tension=0.5] coordinates {(7,1.7320508) (5,3.4)(4,3) (1,1.65)(3.2,0.6) (5.4,0.5)(3,0) (1.1,-1.66)(3,-3.44) (7,-1.732)(7,1.73)};

\end{tikzpicture}

\caption{\label{fig:FourCrossings}}

\end{figure}

Consider $\partial A$ as a curve $\gamma:[0,1]\rightarrow\mathbb{R}^2$ oriented counterclockwise, and with $\gamma(0)\in\Lambda_{\smallsearrow}\left(i_{\smallsearrow}^A\right)$. 
Let $t_1\geq0$ be the last time after $0$ that $\gamma$ intersects $\Lambda_2$ before intersecting $\Lambda_1$, $t_2$ be the first time after $t_1$ that $\gamma$ intersects $\Lambda_1$, $t_3$ be the last time after (or equal to) $t_2$ that $\gamma$ intersects $\Lambda_1$ before intersecting $\Lambda_2$ again, $t_4$ be the first time after $t_3$ that $\gamma$ intersects $\Lambda_2$, and similarly for $t_5$ and $t_6$. The above figure gives an illustration. Note that it is possible that the second coordinate of $\gamma\left(t_5\right)$ is greater than the second coordinate of $\gamma\left(t_3\right)$. 

Consider the union of $\gamma\left([t_3,t_6]\right)$ and the geodesic in $\Lambda_1$ that connects $\gamma(t_3)$ to $\gamma(t_6)$. Put together, these paths form a closed curve, which encloses a subset of $A^{\varhexagon}$, call it $S_1$. Either $S_1\cap A=\emptyset$ (as is the case in the above figure), or $S_1\subset A$. If $S_1\subset A$, then we must consider instead the curve $\gamma\left([t_5,t_8]\right)$ along with the geodesic in $\Lambda_2$ that connects $\gamma\left(t_5\right)$ to $\gamma\left(t_8\right)$. This also forms a closed curve which forms a subset $S_2$ of $A^{\varhexagon}$. If $S_1\subset A$, then $S_2\cap A=\emptyset$. The proofs in both cases are very similar, so we will do only the former, as in Figure \ref{fig:FourCrossings}. 

Define the set $A'$ to be the union of $A$, $S_1$, and $\partial A\cap\partial S_1$, and define $A_1$ to be the union of $A'$ and any bounded and connected components of $\mathbb{R}^2\setminus A'$. 
Let $\delta_1$ be the distance between $\gamma\left(t_3\right)$ and $\gamma\left(t_6\right)$. Then, the path (in $\partial A$) between $\gamma\left(t_3\right)$ and $\gamma\left(t_6\right)$ must be at least as long as $\delta_1+\delta$. Therefore, $\rho\left(\partial A\right)\geq\rho\left(A_1\right)+\delta\geq\rho\left(A^{\varhexagon}\right)+\delta$. The last inequality follows because $A_1\subset A^{\varhexagon}$.

\end{proof}

\begin{lemma}\label{lemma:TwoSixtyDegreeAngles}

Let $(A,B)\in\gamma_{\alpha}$ be such that $\mu\left(A^{\varhexagon}\setminus B^{\varhexagon}\right),\mu\left(A^{\varhexagon}\cap B^{\varhexagon}\right),\mu\left(B^{\varhexagon}\setminus A^{\varhexagon}\right)>0$.  If both $\theta_1$ and $\theta_2$ are sixty degree angles, then, using the notation from before this lemma, one of either $\left(\tilde{A},B^{\varhexagon}\right)$, $\left(A^{\varhexagon},\tilde{B}\right)$, $\left(A^{\varhexagon}\setminus B^{\varhexagon},B^{\varhexagon}\right)$, or $\left(A^{\varhexagon},B^{\varhexagon}\setminus A^{\varhexagon}\right)$ has double bubble perimeter no greater than the double bubble perimeter of $(A,B)$. Furthermore, we can, in each case, alter the volumes of these sets so that we obtain a configuration $\left(A',B'\right)\in\mathfs{F}_{\alpha}$ such that $\rho_{DB}\left(A',B'\right)\leq\rho_{DB}\left(A,B\right)$. 

\end{lemma}

\begin{proof}

Consider first when only one of the line segments forming $\partial A^{\varhexagon}$ intersects the interior of $B^{\varhexagon}$. It must be some line, and by a sequence of reflections and/or $60^{\circ}$ rotations, we may assume that it is the line contained in $\Lambda_{\smallnearrow}\left(s_{\smallnearrow}^A\right)$. Since we are assuming that both points of intersection of $\partial A^{\varhexagon}$ and $\partial B^{\varhexagon}$ occur at the corner of a sixty degree angle, we know that the two lines from $\partial B^{\varhexagon}$ that intersect $\Lambda_{\smallnearrow}\left(s_{\smallnearrow}^A\right)$ must be contained in $\Lambda_{\smallrightarrow}\left(s_{\smallrightarrow}^B\right)$ and $\Lambda_{\smallsearrow}\left(i_{\smallsearrow}^B\right)$. See the lefthand side of Figure \ref{fig:OneLineSegmentProcess} for illustration. The exact shapes are unimportant for this analysis. What matters is which lines must cross $\Lambda_{\smallnearrow}\left(s_{\smallnearrow}^A\right)$ at sixty degree angles, and the only possibility is shown in the lefthand side of this figure.

The beginning of our construction is similar to previous constructions. First, we replace $B$ with $B^{\varhexagon}$. Also, there must be at least one point in the intersection of $\partial A$ and each of the line segments forming $\partial A^{\varhexagon}$ that don't intersect $B^{\varhexagon}$. We can connect these points via the geodesics in $\partial A^{\varhexagon}$ and be sure that the length of this path is no more than the length of the path in $\partial A$ that connects the same points (and doesn't intersect $B^{\varhexagon}$).

\begin{figure}[H]

\begin{tikzpicture}[scale=0.4]

\draw[blue, thin] (-2,0) to (2,0);
\draw[blue, thin] (2,0) to (2.5,-0.86602540);
\draw[blue, thick, dotted] (2.5,-0.86602540) to (3,-1.7320508);
\draw[blue, thick, dotted] (3,-1.7320508) to (1,-5.19615242);
\draw[blue, thick, dotted] (0.5,-6.06217782) to (0,-6.9282032);
\draw[blue, thin] (-1,-6.9282032) to (-2,-6.9282032);
\draw[blue, thick, dotted] (0,-6.9282032) to (-1,-6.9282032);
\draw[blue, thin] (-2,-6.9282032) to (-4,-3.4641016);
\draw[blue, thin] (-4,-3.4641016) to (-2,0);

\draw[red, thin] (-1.5,-3.4641016) to (0,-3.4641016);
\draw[red, thin] (0,-3.4641016) to (1.5,-6.0621778);
\draw[red, thin] (1.5,-6.0621778) to (-1,-6.0621778);
\draw[red, thin] (-1,-6.0621778) to (-2,-4.330127);
\draw[red, thin] (-2,-4.330127) to (-1.5,-3.4641016);

\draw (-1,-6.9282032) node{$\bullet$};
\draw[black, thin] (-1,-6.9282032) to (-1,-8);
\draw (-1,-9) node{$p\left(s_{\smallrightarrow}^A\right)$};

\draw (-3,-5.19615242) node{$\bullet$};
\draw[black, thin] (-3,-5.19615242) to (-4,-5.3);
\draw (-5.1,-5.4) node{$p\left(s_{\smallsearrow}^A\right)$};

\draw (-3,-1.7320508) node{$\bullet$};
\draw[black, thin] (-3,-1.7320508) to (-4,-1.5);
\draw (-5.1,-1.4) node{$p\left(i_{\smallnearrow}^A\right)$};

\draw (-1,0) node{$\bullet$};
\draw[black, thin] (-1,0) to (-1,1);
\draw (-1,1.8) node{$p\left(i_{\smallrightarrow}^A\right)$};

\draw (2.5,-0.86602540) node{$\bullet$};
\draw[black, thin] (2.5,-0.86602540) to (3.2,-0.7);
\draw (4.5,-0.6) node{$p\left(i_{\smallsearrow}^A\right)$};

\draw[black, thin] (2,-3.46410161) to (4.5,-4);
\draw (5,-4.1) node{$l_1^A$};

\draw[black, thin] (0.25,-6.495190528) to (3,-6.495190528);
\draw (3.5,-6.495190528) node{$l_2^A$};

\draw[->] (6.5,-3) to (7.5,-3);

\draw[blue, thin] (12,0) to (16,0);
\draw[blue, thin] (16,0) to (16.5,-0.86602540);
\draw[blue, thick, dotted] (16.5,-0.86602540) to (17,-1.7320508);
\draw[blue, thick, dotted] (16.5,-0.86602540) to (14.5,-4.33012701);
\draw[blue, thick, dotted] (14,-6.06217782) to (13.5,-6.9282032);
\draw[blue, thin] (13,-6.9282032) to (12,-6.9282032);
\draw[blue, thick, dotted] (14,-6.9282032) to (13,-6.9282032);
\draw[blue, thin] (12,-6.9282032) to (10,-3.4641016);
\draw[blue, thin] (10,-3.4641016) to (12,0);

\draw[red, thin] (12.5,-3.4641016) to (14,-3.4641016);
\draw[red, thin] (14,-3.4641016) to (15.5,-6.0621778);
\draw[red, thin] (15.5,-6.0621778) to (13,-6.0621778);
\draw[red, thin] (13,-6.0621778) to (12,-4.330127);
\draw[red, thin] (12,-4.330127) to (12.5,-3.4641016);

\draw (13,-6.9282032) node{$\bullet$};
\draw[black, thin] (13,-6.9282032) to (13,-8);
\draw (13,-9) node{$p\left(s_{\smallrightarrow}^A\right)$};

\draw (11,-5.19615242) node{$\bullet$};

\draw (11,-1.7320508) node{$\bullet$};

\draw (13,0) node{$\bullet$};

\draw (16.5,-0.86602540) node{$\bullet$};
\draw[black, thin] (16.5,-0.86602540) to (17.2,-0.7);
\draw (19.8,-0.4) node{$p\left(i_{\smallsearrow}^A\right)=p_1$};

\draw (13.75,-6.495190528) node{$\bullet$};
\draw[black, thin] (13.75,-6.495190528) to (15,-6.5);
\draw (15.6,-6.58) node{$p_2$};

\draw[->] (20.5,-3) to (21.5,-3);

\draw[blue, thin] (26,0) to (30,0);
\draw[blue, thin] (30,0) to (30.5,-0.86602540);
\draw[blue, thick, dotted] (30.5,-0.86602540) to (31,-1.7320508);
\draw[blue, thin] (30.5,-0.86602540) to (28.5,-4.33012701);
\draw[blue, thin] (28,-6.06217782) to (27.5,-6.9282032);
\draw[blue, thin] (27,-6.9282032) to (27.5,-6.9282032);
\draw[blue, thin] (27,-6.9282032) to (26,-6.9282032);
\draw[blue, thick, dotted] (28,-6.9282032) to (27.5,-6.9282032);
\draw[blue, thin] (26,-6.9282032) to (24,-3.4641016);
\draw[blue, thin] (24,-3.4641016) to (26,0);

\draw[red, thin] (26.5,-3.4641016) to (28,-3.4641016);
\draw[red, thin] (28,-3.4641016) to (29.5,-6.0621778);
\draw[red, thin] (29.5,-6.0621778) to (27,-6.0621778);
\draw[red, thin] (27,-6.0621778) to (26,-4.330127);
\draw[red, thin] (26,-4.330127) to (26.5,-3.4641016);

\draw (27,-6.9282032) node{$\bullet$};
\draw[black, thin] (27,-6.9282032) to (27,-8);
\draw (27,-9) node{$p\left(s_{\smallrightarrow}^A\right)$};

\draw (25,-5.19615242) node{$\bullet$};

\draw (25,-1.7320508) node{$\bullet$};

\draw (27,0) node{$\bullet$};

\draw (30.5,-0.86602540) node{$\bullet$};
\draw[black, thin] (30.5,-0.86602540) to (31.2,-0.7);
\draw (33.8,-0.4) node{$p\left(i_{\smallsearrow}^A\right)=p_1$};

\draw (27.75,-6.495190528) node{$\bullet$};
\draw[black, thin] (27.75,-6.495190528) to (29.2,-6.8);
\draw (31,-6.58) node{$p_2\in\tilde{l}_2^A$};

\draw[black, thin] (29,-3.46410161) to (30,-3.46410161);
\draw (30.8,-3.46410161) node{$\tilde{l}_1^A$};

\end{tikzpicture}

\caption{\label{fig:OneLineSegmentProcess}}

\end{figure}

We now perform a procedure that is similar to what was done in Lemma \ref{lemma:ExactlyOneSixty}. We move $l_1^A$ to the left and up so that one of its endpoints is always in $\partial B^{\varhexagon}$ until it intersects $\partial A$ in some point $p_1$. Similarly, we move $l_2^A$ to the left until it intersects $\partial A$ in some point $p_2$. This is illustrated in the middle of Figure \ref{fig:OneLineSegmentProcess}. Since we have moved $l_1^A$ and $l_2^A$, we call the translated version of these lines $\tilde{l}_1^A$ and $\tilde{l}_2^A$, respectively. From $p\left(s_{\smallrightarrow}^A\right)$ to $p_2$, there is a geodesic that is partially in $a_2$, and partially in $\tilde{l}_2^A$. This geodesic can be no longer than the path in $\partial A$ that connects these same two points and doesn't pass through $p\left(s_{\smallsearrow}^A\right)$. There must also be a path in $\partial A$ that connects $p_2$ to the line $\Lambda_{\smallrightarrow}\left(s_{\smallrightarrow}^B\right)$. This path must be at least as long as the portion of $\tilde{l}_2^A$ that connects this same point and line. We can similarly argue that there is uncounted perimeter in $\partial A$ that is at least as long as the geodesics we have created from $p\left(i_{\smallsearrow}^A\right)$ to $p_1$ (which happen to be the same point in Figure \ref{fig:OneLineSegmentProcess}), and then from $p_1$ to $\Lambda_{\smallsearrow}\left(i_{\smallsearrow}^B\right)$. This is illustrated in the righthand side of Figure \ref{fig:OneLineSegmentProcess}.

Now, let $\delta_1$ be the distance between $l_1^A$ and $\tilde{l}_1^A$, and let $\delta_2$ be the distance between $\tilde{l}_2^A$ and $l_2^A$. Either $\delta_1\geq\delta_2$ or vice versa. Let's assume, without loss of generality, that $\delta_1\geq\delta_2$. In this case, let $\Lambda$ be the line with slope $\sqrt{3}$ that passes through $\tilde{l}_2^A$. In $\partial A$, there must be two crossings between $\Lambda_{\smallnearrow}\left(s_{\smallnearrow}^A\right)$ and $\Lambda$. Similarly, there must be two crossings in $\partial B$ that pass between these same two lines. Accounting for possible joint boundary, this leaves at least three separate crossings between $\Lambda$ and $\Lambda_{\smallnearrow}\left(s_{\smallnearrow}^A\right)$. Only two such crossings are required in the formation of $B^{\varhexagon}$. This leaves a third path between these two lines that we have yet to count, and this third path must be at least as long as $\delta_2$. Therefore, there is enough boundary in $\partial A$ that we have not counted to be able to move $\tilde{l}_2^A$ back to its original position and recover $l_2^A$. This resulting set is not quite $A^{\varhexagon}\setminus B^{\varhexagon}$, so we call it $A_1$

We may have to adjust the volumes of this configuration. First, we must ensure that if the volume of $B^{\varhexagon}$ is too large, then we can reduce it to a new set $\tilde{B}$ without increasing the double bubble perimeter. Then, we must be sure that, once $\mu\left(\tilde{B}\right)=\alpha$, we can alter $A_1$, if necessary, to create a new set $\tilde{A}$ that has volume $1$. Again, we must ensure that this process does not increase the double bubbler perimeter, and also that $\left(\tilde{A},\tilde{B}\right)\in\mathfs{F}_{\alpha}$. In this case, the process is quite simple. 

We begin by moving $\Lambda_{\smallnearrow}\left(i_{\smallnearrow}^B\right)\cap\partial B^{\varhexagon}$ to the right, adjusting the other line segments in $\partial B^{\varhexagon}$ accordingly. We can do this until one of the endpoints of this line is also one of the endpoints of $\tilde{l}_1^A$. If we need to continue to reduce the volume, we can keep moving this same line to the right, and simultaneously move $\tilde{l}_1^A$ to the right, increasing the length of $a_1$. If necessary, we can do this until we have regained all of $A^{\varhexagon}$. The following figure illustrates this procedure:

\begin{figure}[H]

\begin{tikzpicture}[scale=0.35]

\draw[blue, thin] (0,0) to (4,0);
\draw[blue, thin] (4,0) to (4.5,-0.86602540);
\draw[blue, thick, dotted] (4.5,-0.86602540) to (5,-1.7320508);
\draw[blue, thin] (4.5,-0.86602540) to (2.5,-4.33012701);
\draw[blue, thin] (2.5,-6.06217782) to (2,-6.9282032);
\draw[blue, thin] (0,-6.9282032) to (2,-6.9282032);
\draw[blue, thin] (0,-6.9282032) to (-2,-3.4641016);
\draw[blue, thin] (-2,-3.4641016) to (0,0);

\draw[red, thin] (0.5,-3.4641016) to (2,-3.4641016);
\draw[red, thin] (2,-3.4641016) to (3.5,-6.0621778);
\draw[red, thin] (3.5,-6.0621778) to (1,-6.0621778);
\draw[red, thin] (1,-6.0621778) to (0,-4.330127);
\draw[red, thin] (0,-4.330127) to (0.5,-3.4641016);

\draw (1.2,-1) node{$\scaleobj{0.7}{A_1}$};
\draw (1.5,-4) node{$\scaleobj{0.7}{B^{\varhexagon}}$};

\draw[->] (6,-3) to (7,-3);


\draw[blue, thin] (10,0) to (14,0);
\draw[blue, thin] (14,0) to (14.5,-0.86602540);
\draw[blue, thick, dotted] (14.5,-0.86602540) to (15,-1.7320508);
\draw[blue, thin] (14.5,-0.86602540) to (12.5,-4.33012701);
\draw[blue, thin] (12.5,-6.06217782) to (12,-6.9282032);
\draw[blue, thin] (10,-6.9282032) to (12,-6.9282032);
\draw[blue, thin] (10,-6.9282032) to (8,-3.4641016);
\draw[blue, thin] (8,-3.4641016) to (10,0);

\draw[red, thick, dotted] (10.5,-3.4641016) to (11.5,-3.4641016);
\draw[red, thin] (11.5,-3.46410161) to (12,-3.4641016);
\draw[red, thin] (12,-3.4641016) to (13.5,-6.0621778);
\draw[red, thin] (13.5,-6.0621778) to (11,-6.0621778);
\draw[red, thick, dotted] (10,-4.330127) to (10.5,-5.19615242);
\draw[red, thin] (10.5,-5.19615242) to (11,-6.0621778);
\draw[red, thick, dotted] (10,-4.330127) to (10.5,-3.4641016);
\draw[red, thin] (10.5,-5.19615242) to (11.5,-3.4641016);

\draw[->] (16,-3) to (17,-3);

\draw[blue, thin] (20,0) to (24,0);
\draw[blue, thin] (24,0) to (24.5,-0.86602540);
\draw[blue, thick, dotted] (24.5,-0.86602540) to (25,-1.7320508);
\draw[blue, thin] (24.5,-0.86602540) to (22.5,-4.33012701);
\draw[blue, thin] (22.5,-6.06217782) to (22,-6.9282032);
\draw[blue, thin] (20,-6.9282032) to (22,-6.9282032);
\draw[blue, thin] (20,-6.9282032) to (18,-3.4641016);
\draw[blue, thin] (18,-3.4641016) to (20,0);

\draw[red, thick, dotted] (20.5,-3.4641016) to (22,-3.4641016);
\draw[red, thin] (22,-3.4641016) to (23.5,-6.0621778);
\draw[red, thin] (23.5,-6.0621778) to (21,-6.0621778);
\draw[red, thick, dotted] (20,-4.330127) to (20.75,-5.62916512);
\draw[red, thin] (20.75,-5.62916512) to (21,-6.0621778);
\draw[red, thick, dotted] (20,-4.330127) to (20.5,-3.4641016);

\draw[red, thin] (20.75,-5.62916512) to (22,-3.4641016);


\draw[->] (26,-3) to (27,-3);

\draw[blue, thin] (30,0) to (34,0);
\draw[blue, thin] (34,0) to (34.5,-0.86602540);
\draw[blue, thick, dotted] (34.5,-0.86602540) to (35,-1.7320508);
\draw[blue, thin] (34.5,-0.86602540) to (32.5,-4.33012701);
\draw[blue, thin] (32.5,-6.06217782) to (32,-6.9282032);
\draw[blue, thin] (30,-6.9282032) to (32,-6.9282032);
\draw[blue, thin] (30,-6.9282032) to (28,-3.4641016);
\draw[blue, thin] (28,-3.4641016) to (30,0);

\draw[red, thick, dotted] (30.5,-3.4641016) to (32,-3.4641016);
\draw[red, thick, dotted] (32,-3.46410161) to (32.5,-4.33012701);
\draw[red, thin] (32.5,-4.33012701) to (33.5,-6.0621778);
\draw[red, thin] (33.5,-6.0621778) to (31.5,-6.0621778);
\draw[red, thick, dotted] (30,-4.330127) to (31,-6.06217782);
\draw[red, thick, dotted] (31,-6.06217782) to (31.5,-6.06217782);
\draw[red, thick, dotted] (30,-4.330127) to (30.5,-3.4641016);

\draw[red, thin] (31.5,-6.0621778) to (32.5,-4.33012701);

\draw[->] (36,-3) to (37,-3);


\draw[blue, thin] (40,0) to (44,0);
\draw[blue, thin] (44,0) to (44.75,-1.29903810);
\draw[blue, thin] (44.75,-1.29903810) to (42.75,-4.76313972);
\draw[blue, thick, dotted] (44.5,-0.86602540) to (42.5,-4.33012701);
\draw[blue, thin] (42.5,-6.06217782) to (42,-6.9282032);

\draw[blue, thin] (40,-6.9282032) to (42,-6.9282032);

\draw[blue, thin] (40,-6.9282032) to (38,-3.4641016);
\draw[blue, thin] (38,-3.4641016) to (40,0);

\draw[red, thick, dotted] (40.5,-3.4641016) to (42,-3.4641016);
\draw[red, thick, dotted] (42,-3.46410161) to (43,-5.19615242);

\draw[red, thin] (42.75,-4.76313972) to (43.5,-6.0621778);

\draw[red, thin] (43.5,-6.0621778) to (42,-6.0621778);
\draw[red, thick, dotted] (40,-4.330127) to (41,-6.06217782);
\draw[red, thick, dotted] (41,-6.06217782) to (42,-6.06217782);
\draw[red, thick, dotted] (40,-4.330127) to (40.5,-3.4641016);

\draw[red, thin] (42,-6.0621778) to (42.75,-4.76313972);







\draw[black, thin] (13.03,-3.46410161) to (14,-3.46410161);
\draw (14.7,-3.46410161) node{$\tilde{l}_1^A$};

\draw[black, thin] (12.25,-6.495190528) to (14,-6.495190528);
\draw (14.4,-6.495190528) node{$l_2^A$};

\draw (42,-1) node{$\scaleobj{0.7}{A_2}$};
\draw (42.8,-5.5) node{$\scaleobj{0.6}{B_1}$};

\end{tikzpicture}

\caption{\label{fig:OneLineSegmentVolAdj}}

\end{figure}

This process creates two new sets, $B_1$ to replace $B^{\varhexagon}$, and $A_2$ to replace $A_1$. If $\mu\left(B_1\right)=\alpha$, then we can adjust $A_2$. If this process results in $A_2=A^{\varhexagon}$, and $\mu\left(B_1\right)>\alpha$, then we can just rescale both $A^{\varhexagon}$ and $B_1$ until they have the correct volume ratio. If, however, our process results in $A_2\neq A^{\varhexagon}$, and $B_1$ such that $\mu\left(B_1\right)=\alpha$, then we may need to adjust the volume of $A_2$. This is straightforward. We can move the top of $A_2$ down to create a set $A_3$ such that either $\mu\left(A_3\right)=1$, or $i_{\smallrightarrow}^{A_3}=i_{\smallrightarrow}^{B_1}$, whichever comes first. Then, if necessary, we can move the bottom of $A_3$ up to create a set $A_4$ such that either $\mu\left(A_4\right)=1$, or such that $s_{\smallrightarrow}^{A_4}=s_{\smallrightarrow}^{B_1}$, whichever comes first. Finally, if necessary, we can move the left sides of $A_4$ to the right until we have a set of the correct volume. We now have a configuration in $\mathfs{F}_{\alpha}$ with shorter double bubble perimeter than $\rho_{DB}\left(A,B\right)$.

We next analyse when there are two line segments from $\partial A^{\varhexagon}$ contained in the interior of $B^{\varhexagon}$. They must be two adjacent lines, and (through a sequence of reflections and/or $60^{\circ}$ rotations) we may assume that they are the line segments contained in $\Lambda_{\smallrightarrow}\left(s_{\smallrightarrow}^A\right)$, and $\Lambda_{\smallnearrow}\left(s_{\smallnearrow}^A\right)$. The leftmost configuration in Figure \ref{fig:TwoLineSegmentsProcess} gives an example of this. It may seem like the two line segments of $\partial A$ can form an interior sixty degree angle, for example if they were contained in $\Lambda_{\smallrightarrow}\left(s_{\smallrightarrow}^A\right)$, and $\Lambda_{\smallsearrow}\left(i_{\smallsearrow}^A\right)$. However, we consider that this is actually three line segments, with the third one being a line segment of length zero contained in $\Lambda_{\smallnearrow}\left(s_{\smallnearrow}^A\right)$, that intersect the interior of $B^{\varhexagon}$.

We begin our construction, as usual, by replacing $B$ with $B^{\varhexagon}$. Then, consider each of the line segments in $\partial A^{\varhexagon}$ that do not intersect $B^{\varhexagon}$. There must be a point in each of these line segments that $\partial A$ intersects. In the orientation that we have chosen here, we call these points $p\left(s_{\smallsearrow}^A\right)$, $p\left(i_{\smallnearrow}^A\right)$, $p\left(i_{\smallrightarrow}^A\right)$, and $p\left(i_{\smallsearrow}^A\right)$. See the left side of Figure \ref{fig:TwoLineSegmentsProcess}. The path in $\partial A$ that connects $p\left(s_{\smallsearrow}^A\right)$ and $p\left(i_{\smallnearrow}^A\right)$, and does not pass through $p\left(i_{\smallsearrow}^A\right)$, must be at least as long as the geodesic in $\partial A^{\varhexagon}$ that connects these two points. We can make a similar argument for the other points just mentioned.

\begin{figure}[H]

\begin{tikzpicture}[scale=0.4]

\draw[blue, thin] (-1,0) to (3,0);
\draw[blue, thin] (3,0) to (4,-1.7320508);
\draw[blue, thick, dotted] (4,-1.7320508) to (5,-3.4641016);
\draw[blue, thick, dotted] (5,-3.4641016) to (4,-5.1961524);
\draw[blue, thick, dotted] (4,-5.1961524) to (0,-5.1961524);
\draw[blue, thick, dotted] (0,-5.1961524) to (-1,-3.46410161);
\draw[blue, thin] (-1,-3.46410161) to (-2,-1.7320508);
\draw[blue, thin] (-2,-1.7320508) to (-1,0);

\draw[red, thin] (1.5,-3.5) to (4,-3.5);
\draw[red, thin] (4,-3.5) to (6,-6.9641016);
\draw[red, thin] (6,-6.9641016) to (5,-8.6961524);
\draw[red, thin] (5,-8.6961524) to (3,-8.6961524);
\draw[red, thin] (3,-8.6961524) to (1,-4.3660254);
\draw[red, thin] (1,-4.3660254) to (1.5,-3.5);

\draw (4,-1.7320508) node{$\bullet$};
\draw[black, thin] (4,-1.7320508) to (5,-1.7);
\draw (6.5,-1.7) node{$p\left(i_{\smallsearrow}^A\right)$};

\draw (1,0) node{$\bullet$};
\draw[black, thin] (1,0) to (1,1);
\draw (2,1.5) node{$p\left(i_{\smallrightarrow}^A\right)$};

\draw (-1.5,-0.86602540) node{$\bullet$};
\draw[black, thin] (-1.5,-0.86602540) to (-2.2,-0.8);
\draw (-3.3,-0.7) node{$p\left(i_{\smallnearrow}^A\right)$};

\draw (-1,-3.46410161) node{$\bullet$};
\draw[black, thin] (-1,-3.46410161) to (-2,-3.46410161);
\draw (-3.4,-3.46410161) node{$p\left(s_{\smallsearrow}^A\right)$};

\draw[black, thin] (0.8,-5.1961524) to (0.6,-6.2);
\draw (0.5,-6.9) node{$l_2^A$};

\draw[black, thin] (4.75,-3.897114317) to (5.5,-3.88);
\draw (5.96,-3.88) node{$l_1^A$};

\draw[->] (8,-4) to (9,-4);

\draw[blue, thin] (14,0) to (18,0);
\draw[blue, thin] (18,0) to (19,-1.7320508);
\draw[blue, thick, dotted] (19,-1.7320508) to (20,-3.4641016);

\draw[blue, thick, dotted] (19.5,-2.59807621) to (19,-3.46410161);
\draw[blue, thick, dotted] (19.5,-4.33012701) to (19,-5.1961524);
\draw[blue, thick, dotted] (19,-5.1961524) to (16.5,-5.1961524);
\draw[blue, thick, dotted] (16,-4.33012701) to (14.5,-4.33012701);
\draw[blue, thick, dotted] (15,-5.1961524) to (14,-3.46410161);
\draw[blue, thin] (14,-3.46410161) to (13,-1.7320508);
\draw[blue, thin] (13,-1.7320508) to (14,0);

\draw[red, thin] (16.5,-3.5) to (19,-3.5);
\draw[red, thin] (19,-3.5) to (21,-6.9641016);
\draw[red, thin] (21,-6.9641016) to (20,-8.6961524);
\draw[red, thin] (20,-8.6961524) to (18,-8.6961524);
\draw[red, thin] (18,-8.6961524) to (16,-4.3660254);
\draw[red, thin] (16,-4.3660254) to (16.5,-3.5);

\draw (19,-1.7320508) node{$\bullet$};
\draw[black, thin] (19,-1.7320508) to (20,-1.7);
\draw (21.5,-1.7) node{$p\left(i_{\smallsearrow}^A\right)$};

\draw (16,0) node{$\bullet$};
\draw[black, thin] (16,0) to (16,1);
\draw (17,1.5) node{$p\left(i_{\smallrightarrow}^A\right)$};

\draw (13.5,-0.86602540) node{$\bullet$};
\draw[black, thin] (13.5,-0.86602540) to (12.8,-0.8);
\draw (11.7,-0.7) node{$p\left(i_{\smallnearrow}^A\right)$};

\draw (14,-3.46410161) node{$\bullet$};
\draw[black, thin] (14,-3.46410161) to (13,-3.46410161);
\draw (11.6,-3.46410161) node{$p\left(s_{\smallsearrow}^A\right)$};

\draw (19.25,-3.03108891) node{$\cdot$};
\draw[black, thin] (19.25,-3.03108891) to (18.25,-2.8);
\draw (17.75,-2.7) node{$p_1$};

\draw[black, thin] (19.1,-3.25) to (20.5,-3.88);
\draw (20.96,-3.88) node{$\tilde{l}_1^A$};

\draw[black, thin] (15.7,-4.33012701) to (15.6,-6.2);
\draw (15.5,-6.9) node{$\tilde{l}_2^A$};

\draw (15,-4.33012701) node{$\cdot$};
\draw[black, thin] (15,-4.33012701) to (15,-3.5);
\draw (15,-3) node{$p_2$};

\draw[->] (23,-4) to (24,-4);

\draw[blue, thin] (26,0) to (30,0);
\draw[blue, thin] (30,0) to (31.5,-2.5980762);

\draw[blue, thin] (31.5,-2.5980762) to (31,-3.5);
\draw[blue, thick, dotted] (31.5,-4.3660254) to (31,-5.1961524);

\draw[blue, thick, dotted] (31,-5.1961524) to (28.38,-5.1961524);
\draw[blue, thin] (28.38,-5.1961524) to (27,-5.1961524);
\draw[blue, thin] (27,-5.1961524) to (25,-1.7320508);
\draw[blue, thin] (25,-1.7320508) to (26,0);

\draw[red, thin] (28.5,-3.5) to (31,-3.5);
\draw[red, thin] (31,-3.5) to (33,-6.9641016);
\draw[red, thin] (33,-6.9641016) to (32,-8.6961524);
\draw[red, thin] (32,-8.6961524) to (30,-8.6961524);
\draw[red, thin] (30,-8.6961524) to (28,-4.3660254);
\draw[red, thin] (28,-4.3660254) to (28.5,-3.5);

\draw (31,-1.7320508) node{$\bullet$};

\draw (28,0) node{$\bullet$};

\draw (25.5,-0.86602540) node{$\bullet$};

\draw (26,-3.46410161) node{$\bullet$};

\end{tikzpicture}

\caption{\label{fig:TwoLineSegmentsProcess}}

\end{figure}

Now, similarly to previous arguments, we move $l_1^A$ up and left between $\Lambda_{\smallsearrow}\left(i_{\smallsearrow}^B\right)$ and $\Lambda_{\smallsearrow}\left(i_{\smallsearrow}^A\right)$ until it intersects a point $p_1\in\partial A$. Similarly, we move $l_2^A$ until it intersects a point $p_2\in\partial A$. See the middle configuration in Figure \ref{fig:TwoLineSegmentsProcess}. 

Let $\delta_1$ be the distance between $\tilde{l}_1^A$ and $l_1^A$, and $\delta_2$ be the distance between $\tilde{l}_2^A$ and $l_2^A$. Either $\delta_1\geq\delta_2$ or vice versa. Suppose, without loss of generality, that $\delta_1\geq\delta_2$. Let $\Lambda$ be the horizontal line passing through $\tilde{l}_2^A$. Then, between $\Lambda$ and $\Lambda_{\smallrightarrow}\left(s_{\smallrightarrow}^A\right)$, there must be two separate crossings in $\partial A$, which means that each has a length of at least $\delta_2$. Similarly, there must be two such paths in $\partial B$ as well. Accounting for possible joint boundary, this leaves at least three separate crossings between these two lines. In order to form $B^{\varhexagon}$ only two such crossings are required, which leaves at least one extra crossing that we have not yet counted. This means that there is uncounted boundary in $\partial A$ that has length at least $\delta_2$, which is enough boundary for us to move $\tilde{l}_2^A$ back to its original location and regain $l_2^A$. See the right side of Figure \ref{fig:TwoLineSegmentsProcess}. Furthermore, from $p\left(i_{\smallsearrow}^A\right)$ to $p_1$ there is a geodesic that is partially in $a_1$ and partially in $\tilde{l}_1^A$. Then, from $p_1$ to $\Lambda_{\smallsearrow}\left(i_{\smallsearrow}^B\right)$, there is a geodesic in $\tilde{l}_1^A$ that is no longer than the path in $\partial A$ that connects this same point and line (and doesn't pass through $p\left(i_{\smallsearrow}^A\right)$). 

Since $B^{\varhexagon}\supset B$, it may be that $\mu\left(B^{\varhexagon}\right)>\mu\left(B\right)$. If volume adjustment is needed, we can rotate the configuration counterclockwise by $60^{\circ}$ and then adjust the volumes in a nearly identical manner as was shown in Figure \ref{fig:OneCornerVolumeAdjustment2}.

Now, we consider the case when there are three line segments from $\partial A^{\varhexagon}$ that intersect the interior of $B^{\varhexagon}$, and again, both $\theta_1$ and $\theta_2$ are sixty degrees. Recall that this case includes the option that one of the three line segments has length zero, and the other two meet at a point, forming an interior sixty degree angle. Through a sequence of reflections and/or $60^{\circ}$ rotations, we can assume that the two line segments in $\partial A^{\varhexagon}$ that intersect $\partial B^{\varhexagon}$ are the ones contained in $\Lambda_{\smallrightarrow}\left(s_{\smallrightarrow}^A\right)$, and $\Lambda_{\smallsearrow}\left(i_{\smallsearrow}^A\right)$. See the following figure for illustration:

\begin{figure}[H]

\begin{tikzpicture}[scale=0.3]

\draw[blue, thin] (-2,3.46410161) to (4,3.46410161);
\draw[blue, thin] (4,3.46410161) to (7,-1.7320508);
\draw[blue, thin] (7,-1.7320508) to (4,-6.9282032);
\draw[blue, thin] (4,-6.9282032) to (-2,-6.9282032);
\draw[blue, thin] (-2,-6.9282032) to (-5,-1.7320508);
\draw[blue, thin] (-5,-1.7320508) to (-2,3.46410161);

\draw (1,3.46410161) node{$\cdot$};
\draw[black, thin] (1,3.46410161) to (1,4.2);
\draw (1,4.6) node{$a_1$};

\draw (5,1.7320508) node{$\cdot$};
\draw[black, thin] (5,1.7320508) to (6,2.4);
\draw (6.4,2.9) node{$l_1^A$};

\draw (6.75,-1.29903810) node{$\cdot$};
\draw[black, thin] (6.75,-1.29903810) to (10,2);
\draw (10.5,2.7) node{$l_1^{A_{ext}}$};

\draw (6,-3.46410161) node{$\cdot$};
\draw[black, thin] (6,-3.46410161) to (5.2,-3.46410161);
\draw (4.65,-3.46410161) node{$l_3^A$};

\draw (3,-6.9282032) node{$\cdot$};
\draw[black, thin] (3,-6.9282032) to (4,-7.5);
\draw (4.2,-8) node{$l_2^{A_{ext}}$};

\draw (-1,-6.9282032) node{$\cdot$};
\draw[black, thin] (-1,-6.9282032) to (-4,-9);
\draw (-4.3,-9.3) node{$l_2^A$};

\draw (-4,-3.46410161) node{$\cdot$};
\draw[black, thin] (-4,-3.46410161) to (-5,-3.46410161);
\draw (-5.5,-3.46410161) node{$a_2$};

\draw[red, thin] (1,-0.8660254) to (9,-0.8660254);
\draw[red, thin] (9,-0.8660254) to (12,-6.0621778);
\draw[red, thin] (12,-6.0621778) to (9,-11.2583302);
\draw[red, thin] (9,-11.2583302) to (3,-11.2583302);
\draw[red, thin] (3,-11.2583302) to (-1,-4.3301270);
\draw[red, thin] (-1,-4.3301270) to (1,-0.8660254);

\draw (2,-0.8660254) node{$\cdot$};
\draw[black, thin] (2,-0.8660254) to (2,0);
\draw (2,0.3) node{$l_1^{B_{ext}}$};

\draw (8,-0.8660254) node{$\cdot$};
\draw[black, thin] (8,-0.8660254) to (11,0);
\draw (11.1,0.45) node{$l_1^B$};

\draw (10.5,-3.46410161) node{$\cdot$};
\draw[black, thin] (10.5,-3.46410161) to (12,-3.46410161);
\draw (12.4,-3.46410161) node{$b_1$};

\draw (6,-11.2583302) node{$\cdot$};
\draw[black, thin] (6,-11.2583302) to (7,-12.3);
\draw (7.3,-12.6) node{$b_2$};

\draw (2,-9.5262794) node{$\cdot$};
\draw[black, thin] (2,-9.5262794) to (1,-10.2);
\draw (0.7,-10.4) node{$l_2^B$};

\draw (-0.5,-5.19615242) node{$\cdot$};
\draw[black, thin] (-0.5,-5.19615242) to (0.3,-5.19615242);
\draw (1.5,-5.19615242) node{$l_2^{B_{ext}}$};

\draw (-0.5,-3.46410161) node{$\cdot$};
\draw[black, thin] (-0.5,-3.46410161) to (-1,-3);
\draw (-1.4,-2.9) node{$l_3^B$};

\end{tikzpicture}

\caption{\label{fig:ThreeLineSegments}}

\end{figure}

Now, we can consider the case when three of the line segments forming $\partial A^{\varhexagon}$ intersect the interior of $B^{\varhexagon}$. We are assuming, as usual, that $\left(A^{\varhexagon}\setminus B^{\varhexagon},B^{\varhexagon}\right)$ has greater joint boundary than $\left(A^{\varhexagon},B^{\varhexagon}\setminus A^{\varhexagon}\right)$. We begin our construction in the usual way and replace $B$ with $B^{\varhexagon}$, and consider the line segments in $\partial A^{\varhexagon}$ that don't intersect the interior of $B^{\varhexagon}$. There must be a point in each of these line segments that $\partial A$ intersects. We call these points $p\left(i_{\smallrightarrow}^A\right)\in\Lambda_{\smallrightarrow}\left(i_{\smallrightarrow}^A\right)$, $p\left(i_{\smallnearrow}^A\right)\in\Lambda_{\smallnearrow}\left(i_{\smallnearrow}^A\right)$, and $p\left(s_{\smallsearrow}^A\right)\in\Lambda_{\smallsearrow}\left(s_{\smallsearrow}^A\right)$.

We next move $l_1^A$ to the left until it intersects a point $p_1\in\partial A$. Call this shifted line $\tilde{l}_1^A$. Similarly, we move $l_2^A$ up and left until it intersects a point $p_2\in\partial A$ and call this shifted line $\tilde{l}_2^A$. Let $\delta_1$ be the distance between $l_1^A$ and $\tilde{l}_1^A$, and $\delta_2$ be the distance between $l_2^A$ and $\tilde{l}_2^A$.  

\begin{figure}[H]

\begin{tikzpicture}[scale=0.23]

\draw[blue, thin] (-6,3.46410161) to (-2,3.46410161);
\draw[blue, thick, dotted ] (-2,3.46410161) to (4,3.46410161);
\draw[blue, thick, dotted] (4,3.46410161) to (7,-1.7320508);
\draw[blue, thick, dotted] (7,-1.7320508) to (4,-6.9282032);
\draw[blue, thick, dotted] (4,-6.9282032) to (-4,-6.9282032);
\draw[blue, thick, dotted] (-4,-6.9282032) to (-5,-5.19615242);
\draw[blue, thin] (-5,-5.19615242) to (-8,0);
\draw[blue, thin] (-8,0) to (-6,3.46410161);

\draw (-2,3.46410161) node{$\bullet$};
\draw[black, thin] (-2,3.46410161) to (-2,4.4);
\draw (-2,5.6) node{$p\left(i_{\smallrightarrow}^A\right)$};

\draw (-7,1.7320508) node{$\bullet$};
\draw[black, thin] (-7,1.7320508) to (-8.5,1.7320508);
\draw (-11,1.7320508) node{$p\left(i_{\smallnearrow}^A\right)$};

\draw (-5,-5.19615242) node{$\bullet$};
\draw[black, thin] (-5,-5.19615242) to (-6.5,-5.19615242);
\draw (-9,-5.19615242) node{$p\left(s_{\smallsearrow}^A\right)$};

\draw [blue, thick] plot [smooth, tension=0.5] coordinates {(-2,3.46410161) (-0.1,1.7320508)(-2.4,-0.8) (6.77,-1.33367912)(2,-6.9282032) (1,-4)(-5,-5.19615242) (-4.7,-4.2) (-7,1.7320508)(-2,3.46410161)};

\draw[red, thin] (-1.5,0) to (8.5,0);
\draw[red, thin] (8.5,0) to (12,-6.0621778);
\draw[red, thin] (12,-6.0621778) to (9,-11.2583302);
\draw[red, thin] (9,-11.2583302) to (3,-11.2583302);
\draw[red, thin] (3,-11.2583302) to (-2.5,-1.7320508);
\draw[red, thin] (-2.5,-1.7320508) to (-1.5,0);

\draw [red, thick] plot [smooth, tension=0.5] coordinates {(8,0) (9.5,-1.7320508)(11.35,-6.92820323) (7,-11.2)(1.15,-7.79422863) (2,-6.9282032)(4.225,-5.19615242)(6.7,-2.13042249)(6.77,-1.33367912) (5.834,-1)(4,-0.87)(-1.9,-0.866) (2,-0.2)(8,0)};


\draw[->] (13,-3) to (14,-3);


\draw[blue, thin] (23,3.46410161) to (27,3.46410161);
\draw[blue, thick, dotted ] (27,3.46410161) to (33,3.46410161);
\draw[blue, thick, dotted] (28,3.46410161) to (30,0);
\draw[blue, thick, dotted] (35,0) to (36,-1.7320508);
\draw[blue, thick, dotted] (36,-1.7320508) to (33,-6.9282032);
\draw[blue, thick, dotted] (33,-6.9282032) to (29.5,-6.9282032);
\draw[blue, thick, dotted] (28.5,-5.19615242) to (24,-5.19615242);
\draw[blue, thick, dotted] (25,-6.9282032) to (24,-5.19615242);
\draw[blue, thin] (24,-5.19615242) to (21,0);
\draw[blue, thin] (21,0) to (23,3.46410161);

\draw (27,3.46410161) node{$\bullet$};
\draw[black, thin] (27,3.46410161) to (20,5.7);
\draw (20,6.6) node{$\scaleobj{0.6}{p\left(i_{\smallrightarrow}^A\right)=\gamma_1\left(0\right)}$};

\draw (22,1.7320508) node{$\bullet$};
\draw[black, thin] (22,1.7320508) to (20.5,1.7320508);
\draw (18,1.7320508) node{$p\left(i_{\smallnearrow}^A\right)$};

\draw (24,-5.19615242) node{$\bullet$};
\draw[black, thin] (24,-5.19615242) to (22.5,-5.19615242);
\draw (20,-5.19615242) node{$p\left(s_{\smallsearrow}^A\right)$};

\draw [blue, thick] plot [smooth, tension=0.5] coordinates {(27,3.46410161) (28.9,1.7320508)(26.6,-0.8) (35.77,-1.33367912)(31,-6.9282032) (30,-4)(24,-5.19615242) (24.3,-4.2) (22,1.7320508)(27,3.46410161)};

\draw[red, thin] (27.5,0) to (37.5,0);
\draw[red, thin] (37.5,0) to (41,-6.0621778);
\draw[red, thin] (41,-6.0621778) to (38,-11.2583302);
\draw[red, thin] (38,-11.2583302) to (32,-11.2583302);
\draw[red, thin] (32,-11.2583302) to (26.5,-1.7320508);
\draw[red, thin] (26.5,-1.7320508) to (27.5,0);

\draw [red, thick] plot [smooth, tension=0.5] coordinates {(37,0) (38.5,-1.7320508)(40.35,-6.92820323) (36,-11.2)(30.15,-7.79422863) (31,-6.9282032)(33.225,-5.19615242)(35.7,-2.13042249)(35.77,-1.33367912) (34.834,-1)(33,-0.87)(27.1,-0.866) (31,-0.2)(37,0)};

\draw (38.5,-1.7320508) node{$\cdot$};
\draw[black, thin] (38.5,-1.7320508) to (40,-1.7320508);
\draw (41.5,-1.7320508) node{$\scaleobj{0.6}{\gamma_2\left(0\right)}$};

\draw (27.1,-0.866) node{$\cdot$};
\draw[black, thin] (27.1,-0.866) to (40.5,5.2);
\draw (40,5.6) node{$\scaleobj{0.6}{\gamma_1\left(\tau_1\right)=\gamma_2\left(t_1\right)}$};

\draw (35.77,-1.33367912) node{$\cdot$};
\draw[black, thin] (35.77,-1.33367912) to (40,1);
\draw (40,1.5) node{$\scaleobj{0.6}{\gamma_1\left(\tau_2\right)=\gamma_2\left(t_2\right)}$};


\draw[black, thick, dotted] (26,6.9282032) to (38,-13.8564064);
\draw (37,-13.8564064) node{$\scaleobj{0.6}{\Lambda_1}$};

\draw[black, thick, dotted] (31,6.9282032) to (43,-13.8564064);
\draw (45,-13.8564064) node{$\scaleobj{0.6}{\Lambda_{\smallsearrow}\left(i_{\smallsearrow}^A\right)}$};



\end{tikzpicture}

\caption{\label{fig:ThreeLineSegmentsConstruction}}

\end{figure}

Now, consider $\partial A$ as a path $\gamma_1:[0,1]\rightarrow\mathbb{R}^2$ that is oriented clockwise in the sense that $\gamma_1\left(0\right)\in a_1$, after which $\gamma_1\in l_1^A\cup l_1^{A_{ext}}$, etc. Similarly, we consider $\partial B$ as a path $\gamma_2$ that is oriented counterclockwise, and $\gamma_2(0)\in b_1$. This can be visualized with the righthand side of Figure \ref{fig:ThreeLineSegmentsConstruction}. 

Let $\tau_1$ be the last time (after $\tau=0$) that $\gamma_1$ intersects $\partial B^{\varhexagon}$ before intersecting $l_1^{A_{ext}}$, and let $t_1$ be the first time (after $t=0$) that $\gamma_2$ intersects $l_3^B$. 

Either the second coordinate of $\gamma_2\left(t_1\right)$ is greater than the second coordinate of $\gamma_1\left(\tau_1\right)$, or vice versa, or they are equal. If these two coordinates are equal, as they are in Figure \ref{fig:ThreeLineSegmentsConstruction}, then we still consider $\gamma_2\left(t_1\right)$ to be above $\gamma_1\left(\tau_1\right)$ if there exists some $\epsilon>0$ such that $\forall t\in(t_1-\epsilon,t)$ and $\forall\tau\in(\tau_1,\tau_1+\epsilon)$, the second coordinate of $\gamma_2\left(t\right)$ is greater than the second coordinate of $\gamma_1\left(\tau\right)$. Similarly, we can define what it means for $\gamma_2\left(t_1\right)$ to be below $\gamma_1\left(\tau_1\right)$. Both proofs are similar, and we will do only the one when $\gamma_2\left(t_1\right)$ is above $\gamma_1\left(\tau_1\right)$. 

Let $t_2$ be the first time after $t_1$ that $\gamma_2$ intersects $l_1^{A_{ext}}$. Also, let $\tau_2$ be the first time after $\tau_1$ that $\gamma_1$ intersects $l_1^{A_{ext}}$. Then $\gamma_2\left(t_2\right)$ must occur above $\gamma_1\left(\tau_2\right)$. Furthermore, between $t=0$ and $t=t_2$, $\gamma_2$ must make two crossings between $\Lambda_1$ and $\Lambda_{\smallsearrow}\left(i_{\smallsearrow}^A\right)$.

After $\tau_2$, there is a first time $\tau_3$ at which $\gamma_1$ intersects $l_3^A$. After $\tau_3$, there is a first time $\tau_4$ at which $\gamma_1$ intersects $l_2^{A_{ext}}$. Finally, after $\tau_4$, there is a final time $\tau_5$ that $\gamma_1$ intersects $\partial B^{\varhexagon}$ before intersecting $a_2$.

If after $t_2$ there is no joint boundary, then between $\tau_2$ and $\tau_5$, there is uncounted boundary in $\partial A$ with length at least $\delta_1+\delta_2$. Suppose, on the other hand that there is joint boundary after $t_2$. Let $t_3\geq t_2$ be the last time after $t_2$ that $\gamma_2$ intersects $\partial A$. Let $\Lambda_2$ be the line with slope $-\sqrt{3}$ passing through $\gamma_2\left(t_3\right)$, and $\Lambda_3$ be the horizontal line passing through this same point. 

Between $\Lambda_2$ and $\Lambda_{\smallsearrow}\left(i_{\smallsearrow}^A\right)$ there are four crossings (in $\partial B$). By Lemma \ref{lem:fourcrossings}, this means that $\rho\left(\partial B\right)$ is greater than $\rho\left(B^{\varhexagon}\right)$ plus the distance between $\Lambda_2$ and $\Lambda_{\smallsearrow}\left(i_{\smallsearrow}^A\right)$. This means that there is enough uncounted boundary to move $\tilde{l}_1^A$ back to its original position and regain $l_1^A$. 

In a similar manner, we can argue that there is enough uncounted boundary to move $\tilde{l}_2^A$ back to its original position and regain $l_2^A$. This means that $\rho_{DB}\left(A^{\varhexagon}\setminus B^{\varhexagon},B^{\varhexagon}\right)\leq\rho_{DB}\left(A,B\right)$. At this point, we may need to adjust the volumes to obtain a configuration in $\mathfs{F}_{\alpha}$. The procedure is described below with reference to the following figure:

\begin{figure}[H]

\begin{tikzpicture}[scale=0.18]

\draw[blue, thin] (-2,3.46410161) to (4,3.46410161);

\draw[blue, thin] (4,3.46410161) to (6.5,-0.86602540);

\draw[blue, thin] (0.5,-6.9282032) to (-2,-6.9282032);
\draw[blue, thin] (-2,-6.9282032) to (-5,-1.7320508);
\draw[blue, thin] (-5,-1.7320508) to (-2,3.46410161);

\draw (1,3.46410161) node{$\cdot$};
\draw[black, thin] (1,3.46410161) to (1,4.2);
\draw (1,4.6) node{$\scaleobj{0.8}{a_1}$};

\draw (5,1.7320508) node{$\cdot$};
\draw[black, thin] (5,1.7320508) to (6,2.4);
\draw (6.4,2.9) node{$\scaleobj{0.8}{l_1^A}$};




\draw (-1,-6.9282032) node{$\cdot$};
\draw[black, thin] (-1,-6.9282032) to (-4,-9);
\draw (-4.3,-9.3) node{$\scaleobj{0.8}{l_2^A}$};

\draw (-4,-3.46410161) node{$\cdot$};
\draw[black, thin] (-4,-3.46410161) to (-5,-3.46410161);
\draw (-5.5,-3.46410161) node{$\scaleobj{0.8}{a_2}$};

\draw[red, thin] (1,-0.8660254) to (9,-0.8660254);
\draw[red, thin] (9,-0.8660254) to (12,-6.0621778);
\draw[red, thin] (12,-6.0621778) to (9,-11.2583302);
\draw[red, thin] (9,-11.2583302) to (3,-11.2583302);
\draw[red, thin] (3,-11.2583302) to (-1,-4.3301270);
\draw[red, thin] (-1,-4.3301270) to (1,-0.8660254);

\draw (2,-0.8660254) node{$\cdot$};
\draw[black, thin] (2,-0.8660254) to (2,0);
\draw (2,0.3) node{$\scaleobj{0.8}{l_1^{B_{ext}}}$};

\draw (8,-0.8660254) node{$\cdot$};
\draw[black, thin] (8,-0.8660254) to (11,0);
\draw (11.1,0.45) node{$\scaleobj{0.8}{l_1^B}$};

\draw (10.5,-3.46410161) node{$\cdot$};
\draw[black, thin] (10.5,-3.46410161) to (12,-3.46410161);
\draw (12.4,-3.46410161) node{$\scaleobj{0.8}{b_1}$};

\draw (6,-11.2583302) node{$\cdot$};
\draw[black, thin] (6,-11.2583302) to (7,-12.3);
\draw (7.3,-12.6) node{$\scaleobj{0.8}{b_2}$};

\draw (2,-9.5262794) node{$\cdot$};
\draw[black, thin] (2,-9.5262794) to (1,-10.2);
\draw (0.7,-10.4) node{$\scaleobj{0.8}{l_2^B}$};

\draw (-0.5,-5.19615242) node{$\cdot$};
\draw[black, thin] (-0.5,-5.19615242) to (0.3,-5.19615242);
\draw (1.5,-5.19615242) node{$\scaleobj{0.8}{l_2^{B_{ext}}}$};

\draw (-0.5,-3.46410161) node{$\cdot$};
\draw[black, thin] (-0.5,-3.46410161) to (-1,-3);
\draw (-1.4,-2.9) node{$\scaleobj{0.8}{l_3^B}$};

\draw[->] (14,-2) to (15,-2);


\draw[blue, thin] (20,3.46410161) to (26,3.46410161);

\draw[blue, thin] (26,3.46410161) to (28.5,-0.86602540);

\draw[blue, thin] (22.5,-6.9282032) to (20,-6.9282032);
\draw[blue, thin] (20,-6.9282032) to (17,-1.7320508);
\draw[blue, thin] (17,-1.7320508) to (20,3.46410161);

\draw (23,3.46410161) node{$\cdot$};
\draw[black, thin] (23,3.46410161) to (23,4.2);
\draw (23,4.6) node{$\scaleobj{0.8}{a_1}$};

\draw (27,1.7320508) node{$\cdot$};
\draw[black, thin] (27,1.7320508) to (28,2.4);
\draw (28.4,2.9) node{$\scaleobj{0.8}{l_1^A}$};

\draw (21,-6.9282032) node{$\cdot$};
\draw[black, thin] (21,-6.9282032) to (18,-9);
\draw (17.7,-9.3) node{$\scaleobj{0.8}{l_2^A}$};

\draw (18,-3.46410161) node{$\cdot$};
\draw[black, thin] (18,-3.46410161) to (17,-3.46410161);
\draw (16.5,-3.46410161) node{$\scaleobj{0.8}{a_2}$};

\draw[red, thick, dotted] (23,-0.8660254) to (26,-0.86602540);
\draw[red, thin] (26,-0.86602540) to (31,-0.8660254);
\draw[red, thin] (31,-0.8660254) to (34,-6.0621778);
\draw[red, thin] (34,-6.0621778) to (31,-11.2583302);
\draw[red, thin] (31,-11.2583302) to (25,-11.2583302);
\draw[red, thin] (25,-11.2583302) to (22.5,-6.92820323);
\draw[red, thick, dotted] (22.5,-6.92820323) to (21,-4.3301270);
\draw[red, thin] (22.5,-6.92820323) to (26,-0.8660254);

\draw (30,-0.8660254) node{$\cdot$};
\draw[black, thin] (30,-0.8660254) to (33,0);
\draw (33.1,0.45) node{$\scaleobj{0.8}{l_1^B}$};

\draw (32.5,-3.46410161) node{$\cdot$};
\draw[black, thin] (32.5,-3.46410161) to (34,-3.46410161);
\draw (34.4,-3.46410161) node{$\scaleobj{0.8}{b_1}$};

\draw (28,-11.2583302) node{$\cdot$};
\draw[black, thin] (28,-11.2583302) to (29,-12.3);
\draw (29.3,-12.6) node{$\scaleobj{0.8}{b_2}$};

\draw (24,-9.5262794) node{$\cdot$};
\draw[black, thin] (24,-9.5262794) to (23,-10.2);
\draw (22.7,-10.4) node{$\scaleobj{0.8}{l_2^B}$};

\draw (23.7,1) node{$\scaleobj{0.8}{A_1}$};
\draw (28,-2) node{$\scaleobj{0.8}{B_1}$};


\draw[->] (36,-2) to (37,-2);

\draw[->] (58,-2) to (59,-2);



\draw[blue, thin] (42,3.46410161) to (48,3.46410161);

\draw[blue, thin] (48,3.46410161) to (50.5,-0.86602540);

\draw[blue, thin] (44.5,-6.9282032) to (42,-6.9282032);
\draw[blue, thin] (42,-6.9282032) to (39,-1.7320508);
\draw[blue, thin] (39,-1.7320508) to (42,3.46410161);
\draw[blue, thin] (44.5,-6.92820323) to (46,-6.92820323);

\draw (45,3.46410161) node{$\cdot$};
\draw[black, thin] (45,3.46410161) to (45,4.2);
\draw (45,4.6) node{$\scaleobj{0.8}{a_1}$};

\draw (49,1.7320508) node{$\cdot$};
\draw[black, thin] (49,1.7320508) to (50,2.4);
\draw (50.4,2.9) node{$\scaleobj{0.8}{l_1^A}$};

\draw (43,-6.9282032) node{$\cdot$};
\draw[black, thin] (43,-6.9282032) to (40,-9);
\draw (39.7,-9.3) node{$\scaleobj{0.8}{l_2^A}$};

\draw (40,-3.46410161) node{$\cdot$};
\draw[black, thin] (40,-3.46410161) to (39,-3.46410161);
\draw (38.5,-3.46410161) node{$\scaleobj{0.8}{a_2}$};

\draw[red, thick, dotted] (49.5,-0.8660254) to (44,-0.86602540);
\draw[red, thin] (49.5,-0.86602540) to (53,-0.8660254);
\draw[red, thin] (53,-0.8660254) to (56,-6.0621778);
\draw[red, thin] (56,-6.0621778) to (53,-11.2583302);
\draw[red, thin] (53,-11.2583302) to (47,-11.2583302);
\draw[red, thin] (47,-11.2583302) to (44.5,-6.92820323);
\draw[red, thick, dotted] (44.5,-6.92820323) to (43,-4.3301270);

\draw[red, thin] (46,-6.92820323) to (49.5,-0.8660254);

\draw (52,-0.8660254) node{$\cdot$};
\draw[black, thin] (52,-0.8660254) to (55,0);
\draw (55.1,0.45) node{$\scaleobj{0.8}{l_1^B}$};

\draw (54.5,-3.46410161) node{$\cdot$};
\draw[black, thin] (54.5,-3.46410161) to (56,-3.46410161);
\draw (56.4,-3.46410161) node{$\scaleobj{0.8}{b_1}$};

\draw (50,-11.2583302) node{$\cdot$};
\draw[black, thin] (50,-11.2583302) to (51,-12.3);
\draw (51.3,-12.6) node{$\scaleobj{0.8}{b_2}$};

\draw (46,-9.5262794) node{$\cdot$};
\draw[black, thin] (46,-9.5262794) to (45,-10.2);
\draw (44.7,-10.4) node{$\scaleobj{0.8}{l_2^B}$};


\draw[->] (58,-2) to (59,-2);



\draw[blue, thin] (-2,-15.5884572) to (4,-15.5884572);

\draw[blue, thin] (4,-15.5884572) to (6.5,-19.9185842);

\draw[blue, thin] (0.5,-25.9807621) to (-2,-25.9807621);
\draw[blue, thin] (-2,-25.9807621) to (-5,-20.7846096);
\draw[blue, thin] (-5,-20.7846096) to (-2,-15.5884572);
\draw[blue, thin] (1.25,-27.279800) to (0.5,-25.9807621);

\draw (1,-26.8467875) node{$\cdot$};
\draw[black, thin] (1,-26.8467875) to (0,-27.5);
\draw (-1,-28) node{$\scaleobj{0.8}{l_4^A}$};

\draw (1,-15.5884572) node{$\cdot$};
\draw[black, thin] (1,-15.5884572) to (1,-14.85255888);
\draw (1,-14.4525588) node{$\scaleobj{0.8}{a_1}$};

\draw (5,-17.3205080) node{$\cdot$};
\draw[black, thin] (5,-17.3205080) to (6,-16.6525588);
\draw (6.4,-16.1525588) node{$\scaleobj{0.8}{l_1^A}$};

\draw (-1,-25.9807621) node{$\cdot$};
\draw[black, thin] (-1,-25.9807621) to (-4,-28.05255888);
\draw (-4.3,-28.35255888) node{$\scaleobj{0.8}{l_2^A}$};

\draw (-4,-22.5166604) node{$\cdot$};
\draw[black, thin] (-4,-22.5166604) to (-5,-22.5166604);
\draw (5.5,-22.5166604) node{$\scaleobj{0.8}{a_2}$};

\draw[red, thick, dotted] (1,-19.9185842) to (5.5,-19.9185842);
\draw[red, thin] (5.5,-19.9185842) to (9,-19.9185842);
\draw[red, thin] (9,-19.9185842) to (12,-25.1147367);
\draw[red, thin] (12,-25.1147367) to (9,-30.3108891);
\draw[red, thin] (9,-30.3108891) to (3,-30.3108891);

\draw[red, thin] (3,-30.3108891) to (1.25,-27.279800);%

\draw[red, thick, dotted] (0.5,-25.9807621) to (-1,-23.3826859);

\draw[red, thin] (2,-25.9807621) to (1.25,-27.279800);

\draw[red, thin] (2,-25.9807621) to (5.5,-19.9185842);


\draw (8,-19.9185842) node{$\cdot$};
\draw[black, thin] (8,-19.9185842) to (11,-19.0525588);
\draw (11.1,-18.60255888) node{$\scaleobj{0.8}{l_1^B}$};

\draw (10.5,-22.5166604) node{$\cdot$};
\draw[black, thin] (10.5,-22.5166604) to (12,-22.5166604);
\draw (12.4,-22.5166604) node{$\scaleobj{0.8}{b_1}$};

\draw (6,-30.3108890) node{$\cdot$};
\draw[black, thin] (6,-30.3108890) to (7,-31.3525588);
\draw (7.3,-31.65255888) node{$\scaleobj{0.8}{b_2}$};

\draw (2,-28.578838) node{$\cdot$};
\draw[black, thin] (2,-28.578838) to (1,-29.2525588);
\draw (0.7,-29.4525588) node{$\scaleobj{0.8}{l_2^B}$};


\draw[->] (14,-24) to (15,-24);




\draw[blue, thin] (20,-15.5884572) to (26,-15.5884572);
\draw[blue, thin] (26,-15.5884572) to (28.5,-19.9185842);
\draw[blue, thin] (23.25,-27.2798002) to (20.75,-27.2798002);
\draw[blue, thin] (20,-25.9807621) to (17,-20.7846096);
\draw[blue, thin] (17,-20.7846096) to (20,-15.5884572);
\draw[blue, thin] (20.75,-27.279800) to (20,-25.9807621);

\draw (21.5,-27.279800) node{$\cdot$};
\draw[black, thin] (21.5,-27.279800) to (21.2,-29);
\draw (21,-30.2) node{$\scaleobj{0.8}{\tilde{l}_2^A}$};

\draw (20.75,-26.8467875) node{$\cdot$};
\draw[black, thin] (20.75,-26.8467875) to (19,-26.8467875);
\draw (18,-26.8467875) node{$\scaleobj{0.8}{\tilde{l}_4^A}$};

\draw (23,-15.5884572) node{$\cdot$};
\draw[black, thin] (23,-15.5884572) to (23,-14.85255888);
\draw (23,-14.4525588) node{$\scaleobj{0.8}{a_1}$};

\draw (27,-17.3205080) node{$\cdot$};
\draw[black, thin] (27,-17.3205080) to (28,-16.6525588);
\draw (28.4,-16.1525588) node{$\scaleobj{0.8}{l_1^A}$};




\draw[red, thick, dotted] (23,-19.9185842) to (27.5,-19.9185842);

\draw[red, thin] (27.5,-19.9185842) to (31,-19.9185842);
\draw[red, thin] (31,-19.9185842) to (34,-25.1147367);
\draw[red, thin] (34,-25.1147367) to (31,-30.3108891);
\draw[red, thin] (31,-30.3108891) to (25,-30.3108891);

\draw[red, thin] (25,-30.3108891) to (23.25,-27.279800);%

\draw[red, thick, dotted] (22.5,-25.9807621) to (21,-23.3826859);
\draw[blue, thick, dotted] (22.5,-25.9807621) to (23.5,-27.712812);
\draw[blue, thick, dotted] (20,-25.9807621) to (22.5,-25.9807621);

\draw[red, thin] (24,-25.9807621) to (23.25,-27.279800);

\draw[red, thin] (24,-25.9807621) to (27.5,-19.9185842);


\draw (30,-19.9185842) node{$\cdot$};
\draw[black, thin] (30,-19.9185842) to (33,-19.0525588);
\draw (33.1,-18.60255888) node{$\scaleobj{0.8}{l_1^B}$};

\draw (32.5,-22.5166604) node{$\cdot$};
\draw[black, thin] (32.5,-22.5166604) to (34,-22.5166604);
\draw (34.4,-22.5166604) node{$\scaleobj{0.8}{b_1}$};

\draw (28,-30.3108890) node{$\cdot$};
\draw[black, thin] (28,-30.3108890) to (29,-31.3525588);
\draw (29.3,-31.65255888) node{$\scaleobj{0.8}{b_2}$};

\draw (24,-28.578838) node{$\cdot$};
\draw[black, thin] (24,-28.578838) to (23,-29.2525588);
\draw (22.7,-29.4525588) node{$\scaleobj{0.8}{l_2^B}$};


\draw[->] (36,-24) to (37,-24);



\draw[blue, thin] (42,-15.5884572) to (48,-15.5884572);
\draw[blue, thin] (48,-15.5884572) to (50.5,-19.9185842);
\draw[blue, thin] (43.25,-28.1458256) to (45.75,-28.1458256);
\draw[blue, thin] (43.25,-28.1458256) to (39,-20.7846096);
\draw[blue, thin] (39,-20.7846096) to (42,-15.5884572);

\draw (45,-15.5884572) node{$\cdot$};
\draw[black, thin] (45,-15.5884572) to (45,-14.85255888);
\draw (45,-14.4525588) node{$\scaleobj{0.8}{a_1}$};

\draw (49,-17.3205080) node{$\cdot$};
\draw[black, thin] (49,-17.3205080) to (50,-16.6525588);
\draw (50.4,-16.1525588) node{$\scaleobj{0.8}{l_1^A}$};

\draw (43,-25.9807621) node{$\cdot$};
\draw[black, thin] (43,-25.9807621) to (40,-28.05255888);
\draw (39.7,-28.35255888) node{$\scaleobj{0.8}{l_2^A}$};


\draw[red, thick, dotted] (45,-19.9185842) to (50.5,-19.9185842);

\draw[red, thin] (50.5,-19.9185842) to (53,-19.9185842);
\draw[red, thin] (53,-19.9185842) to (56,-25.1147367);
\draw[red, thin] (56,-25.1147367) to (53,-30.3108891);
\draw[red, thin] (53,-30.3108891) to (47,-30.3108891);
\draw[red, thin] (47,-30.3108891) to (45.75,-28.1458256);

\draw[red, thin] (45.75,-28.1458256) to (50.5,-19.9185842);

\draw[red, thick, dotted] (45.5,-27.712812) to (43,-23.3826859);
\draw[blue, thick, dotted] (42,-25.9807621) to (44.5,-25.9807621);

\draw (44,-22) node{$\scaleobj{0.8}{A_2}$};
\draw (51,-24) node{$\scaleobj{0.8}{B_2}$};

\draw (52,-19.9185842) node{$\cdot$};
\draw[black, thin] (52,-19.9185842) to (55,-19.0525588);
\draw (55.1,-18.60255888) node{$\scaleobj{0.8}{l_1^B}$};
\draw (54.5,-22.5166604) node{$\cdot$};
\draw[black, thin] (54.5,-22.5166604) to (56,-22.5166604);
\draw (56.4,-22.5166604) node{$\scaleobj{0.8}{b_1}$};

\draw (50,-30.3108890) node{$\cdot$};
\draw[black, thin] (50,-30.3108890) to (51,-31.3525588);
\draw (51.3,-31.65255888) node{$\scaleobj{0.8}{b_2}$};

\draw (46,-28.578838) node{$\cdot$};
\draw[black, thin] (46,-28.578838) to (45,-29.2525588);
\draw (44.7,-29.4525588) node{$\scaleobj{0.8}{l_2^B}$};


\draw[->] (58,-24) to (59,-24);



\draw[blue, thin] (64,-15.5884572) to (70,-15.5884572);
\draw[blue, thin] (70,-15.5884572) to (72.5,-19.9185842);
\draw[blue, thin] (65.25,-28.1458256) to (67.75,-28.1458256);
\draw[blue, thin] (65.25,-28.1458256) to (61,-20.7846096);
\draw[blue, thin] (61,-20.7846096) to (64,-15.5884572);

\draw[blue, thick, dotted] (72.5,-19.9185842) to (73.5,-21.65063509);
\draw[blue, thick, dotted] (73.5,-21.65063509) to (69.75,-28.1458256);
\draw[blue, thick, dotted] (69.75,-28.1458256) to (67.75,-28.1458256);

\draw (67,-15.5884572) node{$\cdot$};
\draw[black, thin] (67,-15.5884572) to (67,-14.85255888);
\draw (67,-14.4525588) node{$\scaleobj{0.8}{a_1}$};

\draw (71,-17.3205080) node{$\cdot$};
\draw[black, thin] (71,-17.3205080) to (72,-16.6525588);
\draw (72.4,-16.1525588) node{$\scaleobj{0.8}{l_1^A}$};

\draw (65,-25.9807621) node{$\cdot$};
\draw[black, thin] (65,-25.9807621) to (62,-28.05255888);
\draw (61.7,-28.35255888) node{$\scaleobj{0.8}{l_2^A}$};


\draw[red, thick, dotted] (67,-19.9185842) to (72.5,-19.9185842);

\draw[red, thin] (72.5,-19.9185842) to (75,-19.9185842);
\draw[red, thin] (75,-19.9185842) to (78,-25.1147367);
\draw[red, thin] (78,-25.1147367) to (75,-30.3108891);
\draw[red, thin] (75,-30.3108891) to (69,-30.3108891);
\draw[red, thin] (69,-30.3108891) to (67.75,-28.1458256);

\draw[blue, thin] (67.75,-28.1458256) to (68.75,-28.1458256);
\draw[red, thin] (68.75,-28.1458256) to (73.5,-19.9185842);

\draw[red, thick, dotted] (67.5,-27.712812) to (65,-23.3826859);
\draw[blue, thick, dotted] (64,-25.9807621) to (66.5,-25.9807621);

\draw (74,-19.9185842) node{$\cdot$};
\draw[black, thin] (74,-19.9185842) to (77,-19.0525588);
\draw (77.1,-18.60255888) node{$\scaleobj{0.8}{l_1^B}$};
\draw (76.5,-22.5166604) node{$\cdot$};
\draw[black, thin] (76.5,-22.5166604) to (78,-22.5166604);
\draw (78.4,-22.5166604) node{$\scaleobj{0.8}{b_1}$};

\draw (72,-30.3108890) node{$\cdot$};
\draw[black, thin] (72,-30.3108890) to (73,-31.3525588);
\draw (73.3,-31.65255888) node{$\scaleobj{0.8}{b_2}$};

\draw (68,-28.578838) node{$\cdot$};
\draw[black, thin] (68,-28.578838) to (67,-29.2525588);
\draw (66.7,-29.4525588) node{$\scaleobj{0.8}{l_2^B}$};

\end{tikzpicture}

\caption{\label{fig:ThreeLineSegmentsVolAdj}}

\end{figure}

First, we can move $l_3^B$ to the right until, if necessary, its lower endpoint intersects the right endpoint of $l_2^A$. We know this happens at the same time or before the upper endpoint of $l_3^B$ intersects $l_1^A$ because we chose to operate on $\left(A^{\varhexagon}\setminus B^{\varhexagon},B^{\varhexagon}\right)$ specifically for the reason that the joint boundary of this configuration is at least as long as the joint boundary of $\left(A^{\varhexagon},B^{\varhexagon}\setminus A^{\varhexagon}\right)$. That is, the length of $l_1^{B_{ext}}$ must be at least as long as $l_2^{A_{ext}}$ (see Figure \ref{fig:ThreeLineSegments}). This gives us two new sets, which we call $A_1$ to replace $A^{\varhexagon}\setminus B^{\varhexagon}$, and $B_1$ to replace $B^{\varhexagon}$. This is shown in the second configuration of the first row in Figure \ref{fig:ThreeLineSegmentsVolAdj}. 

Suppose that at this point $\mu\left(A_1\right)<\mu\left(A\right)$. Then, we can continue to move $l_3^B$ to the right until, if necessary, the length of $l_1^{B_{ext}}$ is zero. In this step, all of the length of $l_1^{B_{ext}}$ is added to what was $l_2^{A_{ext}}$. This is shown in the third configuration of the first row in Figure \ref{fig:ThreeLineSegmentsVolAdj}. This portion of the now joint boundary that was part of $l_2^{A_{ext}}$ we reorient so that its right endpoint does not change, and it has slope $\sqrt{3}$. This is shown in the first configuration of the second row in Figure \ref{fig:ThreeLineSegmentsVolAdj}. Let's call this reoriented line $l_4^A$. Finally, to complete this step, we move $l_4^A$ to the left until its upper endpoint intersects the left endpoint of $l_2^A$. Call this translated line segment $\tilde{l}_4^A$. Then we translate $l_2^A$ to the right until its right endpoint intersects the lower endpoint of $l_2^A$. This is shown in the second configuration of the second row in Figure \ref{fig:ThreeLineSegmentsVolAdj}. We can continue to do this until, if necessary, the upper endpoint of $l_3^B$ intersects the lower endpoint of $l_1^A$. This is shown in the third configuration of the second row in Figure \ref{fig:ThreeLineSegmentsVolAdj}. This gives us two new sets, which we call $A_2$ to replace $A_1$, and $B_2$ to replace $B_1$.

If it is still the case that $\mu\left(A_2\right)<\mu\left(A\right)$, then we continue moving $l_3^B$ to the right. We can do this until, if necessary, we have regained all of the length of $l_2^{A_{ext}}$. Let's say this results in two sets $(A_3,B_3)$. Then the double bubble perimeter of $(A_3,B_3)$ is greater than the double bubble perimeter of $\left(A_2,B_2\right)$. However, since the length of $l_2^{A_{ext}}$ is less than the length of $l_1^{B_{ext}}$, $\rho_{DB}\left(A_3,B_3\right)\leq\rho_{DB}\left(A^{\varhexagon},B^{\varhexagon}\setminus A^{\varhexagon}\right)\leq\rho_{DB}\left(A^{\varhexagon}\setminus B^{\varhexagon},B^{\varhexagon}\right)\leq\rho_{DB}\left(A,B\right)$. This process is shown in the rightmost configuration of the second row in Figure \ref{fig:ThreeLineSegmentsVolAdj}. This configuration is not of the correct form. So, to fix it, as we move $l_3^B$ to the right, we can replace $A_3$ with $A_3^{\varhexagon}$. This results in a configuration in $\mathfs{F}_{\alpha'}$ for some $\alpha'\in(0,1]$.

Once we have a set that has volume at least $1$ and another set with volume at least $\alpha$, the remaining volume adjustments are simple and resemble the other volume adjustments described before.

Finally, we can analyse the final two cases in which either four or five sides forming $\partial A^{\varhexagon}$ intersect the interior of $B^{\varhexagon}$. These are, however, non-cases. This is because if four sides of $\partial A^{\varhexagon}$ intersect the interior of $B^{\varhexagon}$, and the two line segments of $\partial B^{\varhexagon}$ that cross $\partial A^{\varhexagon}$ form interior sixty degree angles with $\partial A^{\varhexagon}$, then the joint boundary of $\left(A^{\varhexagon}\setminus B^{\varhexagon}, B^{\varhexagon}\right)$ is greater than the joint boundary of $\left(A^{\varhexagon},B^{\varhexagon}\setminus A^{\varhexagon}\right)$. Therefore, we would have actually taken the latter configuration instead of the former. See Figure \ref{fig:NonCases}.

\begin{figure}[H]

\begin{tikzpicture}[scale=0.2]

\draw[blue, thin] (-2.74,-2) to (5.15,-2);
\draw[blue, thin] (5.15,-2) to (7,-5.1961524);
\draw[blue, thin] (7,-5.1961524) to (6,-6.9282032);
\draw[blue, thin] (6,-6.9282032) to (-2,-6.9282032);
\draw[blue, thin] (-2,-6.9282032) to (-4,-3.4641016);
\draw[blue, thin] (-4,-3.4641016) to (-2.74,-2);

\draw[red, thin] (2,-0.8660254) to (9,-0.8660254);
\draw[red, thin] (9,-0.8660254) to (12,-6.0621778);
\draw[red, thin] (12,-6.0621778) to (9,-11.2583302);
\draw[red, thin] (9,-11.2583302) to (3,-11.2583302);
\draw[red, thin] (3,-11.2583302) to (-1,-4.3301270);
\draw[red, thin] (-1,-4.3301270) to (2,-0.8660254);

\end{tikzpicture}

\caption{\label{fig:NonCases}}

\end{figure}

A similar argument applies to when five sides of $\partial A^{\varhexagon}$ intersect the interior of $B^{\varhexagon}$, and thus this is also a non-case.

\end{proof}

\begin{corollary}\label{cor:OnlyThreeJointBoundaryLines}
Let $\left(A,B\right)\in\mathfs{F}_{\alpha}$ have been formed by either the construction in Lemmas \ref{lem:Two120DegreeAngles} and \ref{lem:Two120AnglesVolAdj}, the construction in Lemmas \ref{lemma:ExactlyOneSixty} and \ref{lem:OneSixtyVolAdj}, or the construction in Lemma \ref{lemma:TwoSixtyDegreeAngles}. Suppose that four or more adjacent sides of  $B$ form joint boundary with $\partial A\setminus\left(\pi_1\cup\pi_2\right)$ (or four or more adjacent sides of $A$ form joint boundary with $\partial B\setminus\left(\pi_1\cup\pi_2\right)$). Then $\left(A,B\right)\notin\Gamma_{\alpha}.$

\end{corollary}

\begin{proof}

Let's suppose, without loss of generality, that four adjacent sides of $B$ form joint boundary with $\partial A\setminus\left(\pi_1\cup\pi_2\right)$. After some rotations and/or reflections, we may assume that the four sides are $\Lambda_{\smallrightarrow}\left(i_{\smallrightarrow}^B\right)$, $\Lambda_{\smallnearrow}\left(i_{\smallnearrow}^B\right)$, $\Lambda_{\smallsearrow}\left(s_{\smallsearrow}^B\right)$, and $\Lambda_{\smallrightarrow}\left(s_{\smallrightarrow}^B\right)$. Let $p_1$ be the point of intersection of $\Lambda_{\smallrightarrow}\left(i_{\smallrightarrow}^B\right)$ and $\Lambda_{\smallnearrow}\left(i_{\smallnearrow}^B\right)$, and let $p_2$ be the point of intersection of $\Lambda_{\smallsearrow}\left(s_{\smallsearrow}^B\right)$ and $\Lambda_{\smallrightarrow}\left(s_{\smallrightarrow}^B\right)$. See Figure \ref{fig:FourLinesIsTooMany} for illustration.

\begin{figure}[H]

\begin{tikzpicture}[scale=0.3]

\draw[red, thin] (5.26,-2) to (13.15,-2);
\draw[red, thin] (13.15,-2) to (15,-5.1961524);
\draw[red, thin] (15,-5.1961524) to (14,-6.9282032);
\draw[red, thin] (14,-6.9282032) to (6,-6.9282032);
\draw[red, thin] (6,-6.9282032) to (4,-3.4641016);
\draw[red, thin] (4,-3.4641016) to (5.26,-2);

\draw[blue, thin] (2,-0.8660254) to (9,-0.8660254);
\draw[blue, thin] (9,-0.8660254) to (9.65470054,-2);

\draw[blue, thin] (11.4999999,-6.9282032) to (9,-11.2583302);
\draw[blue, thin] (9,-11.2583302) to (3,-11.2583302);
\draw[blue, thin] (3,-11.2583302) to (-1,-4.3301270);
\draw[blue, thin] (-1,-4.3301270) to (2,-0.8660254);

\draw (5.26,-2) node{$\cdot$};
\draw[black, thin] (5.26,-2) to (5.26,0);
\draw (5.26,0.3) node{$p_1$};
\draw (6,-6.9282032) node{$\cdot$};
\draw[black, thin] (6,-6.9282032) to (6,-8.5);
\draw (6,-8.9) node{$p_2$};

\draw(9.65470054,-2) node{$\cdot$};
\draw[black, thin] (9.65470054,-2) to (9.65470054,0);
\draw(9.65470054,0.4) node{$\pi_1$};
\draw(11.4999999,-6.9282032) node{$\cdot$};
\draw[black, thin] (11.4999999,-6.9282032) to (11.4999999,-9);
\draw(11.4999999,-9.4) node{$\pi_2$};

\end{tikzpicture}

\caption{\label{fig:FourLinesIsTooMany}}

\end{figure}

Note that it is possible that, for example, $\Lambda_{\smallsearrow}\left(s_{\smallsearrow}^B\right)\cap\partial B$ has length zero. In that case, it might look like there are only three sides of $B$ that form joint boundary with $\partial A\setminus\left(\pi_1\cup\pi_2\right)$. But, as usual, we consider this to be four sides where one of them has length zero. 

We alter this configuration by moving all of $B$ to the right until either $p_1$ meets $\pi_1$, or $p_2$ meets $\pi_2$, whichever comes first. This strictly increases the volume of $A$. So, we can move $\Lambda_{\smallsearrow}\left(s_{\smallsearrow}^A\right)$ and $\Lambda_{\smallnearrow}\left(i_{\smallnearrow}^A\right)$ to the right until we have the correct volume. We call these two new sets $A_1$ and $B_1$. This process strictly reduces the non-joint boundary of $A$, which means that we have found a new configuration $\left(A_1,B_1\right)\in\mathfs{F}_{\alpha}$ with strictly better double bubble perimeter than $\left(A,B\right)$. A similar proof applies when five sides of $B$ form joint boundary with $\partial A\setminus\left(\pi_1\cup\pi_2\right)$. 

\end{proof}

\subsection{The third case}

Now, we want to start with $(A,B)$ such that $B^{\varhexagon}\subset A^{\varhexagon}$  (or vice versa) and show a process whereby we construct two new sets $\left(\tilde{A},\tilde{B}\right)\in\mathfs{F}_{\alpha}$ with double bubble perimeter no greater than that of $(A,B)$. This we do in the following lemma.

\begin{lemma}\label{lemma:AllSidesContained}
Let $(A,B)\in\gamma_{\alpha}$ have the property that either $B^{\varhexagon}\subset A^{\varhexagon}$ or $A^{\varhexagon}\subset B^{\varhexagon}$. Then, there exists $\left(\tilde{A},\tilde{B}\right)\in\mathfs{F}_{\alpha}$ such that $\rho_{DB}\left(\tilde{A},\tilde{B}\right)\leq\rho_{DB}\left(A,B\right)$. Furthermore, we can ensure that no more than three of the line segments in $\partial\tilde{B}$ form joint boundary with $\partial\tilde{A}\setminus\left(\pi_1\left(\tilde{A},\tilde{B}\right)\cup\pi_2\left(\tilde{A},\tilde{B}\right)\right)$, and similarly for the line segments in $\partial\tilde{A}$. 

\end{lemma}

\begin{proof}

Recall that $A^{\varhexagon}$ is created by first removing the points forming six lines in the plane from $\mathbb{R}^2$, and then taking the closure of the bounded connected component of the resulting set (as is $B^{\varhexagon}$). The proof of this lemma depends on how many of the lines used to form $A^{\varhexagon}$ are the same as the corresponding lines used to form $B^{\varhexagon}$. The four options are that $0$, $1$, $2$, or $3$ of these lines are the same. With this in mind, let us begin. 

Suppose first that one of the lines forming $\partial A^{\varhexagon}$ is the same as one of the lines forming $\partial B^{\varhexagon}$. We may assume, without loss of generality, that the line in question is $\Lambda_{\smallrightarrow}\left(i_{\smallrightarrow}^A\right)=\Lambda_{\smallrightarrow}\left(i_{\smallrightarrow}^B\right)$. See Figure \ref{fig:OneAgreeingLine} for illustration:

\begin{figure}[H]

\begin{tikzpicture}[scale=0.4]

\draw[red, thin] (0,0) to (1,1.732050808);
\draw[red, thin] (1,1.732050808) to (3,1.732050808);
\draw[red, thin] (3,1.732050808) to (4,0);
\draw[red, thin]  (4,0) to (3,-1.732050808);
\draw[red, thin] (3,-1.732050808) to (1,-1.732050808);
\draw[red, thin] (1,-1.732050808) to (0,0);

\draw[blue, thin] (1,1.732050808) to (-1,1.732050808);
\draw[blue, thin] (-1,1.732050808) to (-3,-1.732050808);
\draw[blue, thin] (-3,-1.732050808) to (-1,-5.196152423);
\draw[blue, thin] (-1,-5.196152423) to (5,-5.196152423);
\draw[blue, thin] (5,-5.196152423) to (7,-1.732050808);
\draw[blue, thin] (7,-1.732050808) to (5,1.732050808);
\draw[blue, thin] (5,1.732050808) to (3,1.732050808);


\draw[black, thin] (0,1.7320508) to (0,2.4);
\draw (0,2.9) node{$l_2^A$};
\draw (-2,1) node{$a_2$};

\draw[black, thin] (4,1.7320508) to (4,2.4);
\draw (4,2.9) node{$l_1^A$};
\draw (6,1) node{$a_1$};

\end{tikzpicture}

\caption{\label{fig:OneAgreeingLine}}

\end{figure}

As usual, we replace $B$ with $B^{\varhexagon}$. Then, we consider the line segments in $\partial A^{\varhexagon}$ that do not intersect $B^{\varhexagon}$. Since $\partial A$ must touch each of these line segments in at least one point, the path in $\partial A^{\varhexagon}$ that connects them (that doesn't intersect $B^{\varhexagon}$) can be no longer than the path in $\partial A$ that connects them (and doesn't intersect $B^{\varhexagon}$). This is illustrated in the lefthand side of the top row of Figure \ref{fig:OneAgreeingLineConstruction}. We can now move $l_1^A$ down and right (between $\Lambda_{\smallsearrow}\left(i_{\smallsearrow}^B\right)$ and $\Lambda_{\smallsearrow}\left(i_{\smallsearrow}^A\right)$) until it intersects a point $p_1\in\partial A$. Similarly, we can move $l_2^A$ down and to the left (between $\Lambda_{\smallnearrow}\left(i_{\smallnearrow}^B\right)$ and $\Lambda_{\smallnearrow}\left(i_{\smallnearrow}^A\right)$) until it intersects a point $p_2\in\partial A$. This is also demonstrated in Figure \ref{fig:OneAgreeingLineConstruction}. Since we have moved $l_1^A$ and $l_2^A$, we call these translated line segments $\tilde{l}_1^A$ and $\tilde{l}_2^A$ respectively. 

Let $\delta_1$ be the distance between $l_1^A$ and $\tilde{l}_1^A$, and $\delta_2$ be the distance between $l_2^A$ and $\tilde{l}_2^A$. Suppose, without loss of generality, that $\delta_1\leq\delta_2$. Let $\Lambda$ be the horizontal line passing through $\tilde{l}_1^A$. Then, there must be at least two separate crossings in $\partial A$ that pass between the lines $\Lambda$ and $\Lambda_{\smallrightarrow}\left(i_{\smallrightarrow}^A\right)$. Similarly, there must be two separate paths in $\partial B$ that also pass between these two lines. This is a total of four crossings between $\Lambda$ and $\Lambda_{\smallrightarrow}\left(i_{\smallrightarrow}^A\right)$. Accounting for possible joint boundary, there must be three crossings between these two lines, and, in order to form $B^{\varhexagon}$, we need only two such crossings. Therefore, there is a third path with length at least $\delta_1$ that we have not counted. This means that we can move $\tilde{l}_1^A$ back to its original position of $l_1^A$ without increasing the double bubble perimeter. See the righthand side of the top row of Figure \ref{fig:OneAgreeingLineConstruction}. 

Finally, as usual, we know that there is a geodesic that is partially in $a_2$ and partially in $\tilde{l}_2^A$ that connects $p\left(i_{\smallnearrow}^A\right)$ and $p_2$, which can be no longer than the path in $\partial A$ that connects these same two points. Similarly,  there is a geodesic in $\tilde{l}_2^A$ connecting $p_2$ to $\Lambda_{\smallnearrow}\left(i_{\smallnearrow}^B\right)$ that is no longer than the path in $\partial A$ that connects this same point and line. This is also demonstrated in the righthand side of the top row of Figure \ref{fig:OneAgreeingLineConstruction}.

\begin{figure}

\begin{tikzpicture}[scale=0.36]

\draw[red, thin] (0,0) to (1,1.732050808);
\draw[red, thin] (1,1.732050808) to (3,1.732050808);
\draw[red, thin] (3,1.732050808) to (4,0);
\draw[red, thin]  (4,0) to (3,-1.732050808);
\draw[red, thin] (3,-1.732050808) to (1,-1.732050808);
\draw[red, thin] (1,-1.732050808) to (0,0);

\draw[blue, thick, dotted] (1,1.732050808) to (-1,1.732050808);
\draw[blue, thick, dotted] (-1,1.732050808) to (-2,0);
\draw[blue, thin] (-2,0) to (-3,-1.732050808);
\draw[blue, thin] (-3,-1.732050808) to (-1,-5.196152423);
\draw[blue, thin] (-1,-5.196152423) to (5,-5.196152423);
\draw[blue, thin] (5,-5.196152423) to (7,-1.732050808);
\draw[blue, thin] (7,-1.732050808) to (5.5,0.86602540);
\draw[blue, thick, dotted] (5.5,0.86602540) to (5,1.732050808);
\draw[blue, thick, dotted] (5,1.732050808) to (3,1.732050808);

\draw (-2,0) node{$\bullet$};
\draw[black, thin] (-2,0) to (-3,0);
\draw (-4.5,0) node{$p\left(i_{\smallnearrow}^A\right)$};

\draw (-2,-3.46410161) node{$\bullet$};
\draw[black, thin] (-2,-3.46410161) to (-3,-3.46410161);
\draw (-4.5,-3.46410161) node{$p\left(s_{\smallsearrow}^A\right)$};

\draw (0,-5.196152423) node{$\bullet$};
\draw[black, thin] (0,-5.196152423) to (0,-6.2);
\draw (0,-7) node{$p\left(s_{\smallrightarrow}^A\right)$};

\draw (6,-3.46410161) node{$\bullet$};
\draw[black, thin] (6,-3.46410161) to (7,-3.46410161);
\draw (8.5,-3.46410161) node{$p\left(s_{\smallnearrow}^A\right)$};

\draw (5.5,0.86602540) node{$\bullet$};
\draw[black, thin] (5.5,0.86602540) to (6.5,0.86602540);
\draw (7.8,0.86602540) node{$p\left(i_{\smallsearrow}^A\right)$};

\draw (0.5,1.7320508) node{$\cdot$};
\draw[black, thin] (0.5,1.7320508) to (0.5,2.5);
\draw (0.5,3) node{$l_2^A$};

\draw (3.5,1.7320508) node{$\cdot$};
\draw[black, thin] (3.5,1.7320508) to (3.5,2.5);
\draw (3.5,3) node{$l_1^A$};


\draw[->] (10,-2) to (11,-2);

\draw[red, thin] (15,0) to (16,1.732050808);
\draw[red, thin] (16,1.732050808) to (18,1.732050808);
\draw[red, thin] (18,1.732050808) to (19,0);
\draw[red, thin] (19,0) to (18,-1.732050808);
\draw[red, thin] (18,-1.732050808) to (16,-1.732050808);
\draw[red, thin] (16,-1.732050808) to (15,0);

\draw[blue, thick, dotted] (15.5,0.8660254) to (13.5,0.8660254);
\draw[blue, thick, dotted] (14,1.732050808) to (13,0);
\draw[blue, thin] (13,0) to (12,-1.732050808);
\draw[blue, thin] (12,-1.732050808) to (14,-5.196152423);
\draw[blue, thin] (14,-5.196152423) to (20,-5.196152423);
\draw[blue, thin] (20,-5.196152423) to (22,-1.732050808);
\draw[blue, thin] (22,-1.732050808) to (20.5,0.86602540);
\draw[blue, thick, dotted] (20.5,0.86602540) to (20,1.732050808);
\draw[blue, thick, dotted] (20.5,0.86602540) to (18.5,0.86602540);

\draw (19,0.86602540) node{$\cdot$};
\draw[black, thin] (19,0.86602540) to (20,-0.5);
\draw (20.3,-0.8) node{$p_1$};

\draw[black, thin] (19.5,0.86602540) to (19.5,1.5);
\draw (19.5,2.3) node{$\tilde{l}_1^A$};

\draw (14.5,0.8660254) node{$\cdot$};
\draw[black, thin] (14.5,0.8660254) to (14.5,-0.5);
\draw (14.5,-0.85) node{$p_2$};

\draw[black, thin] (14.9,0.8660254) to (15,1.5);
\draw (15.2,2.4) node{$\tilde{l}_2^A$};

\draw (13,0) node{$\bullet$};

\draw (13,-3.46410161) node{$\bullet$};

\draw (15,-5.196152423) node{$\bullet$};

\draw (21,-3.46410161) node{$\bullet$};

\draw (20.5,0.86602540) node{$\bullet$};

\draw[->] (23,-2) to (24,-2);

\draw[red, thin] (28,0) to (29,1.732050808);
\draw[red, thin] (29,1.732050808) to (31,1.732050808);
\draw[red, thin] (31,1.732050808) to (32,0);
\draw[red, thin] (32,0) to (31,-1.732050808);
\draw[red, thin] (31,-1.732050808) to (29,-1.732050808);
\draw[red, thin] (29,-1.732050808) to (28,0);

\draw[blue, thin] (28.5,0.8660254) to (26.5,0.8660254);
\draw[blue, thick, dotted] (27,1.732050808) to (26,0);
\draw[blue, thin] (26.5,0.8660254) to (25,-1.732050808);
\draw[blue, thin] (25,-1.732050808) to (27,-5.196152423);
\draw[blue, thin] (27,-5.196152423) to (33,-5.196152423);
\draw[blue, thin] (33,-5.196152423) to (35,-1.732050808);
\draw[blue, thin] (35,-1.732050808) to (33.5,0.86602540);
\draw[blue, thin] (33.5,0.86602540) to (33,1.732050808);
\draw[blue, thin] (33,1.732050808) to (31,1.732050808);




\draw[black, thin] (27.4,0.8660254) to (27.5,1.6);
\draw (27.6,2.4) node{$\tilde{l}_2^A$};

\draw[black, thin] (32,1.7320508) to (32,2.5);
\draw (32,3.2) node{$l_1^A$};

\draw (26,0) node{$\bullet$};

\draw (26,-3.46410161) node{$\bullet$};

\draw (28,-5.196152423) node{$\bullet$};

\draw (34,-3.46410161) node{$\bullet$};

\draw (33.5,0.86602540) node{$\bullet$};

\draw (34,-7) node{$\hookleftarrow$};



\draw[red, thin] (0,-10.3923048) to (-0.5,-9.52627944);
\draw[red, thick, dotted] (0,-10.3923048) to (0.5,-9.52627944);
\draw[red, thin] (0.5,-9.52627944) to (1,-8.66025403);
\draw[red, thin] (1,-8.66025403) to (3,-8.66025403);
\draw[red, thick, dotted] (3,-8.66025403) to (4,-10.3923048);
\draw[red, thin] (4,-10.3923048) to (5,-8.66025403);
\draw[red, thin] (4,-10.3923048) to (3,-12.124355);
\draw[red, thin] (3,-12.124355) to (1,-12.124355);
\draw[red, thin] (1,-12.124355) to (0,-10.3923048);

\draw[red, thin] (5,-8.66025403) to (3,-8.66025403);


\draw[red, thin] (0.5,-9.52627944) to (-0.5,-9.52627944);

\draw[blue, thin] (-0.5,-9.52627944) to (-1.5,-9.52627944);
\draw[blue, thin] (-1.5,-9.52627944) to (-3,-12.124355);
\draw[blue, thin] (-3,-12.124355) to (-1,-15.58845726);
\draw[blue, thin] (-1,-15.58845726) to (5,-15.58845726);
\draw[blue, thin] (5,-15.58845726) to (7,-12.124355);
\draw[blue, thin] (7,-12.124355) to (5.5,-9.52627944);
\draw[blue, thin] (5.5,-9.52627944) to (5,-8.66025403);




\draw[red, thin] (13,-10.3923048) to (12.5,-9.52627944);
\draw[red, thick, dotted] (13,-10.3923048) to (13.5,-9.52627944);
\draw[red, thin] (12.5,-9.52627944) to (13,-8.66025403);

\draw[red, thin] (13,-8.66025403) to (16,-8.66025403);
\draw[red, thick, dotted] (16,-8.66025403) to (17,-10.3923048);
\draw[red, thin] (17,-10.3923048) to (18,-8.66025403);
\draw[red, thin] (17,-10.3923048) to (16,-12.124355);
\draw[red, thin] (16,-12.124355) to (14,-12.124355);
\draw[red, thin] (14,-12.124355) to (13,-10.3923048);

\draw[red, thin] (18,-8.66025403) to (16,-8.66025403);



\draw[blue, thin] (12.5,-9.52627944) to (11.5,-9.52627944);
\draw[blue, thin] (11.5,-9.52627944) to (10,-12.124355);
\draw[blue, thin] (10,-12.124355) to (12,-15.58845726);
\draw[blue, thin] (12,-15.58845726) to (18,-15.58845726);
\draw[blue, thin] (18,-15.58845726) to (20,-12.124355);
\draw[blue, thin] (20,-12.124355) to (18.5,-9.52627944);
\draw[blue, thin] (18.5,-9.52627944) to (18,-8.66025403);



\draw[red, thin] (26,-10.3923048) to (25.5,-9.52627944);
\draw[red, thin] (25.5,-9.52627944) to (26,-8.66025403);
\draw[red, thin] (26,-8.66025403) to (29,-8.66025403);
\draw[red, thin] (30,-8.66025403) to (28,-12.124355);
\draw[red, thick, dotted] (29,-12.124355) to (28,-12.124355);
\draw[red, thin] (28,-12.124355) to (27,-12.124355);
\draw[red, thin] (27,-12.124355) to (26,-10.3923048);
\draw[red, thin] (31,-8.66025403) to (29,-8.66025403);

\draw[blue, thin] (25.5,-9.52627944) to (24.5,-9.52627944);
\draw[blue, thin] (24.5,-9.52627944) to (23,-12.124355);
\draw[blue, thin] (23,-12.124355) to (25,-15.58845726);
\draw[blue, thin] (25,-15.58845726) to (31,-15.58845726);
\draw[blue, thin] (31,-15.58845726) to (33,-12.124355);
\draw[blue, thin] (33,-12.124355) to (31.5,-9.52627944);
\draw[blue, thin] (31.5,-9.52627944) to (31,-8.66025403);

\end{tikzpicture}

\caption{\label{fig:OneAgreeingLineConstruction}}

\end{figure}

We no longer have $A^{\varhexagon}\setminus B^{\varhexagon}$, so we call this new set $A_1$. Notice that we started with the two sets $A^{\varhexagon}\setminus B^{\varhexagon}$ and $B^{\varhexagon}$, which is of the correct form to belong to $\mathfs{F}_{\alpha'}$ for some $\alpha\in(0,1]$, but the volumes may be incorrect. Thus far, we have only adjusted the lengths of the line segments of $\left(\partial A^{\varhexagon}\right)\setminus B^{\varhexagon}$, which means that $\left(\tilde{A},B^{\varhexagon}\right)$ is of the correct form to belong to $\mathfs{F}_{\alpha''}$ for some $\alpha''\in(0,1]$, but again, the volumes may not be correct. We may now need to adjust the volume of either or both of $\tilde{A}$ and $B^{\varhexagon}$ so that their volume ratio is $\alpha$. If our process to achieve this consists of increasing some of the existing line segments and decreasing the others, then we will have constructed a set that is in $\mathfs{F}_{\alpha}$, which is to say that not only will our sets have the correct volume ratio, they will also be of the correct form. 

Before beginning the volume adjustments, we want to ensure that our final configuration is such that there are at most three sides of $\partial\tilde{B}$ that form joint boundary with $\tilde{A}$ at points other than $\pi_1\left(\tilde{A},\tilde{B}\right)$ and $\pi_2\left(\tilde{A},\tilde{B}\right)$. To do this, we reorient any of $\Lambda_{\smallsearrow}\left(i_{\smallsearrow}^B\right)\cap\partial B^{\varhexagon}$ that is joint boundary so that one of its endpoints does not change, and so that it is contained in $\Lambda_{\smallnearrow}\left(s_{\smallnearrow}^B\right)$. This is demonstrated in the lefthand side of the second row of Figure \ref{fig:OneAgreeingLineConstruction}. We then reorient any of $\Lambda_{\smallnearrow}\left(i_{\smallnearrow}^B\right)\cap\partial B^{\varhexagon}$ that is joint boundary so that it is contained in $\Lambda_{\smallsearrow}\left(s_{\smallsearrow}^B\right)$. This is also demonstrated in the lefthand side of the second row of Figure \ref{fig:OneAgreeingLineConstruction}. This has altered $B^{\varhexagon}$, and so we call this new set $B_1$. It has also altered $A_1$, and so we call this new set $A_2$. Then, to ensure that our final configuration is of the correct form, we replace $B_1$ with $B_1^{\varhexagon}$. See the middle configuration of the second row of Figure \ref{fig:OneAgreeingLineConstruction}.

To adjust the volumes to ensure we have one set of volume $1$ and the other of volume $\alpha$, we can move the right side of $B_1$ to the left until we have a set of the correct volume. See the righthand side of Figure \ref{fig:OneAgreeingLineConstruction}. We can then easily adjust the volume of $A_2$ by moving the right sides to the left and the bottom side up.

Now, we come to the configuration when two of the lines used to form $A^{\varhexagon}$ are the same as the corresponding two lines used to form $B^{\varhexagon}$. We may assume that the two lines in question are $\Lambda_{\smallrightarrow}\left(i_{\smallrightarrow}^A\right)=\Lambda_{\smallrightarrow}\left(i_{\smallrightarrow}^B\right)$ and $\Lambda_{\smallsearrow}\left(i_{\smallsearrow}^B\right)=\Lambda_{\smallsearrow}\left(i_{\smallsearrow}^A\right)$. See Figure \ref{fig:TwoAgreeingLinesBeginning}:

\begin{figure}

\begin{tikzpicture}[scale=0.24]


\draw[blue, thin] (11,0) to (12,1.732050808);
\draw[blue, thin] (12,1.732050808) to (14,1.732050808);
\draw[blue, thin] (14,1.732050808) to (15,0);
\draw[blue, thin]  (15,0) to (14,-1.732050808);
\draw[blue, thin] (14,-1.732050808) to (12,-1.732050808);
\draw[blue, thin] (12,-1.732050808) to (11,0);

\draw[red, thin] (12,1.732050808) to (10,1.732050808);
\draw[red, thin] (10,1.732050808) to (8,-1.732050808);
\draw[red, thin] (8,-1.732050808) to (10,-5.196152423);
\draw[red, thin] (10,-5.196152423) to (16,-5.196152423);
\draw[red, thin] (16,-5.196152423) to (17,-3.46411616);
\draw[red, thin] (17,-3.46411616) to (15,0);


\end{tikzpicture}

\caption{\label{fig:TwoAgreeingLinesBeginning}}

\end{figure}

We begin, as always, by replacing $B$ with $B^{\varhexagon}$. Then, we consider the sides of $\partial A^{\varhexagon}$ that do not intersect $B^{\varhexagon}$. Since $\partial A$ must touch each of these line segments in at least one point, the path in $\partial A^{\varhexagon}$ that connects them (that doesn't intersect $B^{\varhexagon}$) can be no longer than the path in $\partial A$ that connects them (and doesn't intersect $B^{\varhexagon}$). See the lefthand side of Figure \ref{fig:TwoAgreeingLines}:

\begin{figure}

\begin{tikzpicture}[scale=0.3]


\draw[red, thin] (3,0) to (4,1.732050808);
\draw[red, thin] (4,1.732050808) to (6,1.732050808);
\draw[red, thin] (6,1.732050808) to (7,0);
\draw[red, thin]  (7,0) to (6,-1.732050808);
\draw[red, thin] (6,-1.732050808) to (4,-1.732050808);
\draw[red, thin] (4,-1.732050808) to (3,0);

\draw[blue, thick, dotted] (4,1.732050808) to (2,1.732050808);
\draw[blue, thick, dotted] (2,1.732050808) to (1,0);
\draw[blue, thin] (1,0) to (0,-1.732050808);
\draw[blue, thin] (0,-1.732050808) to (2,-5.196152423);
\draw[blue, thin] (2,-5.196152423) to (8,-5.196152423);
\draw[blue, thin] (8,-5.196152423) to (8.5,-4.3301270189);
\draw[blue, thick, dotted] (8.5,-4.3301270189) to (9,-3.46411616);
\draw[blue, thick, dotted] (9,-3.46411616) to (7,0);


\draw[black, thin] (3,1.7320508) to (3.1,2.6);
\draw (3.5,3.1) node{$l_1^A$};

\draw[black, thin] (8,-1.7320508) to (9,-1.7320508);
\draw (10,-1.7320508) node{$l_2^A$};

\draw (1,0) node{$\bullet$};
\draw[black, thin] (1,0) to (0,0);
\draw (-1.8,0) node{$p\left(i_{\smallnearrow}^A\right)$};

\draw (1,-3.46410161) node{$\bullet$};
\draw[black, thin] (1,-3.46410161) to (0,-3.46410161);
\draw (-1.8,-3.46410161) node{$p\left(s_{\smallsearrow}^A\right)$};

\draw (3,-5.196152423) node{$\bullet$};
\draw[black, thin] (3,-5.196152423) to (3,-6.1);
\draw (3,-7) node{$p\left(s_{\smallrightarrow}^A\right)$};

\draw (8.5,-4.3301270189) node{$\bullet$};
\draw[black, thin] (8.5,-4.3301270189) to (9.5,-4.3301270189);
\draw (11.5,-4.3301270189) node{$p\left(s_{\smallnearrow}^A\right)$};

\draw[->] (14,-2) to (15,-2);


\draw[red, thin] (19,0) to (20,1.732050808);
\draw[red, thin] (20,1.732050808) to (22,1.732050808);
\draw[red, thin] (22,1.732050808) to (23,0);
\draw[red, thin]  (23,0) to (22,-1.732050808);
\draw[red, thin] (22,-1.732050808) to (20,-1.732050808);
\draw[red, thin] (20,-1.732050808) to (19,0);

\draw[blue, thick, dotted] (19.5,0.8660254038) to (17.5,0.8660254038);
\draw[blue, thick, dotted] (18,1.732050808) to (17,0);
\draw[blue, thin] (17,0) to (16,-1.732050808);
\draw[blue, thin] (16,-1.732050808) to (18,-5.196152423);
\draw[blue, thin] (18,-5.196152423) to (24,-5.196152423);
\draw[blue, thin] (24,-5.196152423) to (24.5,-4.3301270189);
\draw[blue, thick, dotted] (24.5,-4.3301270189) to (25,-3.46411616);
\draw[blue, thick, dotted] (24.5,-4.330127019) to (22.5,-0.8660254038);


\draw[black, thin] (3,1.7320508) to (3.1,2.6);
\draw (3.5,3.1) node{$l_1^A$};

\draw[black, thin] (8,-1.7320508) to (9,-1.7320508);
\draw (10,-1.7320508) node{$l_2^A$};

\draw (17,0) node{$\bullet$};

\draw (17,-3.46410161) node{$\bullet$};

\draw (19,-5.196152423) node{$\bullet$};

\draw (24.5,-4.3301270189) node{$\bullet$};

\draw[black, thin] (19,0.8660254038) to (19.2,1.6);
\draw (19.3,2.5) node{$\tilde{l}_1^A$};

\draw (18.3,0.8660254038) node{$\cdot$};
\draw[black, thin] (18.3,0.8660254038) to (18.7,-0.4);
\draw (18.7,-0.8) node{$p_1$};

\draw[black, thin] (23,-1.7320508) to (24,-1.7320508);
\draw (25,-1.7320508) node{$\tilde{l}_2^A$};

\draw (23.5,-2.5980762114) node{$\cdot$};
\draw[black, thin] (23.5,-2.5980762114) to (23,-3.2);
\draw (22.6,-3.5) node{$p_2$};


\draw[->] (27,-2) to (28,-2);


\draw[red, thin] (32,0) to (33,1.732050808);
\draw[red, thin] (33,1.732050808) to (35,1.732050808);
\draw[red, thin] (35,1.732050808) to (36,0);
\draw[red, thin]  (36,0) to (35,-1.732050808);
\draw[red, thin] (35,-1.732050808) to (33,-1.732050808);
\draw[red, thin] (33,-1.732050808) to (32,0);

\draw[blue, thin] (33,1.7320508) to (31,1.7320508);
\draw[blue, thin] (31,1.732050808) to (30,0);
\draw[blue, thin] (30,0) to (29,-1.732050808);
\draw[blue, thin] (29,-1.732050808) to (31,-5.196152423);
\draw[blue, thin] (31,-5.196152423) to (37,-5.196152423);
\draw[blue, thin] (37,-5.196152423) to (37.5,-4.3301270189);
\draw[blue, thick, dotted] (37.5,-4.3301270189) to (38,-3.46411616);
\draw[blue, thin] (37.5,-4.330127019) to (35.5,-0.8660254038);


\draw (30,0) node{$\bullet$};

\draw (30,-3.46410161) node{$\bullet$};

\draw (32,-5.196152423) node{$\bullet$};

\draw (37.5,-4.3301270189) node{$\bullet$};

\draw[black, thin] (19,0.8660254038) to (19.2,1.6);
\draw (19.3,2.5) node{$\tilde{l}_1^A$};

\draw (18.3,0.8660254038) node{$\cdot$};
\draw[black, thin] (18.3,0.8660254038) to (18.7,-0.4);
\draw (18.7,-0.8) node{$p_1$};

\draw[black, thin] (23,-1.7320508) to (24,-1.7320508);
\draw (25,-1.7320508) node{$\tilde{l}_2^A$};

\draw (23.5,-2.5980762114) node{$\cdot$};
\draw[black, thin] (23.5,-2.5980762114) to (23,-3.2);
\draw (22.6,-3.5) node{$p_2$};

\end{tikzpicture}

\caption{\label{fig:TwoAgreeingLines}}

\end{figure}

We then move $l_1^A$ down and to the left until it intersects a point in $\partial A$. Similarly, we move $l_2^A$ until it intersects a point of $\partial A$. These two adjusted lines we call $\tilde{l}_1^A$ and $\tilde{l}_2^A$. Let $\delta_1$ be the distance between $l_1^A$ and $\tilde{l}_2^A$, and $\delta_2$ be the distance between $l_2^A$ and $\tilde{l}_2^A$. Suppose, without loss of generality, that $\delta_1\leq\delta_2$. Let $\Lambda$ be the horizontal line passing through $\tilde{l}_1^A$. Then between $\Lambda_{\smallrightarrow}\left(i_{\smallrightarrow}^A\right)$ and $\Lambda$ there must be two crossings in $\partial A$. There must also be two paths in $\partial B$ that pass between these two lines. This makes a total of four such crossings. Accounting for possible joint boundary, this leaves at least three separate crossings between $\Lambda$ and $\Lambda_{\smallrightarrow}\left(i_{\smallrightarrow}^A\right)$. In order to form $B^{\varhexagon}$, only two such crossings are needed, which means that we have at least one more path with length $\delta_1$ that we have not yet counted. Therefore, there is enough uncounted boundary to be able to return $\tilde{l}_1^A$ to its original position and regain $l_1^A$. This process is demonstrated in Figure \ref{fig:TwoAgreeingLines}. 

Now, there is a path that is partially in $a_2$ and partially in $\tilde{l}_2^A$ that forms a geodesic from $p\left(s_{\smallnearrow}^A\right)$ to $p_2$ which can be no longer than the path in $\partial A$ that connects these two points. Similarly, there is a path in $\tilde{l}_2^A$ that connects $p_2$ to $\Lambda_{\smallnearrow}\left(s_{\smallnearrow}^B\right)$ which can be no longer than the path in $\partial A$ that connects this same point and line. Adding these two geodesics can be seen in the rightmost configuration of Figure \ref{fig:TwoAgreeingLines}. While this process does not alter $B^{\varhexagon}$, we may have changed $A^{\varhexagon}\setminus B^{\varhexagon}$. So, we name this new set $A_1$. 

Before we adjust the volumes, we want to ensure that our final configuration, call it $\left(\tilde{A},\tilde{B}\right)$, has the property that at most three sides of $\partial\tilde{B}$ form joint boundary with $\partial\tilde{A}\setminus\left(\pi_1\left(\tilde{A},\tilde{B}\right)\cup\pi_2\left(\tilde{A},\tilde{B}\right)\right)$. To do this, we first reorient any of $\Lambda_{\smallnearrow}\left(i_{\smallnearrow}^B\right)\cap\partial B^{\varhexagon}$ that is joint boundary so that it is contained in $\Lambda_{\smallsearrow}\left(s_{\smallsearrow}^B\right)$. Similarly, we reorient any of $\Lambda_{\smallnearrow}\left(s_{\smallnearrow}^B\right)\cap\partial B^{\varhexagon}$ that is joint boundary so that it is contained in $\Lambda_{\smallrightarrow}\left(s_{\smallrightarrow}^B\right)$. This process can be seen in the first two configurations of Figure \ref{fig:TwoAgreeingLinesVolAdj}. This process changes $B^{\varhexagon}$, and so we call this new set $B_1$. It also changes $A_1$, and so we call this new set $A_2$. Then, replace $B_1$ with $B_1^{\varhexagon}$. 

Now, we move the left side of $B_1$, which for simplicity we will call $l_1^{B_1}$, to the right. If instead $\delta_2\leq\delta_1$, we would move $\Lambda_{\smallnearrow}\left(s_{\smallnearrow}^B\right)$ to the left to reduce the volume of $B_1$. However, here we move $l_1^{B_1}$ to the right until we have a new set $B_2$ that either has the correct volume of $\alpha$, or until $l_1^{B_1}$ is collinear with $\tilde{l}_2^A$. If, once $l_1^{B_1}$ is collinear with $\tilde{l}_1^A$, the volume of $B_2$ is still greater than $\alpha$, we can move $\Lambda_{\smallsearrow}\left(s_{\smallnearrow}^B\right)\cap\partial B_2$ to the left until we have a set $\tilde{B}$ of volume $\alpha$. Figure \ref{fig:TwoAgreeingLinesVolAdj} demonstrates this process. We have changed the set $A_2$, and so we call this set $A_3$. To reduce the volume of $A_3$ so that it has volume $1$, we can, if necessary, move its bottom side up until it is collinear with the bottom side of $\tilde{B}$. If we still need to reduce volume, we can then move its left sides to the right until we have a set $\tilde{A}$ of volume $1$.

\begin{figure}

\begin{tikzpicture}[scale=0.3]


\draw[red, thin] (3,0) to (4,1.732050808);
\draw[red, thin] (4,1.732050808) to (6,1.732050808);
\draw[red, thin] (6,1.732050808) to (7,0);
\draw[red, thin]  (7,0) to (6,-1.732050808);
\draw[red, thin] (6,-1.732050808) to (4,-1.732050808);
\draw[red, thin] (4,-1.732050808) to (3,0);

\draw[blue, thin] (4,1.7320508) to (2,1.7320508);
\draw[blue, thin] (2,1.732050808) to (1,0);
\draw[blue, thin] (1,0) to (0,-1.732050808);
\draw[blue, thin] (0,-1.732050808) to (2,-5.196152423);
\draw[blue, thin] (2,-5.196152423) to (8,-5.196152423);
\draw[blue, thin] (8,-5.196152423) to (8.5,-4.3301270189);
\draw[blue, thick, dotted] (8.5,-4.3301270189) to (9,-3.46411616);
\draw[blue, thin] (8.5,-4.330127019) to (6.5,-0.8660254038);

\draw (1.8,-1.4) node{$A_1$};
\draw (5,0.5) node{$B^{\varhexagon}$};



\draw[red, thick, dotted] (16,0) to (17,1.732050808);
\draw[red, thin] (16,0) to (15,1.7320508);
\draw[red, thin] (15,1.7320508) to (17,1.7320508);
\draw[red, thin] (17,1.732050808) to (19,1.732050808);
\draw[red, thin] (19,1.732050808) to (20,0);
\draw[red, thin] (20,0) to (19.5,-0.86602540);
\draw[red, thin] (19.5,-0.86602540) to (20,-1.7320508);
\draw[red, thin] (20,-1.7320508) to (19,-1.7320508);

\draw[red, thin] (19,-1.732050808) to (17,-1.732050808);
\draw[red, thin] (17,-1.732050808) to (16,0);

\draw[blue, thin] (15,1.732050808) to (14,0);
\draw[blue, thin] (14,0) to (13,-1.732050808);
\draw[blue, thin] (13,-1.732050808) to (15,-5.196152423);
\draw[blue, thin] (15,-5.196152423) to (21,-5.196152423);
\draw[blue, thin] (21,-5.196152423) to (21.5,-4.3301270189);
\draw[blue, thick, dotted] (21.5,-4.3301270189) to (22,-3.46411616);
\draw[blue, thin] (21.5,-4.330127019) to (20,-1.7320508);

\draw (14.6,-1.4) node{$A_2$};
\draw (17.6,0.5) node{$B_1$};


\draw (14,0) node{$\bullet$};

\draw (14,-3.46410161) node{$\bullet$};

\draw (16,-5.196152423) node{$\bullet$};

\draw (21.5,-4.3301270189) node{$\bullet$};






\draw[red, thin] (29,0) to (28,1.7320508);
\draw[red, thin] (28,1.7320508) to (30,1.7320508);
\draw[red, thin] (30,1.732050808) to (32,1.732050808);
\draw[red, thin] (32,1.732050808) to (33,0);
\draw[red, thin] (33,0) to (33.5,-0.86602540);
\draw[red, thin] (33.5,-0.86602540) to (33,-1.7320508);
\draw[red, thick, dotted] (33,0) to (32.5,-0.86602540);
\draw[red, thick, dotted] (32.5,-0.86602540) to (33,-1.7320508);
\draw[red, thin] (33,-1.7320508) to (32,-1.7320508);

\draw[red, thin] (32,-1.732050808) to (30,-1.732050808);
\draw[red, thin] (30,-1.732050808) to (29,0);

\draw[blue, thin] (28,1.732050808) to (27,0);
\draw[blue, thin] (27,0) to (26,-1.732050808);
\draw[blue, thin] (26,-1.732050808) to (28,-5.196152423);
\draw[blue, thin] (28,-5.196152423) to (34,-5.196152423);
\draw[blue, thin] (34,-5.196152423) to (34.5,-4.3301270189);
\draw[blue, thick, dotted] (34.5,-4.3301270189) to (35,-3.46411616);
\draw[blue, thin] (34.5,-4.330127019) to (33,-1.7320508);

\draw (28.6,-1.4) node{$A_2$};
\draw (30.6,0.5) node{$B_1^{\varhexagon}$};



\draw[red, thin] (5,-13.8564064) to (3,-10.39230484);
\draw[red, thin] (2,-10.39230484) to (6,-10.39230484);
\draw[red, thin] (6,-10.39230484) to (7,-12.1243556);
\draw[red, thin] (7,-12.1243556) to (7.5,-12.99038105);
\draw[red, thin] (7.5,-12.99038105) to (7,-13.8564064);
\draw[red, thin] (7,-13.8564064) to (5,-13.8564064);
\draw[red, thick, dotted] (5,-13.8564064) to (4,-13.8564064);


\draw[blue, thin] (2,-10.39230484) to (1,-12.1243556);
\draw[blue, thin] (1,-12.1243556) to (0,-13.8564064);
\draw[blue, thin] (0,-13.8564064) to (2,-17.3205080);
\draw[blue, thin] (2,-17.3205080) to (8,-17.3205080);
\draw[blue, thin] (8,-17.3205080) to (8.5,-16.4544826);

\draw[blue, thin] (8.5,-16.4544826) to (7,-13.8564064);



\draw[red, thin] (20,-13.8564064) to (18,-10.39230484);
\draw[red, thin] (18,-10.39230484) to (19,-10.39230484);
\draw[red, thin] (19,-10.39230484) to (20,-12.1243556);
\draw[red, thin] (20,-12.1243556) to (20.5,-12.99038105);
\draw[red, thin] (20.5,-12.99038105) to (20,-13.8564064);
\draw[red, thick, dotted] (20,-13.8564064) to (17,-13.8564064);


\draw[blue, thin] (15,-10.39230484) to (18,-10.39230484);

\draw[blue, thin] (15,-10.39230484) to (14,-12.1243556);
\draw[blue, thin] (14,-12.1243556) to (13,-13.8564064);
\draw[blue, thin] (13,-13.8564064) to (15,-17.3205080);
\draw[blue, thin] (15,-17.3205080) to (21,-17.3205080);
\draw[blue, thin] (21,-17.3205080) to (21.5,-16.4544826);

\draw[blue, thin] (21.5,-16.4544826) to (20,-13.8564064);

\draw (15,-13) node{$A_3$};
\draw (18.9,-11.2) node{$\scaleobj{0.5}{B_2}$};



\draw[red, thin] (33,-13.8564064) to (31,-10.39230484);
\draw[red, thin] (31,-10.39230484) to (32,-10.39230484);
\draw[red, thin] (32,-10.39230484) to (33,-12.1243556);
\draw[red, thin] (33,-12.1243556) to (33.5,-12.99038105);
\draw[red, thin] (33.5,-12.99038105) to (33,-13.8564064);
\draw[red, thick, dotted] (33,-13.8564064) to (30,-13.8564064);


\draw[blue, thin] (28,-10.39230484) to (31,-10.39230484);

\draw[blue, thin] (28,-10.39230484) to (26,-13.8564064);
\draw[blue, thick, dotted] (26,-13.8564064) to (28,-17.3205080);

\draw[blue, thin] (26,-13.8564064) to (33,-13.8564064);

\draw[blue, thick, dotted] (34,-17.3205080) to (34.5,-16.4544826);

\draw[blue, thick, dotted] (34.5,-16.4544826) to (33,-13.8564064);


\end{tikzpicture}

\caption{\label{fig:TwoAgreeingLinesVolAdj}}

\end{figure}

Next, we consider when three of the lines used to form $A^{\varhexagon}$ are the same as the corresponding lines used to form $B^{\varhexagon}$. Without loss of generality, we may assume that the three lines in question are $\Lambda_{\smallrightarrow}\left(i_{\smallrightarrow}^A\right)=\Lambda_{\smallrightarrow}\left(i_{\smallrightarrow}^B\right)$, $\Lambda_{\smallsearrow}\left(i_{\smallsearrow}^A\right)=\Lambda_{\smallsearrow}\left(i_{\smallsearrow}^B\right)$, and $\Lambda_{\smallnearrow}\left(s_{\smallnearrow}^A\right)=\Lambda_{\smallnearrow}\left(s_{\smallnearrow}^B\right)$. See Figure \ref{fig:TwoAgreeingLinesProcess}:

\begin{figure}

\begin{tikzpicture}[scale=0.4]


\draw[red, thin] (3,0) to (4,1.732050808);
\draw[red, thin] (4,1.732050808) to (6,1.732050808);
\draw[red, thin] (6,1.732050808) to (7,0);
\draw[red, thin]  (7,0) to (6,-1.732050808);
\draw[red, thin] (6,-1.732050808) to (4,-1.732050808);
\draw[red, thin] (4,-1.732050808) to (3,0);

\draw[blue, thin] (4,1.732050808) to (2,1.732050808);
\draw[blue, thin] (2,1.732050808) to (0,-1.732050808);
\draw[blue, thin] (0,-1.732050808) to (2,-5.196152423);
\draw[blue, thin] (2,-5.196152423) to (4,-5.196152423);
\draw[blue, thin] (6,-1.732050808) to (4,-5.196152423);

\draw (2.8,-2) node{$\scaleobj{0.5}{A^{\varhexagon}\setminus B^{\varhexagon}}$};
\draw (4.8,0) node{$\scaleobj{0.5}{B^{\varhexagon}}$};

\draw[->] (7.5,-1) to (8.2,-1);



\draw[red, thin] (13,0) to (14,1.732050808);
\draw[red, thin] (14,1.732050808) to (16,1.732050808);
\draw[red, thin] (16,1.732050808) to (17,0);
\draw[red, thin] (17,0) to (16,-1.732050808);
\draw[red, thin] (16,-1.732050808) to (14,-1.732050808);
\draw[red, thin] (14,-1.732050808) to (13,0);

\draw[blue, thick, dotted] (14,1.732050808) to (12,1.732050808);
\draw[blue, thick, dotted] (12,1.7320508) to (11.5,0.86602540);
\draw[blue, thin] (11.5,0.86602540) to (10,-1.732050808);
\draw[blue, thin] (10,-1.732050808) to (12,-5.196152423);
\draw[blue, thin] (12,-5.196152423) to (13,-5.196152423);
\draw[blue, thick, dotted] (13,-5.196152423) to (14,-5.196152423);
\draw[blue, thick, dotted] (16,-1.732050808) to (14,-5.196152423);

\draw (11.5,0.86602540) node{$\bullet$};
\draw[black, thin] (11.5,0.86602540) to (11,0.9);
\draw (10,1) node{$\scaleobj{0.7}{p\left(i_{\smallnearrow}^A\right)}$};

\draw (11,-3.46410161) node{$\bullet$};
\draw[black, thin] (11,-3.46410161) to (10,-4);
\draw (8.9,-4.2) node{$\scaleobj{0.7}{p\left(s_{\smallsearrow}^A\right)}$};

\draw (13,-5.196152423) node{$\bullet$};
\draw[black, thin] (13,-5.196152423) to (13,-6);
\draw (13,-6.6) node{$\scaleobj{0.7}{p\left(s_{\smallrightarrow}^A\right)}$};

\draw (13,2.4) node{$\scaleobj{0.7}{l_1^A}$};

\draw (15.6,-3.5) node{$\scaleobj{0.7}{l_2^A}$};

\draw[->] (17.5,-1) to (18.2,-1);



\draw[red, thin] (23,0) to (24,1.732050808);
\draw[red, thin] (24,1.732050808) to (26,1.732050808);
\draw[red, thin] (26,1.732050808) to (27,0);
\draw[red, thin] (27,0) to (26,-1.732050808);
\draw[red, thin] (26,-1.732050808) to (24,-1.732050808);
\draw[red, thin] (24,-1.732050808) to (23,0);

\draw[blue, thick, dotted] (23.5,0.86602540) to (21.5,0.86602540);
\draw[blue, thick, dotted] (22,1.7320508) to (21.5,0.86602540);
\draw[blue, thin] (21.5,0.86602540) to (20,-1.732050808);
\draw[blue, thin] (20,-1.732050808) to (22,-5.196152423);
\draw[blue, thin] (22,-5.196152423) to (23,-5.196152423);
\draw[blue, thick, dotted] (23,-5.196152423) to (24,-5.196152423);
\draw[blue, thick, dotted] (25,-1.732050808) to (23,-5.196152423);

\draw (21.5,0.86602540) node{$\bullet$};
\draw[black, thin] (21.5,0.86602540) to (21,0.9);
\draw (19.6,1) node{$\scaleobj{0.6}{p_1=p\left(i_{\smallnearrow}^A\right)}$};

\draw (21,-3.46410161) node{$\bullet$};
\draw[black, thin] (21,-3.46410161) to (20,-4);
\draw (18.9,-4.2) node{$\scaleobj{0.6}{p\left(s_{\smallsearrow}^A\right)}$};

\draw (23,-5.196152423) node{$\bullet$};
\draw[black, thin] (23,-5.196152423) to (23,-6);
\draw (23,-6.6) node{$\scaleobj{0.6}{p_2=p\left(s_{\smallrightarrow}^A\right)}$};

\draw[black, thin] (22.8,2) to (23,0.866);
\draw (23,2.22) node{$\scaleobj{0.7}{\tilde{l}_1^A}$};

\draw[black, thin] (24.6,-3.5) to (24,-3.5);
\draw (24.9,-3.5) node{$\scaleobj{0.7}{\tilde{l}_2^A}$};

\draw[->] (27.5,-1) to (28.2,-1);



\draw[red, thin] (33,0) to (34,1.732050808);
\draw[red, thin] (34,1.732050808) to (36,1.732050808);
\draw[red, thin] (36,1.732050808) to (37,0);
\draw[red, thin] (37,0) to (36,-1.732050808);
\draw[red, thin] (36,-1.732050808) to (34,-1.732050808);
\draw[red, thin] (34,-1.732050808) to (33,0);

\draw[blue, thin] (34,1.7320508) to (32,1.7320508);
\draw[blue, thin] (32,1.7320508) to (30,-1.732050808);
\draw[blue, thin] (30,-1.732050808) to (32,-5.196152423);
\draw[blue, thin] (32,-5.196152423) to (33,-5.196152423);
\draw[blue, thick, dotted] (33,-5.196152423) to (34,-5.196152423);
\draw[blue, thin] (35,-1.732050808) to (33,-5.196152423);

\draw (32.8,-2) node{$\scaleobj{0.5}{A_1}$};
\draw (34.8,0) node{$\scaleobj{0.5}{B^{\varhexagon}}$};



\draw[red, thin] (5,-13.8564064) to (6.5,-11.2583302);
\draw[red, thin] (6,-10.39230484) to (7,-12.1243556);
\draw[red, thin] (7,-12.1243556) to (6,-13.8564064);
\draw[red, thin] (6,-13.8564064) to (4,-13.8564064);
\draw[red, thick, dotted] (4,-13.8564064) to (3,-12.1243556);

\draw[blue, thin] (6,-10.39230484) to (2,-10.39230484);
\draw[blue, thin] (2,-10.39230484) to (0,-13.8564064);
\draw[blue, thin] (0,-13.8564064) to (2,-17.3205080);
\draw[blue, thin] (2,-17.3205080) to (3,-17.3205080);
\draw[blue, thin] (5,-13.8564064) to (3,-17.3205080);

\draw[->] (7.8,-14) to (8.4,-14);



\draw[red, thin] (15.5,-13.8564064) to (16.75,-11.6913429);
\draw[red, thin] (16,-10.39230484) to (17,-12.1243556);
\draw[red, thin] (17,-12.1243556) to (16,-13.8564064);
\draw[red, thin] (16,-13.8564064) to (14,-13.8564064);
\draw[red, thick, dotted] (14,-13.8564064) to (13,-12.1243556);

\draw[blue, thin] (16,-10.39230484) to (12,-10.39230484);
\draw[blue, thin] (12,-10.39230484) to (10,-13.8564064);
\draw[blue, thin] (10,-13.8564064) to (12,-17.3205080);
\draw[blue, thin] (12,-17.3205080) to (13,-17.3205080);
\draw[blue, thin] (15,-13.8564064) to (13,-17.3205080);

\draw[->] (17.8,-14) to (18.4,-14);



\draw[red, thin] (25.5,-13.8564064) to (26.75,-11.6913429);
\draw[red, thin] (26,-10.39230484) to (27,-12.1243556);
\draw[red, thin] (27,-12.1243556) to (26,-13.8564064);
\draw[red, thin] (26,-13.8564064) to (24,-13.8564064);
\draw[red, thick, dotted] (24,-13.8564064) to (23,-12.1243556);

\draw[blue, thin] (26,-10.39230484) to (22,-10.39230484);
\draw[blue, thin] (22,-10.39230484) to (20,-13.8564064);
\draw[blue, thin] (20,-13.8564064) to (22,-17.3205080);
\draw[blue, thin] (22,-17.3205080) to (23.5,-17.3205080);
\draw[blue, thin] (25.5,-13.8564064) to (23.5,-17.3205080);

\draw[blue, thick, dotted] (25,-13.8564064) to (23,-17.3205080);


\end{tikzpicture}

\caption{\label{fig:TwoAgreeingLinesProcess}}

\end{figure}

First, we replace $B$ with $B^{\varhexagon}$, and then consider the sides of $A^{\varhexagon}$ that do not intersect $B^{\varhexagon}$. There must be at least one point in each of these line segments that $\partial A$ intersects, and the path in $\partial A^{\varhexagon}$ that connects them (and doesn't intersect $B^{\varhexagon}$) can be no longer than the path in $\partial A$ that connects them (and doesn't intersect $B^{\varhexagon}$). See the second configuration in the top row of Figure \ref{fig:TwoAgreeingLinesProcess}. Then, we move $l_1^A$ down and left until it intersects a point $p_1\in\partial A$. Similarly, we move $l_2^A$ to the left until it intersects a point $p_2\in\partial A$. Since we have moved $l_1^A$ and $l_2^A$, we call these two new lines $\tilde{l}_1^A$ and $\tilde{l}_2^A$, respectively. See the third configuration in the top row of Figure \ref{fig:TwoAgreeingLinesProcess}. We know that the path that is partially in $a_1$ and partially in $\tilde{l}_1^A$ that connects $p\left(i_{\smallnearrow}^A\right)$ to $p_1$, and then $p_1$ to $\Lambda_{\smallnearrow}\left(i_{\smallnearrow}^B\right)$, can be no longer than the path in $\partial A$ that first connects these same two points, and then $p_1$ to $\Lambda_{\smallnearrow}\left(i_{\smallnearrow}^B\right)$. A similar argument can be made for connecting $p\left(s_{\smallrightarrow}^A\right)$ to $\Lambda_{\smallrightarrow}\left(s_{\smallrightarrow}^B\right)$. 

Let $\delta_1$ be the distance between $l_1^A$ and $\tilde{l}_1^A$, and $\delta_2$ be the distance between $l_2^A$ and $\tilde{l}_2^A$. Suppose, without loss of generality, that $\delta_1\leq\delta_2$. In this case, let $\Lambda$ be the horizontal line that passes through $\tilde{l}_1^A$. In $\partial A$, there must be at least two paths between $\Lambda_{\smallrightarrow}\left(i_{\smallrightarrow}^A\right)$ and $\Lambda$. Similarly, there must be at least two paths in $\partial B$ between these same two lines. Accounting for possible joint boundary, this leaves at least three separate paths between $\Lambda$ and $\Lambda_{\smallrightarrow}\left(i_{\smallrightarrow}^A\right)$. To form $B^{\varhexagon}$, only two such paths are required. This means that there is a path in $\partial A$ that we have not yet counted, and that has length at least $\delta_1$. Therefore, we can move $\tilde{l}_1^A$ back to its original position and regain $l_1^A$. See the rightmost configuration in the top row of Figure \ref{fig:TwoAgreeingLinesProcess}. 

Notice that in this case, there are at most three line segments in $\partial B^{\varhexagon}$ that form joint boundary with $\partial A_1\setminus\left(\pi_1\left(A_1,B^{\varhexagon}\right)\cup\pi_2\left(A_1,B^{\varhexagon}\right)\right)$. So, we are ready to adjust the volumes, beginning with $B^{\varhexagon}$. If necessary, we move $\Lambda_{\smallnearrow}\left(i_{\smallnearrow}^B\right)\cap\partial B^{\varhexagon}$ to the right. We will call this line $l_1^B$. We move $l_1^B$ until we obtain a set $B_1$ that either has the correct volume, or until $l_1^B$ is collinear with $\tilde{l}_2^B$. If we still need to reduce the volume of $B_1$, we can continue to move $l_1^B$ to the right, but now, we also move $l_2^A$ with it. This is demonstrated in the bottom row of Figure \ref{fig:TwoAgreeingLinesProcess}. Eventually, this process will lead to, potentially, two new sets, which we call $A_1$ and $\tilde{B}$, with $\mu\left(\tilde{B}\right)=\alpha$. If we need to reduce the volume of $A_1$, we can move its bottom side up until it is collinear with the bottom side of $\tilde{B}$. If further volume adjustment is required, then we can move the left sides of $A_1$ to the right until we have a new set $\tilde{A}$ with volume $1$

Finally, we come to the possibility that $B^{\varhexagon}\subset A^{\varhexagon}$, but none of the lines used to form $A^{\varhexagon}$ are the same as the corresponding lines used to form $B^{\varhexagon}$. We begin as usual by replacing $B$ with $B^{\varhexagon}$. Consider $\partial A$ as a path $\lambda:[0,1]\rightarrow\mathbb{R}^2$. We may assume that $\lambda(0)=\lambda(1)\in\Lambda_{\smallrightarrow}\left(i_{\smallrightarrow}^A\right)$, that $\lambda$ runs clockwise in the sense that for some $t_1\geq0$ $\lambda\left(t_1\right)\in\Lambda_{\smallsearrow}\left(i_{\smallsearrow}^A\right)$, for some $t_2\geq t_1$ $\lambda\left(t_2\right)\in\Lambda_{\smallnearrow}\left(s_{\smallnearrow}^A\right)$, etc. There must be a last line segment in $\partial A^{\varhexagon}$ that $\lambda$ intersects before intersecting $B^{\varhexagon}$. Then, there must be a first line segment in $\partial A^{\varhexagon}$ that $\lambda$ intersects after the last time that $\lambda$ intersects $\partial B^{\varhexagon}$. We may assume, without loss of generality, that the first line segment in $\partial A$ that $\lambda$ intersects after the last time that it intersects $\partial B$ is $\Lambda_{\smallrightarrow}\left(i_{\smallrightarrow}^A\right)$. Therefore, the last line segment that $\lambda$ must intersect in $\partial A^{\varhexagon}$ before intersecting $B^{\varhexagon}$ is either $\Lambda_{\smallnearrow}\left(i_{\smallnearrow}^A\right)$ or $\Lambda_{\smallrightarrow}\left(i_{\smallrightarrow}^A\right)$.  See Figure \ref{fig:ZeroAgreeingLines} for illustration.

\begin{figure}[H]

\begin{tikzpicture}[scale=0.35]


\draw[blue, thin] (0,0) to (4,6.92820323);
\draw[blue, thin] (4,6.92820323) to (8,6.92820323);
\draw[blue, thin] (8,6.92820323) to (12,0);
\draw[blue, thin] (12,0) to (8,-6.92820323);
\draw[blue, thin] (8,-6.92820323) to (4,-6.92820323);
\draw[blue, thin] (4,-6.92820323) to (0,0);

\draw[red, thin] (2,1.7320508) to (3,3.4641016);
\draw[red, thin] (3,3.4641016) to (4,3.4641016);
\draw[red, thin] (4,3.4641016) to (5,1.7320508);
\draw[red, thin] (5,1.7320508) to (4,0);
\draw[red, thin] (4,0) to (3,0);
\draw[red, thin] (3,0) to (2,1.7320508);

\draw[black,thin] plot [smooth, tension=0.3] coordinates {(6,6.92820323) (6.3,6.8)(8.25,5.2)(8,4.4)(10,3.4641016)};

\draw[black, thin] plot [smooth, tension=0.3] coordinates {(10,3.4641016) (8.8,1)(9.1,-2)(9,-5.1961524)};

\draw[black, thin] plot [smooth, tension=0.3] coordinates {(9,-5.1961524) (8,-6.5)(7,-5.8)(6,-6.92820323)};

\draw[black, thin] plot [smooth, tension=0.3] coordinates {(6,-6.92820323) (5.5,-6.5)(4,-6)(2,-3.4641016)};

\draw[black, thin] plot [smooth, tension=0.3] coordinates {(2,-3.4641016) (1.8,-2.7)(0.9,1.2)(1,1.7320508)};

\draw[green, thin] plot [smooth, tension=0.3] coordinates {(1,1.7320508) (1.4,2)(2.2,0.7)(3,0)};
\draw[green] (1.75,0.7) node{$P_1$};

\draw[purple, thin] plot [smooth, tension=0.3] coordinates {(4,3.4641016) (3.7,4.2)(4.5,5.8)(5.4,6.8)(6,6.92820323)};
\draw[purple] (4.5,5) node{$P_2$};

\draw (6,6.92820323) node{$\cdot$};
\draw[black, thin] (6,6.92820323) to (6,8);
\draw (6,8.5) node{$p\left(i_{\smallrightarrow}^A\right)$};

\draw (10,3.4641016) node{$\cdot$};
\draw[black, thin] (10,3.4641016) to (12,4);
\draw (13.3,4) node{$p\left(i_{\smallsearrow}^A\right)$};

\draw (9,-5.1961524) node{$\cdot$};
\draw[black, thin] (9,-5.1961524) to (11,-6);
\draw (12.3,-6) node{$p\left(s_{\smallnearrow}^A\right)$};

\draw (6,-6.92820323) node{$\cdot$};
\draw[black, thin] (6,-6.92820323) to (6,-8);
\draw (6,-8.5) node{$p\left(s_{\smallrightarrow}^A\right)$};

\draw (2,-3.4641016) node{$\cdot$};
\draw[black, thin] (2,-3.4641016) to (0,-4);
\draw (-1.4,-4) node{$p\left(s_{\smallsearrow}^A\right)$};

\draw (1,1.7320508) node{$\cdot$};
\draw[black, thin] (1,1.7320508) to (0,2);
\draw (-1.4,2) node{$p\left(i_{\smallnearrow}^A\right)$};

\draw (3,0) node{$\cdot$};
\draw[black, thin] (3,0) to (3.1,-0.5);
\draw (3.2,-0.85) node{$\lambda\left(t_{en}\right)$};

\draw (4,3.4641016) node{$\cdot$};
\draw[black, thin] (4,3.4641016) to (5,3.5);
\draw (6,3.5) node{$\lambda\left(t_{ex}\right)$};

\draw (6,0) node{$A$};

\draw (3.4,2) node{$B^{\varhexagon}$};



\draw[blue, thin] (19,0) to (20,1.7320508);
\draw[blue, thick, dotted] (20,1.7320508) to (23,6.92820323);
\draw[blue, thick, dotted] (23,6.92820323) to (25,6.92820323);
\draw[blue, thin] (25,6.92820323) to (27,6.92820323);
\draw[blue, thin] (27,6.92820323) to (31,0);
\draw[blue, thin] (31,0) to (27,-6.92820323);
\draw[blue, thin] (27,-6.92820323) to (23,-6.92820323);
\draw[blue, thin] (23,-6.92820323) to (19,0);

\draw[red, thin] (21,1.7320508) to (22,3.4641016);
\draw[red, thin] (22,3.4641016) to (23,3.4641016);
\draw[red, thin] (23,3.4641016) to (24,1.7320508);
\draw[red, thin] (24,1.7320508) to (23,0);
\draw[red, thin] (23,0) to (22,0);
\draw[red, thin] (22,0) to (21,1.7320508);

\draw[black, thick, dotted] (21,6.92820323) to (26,-1.7320508);
\draw (21,7.3) node{$\scaleobj{0.6}{\Lambda_{\smallsearrow}\left(i_{\smallsearrow}^B\right)}$};

\draw[black, thick, dotted] (19,5.1961524) to (24,-3.46410161);
\draw (19,5.7) node{$\scaleobj{0.6}{\Lambda_{\smallsearrow}\left(s_{\smallsearrow}^B\right)}$};

\draw[black,thin] plot [smooth, tension=0.3] coordinates {(25,6.92820323) (25.3,6.8)(27.25,5.2)(27,4.4)(29,3.4641016)};

\draw[black, thin] plot [smooth, tension=0.3] coordinates {(29,3.4641016) (27.8,1)(28.1,-2)(28,-5.1961524)};

\draw[black, thin] plot [smooth, tension=0.3] coordinates {(28,-5.1961524) (27,-6.5)(26,-5.8)(25,-6.92820323)};

\draw[black, thin] plot [smooth, tension=0.3] coordinates {(25,-6.92820323) (24.5,-6.5)(23,-6)(21,-3.4641016)};

\draw[black, thin] plot [smooth, tension=0.3] coordinates {(21,-3.4641016) (20.8,-2.7)(19.9,1.2)(20,1.7320508)};

\draw[green, thin] plot [smooth, tension=0.3] coordinates {(20,1.7320508) (20.4,2)(21.2,0.7)(22,0)};
\draw[green] (20.75,0.7) node{$P_1$};

\draw[purple, thin] plot [smooth, tension=0.3] coordinates {(23,3.4641016) (22.7,4.2)(23.5,5.8)(24.4,6.8)(25,6.92820323)};
\draw[purple] (23.5,5) node{$P_2$};

\draw (25,6.92820323) node{$\cdot$};
\draw[black, thin] (25,6.92820323) to (25,8);
\draw (25,8.5) node{$p\left(i_{\smallrightarrow}^A\right)$};

\draw (29,3.4641016) node{$\cdot$};
\draw[black, thin] (29,3.4641016) to (31,4);
\draw (32.3,4) node{$p\left(i_{\smallsearrow}^A\right)$};

\draw (28,-5.1961524) node{$\cdot$};
\draw[black, thin] (28,-5.1961524) to (30,-6);
\draw (31.3,-6) node{$p\left(s_{\smallnearrow}^A\right)$};

\draw (25,-6.92820323) node{$\cdot$};
\draw[black, thin] (25,-6.92820323) to (25,-8);
\draw (25,-8.5) node{$p\left(s_{\smallrightarrow}^A\right)$};

\draw (21,-3.4641016) node{$\cdot$};
\draw[black, thin] (21,-3.4641016) to (19,-4);
\draw (17.6,-4) node{$p\left(s_{\smallsearrow}^A\right)$};

\draw (20,1.7320508) node{$\cdot$};
\draw[black, thin] (20,1.7320508) to (19,2);
\draw (17.6,2) node{$p\left(i_{\smallnearrow}^A\right)$};

\draw (22,0) node{$\cdot$};
\draw[black, thin] (22,0) to (22.1,-0.5);
\draw (22.2,-0.85) node{$\lambda\left(t_{en}\right)$};

\draw (23,3.4641016) node{$\cdot$};
\draw[black, thin] (23,3.4641016) to (24,3.5);
\draw (25,3.5) node{$\lambda\left(t_{ex}\right)$};

\draw (25,0) node{$A$};

\draw (22.4,2) node{$B^{\varhexagon}$};


\end{tikzpicture}

\caption{\label{fig:ZeroAgreeingLines}}

\end{figure}

Before we begin, it is important to notice that either the line segment $\Lambda_{\smallnearrow}\left(i_{\smallnearrow}^A\right)\cap\partial A^{\varhexagon}$ is longer than the line segment $\Lambda_{\smallnearrow}\left(i_{\smallnearrow}^B\right)\cap\partial B^{\varhexagon}$, or the line segment $\Lambda_{\smallrightarrow}\left(i_{\smallrightarrow}^A\right)\cap\partial A^{\varhexagon}$ is longer than the line segment $\Lambda_{\smallrightarrow}\left(i_{\smallrightarrow}^B\right)\cap\partial B^{\varhexagon}$ (or both). Otherwise, $B^{\varhexagon}$ could not be contained in $A^{\varhexagon}$. We assume, without loss of generality, that the line segment $\Lambda_{\smallnearrow}\left(i_{\smallnearrow}^A\right)\cap\partial A^{\varhexagon}$ is longer than the line segment $\Lambda_{\smallnearrow}\left(i_{\smallnearrow}^B\right)\cap\partial B^{\varhexagon}$. 

As usual, we first replace $B$ with $B^{\varhexagon}$. Then, we can connect the points $p\left(i_{\smallrightarrow}^A\right)$ and $p\left(i_{\smallsearrow}^A\right)$ via the geodesic in $\partial A^{\varhexagon}$. Similarly, we can connect $p\left(i_{\smallsearrow}^A\right)$ and $p\left(s_{\smallnearrow}^A\right)$ via the geodesic in $\partial A^{\varhexagon}$. We continue like this until we reach $p\left(i_{\smallnearrow}^A\right)$. This is shown in Figure \ref{fig:ZeroAgreeingLines}. 

At some time, $\lambda$ intersects $\partial B^{\varhexagon}$ for the first time. Call this time $t_{en}$. Consider the portion of $\partial A^{\varhexagon}$ between $p\left(i_{\smallnearrow}^A\right)$ and $\lambda\left(t_{en}\right)$, call this path $P_1$. Similarly, let $t_{ex}$ be the last time that $\lambda\in B^{\varhexagon}$, and consider the portion of $\partial A$ between $\lambda\left(t_{ex}\right)$ and $p\left(i_{\smallrightarrow}^A\right)$, call this path $P_2$.

Consider the line $\Lambda_{\smallsearrow}\left(s_{\smallsearrow}^B\right)$. If $p\left(i_{\smallnearrow}^A\right)$ is in the lower half plane $L\left(s_{\smallsearrow}^B\right)$, then there is a geodesic in $\partial A^{\varhexagon}$ between $p\left(i_{\smallnearrow}^A\right)$ and $\Lambda_{\smallsearrow}\left(s_{\smallsearrow}^B\right)$ which can be no longer than $P_1$. If, on the other hand, $p\left(i_{\smallnearrow}^A\right)$ is in the upper half plane $U\left(s_{\smallsearrow}^B\right)$, then we already have enough perimeter for this part of our construction, as will be seen. 

Now, consider the line $\Lambda_{\smallsearrow}\left(i_{\smallsearrow}^B\right)$. If $p\left(i_{\smallrightarrow}^A\right)$ is in the upper half plane $U\left(i_{\smallsearrow}^B\right)$, then $P_2$ is at least as long as the straight line segment with slope $-\sqrt{3}$ between $p\left(i_{\smallrightarrow}^A\right)$ and $\Lambda_{\smallsearrow}\left(i_{\smallsearrow}^B\right)$, which is the same length as the path in $\partial A^{\varhexagon}$ that connects this point and line since they are both geodesics. If $p\left(i_{\smallrightarrow}^A\right)$ is in the lower half plane $L\left(i_{\smallsearrow}^B\right)$, then we already have enough perimeter for this part of our construction, as we shall see. 

Since we assumed that $\Lambda_{\smallnearrow}\left(i_{\smallnearrow}^A\right)\cap\partial A^{\varhexagon}$ is at least as long as $\Lambda_{\smallnearrow}\left(i_{\smallnearrow}^B\right)\cap\partial B^{\varhexagon}$, we can translate $B^{\varhexagon}$ up and to the left between the lines $\Lambda\left(s_{\smallsearrow}^B\right)$ and $\Lambda_{\smallsearrow}\left(i_{\smallsearrow}^B\right)$ until we obtain a translation $B_t^{\varhexagon}$ such that $i_{\smallnearrow}^{B_t^{\varhexagon}}=i_{\smallnearrow}^A$. 

If $\Lambda_{\smallrightarrow}\left(i_{\smallrightarrow}^B\right)\cap\partial B^{\varhexagon}$ is also no longer than $\Lambda_{\smallrightarrow}\left(i_{\smallrightarrow}^A\right)\cap\partial A^{\varhexagon}$, then, we translate $B_t^{\varhexagon}$ up and to the right between $\Lambda
_{\smallnearrow}\left(i_{\smallnearrow}^{B_t^{\varhexagon}}\right)$ and $\Lambda
_{\smallnearrow}\left(s_{\smallnearrow}^{B_t^{\varhexagon}}\right)$. The result looks something like the lefthand side of Figure \ref{fig:FinalTranslation}. 

If, on the other hand $\Lambda_{\smallrightarrow}\left(i_{\smallrightarrow}^B\right)\cap\partial B^{\varhexagon}$ is longer than $\Lambda_{\smallrightarrow}\left(i_{\smallrightarrow}^A\right)\cap\partial A^{\varhexagon}$, then we still translate $B_t^{\varhexagon}$ up and to the right in the same way as before, but we stop when $\Lambda_{\smallrightarrow}\left(i_{\smallrightarrow}^{B_t^{\varhexagon}}\right)\cap\Lambda_{\smallsearrow}\left(i_{\smallsearrow}^{B_t^{\varhexagon}}\right)$ (this is something like the top right corner of $B_t^{\varhexagon}$) intersects $\partial A^{\varhexagon}$. By slight abuse of notation, we will still call this set $B_t^{\varhexagon}$. Then, we can remove the top of $B_t^{\varhexagon}$ (i.e. $\Lambda_{\smallrightarrow}\left(i_{\smallrightarrow}^{B_t^{\varhexagon}}\right)\cap\partial B_t^{\varhexagon}$), which adds some of the volume of $A^{\varhexagon}$ to $B_t^{\varhexagon}$. We call this new set $B_1$. This is shown in the righthand configuration of Figure \ref{fig:FinalTranslation}.

\begin{figure}[H]

\begin{tikzpicture}[scale=0.15]


\draw[blue, thin] (0,0) to (4,6.92820323);
\draw[blue, thin] (4,6.92820323) to (8,6.92820323);
\draw[blue, thin] (8,6.92820323) to (12,0);
\draw[blue, thin] (12,0) to (8,-6.92820323);
\draw[blue, thin] (8,-6.92820323) to (4,-6.92820323);
\draw[blue, thin] (4,-6.92820323) to (0,0);

\draw[red, thin] (3,5.1961524) to (4,6.92820323);
\draw[red, thin] (4,6.92820323) to (5,6.92820323);
\draw[red, thin] (5,6.92820323) to (6,5.1961524);
\draw[red, thin] (6,5.1961524) to (5,3.4641016);
\draw[red, thin] (5,3.4641016) to (4,3.4641016);
\draw[red, thin] (4,3.4641016) to (3,5.1961524);



\draw[blue, thin] (20,0) to (25,8.66025403);
\draw[blue, thin] (25,8.66025403) to (27,8.66025403);
\draw[blue, thin] (27,8.66025403) to (32,0);
\draw[blue, thin] (32,0) to (28,-6.92820323);
\draw[blue, thin] (28,-6.92820323) to (24,-6.92820323);
\draw[blue, thin] (24,-6.92820323) to (20,0);

\draw[red, thin] (23,5.1961524) to (24,6.92820323);
\draw[red, thick, dotted] (24,6.92820323) to (28,6.92820323);
\draw[red, thin] (28,6.92820323) to (29,5.1961524);
\draw[red, thin] (29,5.1961524) to (28,3.4641016);
\draw[red, thin] (28,3.4641016) to (24,3.4641016);
\draw[red, thin] (24,3.4641016) to (23,5.1961524);


\end{tikzpicture}

\caption{\label{fig:FinalTranslation}}

\end{figure}

At this point, the volume adjustments are simple. We can rescale $B_t^{\varhexagon}$ in the lefthand configuration to create a set $\tilde{B}_t^{\varhexagon}$ with volume $\alpha$, and $B_1$ in the righthand configuration to create a set $B_2$ with volume $\alpha$. Similarly, we can reduce the volume of $A^{\varhexagon}\setminus \tilde{B}_t^{\varhexagon}$ in the lefthand configuration and $A^{\varhexagon}\setminus B_2$ in the righthand configuration.

\end{proof}

\section{Eliminating Sixty Degree Angles}
\label{sec:Elimination}

The previous section showed us that we should be able to find a solution to the double bubble problem in the family $\mathfs{F}_{\alpha}$. We now wish to further reduce the number of possible configurations that we have to analyse. 

For ease of notation, we assume, in this section, that $\left(A,B\right)\in\mathfs{F}_{\alpha}$. We will eliminate most cases in which there are $60^{\circ}$ degree angles formed by either two line segments forming the boundary of one set, or one line segment from $A$ and another from $B$ forming an exterior $60^{\circ}$ degree angle. 

Consider the boundary of $\left(A,B\right)$. We want to show that if one of the line segments forming $\partial B$ meets one of the line segments forming $\partial A$ in such a way as to create an exterior $60^{\circ}$ angle, then $(A,B)$ cannot be a minimizing configuration. In other words, the exterior angle formed by $\partial A$ and $\partial B$ at $\pi_1$ and $\pi_2$ must be either $120^{\circ}$ or $180^{\circ}$ (Note that these angles can only be $60^{\circ}$, $120^{\circ}$, or $180^{\circ}$).

\begin{lemma}\label{lemma:NoSixtyDegreeAngles}
Let $(A,B)\in\mathfs{F}_{\alpha}$. Suppose that one of the line segments forming $\partial A$, call it $l_1$, and one of the line segments forming $\partial B$, call it $l_2$, meet in such a way as to form an exterior angle of $60^{\circ}$. Then $(A,B)\notin\Gamma_{\alpha}$. 
\end{lemma}

\begin{proof}
Through a series of reflections and/or $60^{\circ}$ rotations, we may assume that the exterior angle is formed as in Figure \ref{fig:ExteriorSixtyDegrees}. The sides of $\partial A$ and $\partial B$ that are not $l_1$ or $l_2$ may, of course, be different. It will not affect our analysis.

\begin{figure}[H]

\begin{tikzpicture}[scale=0.4]

\draw[black, thick, dotted] (-3,5.19615242) to (2,-3.46410161);\draw (-2.45,5) node{$\Lambda$};

\draw[blue, thin] (0,0) to (-2,3.46410161);
\draw[blue, thin] (-2,3.46410161) to (-4,3.46410161);
\draw[blue, thin] (-4,3.46410161) to (-6,0);
\draw[blue, thin] (-6,0) to (-4,-3.46410161);
\draw[blue, thin] (-4,-3.46410161) to (0,-3.46410161);
\draw[blue, thin] (0,-3.46410161) to (1,-1.7320508);

\draw (-0.7,2) node{$l_1$};
\draw (-3,0) node{$A$};

\draw[red, thin] (0,0) to (1,1.7320508);
\draw[red, thin] (1,1.7320508) to (2,1.7320508);
\draw[red, thin] (2,1.7320508) to (3,0);
\draw[red, thin] (3,0) to (2,-1.7320508);
\draw[red, thin] (2,-1.7320508) to (1,-1.7320508);
\draw[red, thin] (1,-1.7320508) to (0,0);

\draw (1,1) node{$l_2$};
\draw (1.5,0) node{$B$};


\draw[->] (3.5,0) to (4.5,0);

\draw[blue, thin] (11,0) to (9,3.46410161);
\draw[blue, thin] (9,3.46410161) to (7,3.46410161);
\draw[blue, thin] (7,3.46410161) to (5,0);
\draw[blue, thin] (5,0) to (7,-3.46410161);
\draw[blue, thin] (7,-3.46410161) to (11,-3.46410161);
\draw[blue, thin] (11,-3.46410161) to (12,-1.7320508);

\draw (10.3,2) node{$l_1$};
\draw (8,0) node{$A$};

\draw[red, thick, dotted] (11,0) to (12,1.7320508);
\draw[red, thin] (10,1.7320508) to (12,1.7320508);
\draw[red, thin] (12,1.7320508) to (13,1.7320508);
\draw[red, thin] (13,1.7320508) to (14,0);
\draw[red, thin] (14,0) to (13,-1.7320508);
\draw[red, thin] (13,-1.7320508) to (12,-1.7320508);
\draw[red, thin] (12,-1.7320508) to (11,0);

\draw (12,1) node{$l_2$};
\draw (12.5,0) node{$B$};

\end{tikzpicture}

\caption{\label{fig:ExteriorSixtyDegrees}}

\end{figure}

Let $\Lambda$ be the line with slope $-\sqrt{3}$ that passes through $l_1$. Either there is a point in $A$ that is in the upper half-plane formed by this line, i.e. there is a point to the right of this line in Figure \ref{fig:ExteriorSixtyDegrees}, or there is not. Assume first that there is not. In this case, if $l_2<l_1$, we reorient the line segment $l_2$ so that it is horizontal and connects the top of $B$ to $A$, as shown in Figure \ref{fig:ExteriorSixtyDegrees}. This creates a new set $B_1$ with volume strictly greater than the volume of $B$. Then, we merely lower the top of $B_1$ until we have a set of volume $\alpha$. This process strictly reduces the non-joint perimeter of $B_1$ which has the same length as the non-joint perimeter of $B$. Thus, we have found a new configuration with strictly better double bubble perimeter. 

Suppose, on the other hand, that $l_2>l_1$. In this case, let $\Lambda$ be the horizontal line passing through $i_{\smallrightarrow}^A$, and consider the point where this line intersects $l_2$ - call this point $p$. We reorient the part of $l_2$ that is between $p$ and $\pi_1$ so that this line segment connects $p$ to $\partial A$ and is contained in $\Lambda$. See the lefthand side of Figure \ref{fig:ExteriorSixtyDegreesVolAdj2}. Then, replace $B_1$ with $B_1^{\varhexagon}$. See the righthand side of Figure \ref{fig:ExteriorSixtyDegreesVolAdj2}. Then, we can again lower the top side of $B_1$ until we obtain a set of volume $\alpha$. This volume adjustment strictly reduces the double bubble perimeter. 

\begin{figure}[H]

\begin{tikzpicture}[scale=0.4]

\draw[black, thick, dotted] (-3,6.4282) to (5,6.4282);\draw (-2.8,6.75) node{$\Lambda$};

\draw[blue, thin] (2.5,4.330125) to (1.28,6.4282);
\draw[blue, thin] (4.28867584,1.23204756) to (3.28867584,-0.50000324);
\draw[blue, thin] (3.28867584,-0.50000324) to (0.28867584,-0.50000324);
\draw[blue, thin] (0.28867584,-0.50000324) to (-2.71132416,4.69614916);
\draw[blue, thin] (-2.71132416,4.69614916) to (-1.71132416,6.4282);
\draw[blue, thin] (-1.71132416,6.4282) to (1.28,6.4282);

\draw (1.6,5) node{$l_1$};
\draw (0.9,4.5) node{$A$};

\draw[red, thick, dotted] (3.7,6.4282) to (2.5,4.330125);
\draw[red, thin] (3.7,6.4282) to (1.28,6.4282);
\draw[red, thin] (3.7,6.4282) to (4.7,8.1602508);
\draw[red, thin] (4.7,8.1602508) to (6.7,8.1602508);
\draw[red, thin] (6.7,8.1602508) to (8.7,4.69614918);
\draw[red, thin] (8.7,4.69614918) to (6.7,1.23204756);
\draw[red, thin] (6.7,1.23204756) to (4.28867584,1.23204756);
\draw[red, thin] (4.28867584,1.23204756) to (2.5,4.330125);

\draw (3.4,5) node{$l_2$};
\draw (4.9,5) node{$B_1$};

\draw[->] (10.25,4.5) to (11.25,4.5);


\draw[blue, thin] (17.5,4.330125) to (16.28,6.4282);
\draw[blue, thin] (19.28867584,1.23204756) to (18.28867584,-0.50000324);
\draw[blue, thin] (18.28867584,-0.50000324) to (15.28867584,-0.50000324);
\draw[blue, thin] (15.28867584,-0.50000324) to (12.288675,4.69614916);
\draw[blue, thin] (12.288675,4.69614916) to (13.2886758,6.4282);
\draw[blue, thin] (13.2886758,6.4282) to (16.28,6.4282);

\draw (16.6,5) node{$l_1$};
\draw (15.9,4.5) node{$A$};

\draw[red, thick, dotted] (19.7,8.1602508) to (17.5,4.330125);
\draw[red, thin] (16.28,6.4282) to (17.28,8.1602508);
\draw[red, thin] (19.7,8.1602508) to (17.28,8.1602508);
\draw[red, thin] (19.7,8.1602508) to (21.7,8.1602508);
\draw[red, thin] (21.7,8.1602508) to (23.7,4.69614918);
\draw[red, thin] (23.7,4.69614918) to (21.7,1.23204756);
\draw[red, thin] (21.7,1.23204756) to (19.28867584,1.23204756);
\draw[red, thin] (19.28867584,1.23204756) to (17.5,4.330125);

\draw (18.4,5) node{$l_2$};
\draw (16.4,7.5) node{$\tilde{l}_2$};
\draw (19.9,5) node{$B_1^{\varhexagon}$};

\end{tikzpicture}

\caption{\label{fig:ExteriorSixtyDegreesVolAdj2}}

\end{figure}

If there is a point in $A$ that is in the upper half-plane created by $\Lambda$, then we perform an analogous process but by adding volume to $A$ instead of $B$.

\end{proof}

We now wish to prove similar results, but for interior angles of $60^{\circ}$.

\begin{lemma}\label{lemma:NoNonJointBndrySixtyDegAngles}
Let $(A,B)\in\gamma_{\alpha}$. Suppose there are two adjacent sides, call them $l_1$ and $l_2$, both of which form part of $\partial A$ or both of which form part of $\partial B$. Suppose further that the point of intersection of these two line segments is not part of the joint boundary. If the interior angle formed by these two lines is $60^{\circ}$, then $(A,B)$ is not a double bubble minimizing configuration. 

\end{lemma}

\begin{proof}
Suppose, without loss of generality, that the line segments $l_1$ and $l_2$ are in $\partial B$. We know, by Corollary \ref{cor:OnlyThreeJointBoundaryLines} and Lemma \ref{lemma:AllSidesContained}, that there are at most three adjacent line segments in $\partial B$ that form joint boundary at points other than $\pi_1$ and $\pi_2$. Let's assume that $l_1$ is horizontal and $l_2$ has slope $\sqrt{3}$ (see Figure \ref{fig:InteriorSixtyNonJointBoundary}). Then either both $l_1$ and $l_2$ do not contribute to the joint boundary, or at most one of them does, say $l_1$ does and $l_2$ does not. Let $l_3$ be the other line segment in $\partial B$ that shares an endpoint with $l_2$ (and is not $l_1$).

\begin{figure}[H]

\begin{tikzpicture}[scale=0.5]

\draw[red, thin] (-1,0) to (2,0);
\draw[red, thin] (2,0) to (0,-3.4641016);
\draw[red, thin] (0,-3.46410161) to (-1,-3.46410161);
\draw[red, thin] (-1,-3.46410161) to (-2,-1.7320508);
\draw[red, thin] (-2,-1.7320508) to (-1,0);

\draw[blue, thin] (0,-3.46410161) to (-1,-5.19615242);
\draw[blue, thin] (-1,-5.19615242) to (-3,-5.19615242);
\draw[blue, thin] (-3,-5.19615242) to (-5,-1.7320508);
\draw[blue, thin] (-5,-1.7320508) to (-3,1.7320508);
\draw[blue, thin] (-3,1.7320508) to (-2,1.7320508);
\draw[blue, thin] (-2,1.7320508) to (-1,0);

\draw (0.2,-1) node{$B$};
\draw (-0.3,1) node{$l_1$};
\draw (-0.2,-4.55) node{$l_2$};

\draw[->] (3,-1.5) to (4,-1.5);


\draw[red, thin] (9,0) to (11,0);
\draw[black, thick, dotted] (11,0) to (12,0);
\draw[black, thick, dotted] (12,0) to (11.5,-0.8660254);
\draw[red, thin] (11,0) to (11.5,-0.8660254);
\draw[red, thin] (11.5,-0.8660254) to (10,-3.4641016);

\draw[red, thin] (10,-3.46410161) to (9,-3.46410161);
\draw[red, thin] (9,-3.46410161) to (8,-1.7320508);
\draw[red, thin] (8,-1.7320508) to (9,0);

\draw[blue, thin] (10,-3.46410161) to (9,-5.19615242);
\draw[blue, thin] (9,-5.19615242) to (7,-5.19615242);
\draw[blue, thin] (7,-5.19615242) to (5,-1.7320508);
\draw[blue, thin] (5,-1.7320508) to (7,1.7320508);
\draw[blue, thin] (7,1.7320508) to (8,1.7320508);
\draw[blue, thin] (8,1.7320508) to (9,0);

\draw (10.2,-1) node{$B$};
\draw (9.7,1) node{$l_1$};
\draw (9.8,-4.55) node{$l_2$};


\draw[green, thin] (20,0) to (20.5,0);
\draw[red, thin] (20.5,0) to (21.5,0);
\draw[black, thick, dotted] (21.5,0) to (22.5,0);
\draw[black, thick, dotted] (22.5,0) to (22,-0.8660254);
\draw[red, thin] (21.5,0) to (22,-0.8660254);
\draw[red, thin] (22,-0.8660254) to (20.5,-3.4641016);
\draw[green, thin] (20.5,-3.4641016) to (20,-3.4641016);

\draw[red, thin] (20,-3.46410161) to (19,-3.46410161);
\draw[red, thin] (19,-3.46410161) to (18,-1.7320508);
\draw[red, thin] (18,-1.7320508) to (19,0);
\draw[red, thin] (19,0) to (20,0);

\draw[blue, thin] (20,-3.46410161) to (19,-5.19615242);
\draw[blue, thin] (19,-5.19615242) to (17,-5.19615242);
\draw[blue, thin] (17,-5.19615242) to (15,-1.7320508);
\draw[blue, thin] (15,-1.7320508) to (17,1.7320508);
\draw[blue, thin] (17,1.7320508) to (18,1.7320508);
\draw[blue, thin] (18,1.7320508) to (19,0);

\draw (20.2,-1) node{$B$};
\draw (19.7,1) node{$l_1$};
\draw (19.8,-4.55) node{$l_2$};


\draw (1.55,0.2) node{$\epsilon$};

\draw (11.55,0.2) node{$\epsilon$};

\end{tikzpicture}

\caption{\label{fig:InteriorSixtyNonJointBoundary}}

\end{figure}

Figure \ref{fig:InteriorSixtyNonJointBoundary} demonstrates the process by which we reduce the double bubble perimeter. First, let $L$ be the distance between the horizontal line passing through $l_1$ and the  point of intersection of $l_2$ and $l_3$, and let $0<\epsilon<\min\left\{l_1,\frac{L}{2}\right\}$. Then, we can reorient the part of the line segment $l_1$ that has length $\epsilon$ and shares an endpoint with $l_2$. This line segment of length $\epsilon$ has two endpoints, say $p_1$ and $p_2$, with $p_2$ also being an endpoint of $l_2$. We reorient this line segment so that it has slope $-\sqrt{3}$, the endpoint $p_1$ does not change, and the endpoint $p_2\in l_2$. This decreases the volume of $B$ by $\frac{\sqrt{3}}{4}\cdot\epsilon^2$, and also reduces the double bubble perimeter by $\epsilon$. Then, we increase the length of $l_1$ by $\frac{\epsilon}{2}$ and either add a length of $\frac{\epsilon}{2}$ to $l_3$ in the exterior of $A$ if $l_3$ is horizontal, or make a horizontal line segment of length $\frac{\epsilon}{2}$ (in the exterior of $A$) that connects $l_2$ and $l_3$ if these two line segments form a $60^{\circ}$ angle. This operation is demonstrated in Figure \ref{fig:InteriorSixtyNonJointBoundary}. This adds a volume of $\frac{\epsilon}{2}\cdot L>\epsilon^2>\frac{\sqrt{3}}{4}\cdot\epsilon^2$. So, there is some value $\epsilon'\in\left(0,\frac{\epsilon}{2}\right)$ such that we can add two line segments of length $\epsilon'$ in the way just described that results in the correct volume ratio.

\end{proof}

\begin{lemma}\label{lem:JoinBoundarySixty}
Let $(A,B)\in\mathfs{F}_{\alpha}$. Suppose that there are two adjacent sides, call them $l_1$ and $l_2$, such that their meeting point is joint boundary, and also a nontrivial portion of both line segments is part of the joint boundary. If $l_1$ and $l_2$ form a $60^{\circ}$ angle, then $(A,B)$ is not a double bubble minimizing configuration. 
\end{lemma}

\begin{proof}

This proof is similar to that of Lemma \ref{lemma:NoNonJointBndrySixtyDegAngles}. Assume, without loss of generality, that the acute angle formed by $l_1$ and $l_2$ is interior to $B$. We can remove the $60^{\circ}$ angle in the same way as before and find points in $\partial B$ to add the boundary of length $\epsilon'\in\left(0,\frac{\epsilon}{2}\right)$ in the same way. This process is demonstrated in Figure \ref{fig:RemovingJointSixty}. The only difference is that when we originally remove volume from $B$, it must be added to $A$. This means that we have to find a way to reduce the volume of $A$ without increasing the double bubble perimeter, and in such a way that the resulting configuration is still in $\mathfs{F}_{\alpha}$. Notice that we can make $\epsilon$ as small as we wish, so that the extra volume that is added to $A$ is also as small as we wish. If $i_{\smallrightarrow}^A>i_{\smallrightarrow}^B$, then we can move the top of $A$ down to strictly reduce its volume. Since we can make the added volume as small as we wish, we can ensure that this process results in the correct volume for $A$. A similar argument holds if $s_{\smallrightarrow}^A<s_{\smallrightarrow}^B$. If neither of these inequalities holds, then we can either reduce the volume of $A$ by moving its left sides to the right, or its right sides to the left.

\begin{figure}[H]

\begin{tikzpicture}[scale=0.5]

\draw[red, thin] (2,0) to (-1,0);
\draw[red, thin] (-1,0) to (0,-1.7320508);
\draw[red, thin] (0,-1.7320508) to (2,-1.7320508);
\draw[red, thin] (2,-1.7320508) to (2.5,-0.86602540);
\draw[red, thin] (2.5,-0.86602540) to (2,0);

\draw[blue, thin] (1,0) to (0,1.7320508);
\draw[blue, thin] (0,1.7320508) to (-3,1.7320508);
\draw[blue, thin] (-3,1.7320508) to (-4,0);
\draw[blue, thin] (-4,0) to (-2,-3.46410161);
\draw[blue, thin] (-2,-3.46410161) to (-1,-3.46410161);
\draw[blue, thin] (-1,-3.46410161) to (0,-1.7320508);

\draw (0,0.5) node{$l_1$};

\draw (-0.8,-1.2) node{$l_2$};

\draw[->] (3,-0.8) to (4,-0.8);


\draw[red, thin] (11,0) to (9,0);
\draw[red, thin] (9,0) to (8.5,-0.86602540);
\draw[black, thick, dotted] (9,0) to (8,0);
\draw[black, thick, dotted] (8,0) to (8.5,-0.86602540);
\draw[red, thin] (8.5,-0.86602540) to (9,-1.7320508);
\draw[red, thin] (9,-1.7320508) to (11,-1.7320508);
\draw[red, thin] (11,-1.7320508) to (11.5,-0.86602540);
\draw[red, thin] (11.5,-0.86602540) to (11,0);

\draw[blue, thin] (10,0) to (9,1.7320508);
\draw[blue, thin] (9,1.7320508) to (6,1.7320508);
\draw[blue, thin] (6,1.7320508) to (5,0);
\draw[blue, thin] (5,0) to (7,-3.46410161);
\draw[blue, thin] (7,-3.46410161) to (8,-3.46410161);
\draw[blue, thin] (8,-3.46410161) to (9,-1.7320508);

\draw (9.1,0.4) node{$l_1$};

\draw (8.2,-1.2) node{$l_2$};

\draw (8.2,0.3) node{$\epsilon$};

\draw[->] (12,-0.8) to (13,-0.8);


\draw[green, thin] (20,0) to (20.5,0);
\draw[red, thin] (20,0) to (18,0);
\draw[red, thin] (18,0) to (17.5,-0.86602540);
\draw[black, thick, dotted] (18,0) to (17,0);
\draw[black, thick, dotted] (17,0) to (17.5,-0.86602540);
\draw[red, thin] (17.5,-0.86602540) to (18,-1.7320508);
\draw[red, thin] (18,-1.7320508) to (20,-1.7320508);
\draw[green, thin] (20,-1.7320508) to (20.5,-1.7320508);
\draw[red, thin] (20.5,-1.7320508) to (21,-0.86602540);
\draw[red, thin] (21,-0.86602540) to (20.5,0);

\draw[blue, thin] (19,0) to (18,1.7320508);
\draw[blue, thin] (18,1.7320508) to (15,1.7320508);
\draw[blue, thin] (15,1.7320508) to (14,0);
\draw[blue, thin] (14,0) to (16,-3.46410161);
\draw[blue, thin] (16,-3.46410161) to (17,-3.46410161);
\draw[blue, thin] (17,-3.46410161) to (18,-1.7320508);

\draw (18.1,0.4) node{$l_1$};

\draw (17.2,-1.2) node{$l_2$};

\draw (17.2,0.3) node{$\epsilon$};

\end{tikzpicture}

\caption{\label{fig:RemovingJointSixty}}

\end{figure}

\end{proof}

The previous two lemmas show that if there is a $60^{\circ}$ angle, then it must occur at either $\pi_1$ or $\pi_2$. At each of these points, there are three angles: one interior to $A$, one interior to $B$, and one exterior to both $A$ and $B$. Lemma \ref{lemma:NoSixtyDegreeAngles} shows us that the exterior angle at these points cannot be $60^{\circ}$, which means that this angle must be either $120^{\circ}$ or $180^{\circ}$. The possibilities are shown in Figure \ref{fig:RemainingSixties}: 

\begin{figure}[H]

\begin{tikzpicture}[scale=0.8]

\draw[blue, thin] (0,0) to (2,0);
\draw (0.1,0.4) node{$l_1$};
\draw (0.6,-0.6) node{$A$};

\draw[red, thin] (1,-1.7320508) to (3,1.73205081);
\draw (2.4,1.5) node{$l_2$};
\draw (2.8,0.3) node{$B$};

\draw[red, thin] (6,-1) to (7,0.732050808);
\draw[red, thin] (7,0.732050808) to (9,0.732050808);

\draw[blue, thin] (7,0.732050808) to (5,0.732050808);

\draw (5.4,-0.3) node{$A$};
\draw (7.1,-0.3) node{$B$};

\end{tikzpicture}

\caption{\label{fig:RemainingSixties}}

\end{figure}

We know that the joint boundary cannot consist of four or more line segments. Now, we want to eliminate all of the configurations in $\mathfs{F}_{\alpha}$ that have both a $60^{\circ}$ angle as well as joint boundary consisting of three line segments. Consider the family $\mathfs{F}_{\alpha}$. Notice that a configuration in this family has the property that at least one set must be its own hexagon, i.e. either $A=A^{\varhexagon}$ or $B=B^{\varhexagon}$, or potentially both. We suppose, for the remainder of this section, that $B=B^{\varhexagon}$. In the next lemma, we begin to eliminate configurations whose joint boundary consists of three line segments. 

\begin{lemma}

Let $(A,B)\in\mathfs{F}_{\alpha}$, and assume that $B=B^{\varhexagon}$. Suppose that the exterior angle associated with $\pi_1$ is $120^{\circ}$, while the angle interior to $A$ at this point is $60^{\circ}$. Suppose further that either the joint boundary consists of three line segments, or the joint boundary consists of two line segments and the point $\pi_2$ occurs in $\partial B$ at the point of intersection of two line segments of different orientations. Then, $(A,B)$ is not a minimizing configuration. 

\end{lemma}

\begin{proof}

First, after some reflections and/or $60^{\circ}$ rotations, we may assume that $\pi_1$ occurs at the point of intersection of $\Lambda_{\smallrightarrow}\left(i_{\smallrightarrow}^A\right)$ and $\Lambda_{\smallnearrow}\left(i_{\smallnearrow}^B\right)$. This means that $\pi_2$ must occur in $\Lambda_{\rightarrow}\left(s_{\smallrightarrow}^B\right)$. Figure \ref{fig:TheUglySixtyCase} demonstrates the situation:

\begin{figure}[H]

\begin{tikzpicture}[scale=0.6]

\draw[red, thin] (12,0) to (13,1.7320508);
\draw[red, thin] (13,1.7320508) to (14,1.7320508);
\draw[red, thin] (14,1.7320508) to (15,0);
\draw[red, thin] (15,0) to (14,-1.7320508);
\draw[red, thin] (14,-1.7320508) to (13,-1.7320508);
\draw[red, thin] (13,-1.7320508) to (12,0);

\draw[blue, thin] (12.75,1.299038) to (11,1.299038);
\draw[blue, thin] (11,1.299038) to (10.75,0.866025);
\draw[blue, thin] (10.75,0.866025) to (12.5,-2.1650635);
\draw[blue, thin] (12.5,-2.1650635) to (13,-2.1650635);
\draw[blue, thin] (13,-2.1650635) to (13.25,-1.7320508);

\draw (12.75,1.299038) node{$\cdot$};
\draw[black, thin] (12.75,1.299038) to (11.7,2.4);
\draw (11.8,2.6) node{$\pi_1$};

\draw (13.25,-1.7320508) node{$\cdot$};
\draw (13.8,-3) node{$\pi_2$};

\draw[black, thin] (13.25,-1.7320508) to (13.7,-2.7);

\draw (11.6,0.6) node{$\scaleobj{0.6}{A}$};
\draw (13.5,0) node{$\scaleobj{0.6}{B}$};


\end{tikzpicture}

\caption{\label{fig:TheUglySixtyCase}}

\end{figure}

Note that the point $\pi_2$ does not come into the proof in any significant way. We merely need it to occur at some point in $\Lambda_{\smallrightarrow}\left(s_{\smallrightarrow}^B\right)\cap\partial B$. This could be at either endpoint of this line segment. We continue with the notation shown in the lefthand side of Figure \ref{fig:RemainingSixties}. Also, for ease of notation and when clear from context, we will use $l_i$ to represent not only a line segment, but also the length of this line segment.

With this situation, we can perform the following operation: first we reorient the part of $l_2$ that is joint boundary so that it has slope $-\sqrt{3}$. Figure \ref{fig:Breakingl_2} gives an illustration of this:

\begin{figure}[H]

\begin{tikzpicture}[scale=0.5]

\draw[red, thin] (0,0) to (2,3.46410161);
\draw[red, thin] (2,3.46410161) to (4,3.46410161);\draw (1.4,1.8) node{$\scaleobj{0.7}{l_2}$};
\draw[red, thin] (4,3.46410161) to (6,0);
\draw[red, thin] (6,0) to (4,-3.46410161);
\draw[red, thin] (4,-3.46410161) to (2,-3.46410161);
\draw[red, thin] (2,-3.46410161) to (0,0);

\draw (3,0) node{$\scaleobj{0.8}{B}$};

\draw[blue, thin] (3,-3.46410161) to (2,-5.19615242);
\draw[blue, thin] (2,-5.19615242) to (1.5,-5.19615242);
\draw[blue, thin] (1.5,-5.19615242) to (-1.5,0);
\draw[blue, thin] (-1.5,0) to (-0.5,1.7320508);\draw (-1.3,0.94) node{$\scaleobj{0.7}{a_1}$};
\draw[blue, thin] (-0.5,1.7320508) to (1,1.7320508);\draw (0.3,2) node{$\scaleobj{0.7}{l_1}$};

\draw (-0.2,-1) node{$\scaleobj{0.8}{A}$};

\draw[->] (6.8,0) to (7.4,0);


\draw[red, thick, dotted] (10,0) to (11,1.7320508);

\draw[black, thin] (10,0) to (10.3,0);
\draw[black, thin] (10.3,0) to (10.9,-1.03923048);
\draw[black, thin] (12,-3.46410161) to (12.2,-3.11769145);
\draw[black, thin] (12.2,-3.11769145) to (11.4,-1.73205080);

\draw[red, thin] (10,0) to (9.25,1.2990381);\draw (9.6,0.3) node{$\scaleobj{0.7}{l_4}$};
\draw[red, thin] (11,1.7320508) to (12,3.46410161);\draw (11.8,2.5) node{$\scaleobj{0.7}{l_3}$};
\draw[red, thin] (12,3.46410161) to (14,3.46410161);
\draw[red, thin] (14,3.46410161) to (16,0);
\draw[red, thin] (16,0) to (14,-3.46410161);
\draw[red, thin] (14,-3.46410161) to (12,-3.46410161);
\draw[red, thin] (12,-3.46410161) to (10,0);\draw (11.25,-1.265) node{$\scaleobj{0.7}{l_6}$};

\draw (13,0) node{$\scaleobj{0.8}{B_1}$};

\draw[blue, thin] (13,-3.46410161) to (12,-5.19615242);
\draw[blue, thin] (12,-5.19615242) to (11.5,-5.19615242);
\draw[blue, thin] (11.5,-5.19615242) to (8.5,0);
\draw[blue, thin] (8.5,0) to (9.5,1.7320508);\draw (8.7,0.94) node{$\scaleobj{0.7}{a_1}$};
\draw[blue, thin] (9.5,1.7320508) to (11,1.7320508);\draw (10.3,2) node{$\scaleobj{0.7}{l_1}$};

\draw (9.8,-1) node{$\scaleobj{0.8}{A_1}$};

\draw[->] (16.8,0) to (17.4,0);


\draw[red, thick, dotted] (20,0) to (21,1.7320508);

\draw[black, thin] (20,0) to (20.3,0);
\draw[black, thin] (20.3,0) to (20.9,-1.03923048);
\draw[black, thin] (22,-3.46410161) to (22.2,-3.11769145);
\draw[black, thin] (22.2,-3.11769145) to (21.4,-1.73205080);

\draw[red, thin] (20,0) to (19.25,1.2990381);\draw (19.6,0.3) node{$\scaleobj{0.7}{l_4}$};
\draw[red, thin] (19.25,1.2990381) to (20.5,3.46410161);
\draw[red, thin] (20.5,3.46410161) to (22,3.46410161);
\draw[red, thick, dotted] (21,1.7320508) to (22,3.46410161);
\draw[red, thin] (22,3.46410161) to (24,3.46410161);\draw (21.8,2.5) node{$\scaleobj{0.7}{l_3}$};
\draw[red, thin] (24,3.46410161) to (26,0);
\draw[red, thin] (26,0) to (24,-3.46410161);
\draw[red, thin] (24,-3.46410161) to (22,-3.46410161);
\draw[red, thin] (22,-3.46410161) to (20,0);\draw (21.25,-1.265) node{$\scaleobj{0.7}{l_6}$};

\draw (23,0) node{$\scaleobj{0.8}{B}$};

\draw[blue, thin] (23,-3.46410161) to (22,-5.19615242);
\draw[blue, thin] (22,-5.19615242) to (21.5,-5.19615242);
\draw[blue, thin] (21.5,-5.19615242) to (18.5,0);
\draw[blue, thin] (18.5,0) to (19.5,1.7320508);\draw (18.7,0.94) node{$\scaleobj{0.7}{a_1}$};
\draw[blue, thick, dotted] (19.5,1.7320508) to (21,1.7320508);\draw (20.3,2) node{$\scaleobj{0.7}{l_1}$};

\draw (19.8,-1) node{$\scaleobj{0.8}{A}$};


\end{tikzpicture}

\caption{\label{fig:Breakingl_2}}

\end{figure}

The reoriented part of $l_2$ we rename $l_4$, and the part of $l_2$ that did not change we name $l_3$. Notice that $l_4$ may have an endpoint in $l_1$ if $l_1\geq l_2-l_3$, or it could have an endpoint in $a_1$ (at a point that is not the meeting point of $a_1$ and $l_1$), as it does in Figure \ref{fig:Breakingl_2}. In this case, $l_1<l_2-l_3$. So, the double bubble perimeter of this new configuration is less than that of $(A,B)$ by a length of $l_5=l_2-\left(l_3+l_4\right)>0$. Finally, let $l_6$ be the original line segment $\Lambda_{\smallsearrow}\left(s_{\smallsearrow}^B\right)\cap\partial B$. See Figure \ref{fig:VolAdjUglySixtyCase} for reference. 

Thus far, we have created two new sets, which we call $A_1$ (to replace $A$), and $B_1$ (to replace $B$). Once we have done this, we replace $B_1$ with $B_1^{\varhexagon}$, which does not alter $A_1$ or the double bubble perimeter. However, it does translate the line segment $l_3$ to the left, so we call this new line segment $\tilde{l}_3$. Similarly it translates the line segment $l_1$ up and to the right between $l_3$ and $\tilde{l}_3$; we call this translated line segment $\tilde{l}_1$. See the lefthand side of Figure \ref{fig:VolAdjUglySixtyCase}.

\begin{figure}[H]

\begin{tikzpicture}[scale=0.55]


\draw[red, thick, dotted] (0,0) to (1,1.7320508);
\draw[red, thin] (0,0) to (-0.75,1.2990381);\draw (-0.4,0.3) node{$\scaleobj{0.7}{l_4}$};
\draw[red, thin] (-0.5,1.7320508) to (0.5,3.46410161);\draw (1.8,2.5) node{$\scaleobj{0.7}{l_3}$};\draw (-0.5,2.5) node{$\scaleobj{0.7}{\tilde{l}_3}$};

\draw[red, thin] (0.5,3.46410161) to (4,3.46410161);
\draw[red, thin] (4,3.46410161) to (6,0);
\draw[red, thin] (6,0) to (4,-3.46410161);
\draw[red, thin] (4,-3.46410161) to (2,-3.46410161);
\draw[red, thin] (2,-3.46410161) to (0,0);\draw (1.2,-1.2) node{$\scaleobj{0.7}{l_6}$};

\draw (3,0) node{$\scaleobj{0.8}{B_1}$};

\draw (1.56,3.86) node{$\scaleobj{0.7}{\tilde{l}_1}$};

\draw[blue, thin] (3,-3.46410161) to (2,-5.19615242);
\draw[blue, thin] (2,-5.19615242) to (1.5,-5.19615242);
\draw[blue, thin] (1.5,-5.19615242) to (-1.5,0);
\draw[blue, thin] (-1.5,0) to (-0.5,1.7320508);\draw (-1.3,0.94) node{$\scaleobj{0.7}{a_1}$};
\draw (-1.2,1.5) node{$\scaleobj{0.6}{l_5}$};

\draw[blue, thick, dotted] (-0.5,1.7320508) to (1,1.7320508);\draw (0.3,2) node{$\scaleobj{0.7}{l_1}$};

\draw[black, thick, dotted] (-1,1.7320508) to (1.5,1.7320508);
\draw[black, thick, dotted] (-1,1.2990381) to (1.5,1.2990381);
\draw[black, thick, dotted] (0,0) to (3,5.1961524);
\draw[black, thick, dotted] (-0.75,1.2990381) to (1.5,5.1961524);
\draw[black, thick, dotted] (-1,3.46410161) to (3,3.46410161);

\draw (-0.2,-1) node{$\scaleobj{0.8}{A_1}$};


\draw[red, thin] (9.5,1.7320508) to (9.75,1.2990381);
\draw[red, thin] (10.5,0) to (9.75,1.2990381);\draw (9.8,0.7) node{$\scaleobj{0.7}{\tilde{l}_4}$};
\draw[red, thin] (9.5,1.7320508) to (10.5,3.46410161);\draw (11.8,2.5) node{$\scaleobj{0.7}{l_3}$};\draw (9.5,2.5) node{$\scaleobj{0.7}{\tilde{l}_3}$};

\draw[red, thin] (10.5,3.46410161) to (14,3.46410161);
\draw[red, thin] (14,3.46410161) to (16,0);
\draw[red, thin] (16,0) to (14,-3.46410161);
\draw[red, thin] (14,-3.46410161) to (12.5,-3.46410161);
\draw[red, thick, dotted] (12.5,-3.46410161) to (12,-3.46410161);
\draw[red, thin] (12.5,-3.46410161) to (10.5,0);\draw (10.9,-1.2) node{$\scaleobj{0.7}{\tilde{l}_6}$};

\draw (13,0) node{$\scaleobj{0.8}{B_2}$};

\draw (11.56,3.86) node{$\scaleobj{0.7}{\tilde{l}_1}$};

\draw[blue, thin] (13,-3.46410161) to (12,-5.19615242);
\draw[blue, thin] (12,-5.19615242) to (11.5,-5.19615242);
\draw[blue, thin] (11.5,-5.19615242) to (8.5,0);
\draw[blue, thin] (8.5,0) to (9.5,1.7320508);\draw (8.7,0.94) node{$\scaleobj{0.7}{a_1}$};
\draw (8.8,1.5) node{$\scaleobj{0.6}{l_5}$};

\draw[blue, thick, dotted] (9.5,1.7320508) to (11,1.7320508);\draw (10.3,2) node{$\scaleobj{0.7}{l_1}$};

\draw[black, thick, dotted] (9,1.7320508) to (11.5,1.7320508);
\draw[black, thick, dotted] (9,1.2990381) to (11.5,1.2990381);
\draw[black, thick, dotted] (10,0) to (13,5.1961524);
\draw[black, thick, dotted] (9.25,1.2990381) to (11.5,5.1961524);
\draw[black, thick, dotted] (9,3.46410161) to (13,3.46410161);

\draw (9.8,-1) node{$\scaleobj{0.8}{A_2}$};



\draw[red, thin] (20.5,1.7320508) to (20.75,1.2990381);
\draw[red, thin] (21.5,0) to (20.75,1.2990381);
\draw[red, thin] (20.5,1.7320508) to (21.5,3.46410161);

\draw[red, thick, dotted] (20.5,3.46410161) to (21.5,3.46410161);
\draw[red, thin] (21.5,3.46410161) to (24,3.46410161);
\draw[red, thin] (24,3.46410161) to (26,0);
\draw[red, thin] (26,0) to (24,-3.46410161);
\draw[red, thin] (24,-3.46410161) to (23.5,-3.46410161);
\draw[red, thick, dotted] (23.5,-3.46410161) to (22.5,-3.46410161);
\draw[red, thick, dotted] (22.5,-3.46410161) to (19.5,1.7320508);
\draw[red, thin] (23.5,-3.46410161) to (21.5,0);

\draw (23,0) node{$\scaleobj{0.8}{B_3}$};


\draw[blue, thin] (23.5,-3.46410161) to (22.5,-5.19615242);
\draw[blue, thin] (22.5,-5.19615242) to (21.5,-5.19615242);
\draw[blue, thin] (21.5,-5.19615242) to (18.5,0);
\draw[blue, thin] (18.5,0) to (19.5,1.7320508);

\draw[blue, thin] (19.5,1.7320508) to (20.5,1.7320508);

\draw[blue, thick, dotted] (23,-3.46410161) to (22,-5.19615242);

\draw[black, thick, dotted] (19,1.7320508) to (21.5,1.7320508);
\draw[black, thick, dotted] (19.25,1.2990381) to (21.5,5.1961524);

\draw (19.8,-1) node{$\scaleobj{0.8}{A_3}$};

\end{tikzpicture}

\caption{\label{fig:VolAdjUglySixtyCase}}

\end{figure}

Recall that $\rho_{DB}\left(A,B\right)-\rho_{DB}\left(A_1,B_1\right)=\rho\left(l_5\right)$ (the length of $l_5$). So, we can move the line $l_4\cup l_6$ to the right any distance between $0$ and $\rho\left(l_5\right)$, and connect this line to $l_5$ via a line segment with slope $-\sqrt{3}$. While we do this, we continuously take volume from $B_1$ and add it to $A_1$. This creates two new sets, call them $B_2$ to replace $B_1$, and $A_2$ to replace $A_1$. See the middle configuration of Figure \ref{fig:VolAdjUglySixtyCase}. We have increased the double bubble perimeter of $\left(A_1,B_1\right)$ by at most $\rho\left(l_5\right)$, which means that the double bubble perimeter of $\left(A_2,B_2\right)$ is still at most the double bubble perimeter of $\left(A,B\right)$. 

Suppose that we move the line segment $l_4\cup l_6$ to the right by a distance of $\rho\left(l_5\right)$, and the resulting $A_2$ still has volume less than $A$. Then, we can continue to move $l_4\cup l_6$ to the right while shortening $\tilde{l}_1$ and moving $\tilde{l}_3$ to the right the same amount. While we reduce the length of $\tilde{l}_1$, we must add a horizontal line segment of this same length with one endpoint intersecting the upper endpoint of $l_5$ and the other endpoint intersecting $\tilde{l}_3$ (here we are abusing notation slightly since, technically, we have moved $\tilde{l}_3$ to the right). This process does decrease the overall volume of both sets. However, all volume that is lost comes from the exterior of $A\cup B$. This process creates two new sets $A_3$ to replace $A_2$, and $B_3$ to replace $B_2$. We can move the line segments in question until $\tilde{l}_3$ has returned to its original position and we have regained $l_3$. 

Suppose that at this point we still have that $\mu\left(A_3\right)<\mu\left(A\right)$. Then, we can continue to move $\tilde{l}_4\cup\tilde{l}_6$ to the right in the same manner as before; now, however, we do not move $l_3$. Instead, we extend the length of $l_1$. See Figure \ref{fig:LastVolAdjUglySixtyCase}. Suppose that we extend the length of $l_1$ by an amount $\epsilon$ and call this extended line $l_1^+$. Also, we denote the set that has replaced $A_3$ by $A_4$ and the set that has replaced $B_3$ by $B_4$. Then, replace $A_4$ with $A_4^{\varhexagon}$ and $B_4$ with $B_4\setminus A_4^{\varhexagon}$. Again, see Figure \ref{fig:LastVolAdjUglySixtyCase}. The effect of this process does not change the double bubble perimeter of $\left(A_4,B_4\right)$, but it does remove volume from $B_4$ and add it to $A_4$.

\begin{figure}

\begin{tikzpicture}[scale=0.45]

\draw[red, thin] (20.5,1.7320508) to (20.75,1.2990381);
\draw[red, thin] (21.5,0) to (20.75,1.2990381);
\draw[red, thin] (20.5,1.7320508) to (21.5,3.46410161);

\draw[red, thick, dotted] (20.5,3.46410161) to (21.5,3.46410161);
\draw[red, thin] (21.5,3.46410161) to (24,3.46410161);
\draw[red, thin] (24,3.46410161) to (26,0);
\draw[red, thin] (26,0) to (24,-3.46410161);
\draw[red, thin] (24,-3.46410161) to (23.5,-3.46410161);
\draw[red, thick, dotted] (23.5,-3.46410161) to (22.5,-3.46410161);
\draw[red, thick, dotted] (22.5,-3.46410161) to (19.5,1.7320508);
\draw[red, thin] (23.5,-3.46410161) to (21.5,0);

\draw (23,0) node{$\scaleobj{0.8}{B_3}$};


\draw[blue, thin] (23.5,-3.46410161) to (22.5,-5.19615242);
\draw[blue, thin] (22.5,-5.19615242) to (21.5,-5.19615242);
\draw[blue, thin] (21.5,-5.19615242) to (18.5,0);
\draw[blue, thin] (18.5,0) to (19.5,1.7320508);

\draw[blue, thin] (19.5,1.7320508) to (20.5,1.7320508);

\draw[blue, thick, dotted] (23,-3.46410161) to (22,-5.19615242);

\draw[black, thick, dotted] (19,1.7320508) to (21.5,1.7320508);
\draw[black, thick, dotted] (19.25,1.2990381) to (21.5,5.1961524);

\draw (19.8,-1) node{$\scaleobj{0.8}{A_3}$};


\draw[red, thin] (30.5,1.7320508) to (31.5,3.46410161);
\draw[red, thick, dotted] (30.5,3.46410161) to (31.5,3.46410161);
\draw[red, thin] (31.5,3.46410161) to (34,3.46410161);
\draw[red, thin] (34,3.46410161) to (36,0);
\draw[red, thin] (36,0) to (34,-3.46410161);
\draw[red, thin] (34,-3.46410161) to (33.5,-3.46410161);
\draw[red, thick, dotted] (33.5,-3.46410161) to (32.5,-3.46410161);
\draw[red, thick, dotted] (32.5,-3.46410161) to (29.5,1.7320508);
\draw[red, thin] (31,1.7320508) to (34,-3.46410161);
\draw[red, thin] (31,1.7320508) to (30.5,1.7320508);

\draw (33,0) node{$\scaleobj{0.8}{B_3}$};

\draw[blue, thin] (33.5,-3.46410161) to (32.5,-5.19615242);
\draw[blue, thin] (32.5,-5.19615242) to (31.5,-5.19615242);
\draw[blue, thin] (31.5,-5.19615242) to (28.5,0);
\draw[blue, thin] (28.5,0) to (29.5,1.7320508);
\draw[blue, thin] (29.5,1.7320508) to (30.5,1.7320508);
\draw[blue, thick, dotted] (33,-3.46410161) to (32,-5.19615242);

\draw[black, thick, dotted] (29,1.7320508) to (31.5,1.7320508);
\draw[black, thick, dotted] (29.25,1.2990381) to (31.5,5.1961524);

\draw (29.8,-1) node{$\scaleobj{0.8}{A_3}$};

\end{tikzpicture}

\caption{\label{fig:LastVolAdjUglySixtyCase}}

\end{figure}

Notice that at this point the double bubble perimeter of $\left(A_4^{\varhexagon},B_4\setminus A_4^{\varhexagon}\right)$ may be longer than the double bubble perimeter of $\left(A_3,B_3\right)$ by $\epsilon$. However, the reason we choose to analyse $\left(A^{\varhexagon}\setminus B^{\varhexagon},B^{\varhexagon}\right)$ is that it has joint boundary that is no longer than the joint boundary of $\left(A^{\varhexagon},B^{\varhexagon}\setminus A^{\varhexagon}\right)$. This means that the amount of joint boundary in $\left(A^{\varhexagon},B^{\varhexagon}\setminus A^{\varhexagon}\right)$ that is contained in $\Lambda_{\smallrightarrow}\left(i_{\smallrightarrow}^A\right)$ is no longer than the joint boundary in $\left(A^{\varhexagon}\setminus B^{\varhexagon},B^{\varhexagon}\right)$ that is contained in $\Lambda_{\smallrightarrow}\left(s_{\smallrightarrow}^B\right)$.

This, however, must be too far because then $A_2$ strictly contains $A$, which means that $\mu\left(A_3\right)>\mu\left(A\right)$. Therefore, we only have to move the line segments to the right by an amount strictly less than the length of $\tilde{l}_1$. This means that the (closure of the) union of the resulting $A_3$ and $B_3$ strictly contains $A\cup B$. Therefore, $B_3$ has volume strictly greater than $\alpha$ (because we stop the process when $\mu\left(A_3\right)=1$). So, we can lower the top of $B_3$, and potentially move its right sides to the left, until we get a set $B_4$ that has volume $\alpha$. This strictly reduces the double bubble perimeter. Thus, $\left(A_3,B_4\right)\in\mathfs{F}_{\alpha}$, and $\rho_{DB}\left(A_3,B_4\right)<\rho_{DB}\left(A,B\right)$, which is what we wanted to show. 

\end{proof}

We continue eliminating configurations whose joint boundary consists of three line segments. 

\begin{lemma}
Let $\left(A,B\right)\in\mathfs{F}_{\alpha}$. Suppose that $B=B^{\varhexagon}$, that the joint boundary of $A$ and $B$ consists of three line segments, and that the angle interior to $A$ associated with $\pi_1$ is $60^{\circ}$. Suppose further that the angle exterior to both $A$ and $B$ at the point $\pi_1$ is $180^{\circ}$. Then, $(A,B)$ is not a double bubble minimizing configuration. The same result holds when $\pi_1$ is replaced with $\pi_2$. 
\end{lemma}

\begin{proof}

Since we are assuming that $B=B^{\varhexagon}$, and the joint boundary consists of three line segments, we may assume that these three line segments are $\Lambda_{\smallnearrow}\left(i_{\smallnearrow}^B\right)\cap\partial B$, $\Lambda_{\smallsearrow}\left(s_{\smallsearrow}^B\right)\cap\partial B$, and $\Lambda_{\smallrightarrow}\left(s_{\smallrightarrow}^B\right)\cap\partial B$. See Figure \ref{fig:No60InsideA} for illustration:

\begin{figure}[H]

\begin{tikzpicture}[scale=0.45]


\draw[blue, thin] (0,0) to (2,3.46410161);
\draw[blue, thin] (2,3.46410161) to (4,3.46410161);
\draw[blue, thin] (4,3.46410161) to (3,1.7320508);
\draw[blue, thin] (3,1.7320508) to (4,0);
\draw[blue, thin] (4,0) to (5,0);
\draw[blue, thin] (5,0) to (6,-1.7320508);
\draw[blue, thin] (6,-1.7320508) to (5,-3.46410161);
\draw[blue, thin] (5,-3.46410161) to (2,-3.46410161);
\draw[blue, thin] (2,-3.46410161) to (0,0);

\draw[red, thin] (4,3.46410161) to (5,3.46410161);
\draw[red, thin] (5,3.46410161) to (6,1.7320508);
\draw[red, thin] (6,1.7320508) to (5,0);
\draw[red, thin] (5,0) to (4,0);
\draw[red, thin] (4,0) to (3,1.7320508);
\draw[red, thin] (3,1.7320508) to (4,3.46410161);

\draw (2,0) node{$A$};
\draw (4.5,1.7) node{$B$};

\draw[->] (6.75,0) to (7.25,0);


\draw[blue, thin] (8,0) to (10,3.46410161);
\draw[blue, thin] (11,1.7320508) to (12,0);
\draw[blue, thin] (12,0) to (13,0);
\draw[blue, thin] (13,0) to (14,-1.7320508);
\draw[blue, thin] (14,-1.7320508) to (13,-3.46410161);
\draw[blue, thin] (13,-3.46410161) to (10,-3.46410161);
\draw[blue, thin] (10,-3.46410161) to (8,0);

\draw[red, thin] (12,3.46410161) to (13,3.46410161);
\draw[red, thin] (13,3.46410161) to (14,1.7320508);
\draw[red, thin] (14,1.7320508) to (13,0);
\draw[red, thin] (13,0) to (12,0);
\draw[red, thin] (12,0) to (11,1.7320508);
\draw[red, thin] (11,1.7320508) to (10,3.46410161);
\draw[red, thin] (10,3.46410161) to (12,3.46410161);

\draw (10,0) node{$A_1$};
\draw (12.5,1.7) node{$B_1$};

\draw (10.5,1.5) node{$L$};

\end{tikzpicture}

\caption{\label{fig:No60InsideA}}

\end{figure}

Note that at $\pi_2$ the angles don't have to be as they are in Figure \ref{fig:No60InsideA} above for this proof to work. The angle interior to $A$ at this point could be any of $60^{\circ}$, $120^{\circ}$, or $180^{\circ}$, as could the interior angle to $B$, so long as the sum of these two angles is at most $240^{\circ}$ (since the exterior angle cannot be $60^{\circ}$). Similarly, the angle interior to $B$ at the point $\pi_1$ can be $60^{\circ}$ or $120^{\circ}$.

The process is simple enough. We merely reorient the line segment $\Lambda_{\smallnearrow}\left(i_{\smallnearrow}^B\right)\cap\partial B$ so that its endpoint that is not at $\pi_1$ does not change, and so that it is contained in $\Lambda_{\smallsearrow}\left(s_{\smallsearrow}^B\right)$. See Figure \ref{fig:No60InsideA} above for reference.

For simplicity, we denote by $L$ the line segment that we have just reoriented, along with the line segment $\Lambda_{\smallsearrow}\left(s_{\smallsearrow}^B\right)\cap\partial B$. We have changed the volume ratio. So, to fix this, we move the line $L$ to the right, continuously reducing the volume of $B_1$ and adding it to $A_1$. We can do this until we obtain a new set $B_2$ that has volume $\alpha$. Note that this could result in reducing the line segment $\Lambda_{\smallrightarrow}\left(s_{\smallrightarrow}^B\right)\cap\partial B$ to length zero, and then possibly even adding part of $\Lambda_{\smallnearrow}\left(s_{\smallnearrow}^B\right)$ to the boundary of the set $A_2$, which is what we call the set that replaces $A_1$. The result is two sets whose joint boundary consists of either two line segments or one line segment. If these two sets have joint boundary consisting of two line segments, then we have a configuration $(A_2,B_2)$ in $\mathfs{G}_{\alpha}$ that has strictly better double bubble perimeter than did $(A,B)$ since we reduced the length of the bottom line segment in $B$ ($\Lambda_{\smallrightarrow}\left(s_{\smallrightarrow}^B\right)\cap\partial B$). If $(A_2,B_2)$ has joint boundary consisting of one line segment, we replace $A_2$ with $A_2^{\varhexagon}$ and rescale this set so that it has volume $1$. Let's call this new set $A_3$. Then $(A_3,B_2)$ has strictly better double bubble perimeter than did $(A,B)$ because, again, we reduced the length of $\Lambda_{\smallrightarrow}\left(s_{\smallrightarrow}^B\right)\cap\partial B$. Therefore, we have found a new set that is in $\mathfs{G}_{\alpha}$ with strictly better double bubble perimeter than $(A,B)$, as we wished.

\end{proof}

\begin{lemma}\label{lem:ThreeLineSegBndry}
Let $\left(A,B\right)\in\mathfs{F}_{\alpha}$. Suppose that the joint boundary of $(A,B)$ consists of three line segments, and that $B=B^{\varhexagon}$. If the angle interior to $B$ associated with $\pi_1$ is $60^{\circ}$, and the angle interior to $B$ associated with the point $\pi_2$ is also $60^{\circ}$, then $(A,B)$ is not a double bubble minimizing configuration. 
\end{lemma}

\begin{proof}

As in the previous lemma, we may assume that the joint boundary consists of $\Lambda_{\smallnearrow}\left(i_{\smallnearrow}^B\right)\cap\partial B$, $\Lambda_{\smallsearrow}\left(s_{\smallsearrow}^B\right)\cap\partial B$, and $\Lambda_{\smallrightarrow}\left(s_{\smallrightarrow}^B\right)\cap\partial B$. Since we are assuming that there are two $60^{\circ}$ angles interior to $B$ at $\pi_1$ and $\pi_2$, there can only be one other line segment forming $\partial B$. Figure \ref{fig:NoDouble60sInsideB} gives an illustration:

\begin{figure}[H]

\begin{tikzpicture}[scale=0.3]


\draw[blue, thin] (-1,0) to (2,5.19615242);
\draw[blue, thin] (2,5.19615242) to (3,5.19615242);
\draw[blue, thin] (3,5.19615242) to (4,3.46410161);
\draw[blue, thin] (4,3.46410161) to (3,1.7320508);
\draw[blue, thin] (3,1.7320508) to (4,0);
\draw[blue, thin] (4,0) to (6,0);
\draw[blue, thin] (6,0) to (7,-1.7320508);
\draw[blue, thin] (7,-1.7320508) to (6,-3.46410161);
\draw[blue, thin] (6,-3.46410161) to (1,-3.46410161);
\draw[blue, thin] (1,-3.46410161) to (-1,0);

\draw[red, thin] (4,3.46410161) to (6,0);
\draw[red, thin] (6,0) to (4,0);
\draw[red, thin] (4,0) to (3,1.7320508);
\draw[red, thin] (3,1.7320508) to (4,3.46410161);

\draw[->] (7,1) to (7.6,1);


\draw[blue, thin] (9,0) to (12,5.19615242);
\draw[blue, thin] (12,5.19615242) to (13,5.19615242);
\draw[blue, thin] (13,5.19615242) to (14,3.46410161);
\draw[blue, thin] (16,0) to (17,-1.7320508);
\draw[blue, thin] (17,-1.7320508) to (16,-3.46410161);
\draw[blue, thin] (16,-3.46410161) to (11,-3.46410161);
\draw[blue, thin] (11,-3.46410161) to (9,0);


\draw[red, thin] (14,3.46410161) to (16,0);
\draw[red, thin] (16,0) to (17,1.7320508);
\draw[red, thin] (17,1.7320508) to (16,3.46410161);
\draw[red, thin] (16,3.46410161) to (14,3.46410161);


\draw[->] (7,1) to (7.6,1);

\end{tikzpicture}

\caption{\label{fig:NoDouble60sInsideB}}

\end{figure}

The line segments in $\partial A\setminus\partial B$ need not be as shown in Figure \ref{fig:NoDouble60sInsideB}; the angle interior to $A$ associated with $\pi_1$ can form a $60^{\circ}$ or a $120^{\circ}$ angle, and similarly for the line segment in $\partial A$ that has an endpoint at $\pi_2$. 

Again, the process is simple enough. We merely reflect $B$ about the line $\Lambda_{\smallsearrow}\left(i_{\smallsearrow}^B\right)$ as demonstrated in Figure \ref{fig:NoDouble60sInsideB}. Let's denote by $B_1$ this reflected version of $B$. This process adds volume to $A$, which we can rescale to create a new set $A_1$ that has volume $1$. The result is a configuration $(A_1,B_1)\in\mathfs{G}_{\alpha}$ that has strictly better double bubble perimeter than $(A,B)$. 

As a side note, if the reflecting process results in an exterior $60^{\circ}$ angle, then we have a process to eliminate it as presented in Lemma \ref{lemma:NoSixtyDegreeAngles}.

\end{proof}

There are two other options for configurations whose joint boundary consists of three line segments (assuming that there is a $60^{\circ}$ angle). Again, consider the three angles associated with the point $\pi_1$. We know that the angle interior to $A$ cannot be $60^{\circ}$, and the angle exterior to both sets cannot be $60^{\circ}$. So, assume that the angle interior to $B$ is $60^{\circ}$. Now consider the angles associated with the point $\pi_2$. Again, we know that the angle exterior to both $A$ and $B$ cannot be $60^{\circ}$, nor can the angle interior to $A$. Since we are assuming that the angle interior to $B$ associated with $\pi_1$ is $60^{\circ}$, the corresponding angle associated with $\pi_2$ cannot be $60^{\circ}$. Therefore, all three angles associated with $\pi_2$ must be $120^{\circ}$. The final option is quite similar, but in this case the angle interior to $A$ at $\pi_1$ is $180^{\circ}$, while the angle interior to $B$ at $\pi_1$ is, again, $60^{\circ}$. Figure \ref{fig:FinalTwoSixties} gives an illustration of these two options.

\begin{figure}[H]

\begin{tikzpicture}[scale=0.3]


\draw[blue, thin] (0,0) to (2,3.46410161);
\draw[blue, thin] (2,3.46410161) to (4,3.46410161);
\draw[blue, thin] (4,3.46410161) to (5,1.7320508);
\draw[blue, thin] (5,1.7320508) to (7,1.7320508);
\draw[blue, thin] (7,1.7320508) to (7.5,2.59807621);
\draw[blue, thin] (7.5,2.59807621) to (8.5,2.59807621);
\draw[blue, thin] (8.5,2.59807621) to (10,0);
\draw[blue, thin] (10,0) to (8,-3.46410161);
\draw[blue, thin] (8,-3.46410161) to (2,-3.46410161);
\draw[blue, thin] (2,-3.46410161) to (0,0);

\draw[red, thin] (4,3.46410161) to (7,3.46410161);
\draw[red, thin] (7,3.46410161) to (7.5,2.59807621);
\draw[red, thin] (7.5,2.59807621) to (7,1.7320508);
\draw[red, thin] (7,1.7320508) to (5,1.7320508);
\draw[red, thin] (5,1.7320508) to (4,3.46410161);

\draw (4.8,0) node{$A$};
\draw (5.8,2.6) node{$B$};
\draw (3,4.1) node{$l_1^A$};
\draw (8,3.26) node{$l_2^A$};


\draw[blue, thin] (11,0) to (14,5.19615242);
\draw[blue, thin] (14,5.19615242) to (15,5.19615242);
\draw[blue, thin] (15,5.19615242) to (16,3.46410161);
\draw[blue, thin] (16,3.46410161) to (17,1.7320508);
\draw[blue, thin] (17,1.7320508) to (19,1.7320508);
\draw[blue, thin] (19,1.7320508) to (19.5,2.59807621);
\draw[blue, thin] (19.5,2.59807621) to (20.5,2.59807621);
\draw[blue, thin] (20.5,2.59807621) to (22,0);
\draw[blue, thin] (22,0) to (20,-3.46410161);
\draw[blue, thin] (20,-3.46410161) to (13,-3.46410161);
\draw[blue, thin] (13,-3.46410161) to (11,0);

\draw[red, thin] (16,3.46410161) to (19,3.46410161);
\draw[red, thin] (19,3.46410161) to (19.5,2.59807621);
\draw[red, thin] (19.5,2.59807621) to (19,1.7320508);
\draw[red, thin] (19,1.7320508) to (17,1.7320508);
\draw[red, thin] (17,1.7320508) to (16,3.46410161);

\end{tikzpicture}

\caption{\label{fig:FinalTwoSixties}}

\end{figure}

Neither of the configurations in the previous figure are double bubble minimizers, and both can be easily eliminated. Consider the lefthand configuration, with the notation as in the figure. We can move $B$ to the right, which shortens $l_2^A$, and lengthens $l_1^A$ by the same amount. We do this until $l_2^A=0$. This translated version of $B$ we call $\tilde{B}$. The set $A$ has also changed, and we call this new set $A_1$. The volume of $A_1$ is strictly greater than the volume of $A$. So, we can move the left sides of $A_1$ to the right until we have a set $A_2$ of volume $1$. Furthermore, we have a $60^{\circ}$ angle interior to $A_2$ that is created by the boundary of $\tilde{B}$ and $A_2$ at what would be the point $\pi_2\left(A_2,\tilde{B}\right)$. We already have a method to remove this angle and find a configuration with strictly better double bubble perimeter. Thus, the configuration on the lefthand side of Figure \ref{fig:FinalTwoSixties} is not a minimizer. 

The configuration on the righthand side is not in the family $\mathfs{F}_{\alpha}$ since there are three separate horizontal line segments and three separate line segments with slope $-\sqrt{3}$. Indeed, one can see that in this case $B^{\varhexagon}\subset A^{\varhexagon}$, and none of the lines used to form $B^{\varhexagon}$ are the same as the corresponding lines used to form $A^{\varhexagon}$. Lemma \ref{lem:AllSidesContained} shows us that this case can be replaced with a configuration in which two of the lines used to form $A^{\varhexagon}$ are the same as two of the lines used to form $\tilde{B}^{\varhexagon}$, where $\tilde{B}$ is a translated version of $B$. In this case, there would be a $60^{\circ}$ angle in $A$, which is a case which we have already analysed.

\begin{lemma}
Let $\left(A,B\right)\in\mathfs{F}_{\alpha}$. Suppose that $B=B^{\varhexagon}$, and that the joint boundary consists of two line segments. We may assume that these two line segments occur in $\Lambda_{\smallnearrow}\left(i_{\smallnearrow}^B\right)$ and $\Lambda_{\smallsearrow}\left(s_{\smallsearrow}^B\right)$. Suppose also that $\pi_1$ occurs at a point in $\Lambda_{\smallnearrow}\left(i_{\smallnearrow}^B\right)\cap\partial B$ that is not an endpoint. Similarly, suppose that $\pi_2$ occurs at a point in $\Lambda_{\smallsearrow}\left(s_{\smallsearrow}^B\right)\cap\partial B$ that is not an endpoint. Then, $\left(A,B\right)$ is not a double bubble minimizing configuration. 
\end{lemma}

\begin{proof}

We may assume that the angle exterior to both $A$ and $B$ at $\pi_1$ is $120^{\circ}$ and similarly for the angle exterior to both $A$ and $B$ at $\pi_2$. This is because the only options for these two angles are $60^{\circ}$ or $120^{\circ}$, and we already saw that there cannot be an exterior $60^{\circ}$ angle. The configuration looks as in Figure \ref{fig:ASillyConfiguration}:

\begin{figure}[H]

\begin{tikzpicture}[scale=0.6]

\draw[blue, thin] (6,0) to (7,1.7320508);
\draw[blue, thin] (7,1.7320508) to (8,1.7320508);
\draw[blue, thin] (8,1.7320508) to (9,0);
\draw[blue, thin] (9,0) to (8,-1.7320508);
\draw[blue, thin] (8,-1.7320508) to (7,-1.7320508);
\draw[blue, thin] (7,-1.7320508) to (6,0);

\draw[red, thin] (6.75,1.299038) to (5,1.299038);
\draw[red, thin] (5,1.299038) to (4.75,0.866025);
\draw[red, thin] (4.75,0.866025) to (5.5,-0.4330127);
\draw[red, thin] (5.5,-0.4330127) to (6.25,-0.4330127);

\draw (5.6,0.5) node{$A$};
\draw (7.5,0.1) node{$B$};

\end{tikzpicture}
\caption{\label{fig:ASillyConfiguration}}
\end{figure}

We can replace $A$ with $A^{\varhexagon}$ and $B$ with $B\setminus A^{\varhexagon}$. Then, we can adjust the volumes in a manner nearly identical to what we presented in Figure \ref{fig:OneCornerVolumeAdjustment2}. This will result in strictly decreasing the double bubble perimeter because it will strictly increase the volume of $B$.

\end{proof}

\section{Configurations With No Sixty Degree Angles}
\label{sec:NoSixties}

Throughout this section, we assume that $(A,B)\in\mathfs{F}_{\alpha}$. We will include here the analysis of possible solutions that have no interior or exterior $60^{\circ}$ angles. We know that the joint boundary must consist of one, two, or three line segments. If it is one line segment only, then the configuration is what we call a kissing hexagon case, and these are in the family $\mathfs{G}_{\alpha}$. If the joint boundary consists of two line segments, then this is what we call the embedded hexagon case, and these, again, are in the family $\mathfs{G}_{\alpha}$. The case when the joint boundary consists of three line segments, and all angles are $120^{\circ}$, we eliminate now.

\begin{figure}[H]

\begin{tikzpicture}[scale=0.2]

\draw[red, thin] (10,0) to (9,1.732050808);
\draw[red, thin] (9,1.732050808) to (6.5,1.732050808);
\draw[red, thin] (6.5,1.732050808) to (5.5,0);
\draw[red, thin] (5.5,0) to (6.5,-1.732050808);
\draw[red, thin] (6.5,-1.732050808) to (9,-1.732050808);
\draw[red, thin] (9,-1.732050808) to (10,0);

\draw[blue, thin] (6.5,1.732050808) to (5.5,3.464101615);
\draw[blue, thin] (5.5,3.464101615) to (4,3.464101615);
\draw[blue, thin] (4,3.464101615) to (0,-3.464101615);
\draw[blue, thin] (0,-3.464101615) to (2,-6.92820323);
\draw[blue, thin] (2,-6.92820323) to (8,-6.92820323);
\draw[blue, thin] (8,-6.92820323) to (10,-3.464101615);
\draw[blue, thin] (10,-3.464101615) to (9,-1.732050808);


\end{tikzpicture}

\caption{\label{fig:TwoPossibleCases}}

\end{figure}

As can be seen in Figure \ref{fig:TwoLinesNoSixties}, there are two line segments in $\partial A$ that are parallel to $\Lambda_{\smallsearrow}\left(i_{\smallsearrow}^B\right)$, don't contribute positive length to the joint boundary, and intersect $\partial B$. Let $L_1$ and $L_2$ be the lines passing through these two line segments, as in Figure \ref{fig:TwoLinesNoSixties}:

\begin{figure}[H]

\begin{tikzpicture}[scale=0.25]

\draw[red, thin] (10,0) to (9,1.732050808);
\draw[red, thin] (9,1.732050808) to (6.5,1.732050808);
\draw[red, thin] (6.5,1.732050808) to (5.5,0);
\draw[red, thin] (5.5,0) to (6.5,-1.732050808);
\draw[red, thin] (6.5,-1.732050808) to (9,-1.732050808);
\draw[red, thin] (9,-1.732050808) to (10,0);

\draw[blue, thin] (6.5,1.732050808) to (5.5,3.464101615);
\draw[blue, thin] (5.5,3.464101615) to (4,3.464101615);
\draw[blue, thin] (4,3.464101615) to (0,-3.464101615);
\draw[blue, thin] (0,-3.464101615) to (2,-6.92820323);
\draw[blue, thin] (2,-6.92820323) to (8,-6.92820323);
\draw[blue, thin] (8,-6.92820323) to (10,-3.464101615);
\draw[blue, thin] (10,-3.464101615) to (9,-1.732050808);

\draw[black, thick, dashed] (7.5,0) to (3.5,6.92820323028);

\draw[black, thick, dashed] (11,-5.19615242271) to (7,1.732050808);

\draw (5.5,5) node{$L_1$};

\draw (11.2,-4) node{$L_2$};

\end{tikzpicture}

\caption{\label{fig:TwoLinesNoSixties}}

\end{figure}

If, as is the case in Figure \ref{fig:TwoLinesNoSixties}, the $x$-intercept of $L_2$ is greater than that of $L_1$, then we can slide $B$ parallel to $L_1$ (or $L_2$) toward the top of $\partial A$. Moving it at all will increase the volume of $A$. As we slide $B$, we reduce the perimeter of $A$ that is contained in $L_1$ and add it to the perimeter of $A$ that is contained in $L_2$. This will not change the double bubble perimeter, but it will add volume to $A$. We can do this until the line segment $L_1\cap\partial A$ has length zero. Let's denote this translated version of $B$ by $B_t$, and since we have also altered the set $A$, lets call this new set $A_1$. The volume of $A_1$ is strictly greater than the volume of $A$. Adjusting this volume is simple and similar to previous volume adjustments. Furthermore, this procedure results in a configuration whose joint boundary consists of three line segments, and also an interior $60^{\circ}$ angle. Thus, if the volume adjustment does not eliminate this $60^{\circ}$ angle, then the previous section showed us that this cannot be a minimizing configuration either. If the $x$-intercept of $L_1$ is greater than (or equal to) the $x$-intercept of $L_2$, we can perform the same procedure, but by sliding $B$ in the opposite direction. Notice that if the $x$-intercepts of both $L_1$ and $L_2$ are the same, then no volume adjustments are necessary, and this procedure results in an interior $60^{\circ}$ angle and a configuration whose joint boundary consists of three line segments. Again, we know that such a configuration is not a double bubble minimizing configuration and that we can replace it with something in $\mathfs{G}_{\alpha}$. 

If $L_1=L_2$, then we first slide $B$ so that there is a point of intersection of the lines $\Lambda_{\smallnearrow}\left(i_{\smallnearrow}^B\right)$, $\Lambda_{\smallrightarrow}\left(i_{\smallrightarrow}^B\right)$, and $\Lambda_{\smallrightarrow}\left(i_{\smallrightarrow}^A\right)$. This creates a figure with an interior $60^{\circ}$ angles, and so we will analyse it in the following section.

\section{Calculations}
\label{sec:kkt}

Based on the previous sections, in order to find a double bubble minimizer, it is enough to look only at the configurations in the family $\mathfs{G}_{\alpha}$. See Figure \ref{fig:TheGoodFamily}. The configuration on the lefthand side of this figure we call Embedded Hexagons, and the configuration on the righthand side of this figure we call Kissing Hexagons. We begin by analysing this latter case.

\subsection{Kissing Hexagons}
\label{subsec:kissinghex}

In the kissing hexagons case, we begin by first analysing a single hexagon where the length of one side is fixed, call this length $L$, to obtain an expression for the minimizing perimeter that depends on $L$. This fixed side will be at least part of the joint boundary in the double bubble configuration. 

To solve the problem of the kissing hexagons would normally require six parameters. We arrive at this number because there are six sides of each hexagon, making a total of 12. Then, suppose we fix one of the corners of one of the hexagons, say $A$, at the origin. This means that the sum of the first coordinates of the six corners of the hexagon must add to $0$, as must the sum of the second coordinates. This makes two equations. We can make a similar argument for the other hexagon, in this case $B$. This makes two more equations. The last two equations come from the fact that $\mu\left(A\right)=1,\mu\left(B\right)=\alpha$. Thus, in total, we have six equations, which reduces our number of parameters from $12$ to $6$.

For a single set, we start with $5$ parameters (because we fix one side as $L_1$). We can then use the same three equations, which reduces the single bubble problem from five parameters to two. This will result in a solution that depends on $L_1$. Then, we can do the same thing with the other set and with $L_2$ being the length of the fixed side. 

Finally, we will have to solve a problem that depends on $L_1$ and $L_2$, which is again a two parameter problem.

Figure~\ref{fig:params2} shows how we parameterize the single set. 

\begin{figure}[H]

\begin{tikzpicture}[scale=0.4]

\draw[blue, thin] (-1,0) to (1,0);
\draw[blue, thin] (1,0) to (3,3.4641);
\draw[blue, thin] (3,3.4641) to (1,6.9282);
\draw[blue, thin] (1,6.9282) to (-1,6.9282);
\draw[blue, thin] (-1,6.9282) to (-3,3.4641);
\draw[blue, thin] (-3,3.4641) to (-1,0);

\draw (2.51,5) node{$\scaleobj{0.8}{x_2}$};
\draw (0,7.16) node{$\scaleobj{0.8}{x_3}$};
\draw (-2.54,4.95) node{$\scaleobj{0.8}{x_4}$};
\draw (-2.5,2) node{$\scaleobj{0.8}{x_5}$};
\draw (0,-0.24) node{$\scaleobj{0.8}{L_1}$};
\draw (2.55,2) node{$\scaleobj{0.8}{x_1}$};

\end{tikzpicture}
\caption{\label{fig:params2}}
\end{figure}

First, assume that the left endpoint of $L_1$ is at the origin. Since the boundary of the set is a closed curve, the sum of the sides of the hexagon results in $(0,0)$. This gives 

$$\left(L_1+\frac{1}{2}\left(x_{1}-x_{2}\right)-x_{3}+\frac{1}{2}\left(-x_{4}+x_{5}\right),\frac{\sqrt{3}}{2}\left(x_{1}+x_{2}-x_{4}-x_{5}\right)\right)=\left(0,0\right)$$ 
(this is because $\left(\sin\left(\frac{\pi}{3}\right),\cos\left(\frac{\pi}{3}\right)\right)=\left(\frac{1}{2},\frac{\sqrt{3}}{2}\right)$). Notice that the second coordinate gives $\frac{\sqrt{3}}{2}\left(x_1+x_2-x_4-x_5\right)=0$, which implies that 

\begin{equation}\label{eq:EqForX5}
x_5=x_1+x_2-x_4.
\end{equation} 

The first coordinate gives us (after a little rearranging) that $x_{3}=L_1+\frac{1}{2}\left(x_1-x_2-x_4+x_5\right)=L_1+\frac{1}{2}\left(x_{1}-x_2-x_{4}+\left(x_1+x_2-x_4\right)\right)=L_1+x_1-x_4$. We save this as our second equation:

\begin{equation}\label{eq:EqForX3}
x_3=L_1+x_1-x_4. 
\end{equation}

The final equation comes from the volume, which we will call $V$ and can be either $1$ or $\alpha\in(0,1]$. We rely on the shoelace theorem, also known as Gauss's area formula, to show that the volume is given by: $V=\frac{\sqrt{3}}{4}\left(2(x_1+x_2)(x_3+x_4)-x_1^2-x_4^2\right)=\frac{\sqrt{3}}{4}\left(2(x_1+x_2)(L+x_1)-x_1^2-x_4^2\right)$. Solving this equation for $x_4$ yields:

\begin{equation}\label{eq:EqForX4}
x_{4}=\sqrt{x_{1}^{2}+2x_{1}x_{2}+2L(x_{1}+x_{2})-4\frac{V}{\sqrt{3}}}. 
\end{equation}

We want to minimize the perimeter, which is given by $P=L_1+\sum_{i=1}^5x_i=L_1+x_1+x_2+\left(L_1+x_1-x_4\right)+x_4+\left(x_1+x_2-x_4\right)$. Here, we have used \eqref{eq:EqForX5} and \eqref{eq:EqForX3} to replace $x_3$ and $x_5$. Now, replacing $x_4$ using \eqref{eq:EqForX4} yields a perimeter of $P\left(x_1,x_2\right)=3x_{1}+2x_{2}-\sqrt{x_{1}^{2}+2x_{1}x_{2}+2L(x_{1}+x_{2})-4\frac{V}{\sqrt{3}}}+2L_1$. We want to minimize this perimeter subject to the constraints that each side length is non-negative. Substituting the equations for $x_3$, $x_4$, and $x_5$, we get the following problem:

\small
\begin{equation}
	\label{eq:kktproblem}
	\begin{aligned}
	argmin(\{3x_{1}+2x_{2}-\sqrt{x_{1}^{2}+2x_{1}x_{2}+2L_1(x_{1}+x_{2})-4\frac{V}{\sqrt{3}}}+2L_1|x_{1},x_{2}\ge 0,
	\\L_1+x_{1}-\sqrt{x_{1}^{2}+2x_{1}x_{2}+2L_1(x_{1}+x_{2})-4\frac{V}{\sqrt{3}}}\ge 0,\sqrt{x_{1}^{2}+2x_{1}x_{2}+2L_1(x_{1}+x_{2})-4\frac{V}{\sqrt{3}}}\ge 0,
	\\\textrm{ and }x_{1}+x_{2}\ge \sqrt{x_{1}^{2}+2x_{1}x_{2}+2L_1(x_{1}+x_{2})-4\frac{V}{\sqrt{3}}}\}). 
	\end{aligned}
\end{equation}
\normalsize

This can be solved with elementary calculus. However, the calculations are quite cumbersome and not particularly instructive. So, we do not include them here. The minimizers change depending on whether $3\sqrt{3}L_1^2<16V$ or not. If this inequality holds, then we find that $x_1=x_5=2x_2-L_1$, and $x_2=x_3=x_4=\sqrt{\frac{3L_1^2+4\sqrt{3}V}{21}}$. 

At the point when $3\sqrt{3}L_1^2=16V$, the equation for $x_1$ becomes 

$$x_1=2x_2-L_1=2\sqrt{\frac{3L_1^2+4\sqrt{3}V}{21}}-L_1=...=2\sqrt{\frac{3L_1^2+\frac{3\sqrt{3}L_1^2\cdot\sqrt{3}}{4}}{21}}-L_1=...=0.$$

This indicates that for $3\sqrt{3}L_1^2<16V$, the figure is a hexagon with six sides, whereas when this inequality does not hold, the figure consists of only four sides. In this latter case, we get that $x_1=x_5=0$, $x_2=x_4=L_1-\sqrt{\frac{2L_1^2-4\sqrt{3}V}{3}}$, and $x_3=\sqrt{\frac{2L_1^2-4\sqrt{3}V}{3}}$. 

Now, it is easy to solve the isoperimetric problem that was left unfinished in Lemma \ref{lem:IsoperimetricProblem}. There are two possible functions for the perimeter depending on if $3\sqrt{3}L^2<16V$ or not. Let $P_1\left(L\right)$ be the length of the perimeter when $3\sqrt{3}L^2<16V$, and $P_2\left(L\right)$ be the perimeter when this inequality does not hold. Then, we have an explicit equation for the perimeter given by one of the following: 

$$P_1\left(L\right)=L+2\left(2\sqrt{\frac{3L^2+4\sqrt{3}V}{21}}-L\right)+3\sqrt{\frac{3L^2+4\sqrt{3}V}{21}}=7\sqrt{\frac{3L^2+4\sqrt{3}V}{21}}-L,$$

or

$$P_2\left(L\right)=L+2\left(L-\sqrt{\frac{2L^2-4\sqrt{3}V}{3}}\right)+\sqrt{\frac{2L^2-4\sqrt{3}V}{3}}=3L+3\sqrt{\frac{2L^2-4\sqrt{3}V}{3}}.$$

This second perimeter function has domain $L\geq\sqrt{2V}\sqrt[4]{3}$. It is easily shown that on this domain $P_2>P_1$. See Appendix \ref{sec:Comparison} for a similar calculation. So, we only need to analyse $P_1$. 
Equating the derivative of this function to zero and solving for $L$ yields $L_0=\frac{\sqrt{2V}}{3^{3/4}}$. This gives a total perimeter of $P\left(L_0\right)=2\sqrt{2V}\sqrt[4]{3}$. This is the perimeter of a regular hexagon with volume $V$, as we wanted to show.

We now turn to solving the kissing hexagons problem. Say that the side of $A$ that is at least partially joint boundary has length $L_1$, and the side of $B$ that is at least partially joint boundary has length $L_2$. There are four possible equations for the double bubble perimeter depending on if $3\sqrt{3}L_{1}^{2}<16$ or not, and if $3\sqrt{3}L_{2}^{2}<16\alpha$ or not. To find these equation, we merely add $P_i\left(L_1\right)$ to $P_i\left(L_2\right)$ ($i=1,2$) and subtract the shorter of $L_1$ and $L_2$ since this value represents the shared boundary that we only count once. Equation \ref{eq:kktmult} shows the possible perimeters.

\begin{equation}
\label{eq:kktmult}
\begin{aligned}
&7\sqrt{\frac{3L_{1}^{2}+4\sqrt{3}}{21}}+7\sqrt{\frac{3L_{2}^{2}+4\sqrt{3}\alpha}{21}}-L_{1}-L_{2}-min(L_{1},L_{2}),\hspace{0.5cm}\left(3\sqrt{3}L_{1}^{2}<16,3\sqrt{3}L_{2}^{2}<16\alpha\right)\\
&7\sqrt{\frac{3L_{1}^{2}+4\sqrt{3}}{21}}-\sqrt{\frac{2L_{2}^{2}-4\sqrt{3}\alpha}{3}}+3L_{2}-L_{1}-min(L_{1},L_{2}),\hspace{0.3cm}\left(3\sqrt{3}L_{1}^{2}<16,3\sqrt{3}L_{2}^{2}\ge 16\alpha\right)\\
&7\sqrt{\frac{3L_{2}^{2}+4\sqrt{3}\alpha}{21}}-\sqrt{\frac{2L_{1}^{2}-4\sqrt{3}}{3}}+3L_{1}-L_{2}-min(L_{1},L_{2}),\hspace{0.3cm}\left(3\sqrt{3}L_{1}^{2}\ge 16,3\sqrt{3}L_{2}^{2}<16\alpha\right)\\
&3(L_{1}+L_{2})-\sqrt{\frac{2L_{1}^{2}-4\sqrt{3}}{3}}-\sqrt{\frac{2L_{2}^{2}-4\sqrt{3}\alpha}{3}}-min(L_{1},L_{2}),\hspace{0.3cm}\left(3\sqrt{3}L_{1}^{2}\ge 16,3\sqrt{3}L_{2}^{2}\ge 16\alpha\right).\\
\end{aligned}
\end{equation}

We assume first that either $L_1<L_2$, or vice versa. Taking the partial derivative of these equations with respect to $L_1$ and equating them to zero results in one of the following:

\begin{equation}
    \label{eq:crit}
    \sqrt{\frac{21L_{1}^{2}}{3L_{1}^{2}+4\sqrt{3}V}}-1-\mathbbm{1}_{L_{1}<L_{2}}=0,3-\mathbbm{1}_{L_{1}<L_{2}}-\sqrt{\frac{2L_{1}^{2}}{3(L_{1}^{2}-2\sqrt{3}V)}}=0.
\end{equation}

Here, $V=1$. Taking the partial derivatives with respect to $L_2$ and equating to zero results in the same equations, but with $L_2$ instead of $L_1$ and with $V=\alpha$ instead of $1$. Solving \eqref{eq:crit} results in:

\begin{equation}
    \label{eq:crit2}
    L_{1}=\sqrt{\frac{4\sqrt{3}V(1+\mathbbm{1}_{L_{1}<L_{2}})^{2}}{(21-3(1+\mathbbm{1}_{L_{1}<L_{2}})^{2})}},L_{1}=\sqrt{\frac{(3-\mathbbm{1}_{L_{1}<L_{2}})^{2}(3(2\sqrt{3}V))}{(3(3-\mathbbm{1}_{L_{1}<L_{2}})^{2}-2)}}. 
\end{equation}

This results in eight possible pairs of lengths for $L_1$ and $L_2$. They are given below. The values of $\alpha$ for which these lengths hold are also given. These are obtained by comparing the side lengths and solving for $\alpha$. For example, consider when $L_1=\sqrt{\frac{16\sqrt{3}}{9}}$, and $L_2=\sqrt{\frac{4\sqrt{3}\alpha}{18}}$, as in the first row of \eqref{KissingHexL1NEQL2}. These lengths occur when $L_1<L_2$. So, we solve this inequality for $\alpha$ and obtain $8<\alpha$. Thus, we know that these values of $L_1$ and $L_2$ do not yield a minimizing configuration because we are assuming that $\alpha\in(0,1]$. We denote this range of possible values of $\alpha$ by $\mathfs{R}$.

\bae\label{KissingHexL1NEQL2}
\left(L_1,L_2,\mathfs{R}\right)=&\left(\sqrt{\frac{16\sqrt{3}}{9}},\sqrt{\frac{4\sqrt{3}\alpha}{18}},8<\alpha\right),&L_1<L_2,\left(3\sqrt{3}L_{1}^{2}<16,3\sqrt{3}L_{2}^{2}<16\alpha\right)\\
&\left(\sqrt{\frac{4\sqrt{3}}{18}},\sqrt{\frac{16\sqrt{3}\alpha}{9}},\frac{1}{8}>\alpha\right),&L_2<L_1,\left(3\sqrt{3}L_{1}^{2}<16,3\sqrt{3}L_{2}^{2}<16\alpha\right)\\
&\left(\sqrt{\frac{4(3(2\sqrt{3}))}{10}},\sqrt{\frac{9(3(2\sqrt{3}\alpha))}{25}},\frac{10}{9}<\alpha\right),&L_1<L_2,\left(3\sqrt{3}L_{1}^{2}\ge 16,3\sqrt{3}L_{2}^{2}\ge 16\alpha\right)\\
&\left(\sqrt{\frac{9(3(2\sqrt{3}))}{25}},\sqrt{\frac{4(3(2\sqrt{3}\alpha))}{10}},\frac{9}{10}>\alpha\right)&L_2<L_1,\left(3\sqrt{3}L_{1}^{2}\ge 16,3\sqrt{3}L_{2}^{2}\ge 16\alpha\right)\\
&\left(\sqrt{\frac{16\sqrt{3}}{9}},\sqrt{\frac{9(3(2\sqrt{3}\alpha))}{25}},\frac{200}{243}<\alpha\right),&L_1<L_2,\left(3\sqrt{3}L_{1}^{2}<16,3\sqrt{3}L_{2}^{2}\ge 16\alpha\right)\\
&\left(\sqrt{\frac{4\sqrt{3}}{18}},\sqrt{\frac{4(3(2\sqrt{3}\alpha))}{10}},\frac{5}{54}>\alpha\right),&L_2<L_1,\left(3\sqrt{3}L_{1}^{2}<16,3\sqrt{3}L_{2}^{2}\ge 16\alpha\right)\\
&\left(\sqrt{\frac{4(3(2\sqrt{3}))}{10}},\sqrt{\frac{4\sqrt{3}\alpha}{18}},\frac{54}{5}<\alpha\right),&L_1<L_2,\left(3\sqrt{3}L_{1}^{2}\ge 16,3\sqrt{3}L_{2}^{2}<16\alpha\right)\\
&\left(\sqrt{\frac{9(3(2\sqrt{3}))}{25}},\sqrt{\frac{16\sqrt{3}\alpha}{9}},\frac{243}{200}>\alpha\right),&L_2<L_1,\left(3\sqrt{3}L_{1}^{2}\ge 16,3\sqrt{3}L_{2}^{2}<16\alpha\right). 
\eae

It is easily seen that some of these possible values for $L_1$ and $L_2$ can be eliminated since the range for $\alpha$ does not include any values in $(0,1]$. Upon closer inspection, it can also be seen that not all values of $\alpha\in(0,1]$ are included. This is because we assumed that either $L_1<L_2$, or $L_2<L_1$. Thus, to find the possible solution for the remaining values of $\alpha$, we set $L_1=L_2$. The possible perimeters are shown in \eqref{eq:kktmult2}. 

\bae
\label{eq:kktmult2}
&P_3=7\sqrt{\frac{3L^{2}+4\sqrt{3}}{21}}+7\sqrt{\frac{3L^{2}+4\sqrt{3}\alpha}{21}}-3L,
\\
&P_4=7\sqrt{\frac{3L^{2}+4\sqrt{3}}{21}}-\sqrt{\frac{2L^{2}-4\sqrt{3}\alpha}{3}}+L,
\\
&P_5=7\sqrt{\frac{3L^2+4\sqrt{3}\alpha}{21}}-\sqrt{\frac{2L^2-4\sqrt{3}}{3}}+L,
\\
&P_6=5L-\sqrt{\frac{2L^{2}-4\sqrt{3}}{3}}-\sqrt{\frac{2L^{2}-4\sqrt{3}\alpha}{3}}.
\eae

We can easily compare these four functions to each other to see inside which regimes, if any, each one is the smallest. This is quite a tedious process and is therefore relegated to an appendix, where we show the outline for how to compare $P_3$ and $P_4$. When these comparisons are made, we find that $P_3\leq P_5$ and $P_3\leq P_6$ for all $L\geq\sqrt{2}\sqrt[4]{3}$, $P_3\leq P_4$ for $L\geq\sqrt{2}\sqrt[4]{3}\sqrt{\alpha}$.  Of course, $P_5$ and $P_6$ are not defined for $L<\sqrt{2}\sqrt[4]{3}$, and $P_4$ is not defined for $L<\sqrt{2}\sqrt[4]{3}\sqrt{\alpha}$. So, we must now compare $P_3$ to the perimeters found when $L_1\neq L_2$.  

In \eqref{KissingHexL1NEQL2}, notice that there are three possible solutions that we can immediately eliminate as contenders because their range does not intersect $(0,1]$. These correspond to the first, third, and seventh rows in this equation. For the other possible lengths of $L_1$ and $L_2$, we merely have to evaluate the perimeter functions that produced each pair of lengths. Then, we compare these results to $P_3$. Again, this is somewhat tedious, and we therefore do not include these calculations here as they are similar enough to the one shown in Appendix \ref{sec:Comparison}. These comparisons show that for $\alpha<\frac{1}{8}$, the double bubble perimeter corresponding to the second row of \eqref{KissingHexL1NEQL2} is less than the double bubble perimeter produced by $P_3$. That is, when $\L_1=\sqrt{\frac{4\sqrt{3}}{18}},L_2=\sqrt{\frac{16\sqrt{3}\alpha}{9}},\frac{1}{8}>\alpha,L_2<L_1,\left(3\sqrt{3}L_{1}^{2}<16,3\sqrt{3}L_{2}^{2}<16\alpha\right)$, the double bubble perimeter in the kissing hexagons case is 
$$7\sqrt{\frac{3L_{1}^{2}+4\sqrt{3}}{21}}+7\sqrt{\frac{3L_{2}^{2}+4\sqrt{3}\alpha}{21}}-L_{1}-L_{2}-min(L_{1},L_{2})=2\sqrt[4]{3}\left(\sqrt{2}+\sqrt{\alpha}\right).$$

For all other values of $\alpha$, $P_3$ produces the shortest double bubble perimeter. Therefore, we must find the critical points of $P_3$ to obtain the global minimizing value of $L$ (which will, naturally, depend on $\alpha$). 

Finding the derivative of $P_3$ is simple. Finding the critical points, however, is not as easy. This is because the values of $L$ that result in minima are the roots of a polynomial of degree $8$. We show in Appendix \ref{sec:TheBigPolynomial} what this polynomial is, call it $p$, and how to obtain it. The root that results in the global minimizer for $P_3$ cannot be expressed in closed form. We can, however, give an order of the roots of this polynomial and find which root gives the global minimizer; the root in question is the third root under the ordering that we give, so we call this root $r_3$. Thus, for the kissing hexagons case, we have our potential double bubble minimizing configurations. For $0<\alpha<\frac{1}{8}$, the minimizing perimeter is given by $2\sqrt[4]{3}\left(\sqrt{2}+\sqrt{\alpha}\right)$. For $\alpha\geq\frac{1}{8}$, the minimizer is $P_3\left(r_3\right)$. 

\subsection{Embedded Hexagon Case}
\label{subsec:embedded}

The next case we consider is the embedded hexagon case, which is demonstrated in Figure~\ref{fig:EmbeddedRectangle}

\begin{figure}[H]

\begin{tikzpicture}[scale=0.6]

\draw[blue, thin] (0,0) to (-1,1.732050808);
\draw[blue, thin] (-1,1.732050808) to (-3.5,1.732050808);
\draw[blue, thin] (-3.5,1.732050808) to (-4.5,0);
\draw[blue, thin] (-4.5,0) to (-3.5,-1.732050808);
\draw[blue, thin] (-3.5,-1.732050808) to (-1,-1.732050808);
\draw[blue, thin] (-1,-1.732050808) to (0,0);

\draw[red, thin] (-3.5,1.732050808) to (-4.5,3.464101615);
\draw[red, thin] (-4.5,3.464101615) to (-6,3.464101615);
\draw[red, thin] (-6,3.464101615) to (-8,0);
\draw[red, thin] (-8,0) to (-6,-3.464101615);
\draw[red, thin] (-4.5,-3.464101615) to (-6,-3.464101615);
\draw[red, thin] (-4.5,-3.464101615) to (-3.5,-1.732050808);


\draw (-2,-2.15) node{$x_1$};
\draw (-0.15,-1) node{$x_2$};
\draw (-0.15,1) node{$x_3$};
\draw (-2,2.05) node{$x_4$};
\draw (-4.3,1) node{$x_5$};
\draw (-4.3,-1) node{$x_6$};


\draw (-5.25,-3.9) node{$y_1$};
\draw (-3.8,-3) node{$y_2$};
\draw (-3.8,3) node{$y_3$};
\draw (-5.25,3.9) node{$y_4$};
\draw (-7,2.6) node{$y_5$};
\draw (-7,-2.6) node{$y_6$};

\end{tikzpicture}

\caption{\label{fig:EmbeddedRectangle}}
\end{figure}



Let's say $B$ is the set with parameters $x_1,...x_6$ and $A$ is the set with parameters $y_1,...,y_6$. We employ a similar method to our process of solving the kissing hexagons problem by fixing the distance (in $\mathcal{D}$, not in the Euclidean norm), call it $L$, between $x_1$ and $x_4$. This means that $x_2+x_3=x_5+x_6=L$. Similarly to before, we can find three equations in terms of our six parameters based off of the facts that $\partial B$ forms a closed curve, and $\mu\left(B\right)=\alpha$. The first of these equations is as follows:

\begin{equation}\label{eq:SumOfFirstCoord}
x_1-x_4+\frac{1}{2}\left(x_2+x_6-x_3-x_5\right)=0.
\end{equation}

This is because $\cos\left(\frac{\pi}{3}\right)=\frac{1}{2}$. 

The second equation, similarly, comes from the fact that the second coordinates of these vectors must sum to zero:

\begin{equation}\label{eq:SumOfSecondCoord}
x_2+x_3-x_5-x_6=0.
\end{equation}

Finally, the third and fourth equation merely come from us setting both $x_5+x_6$ and $x_2+x_3$ equal to $L$:

\begin{equation}\label{eq:VertEq}
x_2+x_3=x_5+x_6=L.
\end{equation}

Using \eqref{eq:VertEq}, we can replace $x_3$ and $x_6$ in \eqref{eq:SumOfFirstCoord}, which, after a little rearranging, is $2(x_1-x_4)=x_5-x_6+x_3-x_2$, whence $2(x_1-x_4)=2x_5-L+L-2x_2=2(x_5-x_2)$. Solving this for $x_5$ gives $x_5=x_1+x_2-x_4$. We now have equations for $x_3$, $x_5$, and $x_6$ in terms of the other variables. 

We again use Gauss's area formula to show that the volume $V$ must be $V=\frac{\sqrt{3}}{4}(2(x_{1}+x_{2})(L+x_{4}-x_{2})-x_{1}^{2}-x_{4}^{2})$. Solving this for $x_4$ gives:

\begin{equation}
x_{4}=x_{1}+x_{2}\pm\sqrt{2L(x_{1}+x_{2})-x_{2}^{2}-\frac{4V}{\sqrt{3}}}.
\end{equation}

Using these equations to replace $x_3,...,x_6$ in the formula for the perimeter gives:

\begin{equation}
P\left(x_1,x_2\right)=2L+2x_1+x_2\pm\sqrt{2L(x_1+x_2)-x_2^2-\frac{4V}{\sqrt{3}}}.
\end{equation}

We can use calculus to find the values of $x_1,x_2$ that minimize this function. Then, using the equations for $x_3,...,x_6$, we can find the remaining side lengths. Doing this results in $x_2=x_3=x_5=x_6=\frac{L}{2}$. We can then use this to solve for $x_1$ and $x_4$. The result is $x_1=x_4=\frac{8\sqrt{3}V-3L^2}{12L}$. This gives a global minimum perimeter of $\frac{9L^{2}+8\sqrt{3}V}{6L}$. One also obtains the hexagon's perimeter when one minimizes over $L$.

Now we consider the red set in Figure~\ref{fig:EmbeddedRectangle}, call this set $A$. Let $y_2+x_5+x_6+y_3=L_{2}$ and $x_5+x_6=L_{1}$, and fix these two values. We want to initially remove the indentation corresponding to $x_5$ and $x_6$. This is achieved by replacing $A$ with $A^{\varhexagon}$. Then, the argument of the previous paragraph gives us the minimum perimeter shape for a given volume, which in this case is $V+\frac{\sqrt{3}}{2}\cdot x_5\cdot x_6$, where $V=\mu\left(A\right)$. In this minimization problem, we make one change to enlarge the parameter space. This change is to allow for the possibility that one of $x_5$ or $x_6$ is greater than $\frac{L_2}{2}$. Without loss of generality, we assume that $x_6\leq\frac{L}{2}$ (otherwise we can just rotate the set about a horizontal axis). This means that it is possible that $x_5>\frac{L_2}{2}$. So, to compensate for this, we allow for the possibility that $y_3<0$. This may look strange, but instead of considering our perimeter function as an expression representing the perimeter of a set that can be drawn, we can, for the moment, consider it merely as an algebraic expression that requires minimizing. Doing this results in $y_1=y_4$, $\frac{L_2}{2}=y_5=y_6=y_2+x_6=y_3+x_5$. After a little bit of algebra, we find the following for $y_1,...,y_6$ and for $x_5$ and $x_6$: 

\begin{align*}
y_1&=y_1, &y_2&=y_2,\\
x_6&=\frac{L_2}{2}-y_2, &x_5&=L_1-\frac{L_2}{2}+y_2,\\
y_3&=L_2-L_1-y_2, &y_4&=y_1,\\
y_5&=\frac{L_2}{2}, &y_6&=\frac{L_2}{2}.\\ 
\end{align*}

Summing $y_1$ to $y_6$ along with  $x_5$ and $x_6$ gives us: 
\begin{equation}
x_5+x_6+\sum_{j=1}^6y_j=2y_1+2L_2=2(\frac{\sqrt{3}L_{2}^2-4 \sqrt{3}L_{2}y_{2}+4V'+2 \sqrt{3}y_{2}^2}{2 \sqrt{3} L_{2}})+2L_{2}.
\end{equation} 

Here, $V'=\mu\left(A\right)+\frac{\sqrt{3}}{2}x_5x_6$. We can then rewrite this perimeter in terms of $y_2, L_1$, and $L_2$ to obtain: $2(\frac{\sqrt{3}L_{2}^2-4 \sqrt{3}L_{2}y_{2}+4(V+\frac{\sqrt{3}(\frac{L_{2}}{2}-y_{2})(L_{1}-\frac{L_{2}}{2}+y_{2})}{2})+2 \sqrt{3}y_{2}^2}{2 \sqrt{3} L_{2}})+2L_{2}$. Minimizing this equation we find that $y_2=\frac{L_2-L_1}{2}$, which then gives us $x_5=x_6=\frac{L_1}{2}$. This gives us two possible perimeters to minimize in terms of $L_{1}$ and $L_{2}$ depending on which configuration is volume 1 or $\alpha$, as given in ~\eqref{eq:finalEm}.

\begin{equation}
\label{eq:finalEm}
\begin{aligned}
&\rho_{1}(\alpha)=\frac{9L_{2}^{2}+8\sqrt{3}(1+\frac{\sqrt{3}}{2}(\frac{L_{1}}{2})^{2})}{6L_{2}}+\frac{9L_{1}^{2}+8\sqrt{3}\alpha}{6L_{1}}-L_{1},\\
&\rho_{2}(\alpha)=\frac{9L_{2}^{2}+8\sqrt{3}(\alpha+\frac{\sqrt{3}}{2}(\frac{L_{1}}{2})^{2})}{6L_{2}}+\frac{9L_{1}^{2}+8\sqrt{3}}{6L_{1}}-L_{1}.
\end{aligned}
\end{equation}

The minimum of this first function must be found in the same way as was shown in Appendix \ref{sec:TheBigPolynomial}. The second function has minimum $\frac{2\sqrt{10(1+\alpha)}}{3^{\frac{1}{4}}}$ for $\alpha\le \frac{2}{3}$ with $L_{2}=\frac{480+1080\alpha+\sqrt{3(19200-57600\alpha+432000\alpha^{2})}}{180(\sqrt[4]{3})\sqrt{10+10\alpha}}$, and $L_{1}=\frac{2\sqrt{\frac{2}{5}}\sqrt{1+\alpha}}{3^{\frac{1}{4}}}$; for $\alpha$ greater than this value, we must again use the methods of Appendix \ref{sec:TheBigPolynomial}. The first minimum is always less than the second, and therefore it is this one that is the the minimum for the Embedded Hexagons case.

Now let us deal with \textbf{case 2}. Our analysis in the previous section already gives us the optimal properties of each shape. Namely the blue hexagon, as a function of $L_{1}$, will have the same optimal shape. Similarly the red curve on the right of Figure~\ref{fig:EmbeddedRectangle} has $x_{1}=x_{2},x_{3}=x_{4}=\frac{L_{1}}{2}$ but $x_{5}=x_{2}=0$ (as $L_{2}=L_{1}$ in this case). This yeilds an identical equation to Equation~\ref{eq:finalEm2} as the removed volume from the red curve has sidlength $L_{2}$ rather than $L_{1}$

\begin{multline}
\label{eq:finalEm2}
\rho_{1}(\alpha)=\frac{9L_{2}^{2}+8\sqrt{3}(1+\frac{\sqrt{3}}{2}(\frac{L_{2}}{2})^{2})}{6L_{2}}+\frac{9L_{1}^{2}+8\sqrt{3}\alpha}{6L_{1}}-L_{2}
\\
\rho_{2}(\alpha)=\frac{9L_{2}^{2}+8\sqrt{3}(\alpha+\frac{\sqrt{3}}{2}(\frac{L_{2}}{2})^{2})}{6L_{2}}+\frac{9L_{1}^{2}+8\sqrt{3}}{6L_{1}}-L_{2}
\end{multline}

Using cylindrical algebraic decomposition, we obtain that the minimizer occurs at all volumes when $L_{2}=L_{1}$.


\subsection{Final Solution}
\label{subsec:finalanal}
Note that up to $\alpha=\frac{1}{8}$ the embedded hexagon case is the minimizer. Therefore, we never see the kissing hexagons case when $L_{1}\neq L_{2}$. The transition from the embedded hexagons case to the kissing hexagons case occurs at another polynomial root with no closed form. However, we can use the methods previously mentioned to find what the polynomial is, and which root it is. If we estimate its value, we find that the phase transition occurs at approximately $\alpha_0\approx0.152$.

\section{Proof of Theorem \ref{thm:maintheorem}}
\label{sec:ProofOfMainTheorem}

In this section we prove Theorem \ref{thm:maintheorem}. The idea is that for fixed $\alpha$, and for any sequence of configurations in $\gamma_{\alpha}$, we can replace each element of the sequence with a configuration in $\mathfs{G}_{\alpha}$. By the previous section, we know that there is a specific element $\chi_{\alpha}\in\mathfs{G}_{\alpha}$ such that $\rho_{DB}\left(\chi_{\alpha}\right)$ is less than the double bubble perimeter of every other element of $\mathfs{G}_{\alpha}$. Thus, $\chi_{\alpha}\in\Gamma_{\alpha}$. \\

\textbf{Proof of Theorem \ref{thm:maintheorem}} Let $\alpha\in(0,1]$, and $\chi_i\in\gamma_{\alpha}$ be a sequence such that $\lim_{i\rightarrow\infty}\rho_{DB}\left(\chi_i\right)=\rho_{DB}(\alpha)$. By Lemmas \ref{lem:Two120DegreeAngles} to \ref{lemma:AllSidesContained}, we can replace each $\chi_i$ with some $\chi'_i$ with $\chi'_i\in\mathfs{F}_{\alpha}$. By Lemmas \ref{lemma:NoSixtyDegreeAngles} to \ref{lem:ThreeLineSegBndry}, as well as the discussion in Section \ref{sec:NoSixties}, we can replace each $\chi'_i$ with a $\chi''_i\in\mathfs{G}_{\alpha}$. It follows from Section \ref{sec:kkt} that there is some $\chi_{\alpha}\in\mathfs{G}_{\alpha}$ such that $\rho_{DB}\left(\chi_{\alpha}\right)\leq\rho_{DB}\left(\chi''_i\right)$ for all $i\in\mathbb{N}$. Therefore, we get:

$$\rho_{DB}\left(\alpha\right)\leq\rho_{DB}\left(\chi_{\alpha}\right)\leq\lim_{i\rightarrow\infty}\rho_{DB}\left(\chi_i\right)\leq\rho_{DB}\left(\alpha\right).$$

Thus, $\chi_{\alpha}\in\Gamma_{\alpha}$, which establishes Theorem \ref{thm:maintheorem} part I. The second part of Theorem \ref{thm:maintheorem} is proved by comparing the three possible minimizing double bubble perimeters that were found in Section \ref{sec:kkt}. Namely, in the case of Kissing Hexagons, the minimizing double bubble perimeter is given by $2\sqrt[4]{3}\left(\sqrt{2}+\sqrt{\alpha}\right)$ when $\alpha\in\left(0,\frac{1}{8}\right)$, and by $P_3\left(r_3\right)$ when $\alpha\in\left[\frac{1}{8},1\right]$. Here, $P_3=7\sqrt{\frac{3L^2+4\sqrt{3}}{21}}+7\sqrt{\frac{3L^2+4\sqrt{3}\alpha}{21}}-3L$, and $r_3$ is the root of the polynomial in Appendix \ref{sec:TheBigPolynomial}. On the other hand, for the Embedded Hexagons case, the minimizing double bubble perimeter is given by the expression in \eqref{eq:finalEm}, where the values of $L_1$ and $L_2$ are given directly afterwards. To find the global minimizing double bubble perimeter, we compare the lengths of these three perimeters for $\alpha\in\left(0,1\right]$. The final result is that there is a value $\alpha_0$, which is a root with no closed form of a polynomial of degree $8$, such that the minimizing configuration is given by the Embedded Hexagons case when $\alpha\leq\alpha_0$, and the minimizing configuration is given by the Kissing Hexagons case when $\alpha\geq\alpha_0$. We can estimate the value of $\alpha_0$ to be approximately $0.152$, which is greater than $\frac{1}{8}$. This means that $2\sqrt[4]{3}\left(\sqrt{2}+\sqrt{\alpha}\right)$ is never the global minimizing double bubble perimeter.

\appendix

\section{Comparison of $P_3$ and $P_4$}\label{sec:Comparison}

Recall from \eqref{eq:kktmult2} that 

\bae
\label{eq:kktmult3}
&P_3=7\sqrt{\frac{3L^{2}+4\sqrt{3}}{21}}+7\sqrt{\frac{3L^{2}+4\sqrt{3}\alpha}{21}}-3L,
\\
&P_4=7\sqrt{\frac{3L^{2}+4\sqrt{3}}{21}}-\sqrt{\frac{2L^{2}-4\sqrt{3}\alpha}{3}}+L.
\\
\eae

Therefore,

\begin{align*}
P_4-P_3&=7\sqrt{\frac{3L^{2}+4\sqrt{3}}{21}}-\sqrt{\frac{2L^{2}-4\sqrt{3}\alpha}{3}}+L-\left(7\sqrt{\frac{3L^{2}+4\sqrt{3}}{21}}+7\sqrt{\frac{3L^{2}+4\sqrt{3}\alpha}{21}}-3L\right)\\
&=4L+7\sqrt{\frac{3L^2+4\sqrt{3}\alpha}{21}}-\sqrt{\frac{2L^2-4\sqrt{3}\alpha}{3}}\\
&\Rightarrow 4L+7\sqrt{\frac{3L^2+4\sqrt{3}\alpha}{21}}=\sqrt{\frac{2L^2-4\sqrt{3}\alpha}{3}}\\
&\Rightarrow\left(4L+7\sqrt{\frac{3L^2+4\sqrt{3}\alpha}{21}}\right)^2=\frac{2L^2-4\sqrt{3}\alpha}{3}\\
&\Rightarrow 16L^2+56L\sqrt{\frac{3L^2+4\sqrt{3}\alpha}{21}}+\frac{3L^2+4\sqrt{3}\alpha}{21}=\frac{2L^2-4\sqrt{3}\alpha}{3}\\
&\Rightarrow 16L^2+\frac{3L^2+4\sqrt{3}\alpha}{21}-\frac{2L^2-4\sqrt{3}\alpha}{3}=-56L\sqrt{\frac{3L^2+4\sqrt{3}\alpha}{21}}\\
&\Rightarrow \left(16L^2+\frac{3L^2+4\sqrt{3}\alpha}{21}-\frac{2L^2-4\sqrt{3}\alpha}{3}\right)^2=3136L^2\frac{3L^2+4\sqrt{3}\alpha}{21}. 
\end{align*}

This is a fourth degree polynomial, and therefore the roots can be calculated. From this, it is easily seen that, on the domain of $P_4$, $P_3<P_4$. Similarly calculations are done for the rest of the comparisons. 

\section{Finding the Minimizer of $P_3$}\label{sec:TheBigPolynomial}

Here, we find the polynomial that has a root which gives the global minimizer of $P_3$. The root cannot be expressed in closed form. So, we merely give an order for the roots and find which one is the root we want. 

Recall that $P_3=7\sqrt{\frac{3L^2+4\sqrt{3}}{21}}+7\sqrt{\frac{3L^2+4\sqrt{3}\alpha}{21}}-3L$. Taking the derivative and simplifying gives us: $P_3'=\frac{L}{\sqrt{\frac{3L^2+4\sqrt{3}}{21}}}+\frac{L}{\sqrt{\frac{3L^2+4\sqrt{3}\alpha}{21}}}-3$. 

We equate the derivative of $P_3$ to zero and solve:

\begin{align*}
&\scaleobj{0.9}{\frac{L}{\sqrt{\frac{3L^2+4\sqrt{3}}{21}}}+\frac{L}{\sqrt{\frac{3L^2+4\sqrt{3}\alpha}{21}}}-3}\scaleobj{0.9}{=0}\\
&\scaleobj{0.9}{\Rightarrow L\sqrt{\frac{3L^2+4\sqrt{3}\alpha}{21}}+L\sqrt{\frac{3L^2+4\sqrt{3}}{21}}}\scaleobj{0.9}{=3\sqrt{\frac{(3L^2+4\sqrt{3}\alpha)(3L^2+4\sqrt{3})}{441}}.}\\
&\scaleobj{0.9}{\Rightarrow L^2\left(\frac{3L^2+4\sqrt{3}\alpha}{21}\right)+2L^2\sqrt{\frac{3L^2+4\sqrt{3}\alpha}{21}}\sqrt{\frac{3L^2+4\sqrt{3}}{21}}+L^2\left(\frac{3L^2+4\sqrt{3}}{21}\right)}\scaleobj{0.9}{=9\left(\frac{(3L^2+4\sqrt{3}\alpha)(3L^2+4\sqrt{3})}{441}\right)}\\
&\scaleobj{0.9}{\Rightarrow4L^4\left(\frac{9L^4+12\sqrt{3}L^2+12\sqrt{3}\alpha L^2+48\alpha}{441}\right)-\left(\frac{9L^4+12\sqrt{3}L^2+12\sqrt{3}\alpha L^2+48\alpha}{49}-\frac{6L^4+4\sqrt{3}\alpha L^2+4\sqrt{3} L^2}{21}\right)^2}\scaleobj{0.9}{=0.}
\end{align*}

The lefthand side of this final expression can be expanded to give a polynomial of degree $8$. If we order its roots in terms of the magnitude (from smallest to largest) of their real parts first and their complex parts second, we can find, by fixing $\alpha$ and using Newton's method, that the third root gives the global minimizer.

\bibliographystyle{amsplain}
\bibliography{Main.org.tug.bib}

\end{document}